\documentclass{amsbook}

\usepackage{amsmath,amsthm,amssymb,amscd,MnSymbol,enumerate,paralist,hyperref,graphicx}

\usepackage{color}
\definecolor{darkgreen}{cmyk}{1,0,1,.2}
\definecolor{m}{rgb}{1,0.1,1}

\newdimen\theight
\def\TeXref#1{%
             \leavevmode\vadjust{\setbox0=\hbox{{\tt
                     \quad\quad  {\small \textrm #1}}}%
             \theight=\ht0
             \advance\theight by \lineskip
             \kern -\theight \vbox to
             \theight{\rightline{\rlap{\box0}}%
             \vss}%
             }}%

\newtheorem{thm}{Theorem}[chapter]
\newtheorem{lemma}[thm]{Lemma}
\newtheorem{cor}[thm]{Corollary}
\newtheorem{prop}[thm]{Proposition}

\theoremstyle{definition}
\newtheorem{defn}[thm]{Definition}
\newtheorem{example}[thm]{Example}
\newtheorem{examples}[thm]{Examples}

\newtheorem{problem}{Problem}

\theoremstyle{remark}
\newtheorem{rem}[thm]{Remark}
\newtheorem{claim}{Claim}
\newtheorem{hyp}{Hypothesis}

\numberwithin{section}{chapter}
\numberwithin{equation}{chapter}

\newcommand{\A}{\mathbb{A}}
\newcommand{\Z}{\mathbb{Z}}
\newcommand{\C}{\mathbb{C}}
\newcommand{\Q}{\mathbb{Q}}
\newcommand{\E}{\mathbb{E}}
\newcommand{\R}{\mathbb{R}}
\newcommand{\N}{\mathbb{N}}

\renewcommand{\AA}{\mathcal{A}}
\newcommand{\BB}{\mathcal{B}}
\newcommand{\CC}{\mathcal{C}}
\newcommand{\DD}{\mathcal{D}}
\newcommand{\EE}{\mathcal{E}}
\newcommand{\FF}{\mathcal{F}}
\newcommand{\GG}{\mathcal{G}}
\newcommand{\HH}{\mathcal{H}}

\newcommand{\KK}{\mathcal{K}}
\newcommand{\LL}{\mathcal{L}}
\newcommand{\MM}{\mathcal{M}}
\newcommand{\NN}{\mathcal{N}}
\newcommand{\OO}{\mathcal{O}}

\newcommand{\TT}{\mathcal{T}}
\newcommand{\UU}{\mathcal{U}}
\newcommand{\VV}{\mathcal{V}}
\newcommand{\WW}{\mathcal{W}}

\newcommand{\ZZ}{\mathcal{Z}}

\newcommand{\bfa}{\mathbf{a}}
\newcommand{\bfb}{\mathbf{b}}
\newcommand{\bfe}{\mathbf{e}}
\newcommand{\bff}{\mathbf{f}}
\newcommand{\bfz}{\mathbf{z}}

\newcommand{\bfA}{\mathbf{A}}
\newcommand{\bfG}{\mathbf{G}}

\newcommand{\XF}{(X,\mathcal{F})}

\newcommand{\sm}{\smallsetminus} 
\renewcommand{\epsilon}{\varepsilon}
\newcommand{\ol}{\overline} 

\DeclareMathOperator{\asdim}{asdim}
\DeclareMathOperator{\tp}{Top}

\DeclareMathOperator{\Cl}{Cl}

\DeclareMathOperator{\im}{im}

\DeclareMathOperator{\vol}{vol} 
 
\DeclareMathOperator{\dom}{dom}

\DeclareMathOperator{\inj}{inj}
\DeclareMathOperator{\Iso}{Iso}

\DeclareMathOperator{\diam}{diam}
\DeclareMathOperator{\id}{id}
\DeclareMathOperator{\Int}{Int}

\DeclareMathOperator{\Pen}{Pen}

\DeclareMathOperator{\pr}{pr}
\DeclareMathOperator{\gr}{gr}

\makeindex

\begin{document}

\frontmatter

\title{Generic coarse geometry of leaves}


\author{Jes\'us A. \'Alvarez L\'opez}
\address{Departamento de Xeometr\'ia e Topolox\'ia\\
Facultade de Matem\'aticas\\
Campus Vida\\
Universidade de Santiago de Compostela\\
15782 Santiago de Compostela\\
Spain}
\curraddr{}
\email{jesus.alvarez@usc.es}
\thanks{}

\author{Alberto Candel}
\address{Department of Mathematics\\
California State University at Northridge\\
18111 Nordhoff Street\\
Northridge, CA 91330\\
U.S.A.}
\curraddr{}
\email{alberto.candel@csun.edu}
\thanks{Research of the authors supported by Spanish MICINN grants MTM2011-25656 and MTM2008-02640}

\date{December 8, 2017}




\maketitle

\tableofcontents

\begin{abstract}
  A  compact Polish foliated space is considered. Part of this work studies coarsely quasi-isometric invariants of leaves in some residual saturated subset when the foliated space is transitive. In fact, we also use ``equi-'' versions of this kind of invariants, which means that the definition is satisfied with the same constants by some given set of leaves. For instance, the following properties are proved.
  
  Either all dense leaves without
  holonomy are equi-coarsely quasi-isometric to each other, or else
  there exist residually many dense leaves without holonomy such 
  that each of them is coarsely quasi-isometric to meagerly many leaves. 
  Assuming that the foliated space is minimal, the first of the above alternatives holds if and if the
  leaves without holonomy satisfy a condition called coarse
  quasi-symmetry.
  
  A similar dichotomy holds for the growth type of the leaves, as
  well as an analogous characterization of the first alternative in
  the minimal case, involving a property called growth symmetry.  Moreover some
  classes of growth are shared, either by residually many
  leaves, or by meagerly many leaves.
  
  If some leaf without holonomy is amenable, then
  all dense leaves without holonomy are equi-amenable, and, in the
  minimal case, they satisfy a property called amenable symmetry.
  
  Residually many leaves have the same asymptotic dimension. 
  
  If the foliated space is minimal, then any pair of nonempty open sets in the Higson coronas of the leaves with holonomy contain homeomorphic nonempty open subsets. 
  
  Another part studies limit sets of leaves at points in the coronas of their compactifications, defined like their usual limit sets at their ends. These sets are nonempty and compact, but they may not be saturated. The following properties are shown. 
  
  The limit sets are saturated for compactifications less or equal than the Higson compactification of the leaves. This establishes a relation between the coarse geometry of the leaves and the structure of closed saturated sets.
  
  For any given leaf, its limit set at every point in its Higson corona is the whole space if and only if the foliated space is minimal.
  
  For some dense open subset of points in the Higson corona of any leaf, the corresponding limit sets are minimal sets. 
\end{abstract}

\tableofcontents

\mainmatter

\chapter{Introduction}\label{c: intro}
The work presented in this monograph is about the metric and
topological properties of Riemannian manifolds that arise as leaves of
compact foliated spaces, and more generally about the metric and
topological porperties of orbits of pseudogroups compactly generated
pseudogroups of homeomorphisms of topological spaces.

Since the inception of foliation theory as a proper mathematical
subject of research, with Reeb and Ehreshman and Haefliger, the
problem of which manifolds can be leaves has been one of fundamental
interest and motive of research, even if not explicitly stated as
such.  In its purest essence, a foliation of a manifold is a
decomposition of the manifold into equidimensional submanifolds, the
leaves, and it is a natural question to understand how they are
assembled, and how the nature of the ambient space affects their
structure.  In fact, it could, in retrospect, be said that the problem
appeared even before the first foliation of the three dimensional
sphere was explicitly known: because of algebraic topological reasons,
a two dimensional leaf of a foliation of \(S^3\) must be a torus. Reeb
constructed the first foliation of \(S^3\), and what it has since
called a Reeb component: a foliated solid torus with one compact leaf
and a bundle of planes each having linear, rather than quadratic, area
growth (later Novikov proved that any foliation of \(S^3\) must
contain at least one Reeb component).

The problems thus suggested by the structure of the Reeb foliation,
and by the structure of the topology of flows (Poincar\'e-Bendixon
theory, the Denjoy flow and the Cherry flow of the torus) took a
definite position in mathematical research in the early 1970's with
the work of many on many topics, for example: growth and invariant
measures (Hirsch, Thurston, Plante, Mossu-Pelletier, Anosov,
Ruelle-Sullivan), realizability of manifolds as leaves (Sullivan,
Sondow, Cantwell-Conlon, Raymond), non-leaves (Ghys, Inaba et al.),
limit sets of leaves (Duminy). Sullivan posed the explicit problem in
\cite{Sullivan1975}: Problem 8. What do 2-dimensional leaves in
\(S^3\) look like?

It is not our intention to give a complete historical account of these
developments because they are well surveyed in the paper of
Hurder~\cite{Hurder1994}, which also introduced to the
study of foliation theory many of the concepts studied in the present
work.

We would also like to make a explitic mention of the paper of
D.~Cass~\cite{Cass1985} who gave the first published results relating
the recurrence properties of leaves of foliations with their
quasi-isometry types, and who quoted an unpublished result of Gromov
that we have developed in~\cite{AlvarezBarralCandel2016}.



Paraphrasing a referee, the work to be presented can be considered as
the complete development of the ideas behind Cass work, the
continuation of our work on the descriptive set theory aspects of
foliations~\cite{AlvarezCandel-Non-reduction} and \cite{AlvarezCandel-turbulent}, and
the further applications of these ideas to the properties of the
coarse geometry of open complete manifolds which arise as leaves of
foliations of compact manifolds.

  We consider (compact) Polish foliated spaces.  The original question
  that motivated our work, here and elsewhere, was not the original
  Problem 8 of Sullivan cited above: what do the leaves of a foliation
  look like?, but rather: when do all leaves of a foliated space look
  alike? (the quantifier ``all leaves'' is taken in the relaxed sense
  of category theory or measure theory). By ``look alike'' we not only
  mean the strongest possibility of ``being quasi-isometric,'' but
  also that possibility of sharing a common quasi-isometry invariant,
  like having the same growth type, or the same end-point
  compactification, for example.  In fact, we also use ``equi-''
  versions of this kind of invariants, which means that the definition
  is satisfied with the same constants by some given set of
  leaves. For instance, the following properties are proved.

  We prove the following complete characterization of the when and how
  all leaves ``look alike.''  Either all dense leaves without holonomy
  are equi-coarsely quasi-isometric to each other, or else there exist
  residually many dense leaves without holonomy such that each of them
  is coarsely quasi-isometric to meagerly many leaves.  Assuming that
  the foliated space is minimal, the first of the above alternatives
  holds if and if the leaves without holonomy satisfy a condition
  called coarse quasi-symmetry.
  
  We also show that a similar dichotomy holds for the growth type of
  the leaves, as well as an analogous characterization of the first
  alternative in the minimal case, involving a property called growth
  symmetry.  Moreover some classes of growth are shared, either by
  residually many leaves, or by meagerly many leaves.
  
  Other quasi-isometry invariants that we study are amenability: if
  some leaf without holonomy is amenable, then all dense leaves
  without holonomy are equi-amenable (in the minimal case, they
  satisfy a property called amenable symmetry), and asymptotic
  dimension: residually many leaves have the same asymptotic
  dimension.

Limit sets of leaves of foliated spaces have always been a fundamental
topic of research (originating with the Poincar\'e-Bendixon theory
for planar differential equations). The limits sets that we study here
arise from points in the coronas of compactifications of
leaves. Tradicionally the compactification studied has been the
end-point compactification, but from the point of view of
quasi-isometry, the Higson compactification is more relevant (as we
have shown elsewhere, this compactification characterizes
quasi-isometric spaces).

First we show that, in certain sense, all leaves look alike at
infinity: if the foliated space is minimal, then any pair of
nonempty open sets in the Higson coronas of the leaves with holonomy
contain homeomorphic nonempty open subsets.
  
After that we study limit sets of leaves at points in the coronas of
their compactifications, defined like their usual limit sets at their
ends. These sets are nonempty and compact, but they may not be
saturated (union of leaves). The following properties are shown: (a)
The limit sets are saturated for compactifications less or equal than
the Higson compactification of the leaves (this establishes a relation
between the coarse geometry of the leaves and the structure of closed
saturated sets); (b) for any given leaf, its limit set at every point
in its Higson corona is the whole space if and only if the foliated
space is minimal; and (c) for some dense open subset of points in the
Higson corona of any leaf, the corresponding limit sets are minimal
sets.

To be more precise, let $\XF$ be a compact Polish foliated space; i.e., this foliated
space is compact, Hausdorff and second countable. A regular foliated
atlas, $\UU$, for \(\XF\) induces a coarse metric, $d_\UU$, on each
leaf $L$ \cite{Hurder1994}: for \(x,y\in L\), $d_\UU(x,y)$ is the
minimum number of $\UU$-plaques whose union is connected and contains
$x$ and $y$. A metric $d_\UU^*$ on $L$ can be obtained by modifying
$d_\UU$ on the diagonal of $L\times L$, where $d_\UU^*$ is declared to
be zero. By the compactness of $X$ and the regularity of $\UU$, it
follows that the coarse quasi-isometry type of the leaves with
$d^*_\UU$ (in the sense of Gromov \cite{Gromov1993}) is independent of
the choice $\UU$. In this way, the leaves are equipped with a coarse
quasi-isometry type (of metrics) determined by $\XF$. The leaves of
$\XF$ belong to a class of metric spaces for which the coarse
quasi-isometry type equals the coarse type \cite{Roe1996},
\cite{Roe2003}. Thus terms like ``coarse
equivalence'' and ``coarse type'' could be used as well in our
statements. Moreover Roe \cite{Roe2003} extended to coarse spaces many
of the quasi-isometric invariants and properties we use (bounded
geometry, growth, amenability, Higson corona and asymptotic
dimension).

Let $\XF$ be leafwise differentiable of class $C^3$ (\(C^k\)
smoothness is defined in \cite{CandelConlon2000-I}), and $g$ be a
$C^2$ leafwise Riemannian metric on $X$. Then the coarse
quasi-isometry type of the leaves is also represented by the
Riemannian distance induced on them by $g$. The restrictions of
$g$ also define a \emph{differentable quasi-isometry type} of the
leaves, which depends only on $\XF$.

The determination of the coarse quasi-isometry type of the leaves of a
foliated space obtained in the present work is in fact slightly
stronger than stated above. A regular foliated atlas $\UU$ for
\(\XF\) determines a set of metric spaces \(\{(L, d^*_\UU)\}\),
where \(\{L\}\) is the set of leaves of \(\XF\).  If \(\UU\)
and $\VV$ are regular foliated atlases for \(\XF\), then the
set of metric spaces $\{(L,d^*_\UU)\}$ is \emph{equi-coarsely
  quasi-isometric} to the set of metric spaces $\{(L,d^*_\VV)\}$ in
the sense that all the coarse quasi-isometries between each pair of
metric spaces \((L,d^*_\UU)\) and \((L, d^*_\VV)\) share common
distortion constants (a coarsely quasi-isometric version of
equicontinuity).  This kind of terminology will be used with several
concepts of metric spaces, with a similar meaning, when their
definition involves constants. 

The first goal of this work is to study the coarse quasi-isometry
properties of the \emph{generic} leaf of $\XF$, that is, to study
what quasi-isometry properties are generic in the topological sense
that there is a residual saturated subset of the foliated space all
whose leaves share that quasi-isometry property. Thus we will focus on
the generic coarse properties of the leaves. For the existence of
interesting generic properties of the leaves, it will be often
required that the foliated space be transitive (some leaf is dense) or
even minimal (all leaves are dense); these conditions, transitivity
and minimality, are the topological counterparts of ergodicity.
For example, the closure of any leaf of $\XF$ is a transitive foliated subspace, and any minimal subset of $\XF$ is a minimal foliated space. The condition on $X$ to be compact is also useful because it determines the quasi-isometric types of the leaves.

This part may be thus be placed in the area of Ghys~\cite{Ghys1995},
and Cantwell and Conlon~\cite{CantwellConlon1998}, who have studied
the topology of the generic leaf, where ``generic'' was used in a
measure theoretic sense by Ghys, and in a topological sense by
Cantwell and Conlon. Indeed, it was Ghys who raised
the problem of studying the quasi-isometry type of the generic
leaf.

Recall that a set in a topological space is said to be residual if it contains a countable intersection of open dense subsets. The complements of residual sets are called meager. The Borel sets (respectively, Baire sets) are the elements of the the smallest \(\sigma\)-algebra containing all open sets (respectively, all open sets and all meager sets). Being ``generic'' in the topological sense refers to a property that holds on a residual set.

For a Polish foliated space \(\XF\), the following notation will be
standing for the rest of this work. The leaf of \(\XF\) that contains
\(x\in X\) is denoted by \(L_x\). The subset of \(X\) consisting of
the union of all leaves that have no holonomy is denoted by \(X_0\),
and the subset of \(X\) consisting of all the leaves that are dense
and have no holonomy is denoted by \(X_{0,\textrm{d}}\).

Hector~\cite{Hector1977a} and Epstein-Millet-Tischler~\cite{EpsteinMillettTischler1977} have proved that that \(X_{0}\)
is a residual subset of $X$, and that, if \(\XF\) is transitive, then
\(X_{0, \textrm{d}}\) is also residual in \(X\). This suggests that, when \(\XF\) is transitive, residually many leaves might have other properties, besides having no holonomy and being dense.  Concerning quasi-isometric invariants, we often find out that there is a dichotomy: either residually many leaves share a common value of the invariant, or else every possible value of the invariant is shared by meagerly many leaves. Of course, only the first of these alternatives can happen if the invariant has a countable number of possible values. We also consider relations in $X$ defined by being in leaves with a common value of an invariant, and show that these relations are Borel in $X_0\times X_0$ and Baire in $X\times X$. 

In most of the results of this work, the generic leaf is a leaf in \(X_0\) or in \(X_{0, \textrm{d}}\). The reason is that an essential tool of the proof is the local Reeb stability theorem, which involves the holonomy of the leaves.

The main theorems are stated presently. 

\begin{thm}\label{t: coarsely q.i. leaves 1}
  Let \(\XF\) be a compact Polish foliated space.  The equivalence
  relation ``$x\sim y$ if and only if the leaf $L_x$ is coarsely
  quasi-isometric to the leaf $L_y$'' has a Borel relation set in
  $X_0\times X_0$, and has a Baire relation set in $X\times X$; in
  particular, it has Borel equivalence classes in $X_0$, and Baire
  equivalence classes in $X$.
\end{thm}

\begin{thm}\label{t: coarsely q.i. leaves 2}
  Let \(\XF\) be a transitive compact Polish foliated space. Then the
  following dichotomy holds:
  \begin{enumerate}[{\rm(}i\/{\rm)}]
		
  \item\label{i: all leaves in X_0,d are equi-coarsely quasi-isometric  to each other} Either all leaves in $X_{0,\text{\rm d}}$ are
    equi-coarsely quasi-isometric to each other; or else
			
  \item\label{i: there are uncountably many coarse quasi-isometry
      types of leaves} every leaf in $X$ is coarsely
    quasi-isometric to meagerly many leaves; in particular, in this
    case, there are uncountably many coarse quasi-isometry types of
    leaves in $X_{0,\text{\rm d}}$.
			
  \end{enumerate}
\end{thm}

In Theorem~\ref{t: coarsely q.i. leaves 2}, the alternative~\eqref{i:
  all leaves in X_0,d are equi-coarsely quasi-isometric to each other}
means that there  exists a generic quasi-isometry type of leaves. This
holds for instance in the case of equicontinuous foliated spaces under
some mild conditions \cite{AlvarezCandel2009}.

The next theorem characterizes the alternative~\eqref{i: all leaves in
  X_0,d are equi-coarsely quasi-isometric to each other} of
Theorem~\ref{t: coarsely q.i. leaves 2} by using a property of metric
spaces called \emph{coarse quasi-symmetry}. A metric space is
\emph{coarsely quasi-symmetric} if it admits a set of equi-coarsely
quasi-isometric transformations that is \emph{transitive} in the sense
that, for every pair of points of the metric space, there is a coarse
quasi-isometry in that set mapping one point onto the other.  (Equi-)
coarse quasi-symmetry is invariant by (equi-) coarse quasi isometries.

\begin{thm}\label{t: coarsely q.i. leaves 3}
The following properties hold for a compact Polish foliated space \(\XF\):
  \begin{enumerate}[{\rm(}i\/{\rm)}]
  
  \item\label{i: if one leaf is coarsely quasi-symmetric, then the
      alternative (i) holds} Suppose that $\XF$ is transitive. If one leaf
    is coarsely quasi-symmetric, then the alternative~{\rm(}\ref{i:
      all leaves in X_0,d are equi-coarsely quasi-isometric to each
      other}\/{\rm)} of Theorem~\ref{t: coarsely q.i. leaves 2} holds.
  
  \item\label{i: if the alternative (i) holds, then all leaves in X_0
      are equi-coarsely quasi-symmetric} Suppose that $\XF$ is
    minimal. If the alternative~{\rm(}\ref{i: all leaves in X_0,d are
      equi-coarsely quasi-isometric to each other}\/{\rm)} of
    Theorem~\ref{t: coarsely q.i. leaves 2} holds, then all leaves in
    $X_0$ are equi-coarsely quasi-symmetric.

  \end{enumerate}
\end{thm}

Roe~\cite[Proposition~2.25]{Roe1993} notes that the
number of ends is a coarse invariant for some class of metric spaces,
which includes the leaves of $X$. Furthermore Bridson and Haefliger 
have proved that the end space is a coarsely quasi-isometric invariant 
of proper geodesic spaces \cite[Proposition~8.29]{BridsonHaefliger1999}. 
Indeed the end space is a coarse invariant of each leaf; 
this is proved by introducing the coarse end
space of any metric space, and showing that it equals the usual end
space under certain conditions.

The space of ends of the generic leaf plays an important role in the
indicated studies of its topology; in particular,
a slight simplification of \cite[Theorem~A]{CantwellConlon1998} states
that, if $\XF$ is minimal, then residually many leaves without
holonomy have zero, one, two or a Cantor space of ends,
simultaneously.  Since the end space is a coarse invariant, we
directly get the following slight improvement of that result when
alternative~\eqref{i: all leaves in X_0,d are equi-coarsely
  quasi-isometric to each other} of Theorem~\ref{t: coarsely
  q.i. leaves 2} holds.

\begin{cor}\label{c: end space of leaves in X_0}
 Suppose that $\XF$ is minimal and satisfies
  the alternative~{\rm(}\ref{i: all leaves in X_0,d are equi-coarsely
    quasi-isometric to each other}\/{\rm)} of Theorem~\ref{t: coarsely
    q.i. leaves 2}. Then all leaves in $X_0$ have zero, one, two or a
  Cantor space of ends, simultaneously.
\end{cor}

Another proof of Corollary~\ref{c: end space of leaves in X_0} is also given,
showing that, under certain mild conditions, any coarsely quasi-symmetric metric space has zero, one, two or a
Cantor space of coarse ends (Theorem~\ref{t: coarse ends}).

Corollary~\ref{c: end space of leaves in X_0} is useful to give
examples of minimal compact Polish foliated spaces satisfying the
alternative~\eqref{i: there are uncountably many coarse quasi-isometry
  types of leaves} of Theorem~\ref{t: coarsely q.i. leaves 2}. For
instance, this alternative is satisfied by the foliated space of
Ghys-Kenyon \cite{Ghys1999} because it has a leaf without holonomy
with four ends. Lozano \cite{LozanoRojo2008} gives variations of the
example of Ghys-Kenyon, producing minimal foliated spaces with leaves
without holonomy that have any finite number $\ge3$ of ends; they also
must satisfy alternative~\eqref{i: there are uncountably many coarse
  quasi-isometry types of leaves} of Theorem~\ref{t: coarsely
  q.i. leaves 2} by Corollary~\ref{c: end space of leaves in X_0}.

Blanc \cite[Th\`eoreme~1]{Blanc2003} has shown that, if $\FF$ is
minimal and residually many leaves have two ends, then all leaves
without holonomy are coarsely quasi-isometric to $\Z$, and the leaves
with holonomy are coarsely quasi-isometric to $\N$ and have holonomy
group $\Z/2\Z$. This completely describes the coarse quasi-isometry
type of the leaves in the case with two ends of Corollary~\ref{c: end
  space of leaves in X_0}. In particular, it follows that $\FF$
satisfies alternative~\eqref{i: there are uncountably many coarse
  quasi-isometry types of leaves} of Theorem~\ref{t: coarsely
  q.i. leaves 2} if it is minimal and has a leaf with two ends that is
not coarsely quasi-isometric to $\Z$. For instance, the Ghys-Kenyon
foliated space also contains leaves with two ends that are not
quasi-isometric to $\Z$ \cite{AlcaldeLozanoMacho2009}.

Observe that Theorems~\ref{t: coarsely q.i. leaves 1}--\ref{t:
  coarsely q.i. leaves 3} and Corollary~\ref{c: end space of leaves in
  X_0} cannot be directly extended to the leaves with finite holonomy
groups by the theorem of Blanc quoted above. In that case, the leaves with holonomy are not coarsely quasi-isometric to the leaves with holonomy, and they are not coarsely quasi-symmetric. However they have a common coarse quasi-isometric type, which could be true with more generality. 

One of the classic coarsely quasi-isometric invariants of a leaf $L$
of a compact Polish foliated space \(\XF\) is its \emph{growth
  type}. It can be defined as the growth type of the function
$r\mapsto v_\UU(x,r)$ ($r>0$), for any regular foliated atlas $\UU$
and $x\in L$, where $v_\UU(x,r)$ is the number of plaques of $\UU$
that meet the $d^*_\UU$-ball of center $x$ and radius $r$---this
growth type is independent of $\UU$ and $x$. If the functions
$r\mapsto v_\UU(x,r)$ ($x\in L$) have equi-equivalent growth, then $L$
is called \emph{growth symmetric}. If $\XF$ is $C^3$ and $g$ is a
$C^2$ leafwise Riemannian metric on $X$, then the growth type of $L$
equals the usual growth type of the connected Riemannian manifold $L$
with the restriction of $g$.

Block and Weinberger~\cite{BlockWeinberger1992} introduced the class
of metric spaces of \emph{coarse bounded geometry}, and  defined
the \emph{growth type} for any metric space in this class; it agrees
with the above definition for leaves of $\FF$, which are included in
this class. Growth symmetry can be also defined with this
generality. (Equi-) growth type and (equi-) growth symmetry are
invariant by (equi-) coarse quasi-isometries.

\begin{thm}\label{t:growth 1}
  Let \(\XF\) be a compact Polish foliated space. The equivalence relation
  ``$x\sim y$ if and only if $L_x$ has the same growth type as $L_y$''
  has a Borel relation set in $X_0\times X_0$, and has a Baire
  relation set in $X\times X$; in particular, it has Borel equivalence
  classes in $X_0$, and Baire equivalence classes in $X$.
\end{thm}

\begin{thm}\label{t:growth 2}
  Let \(\XF\) be a transitive compact Polish foliated space. Then the following
  dichotomy holds:
  \begin{enumerate}[{\rm(}i\/{\rm)}]
		
  \item\label{i: all leaves in X_0,d have equi-equivalent growth}
    Either all leaves in $X_{0,\text{\rm d}}$ have equi-equivalent
    growth; or else
			
  \item\label{i: there are uncountably many growth types of leaves in
      X_0,d} the growth type of each leaf in $X$ is
    comparable with the growth type of meagerly many leaves; in
    particular, in this second case, there are uncountably many growth
    types of leaves in $X_{0,\text{\rm d}}$.
			
  \end{enumerate}
\end{thm}

The alternative~\eqref{i: all leaves in X_0,d have equi-equivalent
  growth} of Theorem~\ref{t:growth 2} is the case where there is a
generic growth type.

\begin{thm}\label{t:growth 3}
  Let \((X,\FF)\) be a compact Polish foliated space. The following
  properties hold:
  \begin{enumerate}[{\rm(}i\/{\rm)}]
  
  \item\label{i: if there is a growth symmetric leaf in X_0,d, then
      the alternative (i) holds} Suppose that $\XF$ is
    transitive. If there is a growth symmetric leaf in $X_{0,\text{\rm
        d}}$, then the alternative~{\rm(}\ref{i: all leaves in X_0,d
      have equi-equivalent growth}\/{\rm)} of Theorem~\ref{t:growth 2}
    holds.

  \item\label{i: if alternative (i) holds, then all leaves in X_0 are
      equi-growth symmetric} Suppose that $\XF$ is minimal. If the
    alternative~{\rm(}\ref{i: all leaves in X_0,d have equi-equivalent
      growth}\/{\rm)} of Theorem~\ref{t:growth 2} holds, then all leaves
    in $X_0$ are equi-growth symmetric.

  \end{enumerate}
\end{thm}

\begin{thm}\label{t:liminf ...}
  Suppose that \(\XF\) is a transitive compact Polish foliated
  space. Then there are $a_1,a_3\in[1,\infty]$, $a_2,a_4\in[0,\infty)$
  and $p\ge1$ such that
 		\begin{alignat*}{2}
    			\limsup_{r\to\infty}\frac{\log v(x,r)}{\log r}&=a_1\;,&\qquad
    			a_2\le\liminf_{r\to\infty}\frac{\log v(x,r)}{r}&\le pa_2\;,\\
			\liminf_{r\to\infty}\frac{\log v(x,r)}{\log r}&=a_3\;,&\qquad
    			\limsup_{r\to\infty}\frac{\log v(x,r)}{r}&=a_4
		\end{alignat*}
	for residually many points $x$ in $X$. Moreover
		\[
			\liminf_{r\to\infty}\frac{\log v(x,r)}{\log r}\ge a_3\;,\qquad
    			\limsup_{r\to\infty}\frac{\log v(x,r)}{r}\le a_4
		\]
	for all $x\in X_{0,\text{\rm d}}$.
\end{thm}

The conditions used in the following result are defined by requiring
that the superior and inferior limits used in Theorem~\ref{t:liminf
  ...} be $<\infty$, $>0$ or $\le0$ (see Sections~\ref{s: growth of non-decreasing functions}
and~\ref{s: growth of met sps}). All of those terms are standard except
pseudo-quasi-polynomial, which is introduced here. It is well known
that the growth type of all non-compact leaves of $\FF$ is at least
linear and at most exponential.

\begin{cor}\label{c: polynomial exponential growth}
  	Let \(\XF\) be a transitive compact Polish foliated space. Then the following sets are either meager or residual in $X$:
  		\begin{enumerate}[{\rm(}i\/{\rm)}]

  			\item\label{i: union of leaves in X_0 with polynomial growth} the union of leaves in $X_0$ with polynomial growth;

  			\item\label{i: union of leaves in X_0 with exponential growth} the union of leaves in $X_0$ with exponential growth;
    
  			\item\label{i: union of leaves in X_0 with quasi-polynomial growth} the union of leaves in $X_0$ with quasi-polynomial growth;

  			\item\label{i: union of leaves in X_0 with quasi-exponential growth} the union of leaves in $X_0$ with quasi-exponential growth; and

  			\item\label{i: union of leaves in X_0 with pseudo-quasi-polynomial growth}  the union of leaves in $X_0$ with pseudo-quasi-polynomial growth.

  		\end{enumerate}
  	Moreover, 
		\begin{enumerate}[{\rm(}a\/{\rm)}]
		
			\item if the set~{\rm(}\ref{i: union of leaves in X_0 with quasi-polynomial growth}\/{\rm)} is residual in $X$, then it contains $X_{0,\text{\rm d}}$; and,
			
			\item if one of the sets~{\rm(}\ref{i: union of leaves in X_0 with quasi-exponential growth}\/{\rm)} or~{\rm(}\ref{i: union of leaves in X_0 with pseudo-quasi-polynomial growth}\/{\rm)} is meager in $X$, then it does not meet $X_{0,\text{\rm d}}$.
			
		\end{enumerate}
\end{cor}

Hector \cite{Hector1977b} constructed a remarkable example of a
$C^\infty$ foliation $\FF$ of codimension one on $X=M\times S^1$
(where $M$ is the closed oriented surface of genus two), which is
transverse to the factor \(S^1\), satisfying the following properties:
	\begin{itemize}
		
        \item For each integer \(n\ge 0\), there is exactly one proper
          leaf of exactly polynomial growth of degree \(n\). Every
          other leaf has non-polynomial growth and is dense.
		
        \item Each non-polynomial growth class of leaves has the
          cardinality of the continuum.
		
        \item It has one growth class of leaves with exponential
          growth.
		
        \item The set of growth classes which are non-polynomial but
          quasi-polynomial (respectively, non-quasi-polynomial but
          non-exponential) has the cardinality of the continuum.

	\end{itemize}
It will be shown that this example satisfies the alternative~\eqref{i: there are uncountably many growth types of leaves in X_0,d} of Theorem~\ref{t:growth 2}. In fact, the following stronger property will be proved. In Hector's example, the relation of being in leaves with the same growth type is generically ergodic with respect to the isomorphism relation on countable models (Theorem~\ref{t: Hector's foln is generically ergodic}). The proof uses our work \cite{AlvarezCandel-turbulent} on turbulent relations. This suggests that, in the alternatives Theorem~\ref{t: coarsely q.i. leaves 2}-\eqref{i: there are uncountably many coarse quasi-isometry types of leaves} and Theorem~\ref{t:growth 2}-\eqref{i: there are uncountably many growth types of leaves in X_0,d}, it might be true that the relations involved are turbulent, and therefore generically ergodic with respect to the isomorphism relation on countable models (Problem~\ref{prob: classification by countable models}).
        
        Cantwell and Conlon \cite{CantwellConlon1982} have proved that
        there is a set $G$ of growth types, containing a continuum of
        distinct quasi-polynomial but non-polynomial types and a
        continuum of distinct non-exponential but non-quasipolynomial
        types, together with the exponential type, such that, for any
        closed $C^\infty$ $3$-manifold $M$ and $\gamma\in G$, there is
        a $C^\infty$ foliation $\FF$ on $M$ with a local minimal set
        $U$ of locally dense type so that:
	\begin{itemize}
		
        \item $\ol U\sm U$ is a finite union of totally proper leaves;
		
        \item the leaves of $\FF$ in $U$ have trivial holonomy and are
          diffeomorphic one another; and
		
        \item each leaf of $\FF|_U$ has growth type $\gamma$.

	\end{itemize}
        Then the foliated space $X=\ol U$, with the restriction of
        $\FF$, satisfies the alternative~\eqref{i: all leaves in X_0,d
          have equi-equivalent growth} of Theorem~\ref{t:growth 2}
        because $X_0=X_{0,d}\supset U$ and all leaves have the same
        growth type on $U$. It also satisfies the 
        alternative~\eqref{i: all leaves in X_0,d 
        are equi-coarsely quasi-isometric to each other} 
        of Theorem~\ref{t: coarsely q.i. leaves 2} (Section~\ref{ss: growth at finite level}).
 
Let $\UU$ be a regular foliated atlas for $\XF$, and let $S$ be a set
of $\UU$-plaques in the same leaf $L$. The boundary $\partial S$ is
the set of $\UU$-plaques in $L$ that meet $\UU$-plaques in $S$ and
outside $S$. It is said that $L$ is \emph{$\UU$-F{\o}lner} when there
is a sequence of finite sets $S_n$ of $\UU$-plaques such that
$|S_n|/|\partial S_n|\to0$ as $n\to\infty$, where $|S_n|$ is the
cardinal of $S_n$. This condition is independent of the choice of
$\UU$, and $L$ is called \emph{amenable} when it is satisfied. If $\XF$
is $C^3$ and $g$ is a $C^2$ leafwise Riemannian metric on
$\XF$, then a leaf $L$ is amenable if and only if it is F{\o}lner as Riemannian
manifold with the restriction of $g$. \emph{Amenable symmetry} is
also introduced for a leaf $L$; roughly speaking, it means that the
F{\o}lner condition is satisfied uniformly close to any point of $L$,
and with the same rate of convergence to zero; the precise statement
of this definition is indeed very involved. We will also use a
property stronger than equi-amenable symmetry called \emph{joint
  amenable symmetry}.

The concept of amenability was extended to arbitrary metric spaces of
coarse bounded geometry by Block and
Weinberger~\cite{BlockWeinberger1992}. All of the above variants of
amenability can be also defined in this setting, becoming
(equi-) coarsely quasi-isometric invariants.

\begin{thm}\label{t: amenable}
  Let \(\XF\) be a compact Polish foliated space.  The following properties hold:
  \begin{enumerate}[{\rm(}i\/{\rm)}]
		
  \item\label{i: if some leaf in X_0 is amenable, then all leaves in
      X_0,d are equi-amenable} If $\XF$ is transitive and some leaf in
    $X_0$ is amenable, then all leaves in $X_{0,\text{\rm d}}$ are
    equi-amenable.
			
  \item\label{i: if some leaf in X_0 is amenable, then all leaves in
      X_0,d are jointly amenably symmetric} If $\XF$ is minimal and
    some leaf in $X_0$ is amenable, then all leaves in $X_0$ are
    jointly amenably symmetric.
			
  \end{enumerate}
\end{thm}

A foliated space is called amenable when it has an invariant mean on the holonomy pseudogroup, which is not equivalent to amenability (or F\o lner property) of the leaves \cite{Kaimanovich2001}, \cite{AlcaldeCuestaRechtman2011}. We remark that the concept of amenable foliated space is not used in Theorem~\ref{t: amenable}; we only use amenability of the leaves, and some ``uniform'' versions of amenability of the leaves. 

Another coarse invariant of a metric space $M$ is its \emph{Higson
  corona} $\nu M$. When $M$ is proper, it is the corona of the
\emph{Higson compactification} $M^\nu$, which is defined by applying
the Gelfand-Naimark theorem to the bounded continuous functions on $M$
whose variation vanish at infinity. The Higson corona plays an
important role in coarse geometry \cite{Roe1993}, \cite{Roe2003};
indeed, a weak version of the germ of $\nu M$ in $M^\nu$ contains all
coarse information of $M$ (Proposition~\ref{p: Higson}; see also \cite{AlvarezCandel2011}). The \emph{semi
  weak homogeneity} used in the following theorem means that, given
any pair of nonempty open sets, there is a non-empty open subset of
one of them homeomorphic to some open subset of the other one---this
is weaker than \emph{weak homogeneity} \cite{Feldman1991}, which means
that, for all points $x$ and $y$, there is a pointed homeomorphism
$(U,x)\to(V,y)$ for some open neighborhoods, $U$ of $x$ and $V$ of
$y$.

\begin{thm}\label{t: semi weakly homogeneous, leaves}
  Let \(\XF\) be a compact Polish foliated space. If $\FF$ is minimal,
  then the space $\bigsqcup_L\nu L$, with $L$ running in the set of
  all leaves in $X_0$, is semi weakly homogeneous.
\end{thm}

A numerical coarsely quasi-isometric invariant of metric spaces is
their \emph{asymptotic dimension\/}, introduced by Gromov
\cite{Gromov1993}. It can be defined by using covers by uniformly
bounded open sets, in a way dual to the definition of Lebesgue
covering dimension (by coarsening the covers instead of refining
them). The relevance of this invariant of a metric space is
illustrated by a theorem of Dranishnikov, Keesling and Uspenkij
\cite{DranishnikovKeeslingUspenskij1998}, stating that
$\dim\nu M\le\asdim M$ for all proper metric space $M$, which is
complemented by another theorem of Dranishnikov
\cite{Dranishnikov2000}, stating that $\dim\nu M=\asdim M$ if
\(\asdim M<\infty\). Moreover Yu \cite{Yu1998} related the asymptotic
dimension to the Novikov conjecture.

\begin{thm}\label{t: asdim leaves} 
  Let \(\XF\) be a transitive compact Polish foliated space. Then
  residually many leaves have de same asymptotic dimension.
\end{thm}

The following theorem is a coarsely quasi-isometric version of the
``Proposition fondamentale'' of Ghys \cite[p.~402]{Ghys1995}, using
generic leaves in a topological sense.

\begin{thm}\label{t: Ghys' ``Proposition fondamentale''}
  Suppose that \(\XF\) is a minimal compact Polish foliated space, and let $B$ be
  a Baire subset of $X$. Then:
  \begin{enumerate}[{\rm(}i\/{\rm)}]
		
  \item\label{i: the saturation of B is meager} either the
    $\FF$-saturation of $B$ is meager; or else
			
  \item\label{i: the intersections of residually many leaves with B
      are equi-nets} the intersections of residually many
    leaves with $B$ are equi-nets in those leaves.
			
  \end{enumerate}
\end{thm}

Due to the relevance of Ghys' result in the description of the
topology of the generic leaf, we expect that this theorem will be
relevant to describe the differentiable quasi-isometry type of generic
leaves. We will give a counterexample showing that the measure
theoretic version of Theorem~\ref{t: Ghys' ``Proposition
  fondamentale''} is false. Thus there may be more difficulties to
study the measure theoretic generic differentiable quasi-isometric
type of leaves. However weaker measure theoretic versions of
Theorems~\ref{t: coarsely q.i. leaves 2},~\ref{t:growth 2}
and~\ref{t:liminf ...} will be proved, involving an ergodic harmonic
measure \cite{Garnett1983} so that $X\sm X_0$ has measure zero.

Another goal of this work is to study the \emph{limit sets}
$\lim_\bfe L$ of a leaf $L$ at a point $\bfe$ in the corona of any
compactification of $L$, which is a straightforward generalization of
the usual limit set of $L$, or its $\bfe$-limit for any end $\bfe$ of
$L$. For a general compactification of $L$, the corresponding limit
sets of $L$ are closed in $X$ and nonempty, but they may not be
$\FF$-saturated. However we will mainly consider the Higson
compactification $L^\nu$, or compactifications $\ol L\le L^\nu$; i.e.,
$\id_L$ has a continuous extension $L^\nu\to\ol L$. Then the following
theorem gives a bridge between the coarse geometry of the leaves and
the structure of closed saturated subsets.

\begin{thm}\label{t: lim e is FF-saturated}
  	Let \(\XF\) be a compact Polish foliated space. Let $\ol L$ be a compactification of a leaf $L$, with corona $\partial L$. If $\ol L\le L^\nu$, then $\lim_\bfe L$ is $\FF$-saturated for all $\bfe\in\partial L$.
\end{thm}

As defined by Cantwell and Conlon~\cite{CantwellConlon1998}, a leaf $L$ of $\FF$ is totally recurrent if $\lim_\bfe=X$ for all end $\bfe$ of $L$.  They showed that the set of totally recurrent leaves of $\FF$, if non-empty, is residual. Correspondingly, using the Higson corona, $L$ is said to be \emph{Higson recurrent} if $\lim_\bfe L=X$ for all $\bfe\in\nu L$. Higson recurrence behaves in the following manner.

\begin{thm}\label{t: Higson recurrent leaves}
  Let \(\XF\) be a compact Polish foliated space. A leaf is Higson
  recurrent if and only if $\FF$ is minimal.
\end{thm}

For each $\FF$-minimal set $Y$ and every leaf $L$, let
$\nu_YL=\{\,\bfe\in\nu L\mid\lim_\bfe L=Y\,\}$.

\begin{thm}\label{t: lim e is an FF-minimal set}
  Let \(\XF\) be a compact Polish foliated space. For any leaf $L$,
  the space $\bigcup_Y\Int_{\nu L}(\nu_YL)$, where $Y$ runs in the
  family of $\FF$-minimal sets, is dense in $\nu L$.
\end{thm}

The proofs of the main theorems of this work will be carried out in the
context of the holonomy pseudogroup of \(\XF\). For a regular
foliated atlas $\UU$ for $\XF$, the set $E$ of the transverse
components of the changes of coordinates generate a pseudogroup $\HH$
on a space $Z$, which is called a representative of the \emph{holonomy
  pseudogroup} of $\XF$. In turn, $E$ defines a metric $d_E$ on the
$\HH$-orbits, by setting $d_E(x,y)$ to be the smallest number of elements of
$E$ whose composition is defined at $x$ and maps $x$ to $y$. Then the
$\FF$-leaves with $d^*_\UU$ are coarsely quasi-isometric to the
$\HH$-orbits with $d_E$. In this way, our main theorems are easy
consequences of their versions for pseudogroups. Most of the work is
devoted to prove those pseudogroup versions, as well as to develop the
needed tools about metric spaces. 

In the case of Theorem~\ref{t: Ghys' ``Proposition fondamentale''}, an
alternative direct proof is also given because it is very short and
conceptually interesting. Also, this second proof is representative of
the type of direct proofs that could be given for other
results. However the pseudogroup versions of these theorems have their
own interest; for instance, they can be directly applied to orbits of
finitely generated group actions on compact Polish spaces.

The structure of this work is the following. Chapters~\ref{c: coarse q-i}--\ref{c: Higson}
recall the needed concepts and results about coarse quasi-isometric invarints,
and gives new concepts and results that will be required in the
subsequent chapters. Chapter~\ref{c: pseudogroups} contains
preliminaries about pseudogroups, including the pseudogroup version of
Theorem~\ref{t: Ghys' ``Proposition fondamentale''} and a
quasi-isometric version of the Reeb local stability for pseudogroups.
Chapter~\ref{c: orbits} contains proofs of the pseudogroup versions of
all other main theorems. In Chapter~\ref{c: leaves}, the needed
preliminaries on foliated spaces are recalled, and the main results
are obtained from their pseudogroup versions. It also includes a
section with the measure theoretic versions of Theorems~\ref{t:
  coarsely q.i. leaves 2},~\ref{t:growth 2} and~\ref{t:liminf ...},
and Corollary~\ref{c: polynomial exponential growth}, and a section
showing that the measure theoretic version of Theorem~\ref{t: Ghys'
  ``Proposition fondamentale''} fails. Chapter~\ref{c: examples} contains a section devoted to examples, and another section with open
problems.

\chapter{Coarse quasi-isometries}\label{c: coarse q-i}

In Chapters~\ref{c: coarse q-i}--\ref{c: Higson}, we recall basic concepts and results about coarse geometry on metric spaces, which will be needed in the study of leaves. Some new concepts and results are also given. To begin with, this chapter is devoted to the study of coarse quasi-isometries, and other related classes of maps. They define the primary relation we want to explore on a compact foliated space, being in coarsely quasi-isometric leaves.

\section{Notation, conventions and terminology}\label{s: notation}


The contents of this section applies to Chapters~\ref{c: coarse q-i}--\ref{c: Higson}.

Symbols $M$, $M'$ and $M''$ will denote metric spaces with attendant
metrics $d$, $d'$ and $d''$, respectively; $\{M_i\}_{i\in I}$ and
$\{M'_i\}$ will denote classes of metric spaces with the same index
 \(I=\{ i\}\). Unless otherwise stated, a subset of a metric space
becomes a metric space with the induced metric.

Let $r,s\ge0$, $x\in M$ and $S,T\subset M$. The open and closed balls
in $M$ of center $x$ and radius $r$, defined by $d_M(x,\cdot)<r$ and $d_M(x,\cdot)\le r$, are denoted by $B_M(x,r)$ and
$\ol{B}_M(x,r)$, respectively; in particular, $\ol{B}_M(x,0)=\{x\}$,
and $B_M(x,0)=\emptyset$. The \emph{penumbra}\footnote{This is
  slightly different from the definition of this concept given in
  \cite{Roe1996}.} \index{penumbra} of $S$ of \emph{radius} $r$ is the
set
$$
\Pen_M(S,r)=\bigcup_{y\in S}\ol{B}_M(y,r)\;.
$$
In particular, $\ol{B}_M(x,r)=\Pen_M(\{x\},r)$. The terms \emph{open/closed $r$-ball} and \emph{$r$-penumbra} will be also used
to indicate the radius $r$. Obviously,
\begin{equation}\label{Pen_M(S cap T,r)}
  \Pen_M(S\cap T,r)\subset\Pen_M(S,r)\cap\Pen_M(T,r)\;.
\end{equation}
Moreover, by the triangle inequality,
\begin{equation}\label{Pen_M(Pen_M(S,r),s)}
  \Pen_M(\Pen_M(S,r),s)\subset\Pen_M(S,r+s)
\end{equation}
for all $r,s>0$, and
	\begin{equation}\label{ol Pen_M(S,r) subset Pen_M(S,s)}
		r<s\Longrightarrow\Pen_M(\ol S,r)\subset\ol{\Pen_M(S,r)}\subset\Pen_M(S,s)\;,
	\end{equation}
where the first inclusion is an equality if $M$ is proper\footnote{Recall that $M$ is called \emph{proper} if its closed balls are compact}. The \emph{$r$-boundary}\footnote{This is also slightly different from
  the definition of this concept given in \cite{BlockWeinberger1992}.} \index{$r$-boundary}
of $S$ is the set
$$
\partial^M_rS=\Pen_M(S,r)\cap\Pen_M(M\sm S,r)\;;
$$
in particular, $\partial^M_0S=\emptyset$. The notation $B(x,r)$,
$\ol{B}(x,r)$, $\Pen(S,r)$ and $\partial_rS$ can be also used if it is
clear which metric space is being considered. The inclusion
\begin{equation}\label{partial_r Pen(S,s)}
  \partial_r\Pen(S,s)\subset\partial_{r+s}S
\end{equation}
can be proved as follows. For each $x\in\partial_r\Pen(S,s)$, there
are $y\in\Pen(S,s)$ and $z\in M\sm\Pen(S,s)\subset M\sm S$ such that
$d(x,y)\le r$ and $d(x,z)\le r$. Then there is $y_0\in S$ such that
$d(y,y_0)\le s$, obtaining $d(x,y_0)\le r+s$ by the triangle
inequality. Thus $x\in\partial_{r+s}S$.

\section{Coarse quasi-isometries}\label{s: coarse q-i}

A map $f:M\to M'$ is \emph{Lipschitz} \index{Lipschitz} if there is
some $C>0$ such that $d'(f(x),f(x'))\le C\,d(x,y)$ for all $x,y\in
M$. Such a $C$ will be called a \emph{Lipschitz distortion}
\index{distortion!Lipschitz} of $f$. The map $f$ is called
\emph{bi-Lipschitz} \index{bi-Lipschitz} when there is $C\ge1$ such
that
\[
\frac{1}{C}\,d(x,y)\le d'(f(x),f(x'))\le C\,d(x,y)
\]
for all $x,y\in M$. In this case, $C$ will be called a \emph{bi-Lipschitz distortion} \index{distortion!bi-Lipschitz} of $f$. The term $C$-(\emph{bi})\emph{-Lipschitz} may be also used for a (bi-)Lipschitz map
with (bi-)Lipschitz distortion $C$. A $1$-Lipschitz map is called \emph{non-expanding}. \index{non-expanding} A class of (bi-)Lipschitz maps is called \emph{equi-}(\emph{bi-})\emph{Lipschitz} \index{equi-(bi-)Lipschitz} when they have some common
(bi-)Lipschitz distortion. Two metrics $d_1$ and $d_2$ on a set $S$
are called \emph{Lipschitz equivalent} if the identity map
$(S,d_1)\to(S,d_2)$ is bi-Lipschitz. Let $\{S_i\}$ be a class of sets
each endowed with two metrics $d_{i,1}$ and $d_{i,2}$; the class
$\{d_{i,1}\}$ is \emph{equi-Lipschitz equivalent} to $\{d_{i,2}\}$
if, for all \(i\), the identity maps $(S_i,d_{i,1})\to(S_i,d_{i,2})$
are equi-bi-Lipschitz.

\begin{rem}\label{r: bi-Lipschitz}
  Any bi-Lipschitz map is injective. If $f:M\to M'$ and $f':M'\to M''$
  are (bi-)Lipschitz maps with respective (bi-)Lipschitz distortions
  $C$ and $C'$, then $f' f:M\to M''$ is a (bi-)Lipschitz map with
  (bi-)Lipschitz distortion $CC'$.
\end{rem}

A subset \(A\) of \( M\) is called a \emph{net} \index{net} in $M$ if there is
\(K\ge 0\) such that $\Pen_M(A,K)=M$. A subset $A$ of $M$ is said to be
\emph{separated} \index{separated} if there is $\delta>0$ such that $d(x,y)>\delta$ for
every $x\ne y$ in $A$. The terms \emph{$K$-net}\footnote{The
  definition of $K$-net is slightly different from the definition used
  in \cite{AlvarezCandel2011}. Our arguments become simpler in this
  way.}  and \emph{$\delta$-separated} will be also used. If \(\{
M_i\}\) is a class of metric spaces, a class of subsets \(\{
A_i\subset M_i\}\) is called an \emph{equi-net} \index{equi-net} if there is \(K\ge
0\) such that every \(A_i\) is a \(K\)-net in \(M_i\).

\begin{rem}\label{r: net}
  If $A$ is a $K$-net in $M$, then it is a $K$-net in any subset of
  $M$ that contains $A$. By the triangle inequality, if $A_1$ is a
  $K_1$-net in $M$, and $A_2$ is a $K_2$-net in $A_1$, then $A_2$ is a
  $(K_1+K_2)$-net in $M$. If $f:M\to M'$ is a bi-Lipschitz bijection
  with bi-Lipschitz distortion $C$, and $A$ is a $K$-net in $M$, then
  $f(A)$ is a $CK$-net in $M'$.
\end{rem}

\begin{lemma}[\'Alvarez-Candel {\cite[Lemma~2.1]{AlvarezCandel2011}}]\label{l: separated net}
  Let $K>0$ and $x_0\in M$. There is some $K$-separated $K$-net $A$ of
  $M$ so that $x_0\in A$.
\end{lemma}

\begin{rem}
  In \cite[Lemma~2.1]{AlvarezCandel2011}, it is not explicitly stated that
  $x_0\in A$. In that proof, $A$ is a maximal element of the
  family of $K$-separated subsets of $M$. But the proof works as well
  using the family of $K$-separated subsets of $M$ that contain $x_0$,
  obtaining Lemma~\ref{l: separated net}. In fact, it can be similarly
  proved that there is some $K$-separated $K$-net containing any given
  $K$-separated subset.
\end{rem}

\begin{defn}[Gromov {\cite{Gromov1993}}]\label{d: coarse quasi-isometry}
  A \emph{coarse quasi-isometry} \index{coarse quasi-isometry} of $M$
  to $M'$ is a bi-Lipschitz bijection $f:A\to A'$, where $A$ and $A'$
  are nets in $M$ and $M'$, respectively; in this case, $M$ and $M'$
  are said to have the same \emph{coarse quasi-isometry type} or to be
  \emph{coarsely quasi-isometric}. If $A$ and $A'$ are $K$-nets, and
  $C$ is a bi-Lipschitz distortion of $f$, then the pair $(K,C)$ is
  called a \emph{coarse distortion} \index{distortion!coarse} of $f$;
  the term \emph{$(K,C)$-coarse quasi-isometry} may be also used in
  this case. It is said that a map $M\to M'$ \emph{induces} a coarse
  quasi-isometry when its restriction to some subsets of $M$ and $M'$
  is a coarse quasi-isometry of $M$ to $M'$. A coarse quasi-isometry
  of $M$ to itself will be called a \emph{coarsely quasi-isometric
    transformation} \index{coarsely quasi-isometric transformation} of
  $M$.
\end{defn}
  
\begin{rem}\label{r: coarse quasi-isometry}
  If $f:A\to A'$ is a $(K,C)$-coarse quasi-isometry of $M$ to $M'$,
  then $f^{-1}:A'\to A$ is a $(K,C)$-coarse quasi-isometry of $M'$ to
  $M$. If moreover $g:B'\to B''$ is a $(K,C)$-coarse quasi-isometry of
  $M'$ to $M''$, and $A'\subset B'$, then, using Remarks~\ref{r:
    bi-Lipschitz} and~\ref{r: net}, it is easy to check that $g f:A\to
  g(A')$ is a coarse quasi-isometry of $M$ to $M''$ with coarse
  distortion $(K+CK,C^2)$.
\end{rem}

\begin{defn}\label{d: equi-coarse quasi-isometries}
  Let $f_i$ be a coarse quasi-isometry of each $M_i$ to $M'_i$. If all
  of them have a common coarse distortion, then $\{f_i\}$ is called a
  family of \emph{equi-coarse quasi-isometries}. \index{equi-coarse
    quasi-isometries} In this case, $\{M_i\}$ and $\{M'_i\}$ are
  called \index{equi-coarsely quasi-isometric} \emph{equi-coarsely
    quasi-isometric}.
\end{defn}

\begin{defn}\label{d: close coarse quasi-isometries}
  Two coarse quasi-isometries $f:A\to A'$ and $g:B\to B'$ of $M$ to
  $M'$ are said to be \emph{close} \index{close} if there are
  $r,s\ge0$ such that, for all $x\in A$, there is some $y\in B$ with
  $d(x,y)\le r$ and $d'(f(x),g(y))\le s$.  (In this case, it may be
  also said that \(f\) and \(g\) are \emph{\((r,s)\)-close}, or that
  $f$ is \emph{\((r,s)\)-close} to $g$.)
\end{defn}

\begin{prop}\label{p: being close}
  ``Being close'' is an equivalence relation on the set of coarse
  quasi-isometries of $M$ to $M'$.
\end{prop}

\begin{proof}
  The relation ``Being close'' is obviously reflexive.  To prove that
  it is symmetric, let \(f:A\to A'\) and \(g:B\to B'\) be coarse
  quasi-isometries of \(M\) to \(M'\) such that \(f\) is
  \((r,s)\)-close to \(g\).  Let $(K,C)$ be a coarse distortion for
  $g$. For any $y\in B$, there is some $x\in A$ such that $d(x,y)\le
  K$. Then there is some $y'\in B$ such that $d(x,y')\le r$ and
  $d'(f(x),g(y'))\le s$. It follows that
  \begin{multline*}
    d'(f(x),g(y))\le d'(f(x),g(y'))+d'(g(y'),g(y))\le s+C\,d(y',y)\\
    \le s+C(d(y',x)+d(x,y))\le s+C(r+K)\;,
  \end{multline*}
  obtaining that $g$ is $(K,s+C(r+K))$-close to $f$.

  To prove that the relation ``Being close'' is transitive, let \(f\)
  and \(g\) be as above, and let $h:D\to D'$ be a coarse
  quasi-isometry of $M$ to $M'$ that is $(t,u)$-close to $g$. By the
  triangle inequality, it easily follows that $f$ is $(r+t,s+u)$-close
  to $h$.
\end{proof}

\section{Coarse composites}\label{s: coarse composites}

Let $f:A\to A'_1$ and $f':A'_2\to A''$ be coarse quasi-isometries of
$M$ to $M'$ and of $M'$ to $M''$, respectively, and let $(K,C)$ be a
coarse distortion for both. The following definition makes
sense by Remark~\ref{r: coarse quasi-isometry}.

\begin{defn}\label{d: coarse composite}
  A \emph{coarse composite} \index{coarse composite} of $f$ and $f'$
  is any coarse quasi-isometry of $M$ to $M''$ that is close to the
  composite $g' g$, where $g$ (respectively, $g'$) is a coarse
  quasi-isometry of $M$ to $M'$ (respectively, of $M'$ to $M''$) close
  to $f$ (respectively, $f'$) such that $\im g\subset\dom g'$.
\end{defn}

\begin{prop}\label{p: coarse composite is well defined}
  Every two coarse composites of $f$ and $f'$ are close.
\end{prop}

\begin{proof}
  Let $g$ and $h$ be coarse quasi-isometries of $M$ to $M'$ close to
  $f$, and let $g'$ and $h'$ be coarse quasi-isometries of $M'$ to
  $M''$ close to $f'$. Suppose that $\im g\subset\dom g'$ and $\im
  h\subset\dom h'$. Then the composites $g' g:\dom g\to g'(\im g)$ and
  $h' h:\dom h\to h'(\im h)$ are coarse quasi-isometries of $M$ to
  $M''$ (Remark~\ref{r: coarse quasi-isometry}). By
  Proposition~\ref{p: being close}, there are $r,s,t,u\ge0$ such that
  $g$ is $(r,s)$-close to $h$, and $g'$ is $(t,u)$-close to
  $h'$. Thus, for each $x\in\dom g$, there is some $y\in\dom h$ so
  that $d(x,y)\le r$ and $d(g(x),h(y))\le s$. Then there is some
  $z'\in\dom h'$ such that $d'(g(x),z')\le t$ and $d''(g'
  g(x),h'(z'))\le u$. Let $C$ be a bi-Lipschitz distortion of $h'$. We
  get
  \begin{multline*}
    d''(g' g(x),h' h(y))\le d''(g' g(x),h'(z'))+d''(h'(z'),h' h(y))\\
    \le u+C\,d'(z',h(y))\le u+C(d'(z',g(x))+d'(g(x),h(y)))\le
    u+C(t+s)\;.
  \end{multline*}
  This shows that $g' g$ is $(r,u+C(t+s))$-close to $h' h$.
\end{proof}

The existence of coarse composites is guaranteed by the following result.

\begin{prop}\label{p: coarse composite}
  There is a coarse composite $g:B\to B''$ of $f$ and $f'$ with coarse
  distortion $(K(5C+1),5C^2)$ such that $B$ is a $5KC$-net of $A$,
  $B''$ is a $3KC$-net of $A''$, and $d'(f(x),{f'}^{-1} g(x))\le 2K$
  for all $x\in B$. Furthermore, if $x_1\in A$ and $x'_2\in A'_2$ are
  given so that $d'(f(x_1),x'_2)\le2K$, then $g$ can be chosen such
  that $x_1\in B$, $f'(x'_2)\in B''$ and $g(x_1)=f'(x'_2)$.
\end{prop}

The following lemma will be used to prove Proposition~\ref{p: coarse
  composite}.

\begin{lemma}\label{l:nets}
  For $K>0$, let $A_1$ and $A_2$ be $K$-nets of $M$. Then there is a
  $(6K,5)$-coarsely quasi-isometric transformation $h:B_1\to B_2$ of
  $M$ such that $B_1$ is a $5K$-net of $A_1$, $B_2$ is a $3K$-net of
  $A_2$, and $d(x,h(x))\le2K$ for all $x\in B_1$. Moreover, if $x_1\in
  A_1$ and $x_2\in A_2$ are given so that $d(x_1,x_2)\le 2K$, then $h$
  can be chosen so that $x_1\in B_1$, $x_2\in B_2$ and $h(x_1)=x_2$.
\end{lemma}

\begin{proof}
  By Lemma~\ref{l: separated net}, $A_k$ has some $K$-separated
  $K$-net $A'_k$ for $k\in\{1,2\}$. Then $A'_k$ is a $2K$-net of $M$
  because $A_k$ is $K$-net of $M$ (Remark~\ref{r: net}). So, for each
  $x\in A'_1$, there is some point $h(x)\in A'_2$ such that
  $d(x,h(x))\le2K$. A map $h:A'_1\to A'_2$ is defined in this way, and
  let $B_2$ denote its image. Choose some point $g(y)\in h^{-1}(y)$
  for each $y\in B_2$, defining a map $g:B_2\to A'_1$, whose image is
  denoted by $B_1$. Then the restriction $h:B_1\to B_2$ is bijective
  with inverse equal to the restriction $g:B_2\to B_1$.

  According to Lemma~\ref{l: separated net}, given points $x_k\in A_k$
  with $d(x_1,x_2)\le2K$, we can take $A'_k$ so that $x_k\in A'_k$,
  and we can choose $h(x_1)=x_2$ and $g(x_2)=x_1$, obtaining that
  $x_k\in B_k$.

  For each $z\in A_1$, there is  $x\in A'_1$ with $d(z,x)\le
  K$. Then $h(x)\in B_2$, $g h(x)\in B_1$, and
\begin{gather*}
  d(z,h(x))\le d(z,x)+d(x,h(x))\le3K\;,\\
  d(z,g h(x))\le d(z,h(x))+d(h(x),g h(x))\le5K\;.
\end{gather*}
Thus $B_1$ is a $5K$-net of $A_1$, and $B_2$ is a $3K$-net of
$A_2$. It follows that $B_1$ and $B_2$ are $6K$-nets of $M$ because
$A_1$ and $A_2$ are $K$-nets of $M$ (Remark~\ref{r: net}).

Since $B_1$ and $B_2$ are $K$-separated and $h:B_1\to B_2$ is
bijective, for $x\ne y$ in $B_1$, we have
\begin{multline*}
  d(h(x),h(y))\le d(h(x),x)+d(x,y)+d(y,h(y))\\
  \le d(x,y)+4K<5\,d(x,y)\;,\\[4pt]
  d(x,y)\le d(x,h(x))+d(h(x),h(y))+d(h(y),y)\\
  \le d(h(x),h(y))+4K<5\,d(h(x),h(y))\;.
\end{multline*}
\end{proof}

\begin{proof}[of Proposition~\ref{p: coarse composite}]
  By Lemma~\ref{l:nets}, there is a $(6K,5)$-coarsely quasi-isometric
  transformation $h:B'_1\to B'_2$ of $M'$ such that $B'_1$ is a
  $5K$-net of $A'_1$, $B'_2$ is a $3K$-net of $A'_2$, and
  $d'(x',h(x'))\le2K$ for all $x'\in B'_1$. By Remark~\ref{r: net},
  $B=f^{-1}(B'_1)$ is a $5KC$-net of $A$, and $B''=f'(B'_2)$ is a
  $3KC$-net of $A''$. Thus, by Remark~\ref{r: net}, $B$ and $B'$ are
  $K(5C+1)$-nets in $M$ and $M'$ because $A$ and $A''$ are $K$-nets in
  $M$ and $M''$, respectively. Moreover the composite
  \[
  \begin{CD}
    B@>f>>B'_1@>h>>B'_2@>{f'}>>B''
  \end{CD}
  \]
is a $5C^2$-bi-Lipschitz bijection (Remark~\ref{r: bi-Lipschitz}),
denoted by $g:B\to B''$, which satisfies
$$
d'(f(x),{f'}^{-1} g(x))=d'(f(x),h f(x))\le2K
$$
for each $x\in B$.

Observe that the coarse quasi-isometry $f:B\to B'_1$ is $(0,0)$-close
to $f:A\to A'_1$ because $B\subset A$.

It will be now shown that the coarse quasi-isometry $f' h:B'_1\to B''$
is close to $f':A'_2\to A''$. For each $x'\in B'_1$ there is 
$y'\in A'_2$ such that $d'(x',y')\le K$ because $A'_2$ is a $K$-net in
$M'$. Furthermore
\[
        d''(f' h(x'),f'(y'))\le C\,d'(h(x'),y')\le
        C(d'(h(x'),x')+d'(x',y'))\le3KC\;,
\]
obtaining that $f' h:B'_1\to B''$ is $(K,3KC)$-close to $f':A'_2\to
A''$.

Fix $x_1\in A$ and $x'_2\in A'_2$ so that $d'(f(x_1),x'_2)\le2K$. By
Lemma~\ref{l:nets}, we can choose $h$ such that $f(x_1)\in B'_1$,
$x'_2\in B'_2$ and $h f(x_1)=x'_2$. Hence $x_1\in B$, $f'(x'_2)\in
B''$ and $g(x_1)=f'(x'_2)$.
\end{proof}

According to Propositions~\ref{p: being close},~\ref{p: coarse
  composite} and~\ref{p: coarse composite is well defined}, the
closeness classes of coarse quasi-isometries between metric spaces
form a category of isomorphisms with the operation induced by coarse
composite. The following direct consequence is well known.

\begin{cor}
  ``Being coarsely quasi-isometric'' is an
  equivalence relation on metric spaces.
\end{cor}

\section{A coarsely quasi-isometric version of Arzela-Ascoli theorem}\label{s: coarsely q-i version of Arzela-Ascoli thm}

The following proposition will be useful to produce coarse
quasi-isometries.  It is a version of the Arzela-Ascoli theorem for
coarse quasi-isometries.

\begin{prop}\label{p: Arzela-Ascoli for coarse quasi-isometries}
  Let $\{F_n\subset M\}$ and $\{{F'}_n\subset M'\}$ be increasing
  sequences of finite subsets such that \(\bigcup_n F_n\) and
  \(\bigcup_n {F'}_n\) are $L$-nets in \(M\) and \(M'\),
  respectively. For each $n$, let $f_n$ be a $(K,C)$-coarse
  quasi-isometry of $F_n$ to $F'_n$, such that $f_n(F_m\cap\dom
  f_n)=F'_m\cap\im f_n$ if $m<n$. Then there is a $(K+L,C)$-coarse
  quasi-isometry $g$ of $M$ to $M'$, which is the combination of
  restrictions $f_{n_m}:F_m\cap\dom f_{n_m}\to F'_m\cap\im f_{n_m}$
  for some subsequence $f_{n_m}$ with $n_m\ge m$.
\end{prop}

\begin{proof}
  For each $m$, let $\CC_m$ be the set of restrictions
  $f_n:F_m\cap\dom f_n\to F'_m\cap\im f_n$ for all $n\ge m$. Define a
  graph structure on $\CC=\bigsqcup_m\CC_m$ by placing an edge between
  each $f\in\CC_{m+1}$ and its restriction $F_m\cap\dom f\to
  F'_m\cap\im f$, which is well defined and belongs to $\CC_m$; such a
  $\CC$ is an infinite tree. Each $\CC_m$ is finite because so are the
  sets $F_m$ and $F'_m$, and thus each vertex of $\CC$ meets a finite
  number of edges.  Therefore $\CC$ contains an infinite ray with
  vertices $g_m\in\CC_m$; every $g_m$ is a restriction
  $f_{n_m}:F_m\cap\dom f_{n_m}\to F'_m\cap\im f_{n_m}$ with $n_m\ge
  n$. All maps $g_m$ can be combined to define a map $g:\bigcup_m\dom
  g_m\to\bigcup_m\im g_m$, which is a $(K,C)$-coarse quasi-isometry of
  \(\bigcup_n F_n\) to \(\bigcup_n {F'}_n\). Since \(\bigcup_n F_n\)
  and \(\bigcup_n {F'}_n\) are $L$-nets in \(M\) and \(M'\),
  respectively, this $g$ is a $(K+L,C)$-coarse quasi-isometry of $M$
  to $M'$ (Remark~\ref{r: net}). 
\end{proof}

\begin{rem}\label{r:coarse quasi-isometry}
  In Proposition~\ref{p: Arzela-Ascoli for coarse quasi-isometries},
  observe that, if $x\in\bigcap_n\dom f_n$ and $f_n(x)=y$ for all $n$,
  then $x\in\dom g$ and $g(x)=y$.
\end{rem}

\section{Large scale Lipschitz maps}\label{s: large scale Lipschitz}

\begin{defn}\label{d: close maps}
  Two maps, $f,g:S\to M$, of a set, \(S\), into a metric space, \(M\),
  are said to be \index{close} \emph{close}\footnote{This terminology is used in
    \cite{HigsonRoe2000}. Other terms used to indicate the same
    property are \emph{coarsely equivalent} \cite{Roe1996}, \emph{parallel} \cite{Gromov1993}, \emph{bornotopic}
    \cite{Roe1993}, and \emph{uniformly close}
    \cite{BlockWeinberger1992}.} if there is some $c\ge0$ such that
  $d(f(x),g(x))\le c$ for all $x\in S$; it may be also said that \(f\)
  and \(g\) are \emph{$c$-close}, or that \(f\) is \emph{\(c\)-close}
  to \(g\).
\end{defn}

\begin{rem}
  ``Being close'' is an equivalence relation on the set
  of maps of a set to a metric space.
\end{rem}

\begin{defn}[Gromov \cite{Gromov1993}]\label{d: large scale Lipschitz map}
  A map\footnote{Continuity is not assumed here.} $\phi:M\to M'$ is
  said to be \emph{large scale Lipschitz} \index{large scale Lipschitz} if there exist
  $\lambda>0$ and $b\ge0$ such that
  $$
  d'(\phi(x),\phi(y))\le\lambda\,d(x,y)+b
  $$
  for all $x,y\in M$; in this case, the pair $(\lambda,b)$ is called a
  \emph{large scale Lipschitz distortion} \index{distortion!large scale Lipschitz} of $\phi$, and $\phi$ is
  said to be \emph{$(\lambda,b)$-large scale Lipschitz}.  The map $\phi$
  is said to be \emph{large scale bi-Lipschitz} \index{large scale bi-Lipschitz} if there exist
  constants $\lambda>0$ and $b\ge0$ such that
  \[
  \frac{1}{\lambda}(d(x,y)-b)\le
  d'(\phi(x),\phi(y))\le\lambda\,d(x,y)+b
  \]
  for all $x,y\in M$; in this case, the pair $(\lambda,b)$ is called a
  \emph{large scale bi-Lipschitz distortion} \index{distortion!large scale bi-Lipschitz} of $\phi$, and $\phi$ is
  said to be \emph{$(\lambda,b)$-large scale bi-Lipschitz}.

  A map $\phi:M\to M'$ is said to be a \emph{large scale Lipschitz
    equivalence} \index{large scale Lipschitz equivalence} if it is
  large scale Lipschitz and there is another large scale Lipschitz map
  $\psi:M'\to M$ so that $\psi \phi$ and $\phi\psi$ are close to
  $\id_M$ and $\id_{M'}$, respectively.  In this case, if
  $(\lambda,b)$ is a large scale Lipschitz distortion of $\phi$ and
  $\psi$, and $\psi\phi$ and $\phi\psi$ are $c$-close to the identity
  maps, then $(\lambda,b,c)$ is called a \emph{large scale Lipschitz
    equivalence distortion} \index{distortion!large scale Lipschitz
    equivalence} of $\phi$; it may be also said that $\phi$ is a
  \emph{$(\lambda,b,c)$-large scale Lipschitz equivalence}. In this
  case, $M$ and $M'$ are said to be \index{large scale Lipschitz
    equivalent} \emph{large scale Lipschitz equivalent}. Two metrics
  $d_1$ and $d_2$ on a set $S$ are called \emph{large scale Lipschitz
    equivalent} if the identity map \((S,d_1)\to(S,d_2)\) is large
  scale bi-Lipschitz.
\end{defn}

\begin{defn}\label{d: equi-large scale Lipschitz maps}
  Let $\phi_i:M_i\to M'_i$ for each $i$. The class $\{\phi_i\}$ is
  said to be \emph{equi-large scale {\rm(}bi-\/{\rm)}Lipschitz} 
  \index{equi-large scale (bi-)Lipschitz} if all
  the maps \(\phi_i\) are large scale (bi-)Lipschitz with a common
  large scale (bi-)Lipschitz distortion. The class $\{\phi_i\}$ is
  said to be \emph{equi-large scale Lipschitz equivalences} 
  \index{equi-large scale Lipschitz equivalences} if the
  maps \(\phi_i\) are large scale Lipschitz equivalences and have a
  common large scale Lipschitz equivalence distortion; in this case,
  $\{M_i\}$ and $\{M'_i\}$ are called \index{equi-large scale Lipschitz
    equivalent} \emph{equi-large scale Lipschitz
    equivalent}. Given a class of sets $\{S_i\}$, and metrics
  $d_{i,1}$ and $d_{i,2}$ on each $S_i$, $\{d_{i,1}\}$ is said to be
  \emph{equi-large scale Lipschitz equivalent} to $\{d_{i,2}\}$ if
  the identity maps $(S_i,d_{i,1})\to(S_i,d_{i,2})$ are equi-large
  scale bi-Lipschitz.
\end{defn}

The qualitative content of the following lemmas is well known, but we
keep track of the constants involved.

\begin{lemma}\label{l: any large scale Lipschitz equivalence is large scale bi-Lipschitz}
  If $\phi:M\to M'$ is a $(\lambda,b,c)$-large scale Lipschitz
  equivalence, then \(\phi\) is $(\lambda,b+2c)$-large scale bi-Lipschitz.
\end{lemma}

\begin{proof}
  Let $\psi:M'\to M$ be a $(\lambda,b)$-large scale Lipschitz map such
  that $\psi\phi$ and $\phi\psi$ are $c$-close to $\id_M$ and
  $\id_{M'}$, respectively. Then, for all \(x, y\in M\), 
  \[
  d(x,y)\le
  d(\psi\phi(x),\psi\phi(y))+2c\le\lambda\,d'(\phi(x),\phi(y))+b+2c,
  \]
by the triangle inequality.
\end{proof}

\begin{lemma}\label{l: large scale bi-Lipschitz and the image is c-net}
  Let $\phi:M\to M'$ be $(\lambda,b)$-large scale bi-Lipschitz. If
  $\phi(M)$ is a $c$-net in $M'$, then $\phi$ is a
  $(\lambda,b+2\lambda c,\max\{b,c\})$-large scale Lipschitz
  equivalence.
\end{lemma}

\begin{proof}
  Because \(\phi(M)\) is a \(c\)-net in \(M'\), a map \(\psi:M'\to M\)
  can be constructed by choosing, for each \(x'\in M'\), one point
  \(\psi(x')\in M\) such that $d'(x',\phi\psi(x'))\le c$, and
  furthermore so that, if $x'\in\phi(M)$, then $\phi\psi(x')=x'$;
  i.e., $\phi\psi\phi=\phi$.

 Then, for all $x',y'\in M'$ and $x\in M$, 
  	\begin{align*}
    		d(x,\psi\phi(x))&\le\lambda\,d'(\phi(x),\phi\psi\phi(x))+b
    		=\lambda\,d'(\phi(x),\phi(x))+b=b\;,\\[5pt]
		d'(x',y')&\le d'(x',\phi\psi(x'))+d'(\phi\psi(x'),\phi\psi(y'))+d'(\phi\psi(y'),y')\\
    		&\le\lambda\,d(\psi(x'),\psi(y'))+b+2c\;,\\[5pt]
		d(\psi(x'),\psi(y'))&\le\lambda\,d'(\phi\psi(x'),\phi\psi(y'))+b\\
    		&\le\lambda(d'(\phi\psi(x'),x')+d'(x',y')+d'(y',\phi\psi(y')))+b\\
    		&\le\lambda\,d'(x',y')+b+2\lambda c\;.
  	\end{align*}
\end{proof}

\begin{rem}
  According to Lemmas~\ref{l: any large scale Lipschitz equivalence is
    large scale bi-Lipschitz} and~\ref{l: large scale bi-Lipschitz and
    the image is c-net}, (equi-) large scale Lipschitz equivalences are
  just (equi-) large scale bi-Lipschitz maps whose images are
  (equi-) nets.
\end{rem}

\begin{lemma}\label{l: composite of large scale Lipschitz maps}
  Let $\phi:M\to M'$ and $\phi':M'\to M''$ be maps. The following properties hold:
  \begin{enumerate}[{\rm(}i\/{\rm)}]
		
  \item\label{i: composite of large scale Lipschitz maps} If $\phi$
    and $\phi'$ are $(\lambda,b)$-large scale Lipschitz, then
    $\phi'\phi$ is $(\lambda^2,\lambda b+b)$-large scale Lipschitz.
			
  \item\label{i: composite of large scale Lipschitz equivalences} If
    $\phi$ and $\phi'$ are $(\lambda,b,c)$-large scale Lipschitz
    equivalences, then $\phi'\phi$ is a $(\lambda^2,\lambda
    b+b,2c)$-large scale Lipschitz equivalence.
			
  \end{enumerate}
\end{lemma}

\begin{proof}
  Property~\eqref{i: composite of large scale Lipschitz maps} is true
  because, for all $x,y\in M$,
  \[
  d''(\phi'\phi(x),\phi'\phi(y))\le\lambda\,d'(\phi(x),\phi(y))+b
  \le\lambda^2\,d(x,y)+\lambda b+b\;.
  \]

  To prove~\eqref{i: composite of large scale Lipschitz equivalences},
  take $(\lambda,b)$-large scale Lipschitz maps $\psi:M'\to M$ and
  $\psi':M''\to M'$ such that $\psi\phi$ is $c$-close to $\id_M$,
  $\phi\psi$ and $\psi'\phi'$ are $c$-close to $\id_{M'}$, and
  $\phi'\psi'$ is $c$-close to $\id_{M''}$. By~\eqref{i: composite of
    large scale Lipschitz maps}, $\phi'\phi$ and $\psi\psi'$ are
  $(\lambda^2,\lambda b+b)$-large scale Lipschitz. Moreover
	$$
        d(\psi\psi'\phi'\phi(x),x)\le
        d(\psi\psi'\phi'\phi(x),\phi'\phi(x))+d(\phi'\phi(x),x)\le 2c
	$$
        for all $x\in M$, obtaining that $\psi\psi'\phi'\phi$ is
        $2c$-close to $\id_M$. Similarly, $\phi'\phi\psi\psi'$ is
        $2c$-close to $\id_{M''}$.
\end{proof}

\begin{lemma}\label{l: composite of large scale Lipschitz maps is
    compatible closeness}
  Let $\phi,\psi:M\to M'$ and $\phi',\psi':M'\to M''$ be
  $(\lambda,b)$-large scale Lipschitz maps. If $\phi$ and $\phi'$ are
  $R$-close to $\psi$ and $\psi'$, respectively, then $\phi'\phi$ is
  $(\lambda R+b+R)$-close to $\psi'\psi$.
\end{lemma}

\begin{proof}
  For all $x\in M$,
  \begin{multline*}
    d''(\phi'\phi(x),\psi'\psi(x))\le d''(\phi'\phi(x),\psi'\phi(x))+d''(\psi'\phi(x),\psi'\psi(x))\\
    \le R+\lambda\,d'(\phi(x),\psi(x))+b\le\lambda R+b+R\;.
  \end{multline*}
\end{proof} 
 
By Lemmas~\ref{l: composite of large scale Lipschitz maps} and~\ref{l:
  composite of large scale Lipschitz maps is compatible closeness},
the closeness classes of large scale Lipschitz maps between metric
spaces form a category, whose isomorphisms are the classes represented
by large scale Lipschitz equivalences.

It is well known that two metric spaces are coarsely quasi-isometric
if and only if they are isomorphic in the category whose objects are
metric spaces and whose morphisms are closeness equivalence classes of
large scale Lipschitz maps. This is part of the content of the
following two results, where the constants involved are specially
analyzed.

\begin{prop}[\'Alvarez-Candel {\cite[Proposition~2.2]{AlvarezCandel2011}}]\label{p: large scale Lipschitz extensions}
  Any $(K,C)$-coarse quasi-isometry $f:A\to A'$ of $M$ to $M'$ is
  induced by a $(C,2CK,K)$-large scale Lipschitz equivalence
  $\phi:M\to M'$.
\end{prop}

\begin{prop}[\'Alvarez-Candel {\cite[Proposition~2.3]{AlvarezCandel2011}}]
\label{p: restrictions of large scale Lipschitz maps} 
  For each $\epsilon>0$ and $x_0\in M$, every $(\lambda,b,c)$-large
  scale Lipschitz equivalence $\phi:M\to M'$ induces a $(K,C)$-coarse
  quasi-isometry $f:A\to A'$ of $M$ to $M'$ such that $x_0\in A$,
  where
  \[
  K=c + 2 \lambda c + \lambda b + \lambda \epsilon+b\;,\quad
  C=\lambda+\frac{\lambda}{\epsilon}(2c+b)\;.
  \]
\end{prop}

\begin{rem}
  In \cite[Proposition~2.3]{AlvarezCandel2011}, it is not explicitly
  stated that $A$ contains any given point $x_0$. But, in the  proof
  of that proposition, the set $A$ is any $(2c+b+\epsilon)$-separated
  $(2c+b+\epsilon)$-net  of $M$, and so,  using Lemma~\ref{l:
    separated net}, it may be further impossed that $x_0\in A$.
\end{rem}

\begin{rem}
  According to Propositions~\ref{p: large scale Lipschitz extensions}
  and~\ref{p: restrictions of large scale Lipschitz maps}, $\{M_i\}$
  and $\{M'_i\}$ are equi-coarsely quasi-isometric if and only if they
  are equi-large scale Lipschitz equivalent.
\end{rem}

\begin{prop}\label{p: close iff close}
  Let $\phi,\psi:M\to M'$ be $(\lambda,b)$-large scale Lipschitz
  equivalences, and let $f:A\to A'$ and $g:B\to B'$ be $(K,C)$-coarse
  quasi-isometries of $M$ to $M'$ induced by $\phi$ and $\psi$,
  respectively. The following properties hold:
  \begin{enumerate}[{\rm(}i\/{\rm)}]
		
  \item\label{i: if phi is close to psi, then f is close to g} If
    $\phi$ is $R$-close to $\psi$, then $f$ is $(K,R+\lambda
    K+b)$-close to $g$.
			
  \item\label{i: if f is close to g, then phi is close to psi} If $f$
    is $(r,s)$-close to $g$, then $\phi$ is
    $(\lambda(r+2K)+s+2b)$-close to $\psi$.
		
  \end{enumerate}
\end{prop}

\begin{proof} \eqref{i: if phi is close to psi, then f is close to
    g} For all $x\in A$, there is some $y\in B$ so that $d(x,y)\le
  K$. Then
  \begin{multline*}
    d'(f(x),g(y))=d'(\phi(x),\psi(y))\le d'(\phi(x),\psi(x))+d'(\psi(x),\psi(y))\\
    \le R+\lambda\, d(x,y)+b\le R+\lambda K+b\;.
  \end{multline*}
	
 \eqref{i: if f is close to g, then phi is close to
    psi} For any $x\in M$, there is some $y\in A$ such that
  $d(x,y)\le K$. Then there is some $z\in B$ so that $d(y,z)\le r$ and
  $d'(f(y),g(z))\le s$. Hence
  \begin{multline*}
    d'(\phi(x),\psi(x))\le d'(\phi(x),\phi(y))+d'(f(y),g(z))+d'(\psi(z),\psi(x))\\
    \le\lambda\, d(x,y)+s+\lambda\, d(z,x)+2b
    \le\lambda K+s+\lambda(d(z,y)+d(y,x))+2b\\
    \le\lambda(r+2K)+s+2b\;.
  \end{multline*}
\end{proof}

By Propositions~\ref{p: large scale Lipschitz extensions},~\ref{p:
  restrictions of large scale Lipschitz maps} and~\ref{p: close iff
  close}, the category whose objects are metric space and whose
morphisms are closeness equivalence classes of coarse quasi-isometries
between metric spaces can be identified to the subcategory of
isomorphisms of the category of closeness classes of large scale
Lipschitz maps between metric spaces. Therefore coarse
quasi-isometries and large scale Lipschitz equivalences are equivalent
concepts. We will often use large scale Lipschitz equivalences in the
proofs because they become simpler. However direct proofs for coarse
quasi-isometries may produce better constants. This will be indicated
in remarks.

\begin{prop}\label{p: large scale Lipschitz, x_0 mapsto x'_0}
  Let $\phi:M\to M'$, $x_0\in M$, $x'_0\in M'$ and $R\ge
  d'(\phi(x_0),x'_0)$, and let $\bar\phi:M\to M'$ be defined by
  $\bar\phi(x_0)=x'_0$ and $\bar\phi(x)=\phi(x)$ if $x\ne x_0$. The
  following properties hold:
  \begin{enumerate}[{\rm(}i\/{\rm)}]
		
  \item\label{i: bar phi is (lambda,b+R)-large scale Lipschitz} If
    $\phi$ is $(\lambda,b)$-large scale Lipschitz, then $\bar\phi$ is
    $(\lambda,b+R)$-large scale Lipschitz.
			
  \item\label{i: bar phi is a (lambda,bar b,bar c)-large scale
      Lipschitz equivalence} If $\phi$ is a $(\lambda,b,c)$-large
    scale Lipschitz equivalence, then $\bar\phi$ is a $(\lambda,\bar
    b,\bar c)$-large scale Lipschitz equivalence, where $\bar b=b+R$
    and $\bar c=\lambda R+b+2c$.
		
  \end{enumerate}
\end{prop}

\begin{proof}
  Since $\bar\phi$ equals $\phi$ on $M\sm\{x_0\}$, it is enough to
  check~\eqref{i: bar phi is (lambda,b+R)-large scale Lipschitz} for
  $x_0$ and any $x\ne x_0$ in $M$:
  \begin{multline*}
    d'(\bar\phi(x_0),\bar\phi(x))=d'(x'_0,\phi(x))\le d'(x'_0,\phi(x_0))+d'(\phi(x_0),\phi(x))\\
    \le R+\lambda\,d(x_0,x)+b\;.
  \end{multline*}
	
  To prove~\eqref{i: bar phi is a (lambda,bar b,bar c)-large scale
    Lipschitz equivalence}, take a $(\lambda,b)$-large scale Lipschitz
  map $\psi:M'\to M$ so that $\psi\phi$ and $\phi\psi$ are $c$-close
  to $\id_M$ and $\id_{M'}$, respectively. Let $\bar\psi:M'\to M$ be
  defined by $\bar\psi(x'_0)=x_0$ and $\bar\psi(x')=\psi(x')$ if
  $x'\ne x'_0$. Then $\bar\phi$ and $\bar\psi$ are
  $(\lambda,b+R)$-large scale Lipschitz by~\eqref{i: bar phi is
    (lambda,b+R)-large scale Lipschitz}. Moreover $\phi$ and $\psi$
  are $(\lambda,b+2c)$-large scale bi-Lipschitz by Lemma~\ref{l: any
    large scale Lipschitz equivalence is large scale
    bi-Lipschitz}. Since $\bar\phi$ and $\bar\psi$ equal $\phi$ and
  $\psi$ on $M\sm\{x_0\}$ and $M'\sm\{x'_0\}$, respectively, to check
  that $\bar\psi\bar\phi$ is $(\lambda R+b+2c)$-close to $\id_M$, the
  only non-trivial case is at every point $x\in M$ with
  $\phi(x)=x'_0$:
\[
d(x,\bar\psi\bar\phi(x))=d(x,x_0)\le\lambda\,d'(\phi(x),\phi(x_0))+b+2c
\le \lambda R+b+2c\;.
\]
Similarly, we get that $\bar\phi\bar\psi$ is $(\lambda R+b+2c)$-close to
$\id_{M'}$.
\end{proof}

\begin{rem}\label{r: large scale Lipschitz, x_0 mapsto x'_0}
  In Proposition~\ref{p: large scale Lipschitz, x_0 mapsto x'_0}, note
  that $\bar\phi$ is $R$-close to $\phi$.
\end{rem}

\begin{cor}\label{c: coarse quasi-isometry, x_0 mapsto x'_0}
  Let $f:A\to A'$ be a $(K,C)$-coarse quasi-isometry of $M$ to
  $M'$. Let $R>0$, $x_0\in M$ and $x'_0\in M'$ such that there is some
  $y_0\in A$ with $d(x_0,y_0)\le R$ and $d'(x'_0,f(y_0))\le R$. Then, for all $\epsilon>0$,
  there is a $(\bar K,\bar C)$-coarse quasi-isometry $\bar f:B\to B'$
  of $M$ to $M'$ $(r,s)$-close to $f$ such that $x_0\in B$, $x'_0\in
  B'$ and $f(x_0)=x'_0$, where
  \begin{gather*}
    \bar K=\bar c+2C\bar c+C\bar b+C\epsilon+\bar b\;,\quad
    \bar C=C+\frac{C}{\epsilon}(2\bar c+\bar b)\;,\\
    r=\bar K\;,\quad s=CR+2CK+R+\bar C\bar K+\bar b\;.
  \end{gather*}
\end{cor}

\begin{proof}
  By Proposition~\ref{p: large scale Lipschitz extensions}, $f$ is
  induced by a $(C,2CK,K)$-large scale Lipschitz equivalence
  $\phi:M\to M'$. We have
  \begin{multline*}
  d'(\phi(x_0),x'_0)\le d'(\phi(x_0),\phi(y_0))+d'(\phi(y_0),x'_0)\\
  \le C\,d(x_0,y_0)+2CK+R\le CR+2CK+R\;.
\end{multline*}
According to Proposition~\ref{p: large scale Lipschitz, x_0 mapsto
  x'_0} and Remark~\ref{r: large scale Lipschitz, x_0 mapsto x'_0},
$\phi$ is $(CR+2CK+R)$-close to a $(C,\bar b,\bar c)$-large scale
Lipschitz equivalence $\bar\phi:M\to M'$ with $\bar\phi(x_0)=x'_0$. By
Propositions~\ref{p: restrictions of large scale Lipschitz maps}
and~\ref{p: close iff close}-\eqref{i: if phi is close to psi, then f
  is close to g}, for each $\epsilon>0$, $\bar\phi$ induces a $(\bar
K,\bar C)$-coarse quasi-isometry $\bar f:B\to B'$ of $M$ to $M'$
$(\bar K,s)$-close to $f$ such that $x_0\in B$.
\end{proof}

\begin{rem}
  Another version of Corollary~\ref{c: coarse quasi-isometry, x_0
    mapsto x'_0}, where
$$
\bar K=2C^2R+2CR+R+K\;,\quad\bar C=2C+1\;,\quad r=s=R\;,
$$
can be proved without passing to large scale Lipschitz equivalences,
with more involved arguments. These constants are simpler, but the
constants of Corollary~\ref{c: coarse quasi-isometry, x_0 mapsto x'_0}
give the following extra information: we can get $\bar C$ as close to
$C$ as desired at the expense of increasing $\bar K$ (by taking
$\epsilon$ large enough).
\end{rem}

\section{Coarse and rough maps}\label{s: coarse and rough}

\begin{defn}\label{d:rough map}
  A map $f:M\to M'$ is called:
  \begin{itemize}
		
  \item \emph{uniformly expansive}\footnote{This name is taken from
      \cite{Roe1996}. \index{uniformly expansive} Other terms used to denote the same property are
      \emph{uniformly bornologous} \cite{Roe1993} and \emph{coarsely
        Lipschitz} \cite{BlockWeinberger1992}.} if, for each
    $r\ge0$, there is some $s_r\ge0$ such that
    \[
    d(x,y)\le r\Longrightarrow d'(f(x),f(y))\le s_r
    \]
    for all $x,y\in M$;
    
  \item \emph{metrically proper}\footnote{This term is used in \cite{Roe1996}.} \index{metrically proper} if $f^{-1}(B)$ is bounded in $M$ for each bounded subset $B\subset M'$;
    
  \item \emph{uniformly metrically proper}\footnote{This term is used in \cite{Roe1996}. Another term used to denote the same property is \emph{effectively proper} \cite{BlockWeinberger1992}.} \index{uniformly metrically proper} if, for each
    $r\ge0$, there is some $t_r\ge0$ so that
    \[
    d'(f(x),f(y)) \le r \Longrightarrow d(x,y) \le t_r
    \]
    for all $x,y\in M$;
    
  \item \emph{coarse}\footnote{This is a particular case of coarse maps between general coarse spaces \cite{Roe1996}, \cite{HigsonRoe2000}.} \index{coarse} if it is uniformly expansive and metrically proper; and
			
  \item \emph{rough} \index{rough} if it is uniformly expansive and uniformly metrically proper.
		
  \end{itemize}
  If $f$ satisfies the conditions of uniform expansiveness and uniform
  metric properness with respective mappings $r\mapsto s_r$ and
  $r\mapsto t_r$ (simply denoted by $s_r$ and $t_r$), then $(s_r,t_r)$
  is called a \emph{rough distortion} \index{distortion!rough} of $f$; the term
  \emph{$(s_r,t_r)$-rough map} may be also used. When $s_r=t_r$, we
  simply say that $s_r$ is a \emph{rough distortion} of $f$, or $f$
  is an \emph{$s_r$-rough map}. If $f$ is a coarse map and there is a
  coarse map $g:M'\to M$ such that $gf$ and $fg$ are close to $\id_M$
  and $\id_{M'}$, respectively, then $f$ is called a \index{coarse equivalence} 
  \emph{coarse equivalence}. If $f$ is an $s_r$-rough map and there is an
  $s_r$-rough map $g:M'\to M$ such that $gf$ and $fg$ are $c$-close to
  $\id_M$ and $\id_{M'}$, respectively, then $f$ is called an
  $(s_r,c)$-\emph{rough equivalence}, \index{rough equivalence} and $(s_r,c)$ is called a 
  \emph{rough equivalence distortion} \index{distortion!rough equivalence} of $f$. If there is a coarse
  (respectively, rough) equivalence $M\to M'$, then $M$ and $M'$ are
  \emph{coarsely} (respectively, \emph{roughly})\footnote{The term \emph{uniform
   closeness} is used in \cite{BlockWeinberger1992} when two
   metric spaces are roughly equivalent.} \emph{equivalent}. \index{coarsely equivalent} \index{roughly equivalent} A coarse
  (respectively, rough) equivalence $M\to M$ is called a \emph{coarse}
    (respectively, \emph{rough}) \emph{transformation} \index{coarse transformation} \index{rough transformation} of $M$.
\end{defn}

Two metrics $d_1$ and $d_2$ on the same set $S$ are called
\emph{coarsely} (respectively, \emph{roughly}) \emph{equivalent} if
the identity map $(S,d_1)\to(S,d_2)$ is a coarse (respectively, rough)
equivalence. When $S$ is equipped with a coarse (respectively, rough)
equivalence class of metrics, it is called a \emph{metric coarse
  space}\footnote{This notion of metric coarse space is equivalent to
  the concept of coarse space induced by a metric \cite{Roe1996},
  \cite{Roe2003}.} \index{metric coarse space} (respectively,
\index{rough space} \emph{rough space}). The metric coarse space and
rough space induced by the metric space $M$ is denoted by $[M]$. The
condition on a map $M\to M'$ to be coarse (respectively, rough)
depends only on the metric coarse spaces (respectively, rough spaces)
$[M]$ and $[M']$. Any composition of coarse (respectively, rough) maps
is (respectively, rough); more precisely, if $f:M\to M'$ and $f':M'\to
M''$ are $s_r$-rough maps, then $f'f$ is $s_{s_r}$-rough. Moreover the
composite of coarse/rough maps is compatible with the closeness
relation in an obvious sense. So the closeness classes of coarse
(respectively, rough) maps between metric coarse spaces (respectively,
rough spaces) form a category called \index{metric coarse category}
\emph{metric coarse category}\footnote{This is a subcategory of the
  coarse category \cite{Roe1996}, \cite{HigsonRoe2000}.}
(respectively, \index{rough category} \emph{rough category}). Thus
rough equivalences are the maps that induce isomorphisms in the rough
category. There are interesting differences between the rough category
and the metric coarse category, cf.~\cite{Roe1996}, but the following
result shows that they have the same isomorphisms.

\begin{prop}[\'Alvarez-Candel {\cite[Proposition~3.8]{AlvarezCandel2011}}]\label{p:uniform metric properness} Any coarse equivalence
  between metric spaces is uniformly metrically proper.  Moreover the
  definition of uniform metric properness is satisfied with constants
  that depend only on the constants involved in the definition of
  coarse equivalence.
\end{prop}

Observe that, if $f:M\to M'$ is an $(s_r,c)$-rough equivalence, then
$f(M)$ is a $c$-net in $M'$.

\begin{prop}\label{p: rough equiv = rough + net}
  If $f:M\to M'$ is $s_r$-rough and $f(M)$ is a $K$-net in $M'$, then
  $f$ is a $(\bar s_r,c)$-rough equivalence, where $\bar
  s_r=\max\{s_{r+2K},s_r+2K\}$ and $c=\max\{K,s_K\}$.
\end{prop}

\begin{proof}
  Let $g:M'\to M$ be defined by choosing, for each $x'\in M'$, a point
  $g(x')\in M$ so that $d'(x',fg(x'))\le K$. Thus $fg$ is $K$-close to
  $\id_{M'}$.
	
  For $x\in M$, we have $d'(f(x),fgf(x))\le K$, giving $d(x,gf(x))\le
  s_K$. Hence $gf$ is $s_K$-close to $\id_M$.
	
  For $x',y'\in M'$ with $d'(x',y')\le r$, we have
		\[
                d(fg(x'),fg(y'))\le
                d(fg(x'),x')+d'(x',y')+d'(y',fg(x')\le r+2K\;,
		\]
                obtaining $d(g(x'),g(y'))\le s_{r+2K}$. Thus $g$ is
                $s_{r+2K}$-uniformly expansive.
	
	Suppose now that $d(g(x'),g(y'))\le r$. Then
		\[
			d'(x',y')\le d'(x',fg(x'))+d'(fg(x'),fg(y'))+d'(fg(y'),y')\le s_r+2K\;.
		\]
	So $g$ is also $(s_r+2K)$-uniformly metrically proper. 
\end{proof}

The following is a direct consequence of Propositions~\ref{p:uniform metric properness} and~\ref{p: rough equiv = rough + net}.

\begin{cor}\label{c: rough = coarse embedding}
	For a map $f:M\to M'$, the following conditions are equivalent:
		\begin{enumerate}[{\rm(}i{\rm)}]
		
			\item $f$ is rough.
			
			\item $f:M\to f(M)$ is a rough equivalence.
			
			\item $f:M\to f(M)$ is a coarse equivalence.
			
		\end{enumerate}
\end{cor}

According to Corollary~\ref{c: rough = coarse embedding}, coarse maps
can be also properly called \index{coarse embedding} \emph{coarse
  embeddings}\footnote{This concept is generalized to arbitrary coarse
  spaces as maps that define a coarse equivalence to their image
  \cite[Section~11.1]{Roe2003}.}.

\begin{prop}[\'Alvarez-Candel {\cite[Proposition~3.13]{AlvarezCandel2011}}]\label{p: any large scale Lipschitz map is rough} 
  The following properties are true:
  \begin{enumerate}

  \item\label{i: f satisfies the condition of uniform expansiveness}
    Any $(\lambda,b)$-large scale Lipschitz map 
    satisfies the condition of uniform expansiveness with
    $s_r=\lambda r+b$.

  \item\label{i: f is an (s_r,c)-rough equivalence} Any $(\lambda,b,c)$-large scale Lipschitz equivalence is an
    $(s_r,c)$-rough equivalence, where $s_r=\lambda r+b+2c$.

  \end{enumerate}
\end{prop}

\begin{rem}
  In fact, the proof of Proposition~\ref{p: any large scale Lipschitz
    map is rough} is elementary:~\eqref{i: f satisfies the condition
    of uniform expansiveness} is obvious, and~\eqref{i: f is an
    (s_r,c)-rough equivalence} follows from Lemma~\ref{l: any large
    scale Lipschitz equivalence is large scale bi-Lipschitz}.
\end{rem}

\begin{defn}\label{d: equi-uniformly expansive}
  A class of maps, \(\{f_i:M_i\to M'_i\}\), is said to be a class
  of:
  \begin{itemize}

  \item \emph{equi-uniformly expansive} \index{equi-uniformly expansive} maps if they satisfy the
    condition of uniform expansiveness with the same mapping $r\mapsto
    s_r$.
			
  \item \emph{equi-metrically proper} \index{equi-metrically proper} maps if they satisfy the
    condition of metric properness with the same mapping $r\mapsto
    t_r$.
			
  \item \emph{equi-rough} \index{equi-rough} maps (or \emph{equi-coarse embeddings}) if they are rough with a common rough
    distortion; and

  \item \emph{equi-rough equivalences} \index{equi-rough equivalences} (or \emph{equi-coarse equivalences}) \index{equi-coarse equivalences} if they are rough equivalences
    with a common rough equivalence distortion.

  \end{itemize}
\end{defn}

It is not possible to define ``equi-coarse maps,'' but the concept of ``equi-rough equivalences'' makes sense according to
Proposition~\ref{p:uniform metric properness}. Given a family of sets $\{S_i\}$, and
metrics $d_{i,1}$ and $d_{i,2}$ on each $S_i$, it is said that
$\{d_{i,1}\}$ is \emph{equi-roughly equivalent} \index{equi-roughly equivalent} (or \emph{equi-coarsely equivalent}) \index{equi-coarsely equivalent} to $\{d_{i,2}\}$ if the identity maps
$(S_i,d_{i,1})\to(S_i,d_{i,2})$ are equi-rough equivalences.

\begin{example}
If the metric spaces $M_i$ are bounded, they are equi-coarsely equivalent to the singleton metric space if and only $\sup_i\diam M_i<\infty$.
\end{example}

\begin{rem}
	\begin{enumerate}[(i)]
	
		\item By Proposition~\ref{p: rough equiv = rough + net}, families of equi-rough maps whose images are equi-nets are equi-rough equivalences.
		
		\item By Proposition~\ref{p: any large scale Lipschitz map is
    rough}, families of equi-large scale Lipschitz maps are
  equi-uniformly expansive, and families of equi-large scale Lipschitz
  equivalences are equi-rough equivalences.
  
  	\end{enumerate}
\end{rem}

There are rough equivalences that are not large scale Lipschitz
equivalences, cf.~\cite[Example~3.14]{AlvarezCandel2011}.

\begin{prop}\label{p: combining equi-rough equivalences}
	If there are disjoint unions, $M=\bigcup_{i=0}^\infty M_i$ and $M'=\bigcup_{i=0}^\infty M'_i$, such that:
		\begin{itemize}
		
			\item $M_i$ and $M'_i$ are bounded for all $i$;
			
			\item $\min_{i<j}d(M_i,M_j)\to\infty$ and $\min_{i<j}d(M'_i,M'_j)\to\infty$ as $j\to\infty$\; and
			
			\item $\{M_i\}$ are equi-coarsely equivalent to $\{M'_i\}$;
		
		\end{itemize}
	then $M$ is coarsely equivalent to $M'$.
\end{prop}

\begin{proof}
	There are sequences $u_i,v_i\to\infty$ such that:
		\begin{itemize}
		
			\item $d(M_i,M_j)\ge u_j$ and $d(M'_i,M'_j)\ge u_j$ if $i<j$; and
			
			\item $\diam(\bigcup_{i\le j}M_i)\le v_j$ and $\diam(\bigcup_{i\le j}M'_i)\le v_j$ for all $j$.
		
		\end{itemize}
	Moreover there are rough equivalences $f_i:M_i\to M'_i$ with a common rough equivalence distortion $(s_r,c)$. Thus there are $s_r$-rough maps $g_i:M'_i\to M_i$ such that $g_if_i$ and $f_ig_i$ are $c$-close to $\id_{M_i}$ and $\id_{M'_i}$, respectively. Since the sets $M_i$ are disjoint from each other, the maps $f_i$ can be combined to define a map $f:M\to M'$. Similarly, the maps $g_i$ can be also combined to define a map $g:M'\to M$. 
	
	Clearly, $gf$ and $fg$ are $c$-close to $\id_M$ and $\id_{M'}$, respectively. Thus it only remains to prove that $f$ and $g$ are rough.
	
	For $r\ge0$, let
		\[
			\bar s_r=\max\{\,s_r,v_j\mid u_j\le r\,\}\;.
		\]
	Suppose that $d(x,y)\le r$ for some $x,y\in M$ and $r\ge 0$. If $x,y\in M_i$ for some $i$, then $d'(f(x),f(y))\le s_r\le\bar s_r$. If $x\in M_i$ and $y\in M_j$ for some $i<j$, then $u_j\le d(x,y)\le r$, and therefore $d'(f(x),f(y))\le v_j\le\bar s_r$. 
	
	Now, suppose that $d(f(x),f(y))\le r$ for some $x,y\in M$ and
        $r\ge 0$. If $x,y\in M'_i$ for some $i$, then $d(x,y)\le
        s_r\le\bar s_r$. If $x\in M_i$ and $y\in M_j$ for some $i<j$,
        then $u_j\le d'(f(x),f(y))\le r$, and therefore $d(x,y)\le
        v_j\le\bar s_r$. Thus $f$ is $\bar s_r$-rough. In the same
        way, we get that $g$ is $\bar s_r$-rough.
\end{proof}

\chapter{Some classes of metric spaces}\label{c: some classes of met sps}

In this chapter, we recall some classes of metric spaces: graphs, metric spaces of coarse bounded geometry, and coarsely quasi-convex metric spaces. They contain the coarse quasi-isometric types of leaves of compact foliated spaces. We also introduce the class of coarsely quasi-symmetric metric spaces, which plays a relevant role in one of our main results, Theorem~\ref{t: coarsely q.i. leaves 3}.

\section{Graphs}\label{s: graphs}

Suppose that $M$ is the set of vertices of a connected graph $G$,
equipped with the metric $d$ defined by setting $d(x,y)$ equal to the
minimum number of consecutive edges of $M$ needed to joint $x$ and $y$
(being $0$ if $x=y$). When equipped with this metric, $M$ is called
the metric space of vertices of the connected graph $G$. Observe that
$M$ (equipped with $d$) and $G$ determine each other.  Since $d$ has
values in $\N$, it will be enough to consider (open, closed)
$r$-balls, $r$-penumbras, $r$-boundaries, $K$-nets and
$\delta$-separated sets with $r,K,\delta\in\N$.

\begin{example}\label{ex: group}
  Let \(\Gamma\) be a group and let \(S\) be a generating set for
  \(\Gamma\). The Cayley graph of \(\Gamma\) relative to \(S\),
  \(G=G(\Gamma,S)\), is the graph that has one vertex for each element
  of \(\Gamma\) and one edge joining \(\gamma_1\) and \(\gamma_2\) if
  \(\gamma_1\gamma_2^{-1} \in S\cup S^{-1}\). The vertex set of \(G\) is the set of elements of \(\Gamma\),
  and the metric induced on \(\Gamma\) is the word metric relative to
  \(S\), denoted by \(d_S\). For every \(\gamma\in\Gamma\), let \(|\gamma|\) denote its length\footnote{The minimum number of elements in \(S\cup S^{-1}\) whose product is \(\gamma\).} with respect to \(S\). Then \(d_S(\gamma_1,\gamma_2)=|\gamma_1\gamma_2^{-1}|\). It is well known that if \(\Gamma\) is
  finitely generated then all those metrics that are induced by finite
  generating sets for \(\Gamma\) are in the same Lipschitz class. The group operation of
  \(\Gamma\) induces an action on the right\footnote{The definition of Cayley graph can be modified so that left translations are isometries, by declaring that the existence one edge joining \(\gamma_1\) and \(\gamma_2\) means \(\gamma_1^{-1}\gamma_2 \in S\cup S^{-1}\).} of \(\Gamma \) on \(G\),
  and the metric \(d_S\) is invariant under this action:
  \(d_S(\gamma_1\gamma, \gamma_2\gamma)= d_S(\gamma_1,\gamma_2)\).
\end{example}

\begin{example}\label{ex: group left invariant metric}
  With the notation of Example~\ref{ex: group}, if \(\Gamma_0\) is a subgroup of \(\Gamma\) (not necessarily a
  normal subgroup), the right invariant graph structure and metric \(d_S\) on \(\Gamma\)
  induces a graph structure and the corresponding metric on the homogeneous space of right cosets
  \(\Gamma/\Gamma_0\). This graph is called \emph{Schreier coset graph}.
\end{example}

\begin{example}\label{ex: group action}
Let \(\Gamma\) be a group. Let \(S\) be a generating set for \(\Gamma\)
and let \(d_S\) be the right invariant metric induced by \(S\) on
\(\Gamma\). 

Let \(\Gamma\) act on the left on a space \(X\) and denote
the action by \(x\mapsto \gamma x\). There is a natural bijection
between the orbit, \(\Gamma(x)\) of a point \(x\), consisting of all the
points \(\gamma(x)\), \(\gamma\in \Gamma\), and the space of right cosets
of \(\Gamma\) relative to the subgroup of \(\Gamma\) that fixes \(x\):
If \(\Gamma_x\) is the set of \(\gamma\in \Gamma\) such that
\(x\gamma=x\), then the assignment \([\gamma]\in \Gamma/\Gamma_x
\mapsto \gamma x\) is independent of the
representative of the class \([\gamma]\), for if \(\gamma' \in
[\gamma]\), then \(\gamma^{-1}\gamma'\) fixes \(x\), so
\(\gamma'x=\gamma x\).

The metric \(d_S\) induced by \(S\) on \(\Gamma\) in turn induces a
metric on the orbit \(\Gamma(x)\). This metric is actually independent
of the chosen point in the orbit: if \(y\in X\) is in the same orbit
as \(x\), then \(\gamma_y x=y\) for some \(\gamma_y\in \Gamma\). The
stabilizer subgroup of \(y\) is related to that of \(x\) via
the conjugation \(\Gamma_y \gamma_y=\gamma_y\Gamma_x\), and so the
 homogeneous spaces \(\Gamma/\Gamma_y\) and \(\Gamma/\Gamma_x\) are
 isometric (with the metric induced by \(d_S\)) via right
 multiplication by \(\gamma_y\).
\end{example}

 The reverse inclusion of~\eqref{Pen_M(Pen_M(S,r),s)} also holds
with natural numbers\footnote{On complete path metric spaces, these
  equality holds for all \(r,s\ge0\).}:
\begin{equation}\label{Pen(S,r+s), graph}
  \Pen(S,r+s)=\Pen(\Pen(S,r),s)
\end{equation}
for \(S\subset M\) and \(r,s\in\N\); more precisely,
\begin{equation}\label{Pen(S,r+s) sm S, graph}
  \Pen(S,r+s)\sm S=\Pen(\Pen(S,r)\sm S,s)\sm S\;.
\end{equation}
Note that~\eqref{Pen(S,r+s), graph} follows from~\eqref{Pen(S,r+s) sm
  S, graph} and~\eqref{Pen_M(Pen_M(S,r),s)}. The inclusion
``\(\supset\)'' of~\eqref{Pen(S,r+s) sm S, graph} is given
by~\eqref{Pen_M(Pen_M(S,r),s)}. To prove the reverse inclusion, assume
that \(r,s>0\) (if one of \(r=0\) or \(s=0\) there is nothing to prove).
If \(x\in\Pen(S,r+s)\sm S\), then there is a finite sequence
\(z_0,z_1,\dots,z_k=x\) in \(M\) with \(z_0\in S\), \(k\le r+s\), and
\(d(z_{l-1},z_l)=1\) for all \(l\in\{1,\dots,k\}\), and furthermore that
\(z_l\in M\sm S\) for \(l\ge 1\). If \(k\le r\), then \(x\in\Pen(S,r)\sm
S\subset\Pen(\Pen(S,r)\sm S,s)\sm S\); and if 
\(k>r\), then \(z_r\in\Pen(S,r)\) and \(d(x,z_r)\le k-r\le s\). This implies
that \(x\in\Pen(\Pen(S,r)\sm S,s)\sm S\), which  concludes  the proof
of~\eqref{Pen(S,r+s) sm S, graph}.

\begin{lemma}\label{l:penumbra, graph}
  Let \(A\) be a \(K\)-net in \(M\) for some \(K\in\N\). Then, for every
  \(S\subset M\) and all natural \(r\ge K\), the set \(\Pen(S,r)\cap A\) is
  a \(2K\)-net in \(\Pen(S,r)\); in particular, for every \(x\in M\), \(\ol
  B(x,r)\cap A\) is a \(2K\)-net in \(\ol B(x,r)\).
\end{lemma}

\begin{proof}
  For each \(x\in\Pen(S,r)\) there is \(y\in\Pen(S,r-K)\) with \(d(x,y)\le
  K\) by~\eqref{Pen(S,r+s), graph}, and there is \(z\in A\) with
  \(d(y,z)\le K\) since \(A\) is a \(K\)-net. So \(z\in\Pen(S,r)\cap A\)
  by~\eqref{Pen(S,r+s), graph}, and \(d(x,z)\le2K\) by the triangle
  inequality.
\end{proof}

Recall that the \emph{degree} \index{degree} (or \emph{valence}) \index{valence} of a vertex \(x\) of \(G\) is the number of edges that meet at \(x\), which is denoted by $\deg(x)=\deg_G(x)$. Suppose from now on that \(G\) is of \emph{finite type} \index{finite type} in the sense that there is some \(K\in\N\) such that each vertex
is of degree \(\le K\), and assume that \(K\ge 2\) (if \(K=0\), then
\(G\) has just one vertex and no edges; if \(K=1\), then \(G\) has at most
two vertices and one edge joining them). For each \(r\in\N\), let
\begin{equation}\label{Lambda_K,r}
  \Lambda_{K,r}=
  \begin{cases}
    1+K\frac{(K-1)^r-1}{K-2} & \text{if \(K>2\)}\\
    1+2r & \text{if \(K=2\)}\;.
  \end{cases}
\end{equation}
Then
\begin{equation}\label{|ol B(x,r)|}
  |\ol{B}(x,r)|\le1+K+K(K-1)+\dots+K(K-1)^{r-1}=\Lambda_{K,r}
\end{equation}
for all \(x\in M\) and \(r\in\N\). Therefore
\begin{equation}\label{|Pen(S,r)|}
  |\Pen(S,r)|\le \Lambda_{K,r}\,|S|
\end{equation}
for any \(S\subset M\) and \(r\in\N\).

The growth type of the function \(r\mapsto|B(x,r)|\) (\(r\ge1\)) is
independent of the choice of \(x\in M\), and is called the \emph{growth type}
of \(M\) (as set of vertices of a connected graph\footnote{This
  definition is indeed valid for any metric space with finite
  balls.}), or of \(G\) (as graph).

The \emph{boundary} \index{boundary} of a subset \(S\subset M\) is \(\partial
S=\partial_1S\). The sets \(\partial S\cap S\) and \(\partial S\sm S\) are
respectively called \emph{inner} and \emph{outer boundaries}. \index{boundary!inner} \index{boundary!outer} Since
\[
\partial S\cap S\subset\Pen(\partial S\sm S,1)\;,\textrm{ and } \; \partial S\sm S\subset\Pen(\partial S\cap S,1)\;,
\]
it follows by~\eqref{|Pen(S,r)|} that
\begin{equation}\label{inner/outer boundaries}
  \frac{1}{\Lambda_{K,1}}\,|\partial S\sm S|\le|\partial S\cap S|\le\Lambda_{K,1}\,|\partial S\sm S|\;.
\end{equation}

\begin{lemma}\label{l: partial S, graph}
	\(\partial_rS=\Pen(\partial S,r-1)\) for all \(r\in\Z^+\).
\end{lemma}

\begin{proof}
  By~\eqref{Pen(S,r+s) sm S, graph},
  \begin{align*}
    \partial_rS&=\Pen(S,r)\cap\Pen(M\sm S,r)\\
    &=((\Pen(S,r)\cap\Pen(M\sm S,r))\sm S)\\
    &\phantom{=\text{}}\text{}\cup(\Pen(S,r)\cap\Pen(M\sm S,r)\cap S)\\
    &=(\Pen(S,r)\sm S)\cup(\Pen(M\sm S,r)\cap S)\\
    &=(\Pen(\Pen(S,1)\sm S,r-1)\sm S)\\
    &\phantom{=\text{}}\text{}\cup(\Pen(\Pen(M\sm S,1)\cap S,r-1)\cap S)\\
    &=(\Pen(\partial S\sm S,r-1)\sm S)\cup(\Pen(\partial S\cap S,r-1)\cap S)\\
    &\subset\Pen(\partial S,r-1)\;.
  \end{align*}
  On the other hand, by~\eqref{Pen_M(S cap T,r)}
  and~\eqref{Pen(S,r+s), graph},
  \begin{multline*}
    \Pen(\partial S,r-1)=\Pen(\Pen(S,1)\cap\Pen(M\sm S,1),r-1)\\
    \subset\Pen(\Pen(S,1),r-1)\cap\Pen(\Pen(M\sm S,1),r-1)\\
    =\Pen(S,r)\cap\Pen(M\sm S,r)=\partial_rS\;.
  \end{multline*}
\end{proof}

Lemma~\ref{l: partial S, graph} and~\eqref{|Pen(S,r)|} give
\begin{equation}\label{|partial_rS|}
  |\partial_rS|\le \Lambda_{K,r-1}\,|\partial S|\;.
\end{equation}

The metric space \(M\) (as set of vertices of a connected graph) is called
\emph{F{\o}lner} \index{F{\o}lner!graph} if it contains a sequence of finite subsets \(S_n\)
such that \(|\partial S_n|/|S_n|\to0\) as \(n\to\infty\).

Let \(\{G_i\}\) be a class of connected graphs, and let
\(\{M_i\}\) be the class of metric spaces defined by their vertices.
Then \(\{G_i\}\) is said to be of
\emph{equi-finite type} \index{equi-finite type} if there is \(K\in\N\) such that each
vertex at each \(G_i\) meets at most \(K\) edges. The class \(\{M_i\}\) is called
\emph{equi-F{\o}lner} \index{equi-F{\o}lner} if each \(M_i\) has a F{\o}lner sequence \(S_{i,n}\) such
that, for some \(a>0\), \(|\partial S_{i,n}|/|S_{i,n}|\le
a\,|\partial S_{j,n}|/|S_{j,n}|\) for all \(i\) and \(j\).

\section{Metric spaces of coarse bounded geometry}\label{s: coarse bounded geometry}

\begin{defn}[Block-Weinberger {\cite{BlockWeinberger1992}}]\label{d:
    coarse bounded geometry}
  A \emph{quasi-lattice} \index{quasi-lattice} \(\Gamma\) of \(M\) is an \(R\)-net of \(M\) for
  some \(R\ge0\) such that \(|\Gamma\cap \ol{B}(x,r)|\le Q_r\) for every
  \(x\in M\), where \(r\mapsto Q_r\) (\(r,Q_r\ge0\)) is a mapping
  independent of \(x\); the term \emph{\((R,Q_r)\)-quasi-lattice} may be
  also used. It is said that \(M\) is of \emph{coarse bounded geometry} \index{bounded geometry!coarse}
  if it has an \((R,Q_r)\)-quasi-lattice for some \((R,Q_r)\); in this
  case, \((R,Q_r)\) is called a \emph{coarse bound} \index{coarse bound} of \(M\).
\end{defn}

\begin{example}[Block-Weinberger {\cite{BlockWeinberger1992}}]\label{ex: coarse bounded geometry} 
	\begin{enumerate}[(i)]
	
        \item\label{i: graph of coarse bd geom} If \(M\) is the metric
          space of vertices of any connected graph \(G\), then \(M\) is of
          coarse bounded geometry if and only if \(G\) is of finite type; indeed, if each vertex meets at most \(K\) edges, then
          \(M\) is a \((0,\Lambda_{K,r})\)-quasi-lattice in itself
          by~\eqref{|ol B(x,r)|}, where \(\Lambda_{K,r}\) is given
          by~\eqref{Lambda_K,r}.
		
        \item\label{i: Riem mfd of coarse bd geom} If \(M\) is a
          connected complete Riemannian manifold with a positive
          invectivity radius and whose Ricci curvature is bounded from
          below, then it is of coarse bounded geometry. Recall
            that a Riemannian manifold is said to be of \emph{bounded
              geometry} \index{bounded geometry!Riemannian manifold} if it has a positive injectivity radius and
            each covariant derivative of arbitrary order of its
            curvature tensor is uniformly bounded. Thus \(M\) is of coarse bounded geometry 
            if it is of bounded geometry.
		
	\end{enumerate}
\end{example}

\begin{defn}\label{d: equi-coarse bounded geometry}
  The class \(\{M_i\}\) is said to be of \emph{equi-coarse bounded
    geometry} \index{equi-coarse bounded geometry} when the metric
  spaces \(M_i\) are of coarse bounded geometry with a common coarse
  bound \((R,Q_r)\). In this case, a class \(\{\Gamma_i\subset M_i\}\)
  of \((R,Q_r)\)-quasi-lattices is called a class of
  \emph{equi-quasi-lattices}. \index{equi-quasi-lattices}
\end{defn}

\begin{example} If \(\{M_i\}\) is the class of metric spaces of
  vertices of a class of corresponding connected graphs \(\{G_i\}\),
  then \(\{M_i\}\) is of equi-coarse bounded geometry if and only if
  \(\{G_i\}\) is of equi-finite type. If a class of connected complete
  Riemannian manifolds have injectivity radius bounded from below by a
  common positive constant, and Ricci curvature uniformly bounded from
  below by a common constant, then they are of equi-coarse bounded
  geometry.
\end{example}
 
\begin{prop}\label{p: coarse bounded geometry and rough equivalences}
  Suppose that \(\Gamma\) is an \((R,Q_r)\)-quasi-lattice of \(M\), \(f:M\to
  M'\) is an \((s_r,t_r)\)-rough map, and that there is a rough map \(g:M'\to
  M\) such that \(g f\) and \(f g\) are \(c\)-close to \(\id_M\) and
  \(\id_{M'}\), respectively. Then \(\Gamma'=f(\Gamma)\) is an
  \((R',Q'_r)\)-quasi-lattice of \(M'\), with \(R'=c+s_R\) and
  \(Q'_r=Q_{t_{r+R'}}\).
\end{prop}

\begin{proof}
  For each \(x'\in M'\), there is \(y\in\Gamma\) such that
  \(d(g(x'),y)\le R\). Then \(y'=f(y)\in\Gamma'\) and
  \[
  d'(x',y')\le d'(x',f g(x'))+d'(f g(x'),y')\le c+s_R\;.
  \] 
  Moreover
  \[
  \Gamma'\cap\ol{B}_{M'}(x',r)\subset\ol{B}_{\Gamma'}(y',r+R')
  \subset f(\ol{B}_\Gamma(y,t_{r+R'}))
  \] 
  for all \(r\ge0\), obtaining
  \[
  |\Gamma'\cap\ol{B}_{M'}(x',r)|\le|\Gamma\cap\ol{B}_{M}(y,t_{r+R'})|\le Q_{t_{r+R'}}\;.
  \]
\end{proof}

\begin{cor}\label{c: coarse bounded geometry and large scale Lipschitz equivalences}
  If \(\Gamma\) is an \((R,Q_r)\)-quasi-lattice of \(M\), and $\phi:M\to M'$
  is a $(\lambda,b,c)$-large scale Lipschitz equivalence, then
  $\Gamma'=\phi(\Gamma)$ is an $(R',Q'_r)$-quasi-lattice of $M'$,
  where $R'=\lambda R+b+c$ and $Q'_r=Q_{\lambda(r+R')+b+2c}$.
\end{cor}

\begin{proof}
  By Proposition~\ref{p: any large scale Lipschitz map is rough}, we
  can apply Proposition~\ref{p: coarse bounded geometry and rough
    equivalences} with $s_r=\lambda r+b$ and $t_r=\lambda r+b+2c$.
\end{proof}

\begin{cor}\label{c: coarse bounded geometry and coarse quasi-isometries}
  If $M$ is of coarse bounded geometry with coarse bound $(R,Q_r)$,
  and there is a $(K,C)$-coarse quasi-isometry of $M$ to $M'$, then
  $M'$ is of coarse bounded geometry with coarse bound $(R',Q'_r)$,
  where
  $$
  R'=K+CR+2CK\;,\quad Q'_r=Q_{C(r+R')+2CK+2K}\;.
  $$ 
\end{cor}

\begin{proof}
  By Proposition~\ref{p: large scale Lipschitz extensions}, there is a
  $(C,2CK,K)$-large scale Lipschitz equivalence $\phi:M\to M'$. Let
  $\Gamma$ be an $(R,Q_r)$-quasi-lattice of $M$. Then $\phi(\Gamma)$
  is an $(R',Q'_r)$-quasi-lattice of $M'$ by Corollary~\ref{c: coarse
    bounded geometry and large scale Lipschitz equivalences}
\end{proof}

\begin{rem}
  According to Corollary~\ref{c: coarse bounded geometry and coarse
    quasi-isometries}, equi-coarse bounded geometry is preserved by
  equi-coarse quasi-isometries.
\end{rem}

\begin{rem}\label{r: coarse bounded geometry and coarse quasi-isometries}
  Another version of Corollary~\ref{c: coarse bounded geometry and
    coarse quasi-isometries}, with $R'=6C\,\max\{R,K\}$ and
  $Q'_r=Q_{C(r+R')}$, can be proved without passing to large scale
  equivalences.
\end{rem}

\section{Coarsely quasi-symmetric metric spaces}\label{s: coarsely quasi-symmetric}

\begin{defn}\label{d:transitive families}
  Let $\TT$ be a set of coarsely quasi-isometric
  transformations of $M$, let $\Phi$ be a set of maps $M\to M$, and
  let $R\ge0$. The set $\TT$ is called:
  \begin{itemize}
		
  \item \emph{transitive} \index{transitive} when, for all $x,y\in M$,
    there is some $f\in\TT$ such that $x\in\dom f$, $y\in\im f$ and
    $f(x)=y$; and
		
  \item \emph{$R$-quasi-transitive} \index{$R$-quasi-transitive} if, for all $x,y\in M$, there is
    some $f\in\TT$ and $z\in\dom f$ such that $d(x,z)\le R$ and
    $d(f(z),y)\le R$.
		
  \end{itemize}
  The set $\Phi$ is called:
  \begin{itemize}
		
  \item \emph{transitive} when, for all $x,y\in M$, there is some
    $\phi\in\Phi$ with $\phi(x)=y$; and
		
  \item \emph{$R$-quasi-transitive} if, for all $x,y\in M$, some
    $\phi\in\Phi$ satisfies $d(f(x),y)\le R$.
		
  \end{itemize}
\end{defn}

\begin{defn}\label{d: coarsely quasi-symmetric}
  A metric space $M$ is called \emph{coarsely quasi-symmetric}
  \index{coarsely quasi-symmetric} if there is a transitive set of
  equi-coarsely quasi-isometric transformations of $M$.
\end{defn}

\begin{defn}\label{d: equi-coarsely quasi-symmetric}
  A class $\{M_i\}$ is called \emph{equi-coarsely quasi-symmetric}
  \index{equi-coarsely quasi-symmetric} if, for some $K\ge0$ and
  $C\ge1$, there is a transitive class of $(K,C)$-coarsely
  quasi-isometric transformations of every $M_i$.
\end{defn}

\begin{lemma}\label{l: coarsely quasi-symmetric}
  The following statements are equivalent:
  \begin{enumerate}[{\rm(}i\/{\rm)}]
		
  \item\label{i: M_i is equi-coarsely quasi-symmetric} $\{M_i\}$ is
    equi-coarsely quasi-symmetric.
						
  \item\label{i: R-quasi-transitive family of (K,C)-quasi-isimetric
      transformations} For some $R,K\ge0$ and $C\ge1$, there is an
    $R$-quasi-transitive class of $(K,C)$-quasi-isometric
    transformations of each $M_i$.
			
  \item\label{i: transitive family of (lambda,b,c)-large scale
      Lipschitz transformations} For some $\lambda\ge1$ and $b,c\ge0$,
    there is a transitive class of $(\lambda,b,c)$-large scale
    Lipschitz transformations of each $M_i$.
			
  \item\label{i: R-quasi-transitive family of (lambda,b,c)-large scale
      Lipschitz transformations} For some $R,b,c\ge0$ and
    $\lambda\ge1$, there is an $R$-quasi-transitive class of
    $(\lambda,b,c)$-large scale Lipschitz transformations of each
    $M_i$.
		
  \end{enumerate}
\end{lemma}

\begin{proof}
  This follows from Propositions~\ref{p: large scale Lipschitz
    extensions},~\ref{p: restrictions of large scale Lipschitz maps}
  and~\ref{p: large scale Lipschitz, x_0 mapsto x'_0}, and
  Corollary~\ref{c: coarse quasi-isometry, x_0 mapsto x'_0}.
\end{proof}

\begin{prop}\label{p:coarsely quasi-symmetric 2}
  {\rm(}Equi-\/{\rm)}coarse quasi-symmetry is preserved by
  {\rm(}equi-\/{\rm)}coarse quasi-isometries.
\end{prop}

\begin{proof}
  Assume that there is some transitive set $\Phi$ of
  $(\lambda,b,c)$-large scale Lipschitz transformations of $M$, and
  there is a $(\lambda,b,c)$-large scale Lipschitz equivalence
  $\xi:M\to M'$. Let $\zeta:M'\to M$ be a $(\lambda,b)$-large scale
  Lipschitz map so that $\zeta\xi$ and $\xi\zeta$ are $c$-close to
  $\id_M$ and $\id_{M'}$, respectively. By Lemma~\ref{l: composite of
    large scale Lipschitz maps}, it follows that
  $\Phi':=\{\,\xi\phi\zeta\mid\phi\in\Phi\,\}$ is a family of
  $(\lambda',b',c')$-large scale Lipschitz transformations of $M'$,
  where $(\lambda',b',c')$ depends only on $(\lambda,b,c)$. For all
  $x',y'\in M'$, there is some $\phi\in\Phi$ such that
  $\phi\zeta(x')=\zeta(y')$. Thus $\phi':=\xi\phi\zeta\in\Phi'$
  satisfies
  \begin{multline*}
    d'(\phi'(x'),y')\le d'(\phi'(x'),\xi\zeta(y'))+d'(\zeta\xi(y'),y')\\
    \le\lambda\,d(\phi\zeta(x'),\zeta(y'))+b+c=b+c\;,
  \end{multline*}
  obtaining that $\Phi'$ is $(b+c)$-quasi-transitive. Hence the
  result follows from Lemma~\ref{l: coarsely quasi-symmetric} and
  Proposition~\ref{p: large scale Lipschitz extensions}.
\end{proof}

\begin{rem}
  A more involved proof can be given by using coarse composites of
  coarse quasi-isometries, whose coarse distortion is controlled by
  Proposition~\ref{p: coarse composite}.
\end{rem}

\section{Coarsely quasi-convex metric spaces}\label{s: coarsely quasi-convex}

\begin{defn}\label{d: coarsely quasi-convex} 
  A metric space, $M$, is said to be \emph{coarsely quasi-convex}
  \index{coarsely quasi-convex} if there are $a,b,c\ge0$ such that,
  for each $x,y\in M$, there exists a finite sequence of points
  $x=x_0,\dots,x_n=y$ in $M$ such that $d(x_{k-1},x_k)\le c$ for all
  $k\in\{1,\dots,n\}$, and
\[
\sum_{k=1}^nd(x_{k-1},x_k)\le a\,d(x,y)+b\;.
\]
A class of metric spaces is said to be \emph{equi-coarsely
  quasi-convex} \index{equi-coarsely quasi-convex} if all of them satisfy this condition with the same
constants $a$, $b$, and $c$.
\end{defn}

\begin{rem}
  Coarse quasi-convexity is a coarsely quasi-isometric version of the
  following condition introduced by Gromov: For each $x,y\in M$ and
  $\varepsilon>0$, there is some $z\in M$ such that
  \[
  \max\{d(x,z),d(y,z)\}<\frac{1}{2}\,d(x,y)+\varepsilon\;.
  \]
  This property may be called \emph{approximate convexity} \index{approximate convexity} because a
  subset of $\mathbf{R}^n$ satisfies it precisely when said subset has
  a convex closure. Gromov has shown that a complete metric space is a
  path metric space if and only if it is approximately convex
  \cite[Theorem 1.8]{Gromov1999}.
\end{rem}

\begin{rem}
  The property of coarse quasi-convexity is slightly weaker than the
  property of monogenicity for coarse spaces defined by metrics
  \cite{Roe2003} (monogenicity means that the condition of
  Definition~\ref{d: coarsely quasi-convex} is satisfied with $a=1$
  and $b=0$).
\end{rem}

\begin{example}
  Any class of metric spaces, each being the space of vertices of a
  connected graph, is equi-coarsely quasi-convex (they satisfy the
  condition of Definition~\ref{d: coarsely quasi-convex} with $a=1$,
  $b=0$ and $c=1$). Of course, any class of connected complete
  Riemannian manifolds is also coarsely quasi-convex since they are
  path metric spaces.
\end{example}

\begin{prop}[\'Alvarez-Candel {\cite[Theorem~3.11]{AlvarezCandel2011}}]\label{p: coarsely quasi-convex} 
  A metric space, $M$, is coarsely quasi-convex if and only if there
  exists a coarse quasi-isometry of $M$ to the metric space of
  vertices of some connected graph. A class, $\{M_i\}$, is
  equi-coarsely quasi-convex if and only if $\{M_i\}$ is equi-coarsely
  quasi-isometric to a family of metric spaces of vertices of
  connected graphs.
\end{prop}

\begin{rem}
  In \cite[Theorem~3.11]{AlvarezCandel2011}, the result was stated
  using complete path metric spaces instead of graphs, but indeed a
  graph is constructed in its proof.
\end{rem}

\begin{rem}
  Proposition~\ref{p: coarsely quasi-convex} is a coarsely
  quasi-isometric version of \cite[Proposition~2.57]{Roe2003}, which
  asserts that the monogenic coarse structures are those that are
  coarsely equivalent to path metric spaces.
\end{rem}

\begin{rem}
  As a consequence of Proposition~\ref{p: coarsely quasi-convex}, we
  get that (equi-) coarse quasi-convexity is invariant by (equi-) coarse
  quasi-isometries.
\end{rem}

\begin{prop}[\'Alvarez-Candel {\cite[Proposition~3.15]{AlvarezCandel2011}}]
\label{p:coarsely quasi-convex} 
{\rm(}Equi-\/{\rm)}uniformly expansive maps with
{\rm(}equi-\/{\rm)}coarsely convex domains are {\rm(}equi-\/{\rm)}large
scale Lipschitz.
\end{prop}

\begin{cor}[\'Alvarez-Candel {\cite[Corollary~3.16]{AlvarezCandel2011}}]
\label{c:coarsely quasi-convex} 
{\rm(}Equi-\/{\rm)}coarse equivalences with {\rm(}equi-\/{\rm)}coarsely
convex domains are {\rm(}equi-\/{\rm)}large scale Lipschitz
equivalences.
\end{cor}

\chapter{Growth of metric spaces}\label{s: growth of met sps}

In this chapter, we recall or introduce the concepts and properties about growth on metric spaces that are needed to show some of our main results, Theorems~\ref{t:growth 1}--\ref{t:liminf ...} and Corollary~\ref{c: polynomial exponential growth}. Specially relevant is the role played by growth symmetry in Theorem~\ref{t:growth 3}.

\section{Growth of non-decreasing functions}\label{s: growth of non-decreasing functions}

Given non-decreasing functions\footnote{The usual definition of growth
  type uses functions $\Z^+\to\R^+$, but functions $\R^+\to\R^+$
  give rise to an equivalent concept.}  $u,v:\R^+\to\R^+$, the function \(u\) is said to be \emph{dominated} by
the function $v$, written $u\preccurlyeq v$, if there are $a,b\ge1$ and $c>0$
such that $u(r)\le a\,v(br)$ for all $r\ge c$. For all $a,b\ge1$, $c>0$ and $e\ge0$, we have
	\begin{multline}\label{u(r) le a v(br+e)}
		b'>b\quad \&\quad u(r)\le a\,v(br+e)\ \forall r\ge c\\
		\Longrightarrow\;u(r)\le a\,v(b'r)\ \forall r\ge\max\left\{c,\frac{e}{b'-b}\right\}\;.
	\end{multline}
        If $u\preccurlyeq v\preccurlyeq u$, then $u$ and $v$ represent
        the same \emph{growth type} \index{growth!type} (or \emph{growth class}) \index{growth!class} or have
        \emph{equivalent growth}; \index{growth!equivalent} this is an
        equivalence relation and ``$\preccurlyeq$'' defines a partial
        order relation between growth types called
        \emph{domination}. \index{growth!domination} The growth type of $u$ may be denoted by $\gr(u)$, and we may write $\gr(u)\le\gr(v)$ when $u\preccurlyeq v$; then $\gr(u)<\gr(v)$ has the obvious meaning. For a class of
        pairs of non-decreasing functions $\R^+\to\R^+$,
        \emph{equi-domination} \index{growth!equi-domination} means
        that all of those pairs satisfy the above condition of
        domination with the same constants $a$, $b$ and $c$. A class
        of non-decreasing functions $\R^+\to\R^+$ will be said to have
        \emph{equi-equivalent growth} \index{growth!equi-equivalent}
        if they equi-dominate one another.

For non-decreasing functions $u,v:\R^+\to\R^+$, and constants $a,b\ge1$ and $c>0$, if $u(r)\le a\,v(br)$ for all $r\ge c$, then
	\begin{align}
		\limsup_{r\to\infty}\frac{\log u(r)}{\log r}&\le\limsup_{r\to\infty}\frac{\log v(r)}{\log r}\;,
		\label{limsup_r to infty frac log u(r) log r}\\
  		\liminf_{r\to\infty}\frac{\log u(r)}{\log r}&\le\liminf_{r\to\infty}\frac{\log v(r)}{\log r}\;,
		\label{liminf_r to infty frac log u(r) log r}\\
  		\liminf_{r\to\infty}\frac{\log u(r)}{r}&\le b\,\liminf_{r\to\infty}\frac{\log v(r)}{r}\;,
		\label{liminf_r to infty frac log u(r) r}\\
  		\limsup_{r\to\infty}\frac{\log u(r)}{r}&\le b\,\limsup_{r\to\infty}\frac{\log v(r)}{r}\;.
		\label{limsup_r to infty frac log u(r) r}
	\end{align}
Thus it makes sense to say that the growth type of $u$ is:
	\begin{itemize}
	
		\item \emph{exactly polynomial} \index{growth!type!exactly polynomial} of \emph{degree} $d\in\N$ if it is the growth type of the function $r\mapsto r^d$;
		
		\item \emph{polynomial} \index{growth!type!polynomial} if it is dominated by a polynomial growth of some exact degree; and
		
		\item \emph{exponential}  \index{growth!type!exponential} if it is the growth type of the function $r\mapsto e^r$, which is the same as the growth of $r\mapsto a^r$ for any $a>1$.
		
	\end{itemize}
Observe that the growth type of $u$ is:
	\begin{itemize}
	
		\item polynomial if and only if $\limsup_{r\to\infty}\frac{\log u(r)}{\log r}<\infty$; and
		
		\item exponential if and only if $0<\liminf_{r\to\infty}\frac{\log u(r)}{r}<\infty$.
		
	\end{itemize}
It is also said that the growth type of $u$ is:
	\begin{itemize}
	
		\item  \emph{quasi-polynomial}\footnote{This property is sometimes called \emph{subexponential}.} \index{growth!type!quasi-polynomial} if $\limsup_{r\to\infty}\frac{\log u(r)}{r}\le0$;
		
		\item \emph{quasi-exponential} \index{growth!type!quasi-exponential} if $0<\limsup_{r\to\infty}\frac{\log u(r)}{r}<\infty$; and
		
		\item \emph{pseudo-quasi-polynomial} \index{growth!type!pseudo-quasi-polynomial} if $\liminf_{r\to\infty}\frac{\log u(r)}{\log r}<\infty$.
		
	\end{itemize}

\section{Growth of metric spaces}\label{s: growth of met sps}

Suppose that $M$ and $M'$ are of coarse bounded geometry, and
$\{M_i\}$ and $\{M'_i\}$ are of equi-coarse bounded geometry.

For a quasi-lattice $\Gamma$ of $M$ and $x\in\Gamma$, the function
$r\mapsto v_\Gamma(x,r)=|B_\Gamma(x,r)|$ ($r\ge1$) is called the \emph{growth function} \index{growth!function} of $M$ induced by $\Gamma$ and $x$.

\begin{prop}\label{p: bounded geometry and growth}
  For $k\in\{1,2\}$, let $\Gamma_k$ be an $(R_k,Q^k_r)$-quasi-lattice
  of $M$, and $x_k\in\Gamma_k$. Take any $\delta\ge d(x_1,x_2)$. Then,
  for all $r\ge1$,
  \[
  v_{\Gamma_1}(x_1,r)\le Q^1_{R_2}\, v_{\Gamma_2}(x_2,r+\delta+R_2)\;.
  \]
\end{prop}

\begin{proof}
  Since $B_M(x_1,r)\subset B_M(x_2,r+\delta)$ for all $r\ge1$, and
  $\Gamma_2$ is an $R_2$-net, then
  \[
  B_{\Gamma_1}(x_1,r) \subset\bigcup_{y\in
    B_{\Gamma_2}(x_2,r+\delta+R_2)}\ol{B}_M(y,R_2)\cap\Gamma_1\;,
  \]
  which implies the stated inequality.
\end{proof}

The following definitions are justified by Proposition~\ref{p: bounded
  geometry and growth}.

\begin{defn}\label{d: bounded geometry and growth}
  The \emph{growth type} \index{growth!type} (or \emph{growth class}) \index{growth!class} of $M$ is the growth type of $r\mapsto
  v_\Gamma(x,r)$ for any quasi-lattice $\Gamma$ of $M$ and
  $x\in\Gamma$. We may also say that $M$ and $M'$ have \emph{equivalent
    growth} \index{equivalent growth} \index{growth!equivalent} when they have the same growth type. The notation $\gr(M)$ may be used for the growth type of $M$.
\end{defn}

\begin{defn}\label{d: equi-bounded geometry and equi-equivalent growth}
  Two classes  of metric spaces,  $\{M_i\}$ and $\{M'_i\}$, have \emph{equi-equivalent growth} \index{growth!equi-equivalent} when there are equi-quasi-latices,
  $\Gamma_i\subset M_i$ and $\Gamma'_i\subset M'_i$, and there are
  points, $x_i\in M_i$ and $x'_i\in M'_i$, such that $r\mapsto
  v_{\Gamma_i}(x_i,r)$ and $r\mapsto v_{\Gamma'_i}(x'_i,r)$ have
  equi-equivalent growth.
\end{defn}

\begin{rem}\label{r: bounded geometry and growth}
	\begin{enumerate}[(i)]
	
		\item\label{i: polynomial, ..., metric space} According to Section~\ref{s: growth of non-decreasing functions} and Definition~\ref{d: bounded geometry and growth}, the following notions make sense for the growth type of $M$: exactly polynomial of degree $d\in\N$, polynomial, exponential, pseudo-quasi-polynomial, subexponential or quasi-polynomial, quasi-exponential and non-exponential.
		
		\item\label{i: limsup_r to infty frac log v_Gamma(x,r) log r} For any quasi-lattice $\Gamma$ of $M$ and $x\in M$, the quantities
			\[
  				\limsup_{r\to\infty}\frac{\log v_\Gamma(x,r)}{\log r}\;,\qquad
				\liminf_{r\to\infty}\frac{\log v_\Gamma(x,r)}{\log r}
			\]
		depend only on the growth type of $M$ by~\eqref{limsup_r to infty frac log u(r) log r} and~\eqref{liminf_r to infty frac log u(r) log r}.
		
		\item\label{i: liminf_r to infty frac log v_Gamma(x,r) r} With the notation of Proposition~\ref{p: bounded geometry and growth},
			\begin{align*}
  				\liminf_{r\to\infty}\frac{\log v_{\Gamma_1}(x_1,r)}{r}
				&\le b\,\liminf_{r\to\infty}\frac{\log v_{\Gamma_2}(x_2,r)}{r}\;,\\
  				\limsup_{r\to\infty}\frac{\log v_{\Gamma_1}(x_1,r)}{r}
				&\le b\,\limsup_{r\to\infty}\frac{\log v_{\Gamma_2}(x_2,r)}{r}\;,
			\end{align*}
		for any $b>Q^1_{R_2}$ by~\eqref{u(r) le a v(br+e)},~\eqref{liminf_r to infty frac log u(r) r} and~\eqref{limsup_r to infty frac log u(r) r}.
		
		\end{enumerate}
\end{rem}

\begin{example}
  If $M$ is the metric space of vertices of a connected graph $G$ of
  finite type, then $M$ is a quasi-lattice in itself, and therefore
  its growth type as metric space of coarse bounded geometry equals
  its growth type as metric space of vertices of a connected graph. 
\end{example}

\begin{example}  
  Let $M$ be a connected complete Riemannian manifold. Recall that 
  the \emph{growth type} of $M$, as Riemannian manifold, is the
  growth type of $r\mapsto\vol B(x,r)$ for any $x\in M$. If
  $M$ is of bounded geometry,
  then its growth type as metric space of coarse bounded geometry
  equals its growth type as Riemannian manifold.
\end{example}

\begin{prop}\label{p: growth type and large scale Lipschitz equivalences}
  Let $\Gamma$ be a quasi-lattice of $M$, $\phi:M\to M'$ a
  $(\lambda,b,c)$-large scale Lipschitz equivalence, and
  $x\in\Gamma$. Let $\Gamma'=\phi(\Gamma)$ and $x'=\phi(x)$. Then, for
  all $r\ge1$,
  $$
  v_{\Gamma'}(x',r)\le v_\Gamma(x,\lambda r+b+2c)\;.
  $$
\end{prop}

\begin{proof}
  The result follows since
  \[
  B_{\Gamma'}(x',r)\subset\phi(B_\Gamma(x,\lambda r+b+2c))
  \]
  for all $r\ge1$ because $\phi$ is $(\lambda,b+2c)$-large scale
  bi-Lipschitz by Lemma~\ref{l: any large scale Lipschitz equivalence
    is large scale bi-Lipschitz}.
\end{proof}

\begin{rem}
  In Proposition~\ref{p: growth type and large scale Lipschitz
    equivalences}, $\Gamma'$ is a quasi-lattice of $M'$ by
  Corollary~\ref{c: coarse bounded geometry and large scale Lipschitz
    equivalences}.
\end{rem}

\begin{cor}\label{c: growth type and large scale Lipschitz equivalences}
  Let $\Gamma$ and $\Gamma'$ be $(R,Q_r)$-quasi-lattices of $M$ and
  $M'$, respectively, and let $\phi:M\to M'$ be a
  $(\lambda,b,c)$-large scale Lipschitz equivalence. Take
  $x\in\Gamma$, $x'\in\Gamma'$ and $\delta\ge d'(\phi(x),x')$. Then
  $$
  v_{\Gamma'}(x',r)\le p\, v_\Gamma(x,\lambda r+q)
  $$
  for all $r\ge1$, where 
  $$
  p=Q_{\lambda R+b+c}\;,\quad q=\lambda(\delta+\lambda R+b+c)+b+2c\;.
  $$
\end{cor}

\begin{proof}
  This is a direct consequence of Corollary~\ref{c: coarse bounded
    geometry and large scale Lipschitz equivalences} and
  Propositions~\ref{p: growth type and large scale Lipschitz
    equivalences} and~\ref{p: bounded geometry and growth}.
\end{proof}

\begin{cor}\label{c: growth type and coarse quasi-isometries}
  Let $\Gamma$ and $\Gamma'$ be $(R,Q_r)$-quasi-lattices of $M$ and
  $M'$, respectively, and let $f:A\to A'$ be a $(K,C)$-coarse
  quasi-isometry of $M$ to $M'$. Take $x\in\Gamma$, $x'\in\Gamma'$,
  $y\in A$, $\delta\ge d(x,y)$ and $\delta'\ge d'(x',f(y))$. Then
  \[
  v_{\Gamma'}(x,r)\le p\, v_\Gamma(x,Cr+q)
  \]
  for all $r\ge1$, where
		$$
			p=Q_{CR+2CK+K}\;,\quad q=C(C\delta+4CK+2K+\delta'+CR)+2CK+2K\;.
		$$
\end{cor}

\begin{proof}
  By Proposition~\ref{p: large scale Lipschitz extensions}, $f$ is
  induced by a $(C,2CK,K)$-large scale Lipschitz equivalence
  $\phi:M\to M'$. Moreover
  \begin{multline*}
    d'(\phi(x),x')\le d'(\phi(x),\phi(y))+d'(\phi(y),x')\\
    \le C\,d(x,y)+2CK+\delta'\le C\delta+2CK+\delta'\;.
  \end{multline*}
  Then the result follows from Corollary~\ref{c: growth type and large
    scale Lipschitz equivalences}.
\end{proof}

\begin{rem}
  According to Corollary~\ref{c: growth type and coarse
    quasi-isometries}, (equi-) coarsely quasi-isometric metric spaces
  of (equi-) coarse bounded geometry have (equi-) equivalent growth.
\end{rem}

\begin{rem}\label{r: growth type}
  The following version of Corollary~\ref{c: growth type and coarse
    quasi-isometries} can be proved without using large scale
  Lipschitz equivalences, but using instead coarse composites
  (Proposition~\ref{p: coarse composite}):  With the hypothesis of
  Corollary~\ref{c: growth type and coarse quasi-isometries}, if 
  $d(x,y)\le2K^*$ and $d'(x',f(y)\le2K^*(5C+1)$, where
  $K^*=\max\{R,K\}$, then
  $$
  v_\Gamma(x_1,r)\le Q_{\widehat{K}}\cdot v_{\Gamma'}(x'_2,250\,C^4r),
  $$
  for all $r\ge \widehat{K}$, where
  $$
  \widehat{K}=25\,K^*(5C+1)C^2+5K^*C\;.
  $$
\end{rem}

\section{Growth symmetry}\label{s: growth symmetry}

\begin{defn}\label{d: growth symmetry}
  A metric space, $M$, is called \emph{growth symmetric} \index{growth!symmetric} if there is
  a quasi-lattice $\Gamma$ in $M$ so that the growth functions
  $r\mapsto v_\Gamma(x,r)$, for all $x\in\Gamma$, equi-dominate each
  other.
\end{defn}

\begin{rem}\label{r: growth symmetry}
  \begin{enumerate}[(i)]
  		
  \item\label{i: growth symmetry is independent of the choice of
      Gamma} Definition~\ref{d: growth symmetry} is independent of the
    choice of $\Gamma$ by Proposition~\ref{p: bounded geometry and
      growth}.
		
  \item\label{i: equi-domination gives growth symmetry} From
    Proposition~\ref{p: bounded geometry and growth}, it also follows
    that, given quasi-lattices $\Gamma_1$ and $\Gamma_2$ of $M$, if
    all growth functions $r\mapsto v_{\Gamma_1}(x,r)$, with
    $x\in\Gamma_1$, equi-dominate all growth functions $r\mapsto
    v_{\Gamma_2}(y,r)$, with $y\in\Gamma_2$, then $M$ is
    growth symmetric.
		
  \end{enumerate}
\end{rem}

\begin{defn}\label{d: equi-growth symmetry}
  A class $\{M_i\}$ is called \emph{equi-growth symmetric} \index{equi-growth symmetric} if there
  are equi-quasi-lattices $\Gamma_i\subset M_i$ so that the growth
  functions $r\mapsto v_{\Gamma_i}(x,r)$, for all $i$ and
  $x\in\Gamma_i$, equi-dominate one another.
\end{defn}

\begin{prop}\label{p: growth symmetry}
  {\rm(}Equi-\/{\rm)}large scale Lipschitz equivalences preserve
  {\rm(}equi-\/{\rm)}growth symmetry.
\end{prop}

\begin{proof}
  Let $\Gamma$ and $\Gamma'$ be $(R,Q_r)$-quasi-lattices of $M$ and
  $M'$, respectively. Assume that there are $a,b,c\ge1$ such that
  $v_\Gamma(x,r)\le a\, v_\Gamma(y,br)$ for all $r\ge c$ and
  $x,y\in\Gamma$. Let $\phi:M\to M'$ and $\psi:M'\to M$ be
  $(\lambda,b,c)$-large scale Lipschitz equivalences so that
  $\psi\phi$ and $\phi\psi$ are $c$-close to $\id_M$ and $\id_{M'}$,
  respectively. For all $x',y'\in\Gamma'$, there are $x,y\in\Gamma$
  such that $d'(\phi(x),x')\le R$ and $d(y,\psi(y'))\le R$. Let $p$
  and $q$ be the constants defined in Corollary~\ref{c: growth type
    and large scale Lipschitz equivalences} with $\delta=R$. Applying
  Corollary~\ref{c: growth type and large scale Lipschitz
    equivalences} to $\phi$ and $\psi$, it follows that
  \[
  v_{\Gamma'}(x',r)\le p\, v_\Gamma(x,\lambda r+q) \le ap\,
  v_\Gamma(y,b(\lambda r+q)) \le ap^2\, v_{\Gamma'}(y',\lambda
  b(\lambda r+q)+q)
  \]
  if $r\ge\max\{1,\frac{c-q}{\lambda}\}$.
\end{proof}

\begin{cor}\label{c: growth symmetry}
  {\rm(}Equi-\/{\rm)}coarse quasi-isometries preserve
  {\rm(}equi-\/{\rm)}growth symmetry.
\end{cor}

\begin{proof}
  This follows from Propositions~\ref{p: growth symmetry} and~\ref{p:
    large scale Lipschitz extensions}.
\end{proof}

\chapter{Amenability of metric spaces}\label{c: amenability of met sps}

In this chapter, we recall or introduce the concepts and properties about amenability on metric spaces. They needed to show some of our main result Theorems~\ref{t: amenable}, where equi-amenability and joint amenability are relevant conditions.

\section{Amenability}\label{s: amenability}

Suppose that $M$ and $M'$ are of coarse bounded geometry, and
$\{M_i\}$ and $\{M'_i\}$ are of equi-coarse bounded geometry.

\begin{defn}[Block-Weinberger \cite{BlockWeinberger1992}]\label{d: amenable}
  A metric space, $M$, is called \emph{amenable} \index{amenable} if it has a
  quasi-lattice $\Gamma$ and a sequence of finite subsets
  $S_n\subset\Gamma$ such that $|\partial^\Gamma_rS_n|/|S_n|\to0$ as
  $n\to\infty$ for each $r>0$. Such a sequence $S_n$ is called a \emph{F{\o}lner sequence} in $\Gamma$.
\end{defn}

\begin{example}
  If $M$ is the metric space of vertices of some connected graph of
  finite type, then \(M\) is amenable if and only it is F{\o}lner as
  metric space of vertices of some connected graph
  (cf.~\eqref{|partial_rS|}). 
\end{example}

\begin{example}
  Let $M$ is a connected complete Riemannian manifold. Recall that a Riemannian
    manifold is called \emph{F{\o}lner} \index{F{\o}lner!Riemannian manifold} if it has a sequence of
    smooth compact domains $\Omega_n$ such that
    $\vol\partial\Omega_n/\vol\Omega_n\to0$ as $n\to\infty$.
  If $M$ is of bounded geometry, then it is amenable as
  metric space of coarse bounded geometry if and only if it is
  F{\o}lner as Riemannian manifold.
\end{example}

\begin{defn}\label{d: (weakly) equi-amenable}
  $\{M_i\}$ is called \emph{weakly equi-amenable} \index{weakly equi-amenable} when
  \begin{itemize}
    
  \item there are equi-quasi-lattices $\Gamma_i\subset M_i$;
    
  \item there are subsets $S_{i,m,n}\subset\Gamma_i$ ($m,n\in\N$);
    
  \item there is $a\ge1$, and mappings $r\mapsto p_r$ ($r,p_r>0$) and $t\mapsto q_t$ ($t,q_t\ge0$); and
			
  \item there is a nonempty subset $\NN_{i,j,m,n,t}\subset\N$ for every
    $i$, $j$, $m$, $n$ and $t$;
		
  \end{itemize}
  such that:
  \begin{itemize}
    
  \item for each $i$ and $m$, the sequence $S_{i,m,n}$ is F{\o}lner in
    $\Gamma_i$; and,
			
  \item for each $i$, $m$ and $t$, and any F{\o}lner sequence $X_n$ of
    $\Gamma_i$ with $S_{i,m,n}\subset
    X_n\subset\Pen_{\Gamma_i}(S_{i,m,n},t)$, there is some
    $Y_{i,j,m',n}\subset\Gamma_j$ for all $j$, $n$ and
    $m'\in\NN_{i,j,m,n,t}$ so that
    \begin{gather}
      S_{j,m',n}\subset Y_{i,j,m',n}\subset\Pen_{\Gamma_i}(S_{j,m',n},q_t)\;,
      \label{S_j,m',n subset Y_i,j,m',n subset Pen_Gamma_i(S_j,m',n,q_t)}\\
      \frac{|\partial^{\Gamma_j}_rY_{i,j,m',n}|}{|Y_{i,j,m',n}|}
      \le a\,\frac{|\partial^{\Gamma_i}_{p_r}X_n|}{|X_n|}\;,
      \label{(weakly) equi-amenable}
    \end{gather}
    for all $r$. 
    
  \end{itemize}
  If moreover the mappings $r\mapsto p_r$ and $t\mapsto q_t$ can be
  chosen to be affine, then $\{M_i\}$ is called \index{equi-amenable} \emph{equi-amenable}.
\end{defn}

\begin{prop}\label{p: equi-rough equivalences preserve weak equi-amenability}
  The following are true:
  \begin{enumerate}[{\rm(}i\/{\rm)}]
    
  \item\label{i: rough equivalences preserve weak amenability}
    {\rm(}Equi-\/{\rm)}rough equivalences preserve {\rm(}weak
    equi-\/{\rm)}amenability.
    
  \item\label{i: coarse quasi-isometries preserve amenability}
    {\rm(}Equi-\/{\rm)}coarse quasi-isometries preserve
    {\rm(}equi-\/{\rm)}amenability.
		
  \end{enumerate}
\end{prop}

\begin{proof}
  Let $f:M\to M'$ be an $(s_r,c)$-rough equivalence. By
  Definition~\ref{d:rough map}, there is an $s_r$-rough map $g:M'\to
  M$ such that $g f$ and $f g$ are $c$-close to $\id_M$ and
  $\id_{M'}$, respectively. Let $\Gamma$ be an $(R,Q_r)$-quasi-lattice
  of $M$, for some coarse bound $(R,Q_r)$. By Proposition~\ref{p:
    coarse bounded geometry and rough equivalences},
  $\Gamma'=f(\Gamma)$ is an $(R',Q'_r)$-quasi-lattice of $M'$, where
  $R'=c+s_R$ and $Q'_r=Q_{s_{r+R'}}$.
	
  Let $S\subset\Gamma$ be finite and  let $S'=f(S)\subset\Gamma'$.
  Then $S\subset\Pen_\Gamma(g(S'),c)$, and so
  \begin{equation}\label{|S| le Q_c |S'|}
    |S|\le Q_c\,|S'|\;,
  \end{equation}

  \begin{claim}\label{partial^Gamma'_rS'}
    $\partial^{\Gamma'}_rS'\subset f(\partial^\Gamma_{s_r}S)$ for all $r>0$.
  \end{claim}
	
  Indeed, let $x'\in\partial^{\Gamma'}_rS'$ and $x\in\Gamma$ with
  $f(x)=x'$. There are points $y'\in S'$ and $z'\in\Gamma'\sm S'$ such
  that $d'(x',y')\le r$ and $d'(x',z')\le r$. If $y\in S$ and
  $z\in\Gamma\sm S$ are such that $f(y)=y'$ and $f(z)=z'$, then $d(x,y)\le
  s_r$ and $d(x,z)\le s_r$, and so $x\in\partial^\Gamma_{s_r}S$.

  It follows from Claim~\ref{partial^Gamma'_rS'} that
  \begin{equation}\label{|partial^Gamma'_rS'|}
    |\partial^{\Gamma'}_rS'|\le|\partial^\Gamma_{s_r}S|\;,
  \end{equation}
  and combining ~\eqref{|S| le Q_c |S'|}
  and~\eqref{|partial^Gamma'_rS'|} that
  \begin{equation}\label{frac |partial^Gamma'_rS'| |S'|}
    \frac{|\partial^{\Gamma'}_rS'|}{|S'|}\le Q_c\,\frac{|\partial^\Gamma_{s_r}S|}{|S|}\;.
  \end{equation}
  If $S_n$ is a F{\o}lner sequence in $\Gamma$, then $S_n'$ is a
  F{\o}lner sequence in $\Gamma'$ by~\eqref{frac |partial^Gamma'_rS'|
    |S'|}, and therefore $M'$ is amenable. This shows the weaker
  version of~\eqref{i: rough equivalences preserve weak
    amenability}. Then the weaker version of~\eqref{i: coarse
    quasi-isometries preserve amenability} follows by
  Propositions~\ref{p: large scale Lipschitz extensions} and~\ref{p:
    any large scale Lipschitz map is rough}-\eqref{i: f is an
    (s_r,c)-rough equivalence}.
		
  \begin{claim}\label{cl: partial^Gamma_r Pen_Gamma(S,s_2c)}
    $f(\partial^\Gamma_r\Pen_\Gamma(S,2c))\subset\partial^{\Gamma'}_{u_r}S'$ for all $r>0$, where $u_r=s_r+s_{2c}$.
  \end{claim}
	
  Let $x\in\partial^\Gamma_r\Pen_\Gamma(S,2c)$. There are some
  $y\in\Pen_\Gamma(S,2c)$ and $z\in\Gamma\sm\Pen_\Gamma(S,2c)$ such
  that $d(x,y)\le r$ and $d(x,z)\le r$. Then there is some $y_0\in S$
  so that $d(y,y_0)\le 2c$. Let $x'=f(x)$, $y'=f(y)$ and $z'=f(z)$ in
  $\Gamma'$, and $y_0'=f(y_0)$ in $S'$. We have $z'\in\Gamma'\sm S'$;
  otherwise, $z'=f(z_0)$ for some $z_0\in S$, obtaining
  $$
  d(z,z_0)\le d(z,g(z'))+d(g(z'),z_0)\le2c\;,
  $$
  which is a contradiction because
  $z\notin\Pen_\Gamma(S,2c)$. Moreover
  $$
  d'(x',y_0')\le d'(x',y')+d'(y',y_0')\le s_r+s_{2c}=u_r
  $$
  and $d'(x',z')\le s_r$, giving $x'\in\partial^{\Gamma'}_{u_r}S'$,
  which shows Claim~\ref{cl: partial^Gamma_r Pen_Gamma(S,s_2c)}.
	
  Applying~\eqref{|S| le Q_c |S'|} to the set
  $\partial^\Gamma_r\Pen_\Gamma(S,2c)$, and using Claim~\ref{cl:
    partial^Gamma_r Pen_Gamma(S,s_2c)}, we get
  \begin{equation}\label{frac |partial^Gamma_r Pen_Gamma(S,2c)| |Pen_Gamma(S,2c)|}
    \frac{|\partial^\Gamma_r\Pen_\Gamma(S,2c)|}{|\Pen_\Gamma(S,2c)|}
    \le Q_c\,\frac{|f(\partial^\Gamma_r\Pen_\Gamma(S,2c))|}{|S|}
    \le Q_c\,\frac{|\partial^{\Gamma'}_{u_r}S'|}{|S'|}\;.
  \end{equation}
	
  Suppose that $\{M_i\}$ satisfies the condition of weak
  equi-amenability (Definition~\ref{d: (weakly) equi-amenable}) with a
  class  $\{\Gamma_i\}$ of corresponding $(R,Q_r)$-quasi-lattices,
  subsets $S_{i,m,n}\subset\Gamma_i$, a constant $a\ge1$, mappings
  $r\mapsto p_r$ and $t\mapsto q_t$, and nonempty subsets
  $\NN_{i,j,m,n,t}\subset\N$. Consider a family of equi-rough
  equivalences $f_i:M_i\to M'_i$ with common rough equivalence
  distortion $(s_r,c)$. By Proposition~\ref{p: coarse bounded geometry
    and rough equivalences}, each $\Gamma'_i=f_i(\Gamma_i)$ is an
  $(R',Q'_r)$-quasi-lattice of $M'_i$, where $R'=c+s_R$ and
  $Q'_r=Q_{s_{r+R'}}$. By~\eqref{frac |partial^Gamma'_rS'| |S'|}, for
  each $i$ and $m$, $S'_{i,m,n}=f_i(S_{i,m,n})$ is a F{\o}lner
  sequence in $\Gamma'_i$. Given $i$, $m$ and $t'\ge0$, take any
  F{\o}lner sequence $X'_n$ in any $\Gamma'_i$ with $S'_{i,m,n}\subset
  X'_n\subset\Pen_{\Gamma'_i}(S'_{i,m,n},t')$. Then
  $$
  X_n:=\Pen_{\Gamma_i}(f_i^{-1}(X'_n),2c)\subset\Pen_{\Gamma_i}(S_{i,m,n},s_{t'}+2c)\;;
  $$
  in particular, every $X_n$ is finite. Moreover $X_n$ is a F{\o}lner
  sequence in $\Gamma_i$ by~\eqref{frac |partial^Gamma_r
    Pen_Gamma(S,2c)| |Pen_Gamma(S,2c)|} since
  $f_i(f_i^{-1}(X'_n))=X'_n$. Thus~\eqref{S_j,m',n subset Y_i,j,m',n
    subset Pen_Gamma_i(S_j,m',n,q_t)} and~\eqref{(weakly)
    equi-amenable} are satisfied with $t=s_{t'}+2c$ and some subsets
  $Y_{j,m',n}\subset\Gamma_j$ ($m'\in\NN_{i,j,m,n}$). Let
  $q'_{t'}=s_{q_t}$ and $Y'_{j,m',n}=f_j(Y_{j,m',n})\subset\Gamma'_j$
  for $m'\in\NN_{i,j,m,n}$. We have
  $$
  S'_{j,m',n}\subset Y'_{j,m',n}\subset\Pen_{\Gamma'_j}(S'_{j,m',n},q'_{t'})\;.
  $$
  For $r'>0$, let $r=s_{r'}$ and $p'_{r'}=u_{p_r}$. By~\eqref{frac |partial^Gamma'_rS'| |S'|},~\eqref{(weakly) equi-amenable} and~\eqref{frac |partial^Gamma_r Pen_Gamma(S,2c)| |Pen_Gamma(S,2c)|},
  $$
  \frac{|\partial_{r'}^{\Gamma'_j}Y'_{j,m',n}|}{|Y'_{j,m',n}|}
  \le Q_c\,\frac{|\partial_r^{\Gamma_j}Y_{j,m',n}|}{|Y_{j,m',n}|}
  \le aQ_c\,\frac{|\partial_{p_r}^{\Gamma_i}X_n|}{|X_n|}
  \le aQ_c^2\,\frac{|\partial_{p'_{r'}}^{\Gamma_i}X'_n|}{|X'_n|}\;.
  $$
  for all $r'$. So $\{M'_i\}$ is weakly equi-amenable (the stronger version of~\eqref{i: rough equivalences preserve weak amenability}).
	
  The stronger version of~\eqref{i: coarse quasi-isometries preserve
    amenability} follows like the stronger version of~\eqref{i: rough
    equivalences preserve weak amenability}, assuming that the family
  $\{f_i\}$ is equi-large scale Lipschitz by Propositions~\ref{p:
    large scale Lipschitz extensions} and~\ref{p: restrictions of
    large scale Lipschitz maps}, and using the expression of $s_r$ in
  Proposition~\ref{p: any large scale Lipschitz map is
    rough}-\eqref{i: f is an (s_r,c)-rough equivalence} and the
  expression of $u_r$ in Claim~\ref{cl: partial^Gamma_r
    Pen_Gamma(S,s_2c)}.
\end{proof}

\begin{rem}
  In the proof of Proposition~\ref{p: equi-rough equivalences preserve
    weak equi-amenability}, since the fibers of $f$ are of diameter
  $\le s_0$, we could use $Q_{s_0}$ instead of $Q_c$ in~\eqref{|S| le
    Q_c |S'|}, and Claim~\ref{cl: partial^Gamma_r Pen_Gamma(S,s_2c)}
  could be stated with $s_0$ instead of $2c$.
\end{rem}
 
\begin{rem}
  The weaker version of Proposition~\ref{p: equi-rough equivalences
    preserve weak equi-amenability}-\eqref{i: rough equivalences
    preserve weak amenability} was proved by Block and Weinberger~\cite{BlockWeinberger1992}. Their proof has the following three
  steps. First, they introduce the uniformly finite homology,
  $H^{\text{\rm uf}}_\bullet(M)$;  second, they show that
  $H^{\text{\rm uf}}_\bullet(M)\cong H^{\text{\rm uf}}_\bullet(M')$ if
  $M$ and $M'$ are roughly equivalent; and third, they prove that $M$ is
  amenable if and only if $H^{\text{\rm uf}}_0(M)\neq0$.
\end{rem}

\begin{prop}\label{p: amenability is independent of the quasi-lattice}
  If $M$ is amenable, then each quasi-lattice of $M$ has a F{\o}lner
  sequence.
\end{prop}

\begin{proof}
  Let $\Gamma$ be a quasi-lattice of $M$. Since $M$ is amenable,
  $\Gamma$ is also amenable by Proposition~\ref{p: equi-rough
    equivalences preserve weak equi-amenability}. Therefore some
  $(R,Q_r)$-quasi-lattice $\Gamma_0$ of $\Gamma$ has a F{\o}lner
  sequence $S_{0,n}$. It has to be shown that $\Gamma$ also has a
  F{\o}lner sequence. Let $S_n=\Pen_\Gamma(S_{0,n},R)$.
	
  \begin{claim}\label{cl: partial^Gamma_rS_n}
    $\partial^\Gamma_rS_n\subset\Pen_\Gamma(\partial^{\Gamma_0}_{r+2R}S_{0,n},R)$ for all $r>0$.
  \end{claim}
	
  For each $x\in\partial^\Gamma_rS_n$, there are $x_0\in\Gamma_0$,
  $y\in S_n$ and $z\in\Gamma\sm S_n$ such that $d(x,x_0)\le R$,
  $d(x,y)\le r$ and $d(x,z)\le r$. Then there are points $y_0\in
  S_{0,n}$ and $z_0\in\Gamma_0$ such that $d(y,y_0)\le R$ and
  $d(z,z_0)\le R$. Observe that $z_0\notin S_{0,n}$ because $z\notin
  S_n$ and $d(z,z_0)\le R$. By the triangle inequality, the distances
  $d(x_0,y_0)\le r+2R$ and $d(x_0,z_0)\le r+2R$, and so
  $x_0\in\partial^{\Gamma_0}_{r+2R}S_{0,n}$. Hence
  $x\in\Pen_\Gamma(\partial^{\Gamma_0}_{r+2R}S_{0,n},R)$, which
  completes the proof of Claim~\ref{cl: partial^Gamma_rS_n}.
	
  By Claim~\ref{cl: partial^Gamma_rS_n},
  \begin{equation}\label{frac |partial^Gamma_rS_n| |S_n|}
    \frac{|\partial^\Gamma_rS_n|}{|S_n|}
    \le\frac{|\Pen_\Gamma(\partial^{\Gamma_0}_{r+2R}S_{0,n},R)|}{|S_{0,n}|}
    \le Q_R\,\frac{|\partial^{\Gamma_0}_{r+2R}S_{0,n}|}{|S_{0,n}|}\;.
  \end{equation}
  Then $S_n$ is a F{\o}lner sequence in $\Gamma$ because $S_{0,n}$ is
  a F{\o}lner sequence in $\Gamma_0$.
\end{proof}

\begin{rem}\label{r: amenability is independent of the quasi-lattice}
  In the proof of Proposition~\ref{p: amenability is independent of
    the quasi-lattice}, it is easy to check that
  $\partial_r^{\Gamma_0}(S_n\cap\Gamma_0)\subset\partial_r^\Gamma S_n$
  for all $r>0$, and therefore
		$$
                \frac{|\partial_r^{\Gamma_0}(S_n\cap\Gamma_0)|}{|S_n\cap\Gamma_0|}
                \le\frac{|\partial_r^\Gamma S_n|}{|S_n|}\;.
		$$
 \end{rem}

\section{Amenable symmetry}\label{s: amenable symmetry}

\begin{defn}\label{d: (weakly) amenably symmetric}
  A metric space, $M$, is called \emph{weakly amenably symmetric} \index{weakly amenably symmetric} if there
  are:
  \begin{itemize}
		
  \item a quasi-lattice $\Gamma$ of $M$;
		
  \item subsets $S_{m,n}\subset\Gamma$ ($m,n\in\N$);
		
  \item some $a\ge1$, and mappings $r\mapsto p_r$ ($r,p_r>0$) and
    $t\mapsto q_t$ ($t,q_t\ge0$); and
			
  \item a nonempty subset $\NN_{m,n,t}\subset\N$ for each $m$, $n$
    and $t$;
		
  \end{itemize}
	such that:
        \begin{itemize}
		
        \item for each $m$, the sequence $S_{m,n}$ is F{\o}lner in
          $\Gamma_i$;
			
        \item for each $m$, $n$ and $t\ge0$,
          $\bigcup_{m'\in\NN_{m,n,t}}S_{m',n}$ is a net in $M$; and
			
        \item for each $m$ and $t\ge0$, and any F{\o}lner sequence
          $X_n$ of $\Gamma$ with $S_{m,n}\subset
          X_n\subset\Pen_\Gamma(S_{m,n},t)$, there is some
          $Y_{m',n}\subset\Gamma$ for each $n$ and $m'\in\NN_{m,n}$ so
          that
          \begin{gather}
            S_{m',n}\subset
            Y_{m',n}\subset\Pen_{\Gamma_i}(S_{m',n},q_t)\;,
            \label{S_m',n subset Y_m',n subset Pen_Gamma_i(S_m',n,q_t)}\\
            \frac{|\partial^\Gamma_rY_{m',n}|}{|Y_{m',n}|} \le
            a\,\frac{|\partial^\Gamma_{p_r}X_n|}{|X_n|}\;,
            \label{(weakly) amenably symmetric}
          \end{gather}
          for all $r$.
        \end{itemize}
        In this case, $\Gamma$ is called \emph{weakly F{\o}lner
          symmetric}. \index{weakly F{\o}lner!symmetric} If moreover the mappings $r\mapsto p_r$ and
        $t\mapsto q_t$ can be chosen to be affine, then $M$ is called
        \emph{amenably symmetric} \index{amenably!symmetric} and $\Gamma$ is called \emph{F{\o}lner
          symmetric}. \index{F{\o}lner!symmetric}
\end{defn}

\begin{example}
	Let $M\subseteq \R^3$ be the union of a plane and a line
        orthogonal to the plane.  With the subspace metric, $M$ is
        amenable but not amenably symmetric.
\end{example}

\begin{defn}\label{d: (weakly) equi-amenably symmetric}
  A class of metric spaces, $\{M_i\}$, is (\emph{weakly}) \emph{equi-amenably symmetric} 
  \index{equi-amenably symmetric}  \index{weakly equi-amenably symmetric} if every $M_i$ satisfies the conditions
  of (weak) amenable symmetry (Definition~\ref{d: (weakly) amenably
    symmetric}) with equi-quasi-lattices $\Gamma_i\subset M_i$,
  subsets $S_{i,m,n}\subset\Gamma_i$, the same constant $a\ge1$, the
  same mappings $r\mapsto p_r$ and $t\mapsto q_t$, and subsets
  $\NN_{i,m,n,t}\subset\N$, such that, for some
  $L_n,L_{i,m,n,t}\in\N$, $\bigcup_mS_{i,m,n}$ is an $L_n$-net of
  $\Gamma_i$ for all $i$, and
  $\bigcup_{m'\in\NN_{i,j,m,n,t}}S_{j,m',n}$ is an $L_{i,m,n,t}$-net
  of $\Gamma_j$ for all $j$.
\end{defn}

\begin{rem}
  In Definition~\ref{d: (weakly) equi-amenably symmetric}, every $M_i$
  satisfies the conditions of (weak) amenable symmetry with subsets
  $\NN_{i,m,n,t}\subset\N$ that may depend on $i$.
\end{rem}

\begin{prop}\label{p: (weakly) amenably symmetric is invariant}
  \begin{enumerate}[{\rm(}i\/{\rm)}]
  
  \item\label{i: rough equivalences preserve weak amenable symmetry}
    {\rm(}Equi-\/{\rm)}rough equivalences preserve weak
    {\rm(}equi-\/{\rm)}amenable symmetry.
				
  \item\label{i: coarse quasi-isometries preserve amenable symmetry}
    {\rm(}Equi-\/{\rm)}coarse quasi-isometries preserve
    {\rm(}equi-\/{\rm)}amenable symmetry.
		
  \end{enumerate}
\end{prop}

\begin{proof}
  Let $f:M\to M'$ be a rough equivalence with rough equivalence
  distortion $(s_r,c)$. Suppose that $M$ satisfies the conditions to
  be weakly amenably symmetric with a quasi-lattice $\Gamma$, sets
  $S_{m,n}$, a constant $a\ge1$, mappings $r\mapsto p_r$ and $t\mapsto
  q_t$, and subsets $\NN_{m,n,t}\subset\N$; in particular, each union
  $\bigcup_{m'\in\NN_{m,n,t}}S_{m',n}$ is a $K_{m,n,t}$-net in $M$ for
  some $K_{m,n,t}\ge0$. Let $S'_{m,n}=f(S_{m,n})$. For each $m$, the
  sequence $S'_{m,n}$ is F{\o}lner in $\Gamma'$ by~\eqref{frac
    |partial^Gamma'_rS'| |S'|}. Moreover
  $\bigcup_{m'\in\NN_{m,n,t}}S'_{m',n}$ is a $(s_{K_{m,n,t}}+c)$-net
  in $M'$, which can be proved as follows. For each $x'\in M'$, there
  is some $y\in\bigcup_{m'\in\NN_{m,n,t}}S_{m',n}$ such that
  $d(y,g(x'))\le K_n$. Then $f(y)\in
  \bigcup_{m'\in\NN_{m,n,t}}S'_{m',n}$ and
  $$
  d' (f(y),x')\le d'(f(y),f g(x'))+d'(f g(x'),x')\le
  s_{K_{m,n,t}}+c\;.
  $$
	
  The rest of the proof is analogous to the proof of
  Proposition~\ref{p: equi-rough equivalences preserve weak
    equi-amenability}.
\end{proof}

\begin{prop}\label{p: (weakly) amenably symmetric is independent of the quasi-lattice}
  If $M$ is {\rm(}weakly\/{\rm)} amenably symmetric, then every
  quasi-lattice of $M$ is {\rm(}weakly\/{\rm)} F{\o}lner symmetric.
\end{prop}

\begin{proof}
  Suppose that $M$ is (weakly) amenably symmetric, and let $\Gamma$ be
  a quasi-lattice of $M$. By Proposition~\ref{p: (weakly) amenably
    symmetric is invariant}, $\Gamma$ satisfies the condition of
  (weak) amenable symmetry (Definition~\ref{d: (weakly) amenably
    symmetric}) with some $(R,Q_r)$-quasi-lattice $\Gamma'$ of
  $\Gamma$, a family of F{\o}lner sequences $S'_{m,n}$ in $\Gamma'$,
  some constant $a'\ge1$, and some mappings $r\mapsto p'_r$ and
  $t\mapsto q'_t$. For each $m$ and $n$, let
  $S_{m,n}=\Pen_\Gamma(S'_{m,n},R)$.  For every $m$, $S_{m,n}$ is a
  F{\o}lner sequence in $\Gamma$ by~\eqref{frac |partial^Gamma_rS_n|
    |S_n|}. Given $m$ and $t\ge0$, let $X_n$ be a F{\o}lner sequence
  in $\Gamma$ with $S_{m,n}\subset
  X_n\subset\Pen_\Gamma(S_{m,n},t)$. Then $X'_n:=X_n\cap\Gamma'$ is a
  F{\o}lner sequence in $\Gamma'$ by Remark~\ref{r: amenability is
    independent of the quasi-lattice}, and moreover $S'_{m,n}\subset
  X'_n\subset\Pen_{\Gamma'}(S'_{m,n},t+R)$. Hence $X'_n$
  satisfies~\eqref{S_m',n subset Y_m',n subset
    Pen_Gamma_i(S_m',n,q_t)} and~\eqref{(weakly) amenably symmetric}
  with some subsets $Y'_{m,n}\subset\Gamma'$, using $q'_{t+R}$. For
  $Y_{m,n}=\Pen_\Gamma(Y'_{m,n},R)$, we have
  \begin{gather*}
    S_{m,n}\subset
    Y_{m,n}\subset\Pen_\Gamma(\Pen_{\Gamma'}(S'_{m,n},q'_{t+R}),R)
    \subset\Pen_\Gamma(S_{m,n},q'_{t+R}+R)\;,\\
    \frac{|\partial_r^\Gamma Y_{m,n}|}{|Y_{m,n}|} \le
    Q_R\,\frac{|\partial_r^{\Gamma'}Y'_{m,n}|}{|Y'_{m,n}|} \le
    a'Q_R\,\frac{|\partial_{p'_r}^{\Gamma'}X'_n|}{|X'_n|} \le
    a'Q_R\,\frac{|\partial_{p'_r}^\Gamma X_n|}{|X_n|}\;,
  \end{gather*}
  by~\eqref{Pen_M(Pen_M(S,r),s)} and~\eqref{frac |partial^Gamma_rS_n|
    |S_n|}--\eqref{(weakly) amenably symmetric}. Thus the condition of
  (weak) amenable symmetry of $M$ is satisfied with $\Gamma$.
\end{proof}

\begin{defn}\label{d: jointly (weakly) amenably symmetric}
  A class of metric spaces, $\{M_i\}$, is \emph{jointly weakly amenably
    symmetric} \index{jointly weakly amenably symmetric} if there are:
  \begin{itemize}
		
  \item equi-quasi-lattices $\Gamma_i\subset M_i$;
		
  \item subsets $S_{i,m,n}\subset\Gamma_i$ ($m,n\in\N$);
		
  \item mappings $r\mapsto p_r$ ($r,p_r>0$) and
    $t\mapsto q_t$ ($t,q_t\ge0$);
			
  \item nonempty subsets $\NN_{i,j,m,n,t}\subset\N$, one  for each $i$,
    $j$, $m$, $n$ and $t$;
			
  \item numbers $L_n,L_{i,m,n,t}\in\N$, one for each $i$, $m$, $n$ and $t$;
   and  $a\ge1$,		
  \end{itemize}
  such that:
  \begin{itemize}
		
  \item for each $i$ and $m$, the sequence $S_{i,m,n}$ is F{\o}lner in
    $\Gamma_i$;
			
  \item for each $i$ and $n$, $\bigcup_mS_{i,m,n}$ is an $L_n$-net in
    $M_i$;
			
  \item for each $i$, $j$, $m$, $n$ and $t$,
    $\bigcup_{m'\in\NN_{i,j,m,n,t}}S_{j,m',n}$ is an $L_{i,m,n,t}$-net
    in $M_j$; and
			
  \item for each $i$, $m$ and $t$, and any F{\o}lner sequence $X_n$ of
    $\Gamma_i$ with $S_{i,m,n}\subset
    X_n\subset\Pen_\Gamma(S_{i,m,n},t)$, there is some
    $Y_{i,j,m',n}\subset\Gamma_j$ for all $j$, $n$ and
    $m'\in\NN_{m,n}$ so that
    \begin{gather*}
      S_{j,m',n}\subset Y_{i,j,m',n}\subset\Pen_{\Gamma_i}(S_{i,m',n},q_t)\;,\\
      \frac{|\partial^\Gamma_rY_{i,j,m',n}|}{|Y_{i,j,m',n}|} \le
      a\,\frac{|\partial^\Gamma_{p_r}X_n|}{|X_n|}\;,
    \end{gather*}
    for all $r$.
  \end{itemize}
  In this case, $\{\Gamma_i\}$ is called \emph{jointly weakly F{\o}lner
    symmetric}. \index{jointly weakly F{\o}lner symmetric} If moreover the mappings $r\mapsto p_r$ and
  $t\mapsto q_t$ can be chosen to be affine, then $\{M_i\}$ is called
  \emph{jointly amenably symmetric} \index{jointly amenably symmetric} and $\{\Gamma_i\}$ \emph{jointly
    F{\o}lner symmetric}. \index{jointly F{\o}lner symmetric}
\end{defn}

\begin{rem}
  Joint (weak) amenable symmetry is stronger than (weak)
  equi-amenability and (weak) equi-amenable symmetry.
\end{rem}

\begin{prop}\label{p: (weakly) amenably symmetric is invariant BIS}
  \begin{enumerate}[{\rm(}i\/{\rm)}]
  
  \item Equi-rough equivalences preserve weak joint amenable symmetry.
    \label{i : equi-rough equivalences preserve weak joint amenable
      symmetry}
		
  \item Equi-coarse quasi-isometries preserve joint amenable symmetry.
    \label{i : equi-coarse quasi-isometries preserve joint amenable
      symmetry}
		
  \end{enumerate}
\end{prop}

\begin{proof}
  Similar to the proof of Proposition~\ref{p: (weakly) amenably
    symmetric is invariant}.
\end{proof}

\chapter{Coarse ends}\label{c: coarse ends}

The end space is a topological invariant. In the case of proper
geodesic spaces, the end space turns out to be a coarsely
quasi-isometric invariant
\cite[Proposition~8.29]{BridsonHaefliger1999}. Proceeding further with
this idea, we introduce a version of the end space for metric spaces,
which is invariant by coarse equivalences. Indeed it can be directly
generalized to arbitrary coarse spaces.

\section{Ends}\label{s: ends}

The end space of a topological space $X$ is constructed as
follows. Let $\KK$ be the family of compact subsets of $X$. For each
$K\in\KK$, let $\UU_K$ be the discrete space of connected components
of $X\sm K$ with non-compact closure. Each inclusion $K\subset L$ in
$\KK$ induces a map $\eta_{K,L}:\UU_L\to\UU_K$ which assigns to \(U\in
\UU_L\) the component \(\eta_{K,L}(U)\in \UU_K\) which contains
$U$. These spaces \(\UU_K\) and maps
\(\eta_{K,L}\), \(K, L\in \KK\), form an inverse system, whose
inverse limit is the space of \emph{ends} \index{end} of $X$, denoted
by $\EE(X)$, which is Hausdorff and totally disconnected. Thus any
$\bfe\in\EE(X)$ can be described as a map defined on $\KK$ such that
$\bfe(K)\in\UU_K$ and $\bfe(K)\supset\bfe(L)$ if $K\subset L$ in
$\KK$. Suppose that $X$ has a an increasing sequence of compact
subsets, $(K_n)$, whose interiors cover $X$. Then the topology of
$\EE(X)$ is induced by the ultrametric $d_{(K_n)}$ defined by
\[
d_{(K_n)}(\bfe,\bff)=\exp(-\sup\{\,n\in\N\mid\bfe(K_n)=\bff(K_n)\,\})\;.
\]
The coarse version of these concepts for metric spaces is obtained by
replacing compact sets by bounded sets, as shown next.

\section{Coarse connectivity}\label{s: coarse connectivity}

\begin{defn}\label{d: coarse conexion}
  Let $\mu>0$. Two points $x, y\in M$ are \emph{coarsely
    $\mu$-connected} \index{coarsely $\mu$-connected} if there is, for some \(k\), a finite sequence
  $\{z_l\}_{l=0}^k$ in $M$ such that $x=z_0,z_1,\dots,z_k=y$ and that
  $d(z_{l-1},z_l)\le\mu$ for all $l\in\{1,\dots,k\}$. This concept
  defines an equivalence relation on $M$ whose equivalence classes are
  called \emph{coarse $\mu$-connected components}. \index{coarse $\mu$-connected components} If all points in
  $M$ are $\mu$-connected, then $M$ is called \emph{coarsely $\mu$-connected}. If $M$ is coarsely $\mu$-connected
  for some $\mu>0$, then $M$ is called \emph{coarsely
    connected}. \index{coarsely connected}
\end{defn}

\begin{rem}\label{r: mu-connected}
  The following properties are elementary:
  \begin{enumerate}[(i)]
	
  \item\label{i: mu-conn comp are maximal} The coarse $\mu$-connected
    components are the maximal coarse $\mu$-connected subsets.
		
  \item\label{i: nu-conn for all nu ge mu} If $M$ is coarsely
    $\mu$-connected, then it is coarsely $\nu$-connected for all
    $\nu\ge\mu$.
		
  \item\label{i: d(x,y) le mu} If $M$ is coarsely $\mu$-connected,
    then, for all non-trivial partition of $M$ into two sets,
    $\{A,B\}$, there is some $x\in A$ and $y\in B$ such that
    $d(x,y)\le\mu$.
		
  \item\label{i: A cup B is mu-conn} If $A$ and $B$ are coarsely
    $\mu$-connected subsets of $M$, and $d(x,y)\le\mu$ for some $x\in
    A$ and $y\in B$, then $A\cup B$ is coarsely $\mu$-connected.
	
  \end{enumerate}
\end{rem}

\begin{rem}\label{r: coarse connectivity}
  Coarse connectivity of $M$ means that the coarse space $[M]$ is
  monogenic \cite{Roe2003}, but we prefer the term coarse connectivity
  because it plays the same role here as connectivity in the
  definition of ends.
\end{rem}

\begin{lemma}\label{l: coarse connectivity is invariant by coarse equivalences}
  If $M$ is coarsely $\mu$-connected, $f:M\to M'$ satisfies the
  condition of uniform expansiveness with the mapping $s_r$, and
  $f(M)$ is a $c$-net in $M'$, then $M'$ is coarsely $\mu'$-connected,
  where $\mu'=\max\{s_\mu,c\}$.
\end{lemma}

\begin{proof}
  For $x',y'\in M'$, there are points $x,y\in M$ such that
  $d'(x',f(x))\le c$ and $d'(y',f(y))\le c$, and there is a finite
  sequence $x=z_0,z_1,\dots,z_k=y$ in $M$ such that
  $d(z_{l-1},z_l)\le\mu$ for all $l\in\{1,\dots,k\}$. Then
  $d'(f(z_{l-1}),f(z_l))\le s_\mu$, obtaining that $x'$ is coarsely
  $\mu'$-connected to $y'$.
\end{proof}

\begin{cor}\label{c: coarse connectivity is invariant by coarse equivalences}
  Coarse connectivity is invariant by coarse equivalences.
\end{cor}

\begin{rem}
  Corollary~\ref{c: coarse connectivity is invariant by coarse
    equivalences} is indeed trivial by Remark~\ref{r: coarse
    connectivity}.
\end{rem}

\begin{lemma}\label{l: mu-connected components of U sm C}
  Let $B\subset C\subset M$, and let $U$ be a coarse $\mu$-connected
  component of $M\sm B$. Then each coarse $\mu$-connected component of
  $U\sm C$ is a coarse $\mu$-connected component of $M\sm C$.
\end{lemma}

\begin{proof}		
  Each coarse $\mu$-connected component $V$ of $U\sm C$ is contained
  in some coarse $\mu$-connected component $W$ of $M\sm C$
  (Remark~\ref{r: mu-connected}-\eqref{i: mu-conn comp are
    maximal}). If $V\ne W$, then there are points, $y\in V$ and $z\in
  W\sm V$, such that $d(y,z)\le\mu$ (Remark~\ref{r:
    mu-connected}-\eqref{i: d(x,y) le mu}). Hence $z$ is coarsely
  $\mu$-connected to $y$ in $M\sm C$, and therefore in $M\sm B$. So
  $z\in U$, and $z$ is coarsely $\mu$-connected to $y$ in $U\sm C$,
  giving $z\in V$, a contradiction. Thus $V=W$.
\end{proof}

\begin{cor}\label{c: mu-connected components of U sm(A cup B)}
  Let $A, B\subset M$, and let $U$ be a coarse $\mu$-connected
  component of $M\sm B$ such that $U\cap A=\emptyset$. Then $U$ is a
  coarse $\mu$-connected component of $M\sm(A\cup B)$.
\end{cor}

\begin{proof}
  Take $C=A\cup B$ in Lemma~\ref{l: mu-connected components of U sm
    C}.
\end{proof}

\begin{lemma}\label{l: the mu-connected components of M setminus B are close to B}
  Suppose that $M$ is coarsely $\mu$-connected. Let $\emptyset\ne
  B\varsubsetneq M$, and let $U$ be a coarse $\mu$-connected component
  of $M\sm B$. Then there are points $x\in B$ and $y\in U$ such that
  $d(x,y)\le\mu$.
\end{lemma}

\begin{proof}		
  Let $x_0\in B$ and $y_0\in U\sm B$. Since $M$ is coarsely
  $\mu$-connected, there is a finite sequence
  $x_0=z_0,z_1,\dots,z_k=y_0$ such that $d(z_{l-1},z_l)\le\mu$ for all
  $l\in\{1,\dots,k\}$. Let
  \[
  p=\min\{\,l\in\{1,\dots,k\}\mid \{z_l,\dots,z_k\}\subset M\sm B\,\}\;.
  \]
  Then the statement holds with $x=z_{p-1}$ and $y=z_p$.
\end{proof}

\begin{rem}
  Lemma~\ref{l: the mu-connected components of M setminus B are close
    to B} is a refinement of Remark~\ref{r: mu-connected}-\eqref{i:
    d(x,y) le mu}.
\end{rem}

\begin{cor}\label{c: the mu-connected components meet a penumbra}
  Suppose that $M$ is coarsely $\mu$-connected. Let $A$ be a $\mu$-net
  of $M$, let $\emptyset\ne B\subset M$, and let $U$ be a coarse
  $\mu$-connected component of $M\sm B$ such that
  $U\sm\Pen(B,2\mu)\ne\emptyset$. Then $A\cap
  U\cap\Pen(B,3\mu)\ne\emptyset$.
\end{cor}

\begin{proof}		
  The set $U\sm\Pen(B,2\mu)$ is a coarse $\mu$-connected component of
  $M\sm\Pen(B,2\mu)$ by Lemma~\ref{l: mu-connected components of U sm
    C}. Then, by Lemma~\ref{l: the mu-connected components of M
    setminus B are close to B}, there are points $x\in\Pen(B,2\mu)$
  and $y\in U\sm\Pen(B,2\mu)$ with $d(x,y)\le\mu$. Since $x\in
  M\sm\Pen(B,\mu)$ by the triangle inequality, we also get $x\in
  U$. Take some $z\in A$ such that $d(x,z)\le\mu$. Since
  $z\in\Pen(B,3\mu)\sm B$ by the triangle inequality, we have $z\in
  A\cap U\cap\Pen(B,3\mu)$.
\end{proof}

\begin{cor}\label{c: finite number of mu-connected components}
  Suppose that $M$ is of coarse bounded geometry with coarse bound
  $(R,Q_r)$, and coarsely $\mu$-connected for some $\mu\ge R$. Let $B$
  be a nonempty bounded subset of $M$. Then $M\sm B$ has at most
  $Q_{\diam(B)+3\mu}$ coarse $\mu$-connected components that meet
  $M\sm\Pen(B,2\mu)$.
\end{cor}

\begin{proof}
  Fix an $(R,Q_r)$-quasi-lattice $\Gamma$ of $M$, and a point $x\in
  B$. By Corollary~\ref{c: the mu-connected components meet a
    penumbra}, the number of coarse $\mu$-connected components of
  $M\sm B$ that meet $M\sm\Pen(B,2\mu)$ is bounded by
  \[
  |\Gamma\cap\Pen(B,3\mu)|\le|\Gamma\cap\ol{B}(x,\diam(B)+3\mu)|\le
  Q_{\diam(B)+3\mu}\;.
  \]
\end{proof}

\begin{cor}\label{c: the same unbounded mu-connected components}
  Suppose that $M$ is of coarse bounded geometry with coarse bound
  $(R,Q_r)$, and coarsely $\mu$-connected for some $\mu\ge R$. For any
  nonempty bounded subset $B\subset M$, let $C$ be the union of $B$
  and the bounded coarse $\mu$-connected components of $M\sm B$. Then
  $C$ is bounded, and the coarse $\mu$-connected components of $M\sm
  C$ are the unbounded coarse $\mu$-connected components of $M\sm B$.
\end{cor}

\begin{proof}
  This is a consequence of Corollaries~\ref{c: finite number of
    mu-connected components} and~\ref{c: mu-connected components of U
    sm(A cup B)}.
\end{proof}

\begin{cor}\label{c: there are unbounded mu-connected component}
  Suppose that $M$ is of coarse bounded geometry with coarse bound
  $(R,Q_r)$, coarsely $\mu$-connected for some $\mu\ge R$, and
  unbounded. Then the complement of each bounded subset of $M$ has an
  unbounded coarse $\mu$-connected component.
\end{cor}

\begin{proof}
  Take any bounded $B\subset M$. We can assume $B\ne\emptyset$ because
  $M$ is unbounded. Thus, if all coarse $\mu$-connected components of
  $M\sm B$ were bounded, then $M$ would be bounded by
  Corollary~\ref{c: the same unbounded mu-connected components}, a
  contradiction.
\end{proof}

\section{Coarse ends}\label{s: coarse ends}

Let $\BB(M)$ (or simply $\BB$) be the set of bounded subsets of
$M$. For all $\mu>0$ and $B\in\BB$, let $\UU_{\mu,B}(M)$ (or simply
$\UU_{\mu,B}$) denote the discrete space of unbounded coarse
$\mu$-connected components of $M\sm B$. According to Remark~\ref{r:
  mu-connected}-\eqref{i: mu-conn comp are maximal}, for $B\subset C$
in $\BB$, we get a map $\eta_{\mu,B,C}:\UU_{\mu,B}\to\UU_{\mu,C}$
determined by $\eta_{\mu,B,C}(U)\supset U$. These spaces $\UU_{\mu,B}$
and maps $\eta_{\mu,B,C}$ form a projective system (over $\BB$ with
``$\subset$''), denoted by $\{\UU_{\mu,B},\eta_{\mu,B,C}\}$.

\begin{defn}
  The projective limit of $\{\UU_{\mu,B},\eta_{\mu,B,C}\}$, denoted by
  $\EE_\mu(M)$, is the space of \emph{$\mu$-ends} \index{$\mu$-end} of $M$.
\end{defn}

\begin{rem}\label{r: Ends_mu(M) is Hausdorff and totally disconnected}
  $\EE_\mu(M)$ is Hausdorff and totally disconnected because each
  space $\UU_{\mu,B}$ is discrete.
\end{rem}

Each $\bfe\in\EE_\mu(M)$ can be described as a map defined on $\BB$
such that $\bfe(B)\in\UU_{\mu,B}$ and $\bfe(B)\supset\bfe(C)$ if
$B\subset C$. The maps $\eta_{\mu,B}:\EE_\mu(M)\to\UU_{\mu,B}$
satisfying the universal property of the inverse limit are given by
$\eta_{\mu,B}(\bfe)=\bfe(B)$. Hence, for $B\subset C$ in $\BB$,
\begin{equation}\label{eta_mu,B,C(e(B))}
  \eta_{\mu,B,C}(\bfe(B))=\bfe(C)\;.
\end{equation}

\begin{rem}\label{r: bfe(B) cap bfe(C) ne emptyset}
  We have $\bfe(B)\cap\bfe(C)\ne\emptyset$ for all $\bfe\in\EE_\mu(M)$
  and $B,C\in\BB$, because $\bfe(B)\cap\bfe(C)\supset\bfe(B\cup C)$.
\end{rem}

For each $B\in\BB$ and $U\in\UU_{\mu,B}$, let 
\[
\NN_\mu(B,U)=\{\,\bfe\in\EE_\mu(M)\mid\bfe(B)=U\,\}\;.
\]
The family of the sets $\NN_\mu(B,U)$ is a base of the topology of
$\EE_\mu(M)$. For any fixed $x_0\in M$, an ultrametric $d_{\mu,x_0}$
inducing the topology of $\EE_\mu(M)$ is given by
\begin{align*}
  d_{\mu,x_0}(\bfe,\bff)
  =\exp(-\sup\{\,n\in\N\mid\bfe(B(x_0,n))=\bff(B(x_0,n))\,\})\;.
\end{align*}
	
\begin{rem}\label{r: e(B(x_0,n))}
  Since $\{\,B(x_0,n)\mid n\in\N\,\}$ is cofinal in $\BB$, for each
  nested sequence $U_0\supset U_1\supset\cdots$ with
  $U_n\in\UU_{\mu,B(x_0,n)}$, there is a unique $\bfe\in\EE_\mu(M)$
  such that $U_n=\bfe(B(x_0,n))$ for all $n$.
\end{rem}

\begin{rem}\label{r: the Lipschitz equivalence class of d_mu,x_0 is independent of x_0}
  The Lipschitz equivalence class of $d_{\mu,x_0}$ is independent of
  $x_0$; in fact, for another point $x_1\in M$ and an integer $N\ge
  d(x_0,x_1)$, we easily get $d_{\mu,x_1}\le e^N\,d_{\mu,x_0}$.
\end{rem}

According to Remark~\ref{r: mu-connected}-\eqref{i: nu-conn for all nu
  ge mu}, for $0<\mu\le\nu$ and $B\in\BB$, there is a map
$\theta_{\mu,\nu,B}:\UU_{\mu,B}\to\UU_{\nu,B}$ determined by
$\theta_{\mu,\nu,B}(U)\supset U$. The diagram
\begin{equation}\label{CD}
  \begin{CD}
    \UU_{\nu,C} @>{\eta_{\nu,B,C}}>> \UU_{\nu,B }\\
    @A{\theta_{\mu,\nu,C}}AA @AA{\theta_{\mu,\nu,B}}A \\
    \UU_{\mu,C} @>{\eta_{\mu,B,C}}>> \UU_{\mu,B}
  \end{CD}
\end{equation}
is commutative for $B\subset C$ in $\BB$ because, for all
$U\in\UU_{\mu,B}$,
\begin{gather*}
  \eta_{\nu,B,C}\,\theta_{\mu,\nu,C}(U)\supset\theta_{\mu,\nu,C}(U)\supset U\;,\\
  \theta_{\mu,\nu,B}\,\eta_{\mu,B,C}\supset\eta_{\mu,B,C}(U)\supset
  U\;.
\end{gather*}
Hence the maps $\theta_{\mu,\nu,B}$ induce a continuous map
$\theta_{\mu,\nu}:\EE_\mu(M)\to\EE_\nu(M)$ determined by the condition
$\eta_{\nu,B}\,\theta_{\mu,\nu}=\theta_{\mu,\nu,B}\,\eta_{\mu,B}$ for
all $B\in\BB$. Observe that
\begin{equation}\label{theta_mu,nu(e)(B)}
  \theta_{\mu,\nu}(\bfe)(B)
  =\eta_{\nu,B}\,\theta_{\mu,\nu}(\bfe)
  =\theta_{\mu,\nu,B}\,\eta_{\mu,B}(\bfe)
  =\theta_{\mu,\nu,B}(\bfe(B))\;.
\end{equation} 
On the other hand, like the commutativity of~\eqref{CD}, it can be
proved that
\begin{equation}\label{theta_mu,nu,B circ theta_lambda,mu,B =
    theta_lambda,nu,B}
  \theta_{\mu,\nu,B}\circ\theta_{\lambda,\mu,B}=\theta_{\lambda,\nu,B}
\end{equation}
for $0<\lambda\le\mu\le\nu$. Hence,
by~\eqref{theta_mu,nu(e)(B)},
\begin{multline*}
  \theta_{\mu,\nu}\theta_{\lambda,\mu}(\bfe)(B)
  =\theta_{\mu,\nu,B}(\theta_{\lambda,\mu}(\bfe)(B))
  =\theta_{\mu,\nu,B}\theta_{\lambda,\mu,B}(\bfe(B))\\
  =\theta_{\lambda,\nu,B}(\bfe(B)) =\theta_{\lambda,\nu}(\bfe)(B),
\end{multline*}
 giving
$\theta_{\mu,\nu}\theta_{\lambda,\mu}=\theta_{\lambda,\nu}$. Thus the
spaces $\EE_\nu(M)$ and maps $\theta_{\mu,\nu}$ form a direct system
of topological spaces, denoted by $\{\EE_\mu(M),\theta_{\mu,\nu}\}$.

\begin{defn}\label{d: coarse ends}
  The injective limit of $\{\EE_\mu(M),\theta_{\mu,\nu}\}$, denoted by
  $\EE_\infty(M)$, is called the space of \emph{coarse ends} \index{coarse end} of $M$.
\end{defn}

Let $\theta_\mu:\EE_\mu(M)\to\EE_\infty(M)$ be the maps that satisfy
the universal property of the injective limit.

\begin{rem}\label{l: theta_mu,nu is non-expanding}
  It is easy to see that
	 	$$
                \theta_{\mu,\nu}:(\EE_\mu(M),d_{\mu,x_0})\to(\EE_\nu(M),d_{\nu,x_0})
		$$
	is non-expanding for $\nu\ge\mu$. 
\end{rem}

\begin{rem}\label{r: ends of coarse spaces}
  The definition of the space of coarse ends can be generalized to
  arbitrary coarse spaces as follows. With the terminology of
  \cite{Roe2003}, for each entourage $E$ of a coarse space $X$, define
  the \emph{coarse $E$-connected components} like the above coarse
  $\mu$-connected components by using the condition $(x,y)\in E$
  instead of $d(x,y)\le\mu$. Let $\BB$ be the family of bounded
  subsets of $X$ (those $B\subset X$ so that $B\times B$ is an
  entourage). Then we can define $\UU_{E,B}$ and $\EE_E(X)$ like the
  above $\UU_{\mu,B}$ and $\EE_\mu(M)$. For entourages $E\subset F$,
  we get a continuous map $\theta_{E,F}:\EE_E(X)\to\EE_F(X)$ defined
  like the above $\theta_{\mu,\nu}$ for $\mu\le\nu$. Then
  $\{\EE_E(X),\theta_{E,F}\}$ is a direct system whose injective limit
  is $\EE_\infty(X)$.
\end{rem}

\begin{rem}\label{r: Ends_infty(X) is nonempty}
  Observe that the spaces $\EE_\mu(M)$ and $\EE_\infty(M)$ are
  nonempty if and only if $M$ is unbounded.
\end{rem}

\begin{prop}\label{p: Ends_mu(X) is compact}
  If $M$ is of coarse bounded geometry with coarse bound $(R,Q_r)$,
  and coarsely $\mu$-connected for some $\mu\ge R$, then $\EE_\mu(X)$
  is compact.
\end{prop}

\begin{proof}
  This holds because the spaces $\UU_{\mu,B}$ are finite by
  Corollary~\ref{c: finite number of mu-connected components}.
\end{proof}

\begin{lemma}\label{l: eta_mu,B,C is surjective}
  If $M$ is of coarse bounded geometry with coarse bound $(R,Q_r)$,
  and coarsely $\mu$-connected for some $\mu\ge R$, then
  $\eta_{\mu,B,C}$ is surjective for $\emptyset\ne B\subset C$ in
  $\BB$.
\end{lemma}

\begin{proof}
  Take any $U\in\UU_{\mu,B}$. By Corollary~\ref{c: there are unbounded
    mu-connected component}, $U\sm C$ has some unbounded coarse
  $\mu$-connected component $V$. Then $V\in\UU_{\mu,C}$ by
  Lemma~\ref{l: mu-connected components of U sm C}, and
  $\eta_{\mu,B,C}(V)=U$.
\end{proof}

\begin{cor}\label{c: eta_mu,B is surjective}
  If $M$ is of coarse bounded geometry with coarse bound $(R,Q_r)$,
  and coarsely $\mu$-connected for some $\mu\ge R$, then
  $\eta_{\mu,B}$ is surjective for $\emptyset\ne B\in\BB$.
\end{cor}

\begin{proof}
  By inductively applying Lemma~\ref{l: eta_mu,B,C is surjective}, it
  follows that for every $U\in\UU_{\mu,B}$ there is a nested sequence
  $U=U_0\supset U_1\supset\cdots$ such that
  $U_n\in\UU_{\mu,\Pen(B,n)}$ for all $n\in\N$. Since
  $\{\,\Pen(B,n)\mid n\in\N\,\}$ is cofinal in $\BB$,
  there is a unique $\bfe\in\EE_\mu(M)$ such that
  $\bfe(\Pen(B,n))=U_n$ for all $n\in\N$; in particular,
  $U=\bfe(B)=\eta_{\mu,B}(\bfe)$.
\end{proof}

\begin{lemma}\label{l: theta_mu,nu,B is surjective}
  If $M$ is of coarse bounded geometry with coarse bound $(R,Q_r)$,
  and coarsely $\mu$-connected for some $\mu\ge R$, then
  $\theta_{\mu,\nu,B}$ is surjective for nonempty sets  \(B\in\BB\) and
  $\nu\ge\mu>0$.
\end{lemma}

\begin{proof}
  Every $V\in\UU_{\nu,B}$ is union of coarse $\mu$-connected components
  of $M\sm B$ (Remark~\ref{r: mu-connected}-\eqref{i: nu-conn for all
    nu ge mu}). Moreover $M\sm B$ has a finite number of coarse
  $\mu$-connected components that meet $M\sm\Pen(B,2\mu)$ by
  Corollary~\ref{c: finite number of mu-connected components}. Since
  $V$ is unbounded, it follows that $V$ contains some unbounded coarse
  $\mu$-connected component $U$ of $M\sm B$. Thus $U\in\UU_{\mu,B}$
  and $\theta_{\mu,\nu,B}(U)=V$.
\end{proof}

\begin{lemma}\label{l: theta_mu,nu has dense image}
  If $M$ is of coarse bounded geometry with coarse bound $(R,Q_r)$,
  and coarsely $\mu$-connected for some $\mu\ge R$, then
  $\theta_{\mu,\nu}:\EE_\mu(M)\to\EE_\nu(M)$ has dense image for
  $\nu\ge\mu$.
\end{lemma}

\begin{proof}
  This follows from Corollary~\ref{c: eta_mu,B is surjective},
  Lemma~\ref{l: theta_mu,nu,B is surjective} and \cite[Appendix Two,
  Theorem~2.5-(2), p.~430]{Dugundji1978}.
\end{proof}

\begin{prop}\label{p: Ends_infty(X) is compact}
  If $M$ is of coarse bounded geometry and coarsely connected, then
  $\EE_\infty(M)$ is compact.
\end{prop}

\begin{proof}
  $\EE_\mu(M)$ is compact for $\mu$ large enough by
  Proposition~\ref{p: Ends_mu(X) is compact}. Moreover it is also
  Hausdorff. So $\theta_{\mu,\nu}$ is surjective for $\mu$ large
  enough and $\nu\ge\mu$ by Lemma~\ref{l: theta_mu,nu has dense
    image}, obtaining that $\theta_\mu:\EE_\mu(M)\to\EE_\infty(M)$ is
  surjective. Therefore $\EE_\infty(M)$ is compact.
\end{proof}

\begin{lemma}\label{l: mu-connected subsets of graphs and Riem mfds}
  \begin{enumerate}[{\rm(}i\/{\rm)}]
		
  \item\label{i: mu-connected subsets of graphs} Suppose that $M$ is
    the metric space of vertices of a connected graph. For $B\in\BB$
    and $N\in\Z^+$, let
    $\widetilde{B}=\Pen(B,\lceil\frac{N-1}{2}\rceil)$. Then every
    coarse $N$-connected component of $M\sm\widetilde{B}$ is contained
    in some coarse $1$-connected component of $M\sm B$.
		
  \item\label{i: mu-connected subsets of Riem mfds} Assume that $M$ is
    a connected complete Riemannian manifold\footnote{In fact, the
      proof of this property applies to any complete path metric
      space, as well as the proof of Proposition~\ref{p: coarse ends
        of Riemannian mfds}.}. For every closed $B\in\BB$ and
    $\mu,\epsilon>0$, let
    $\widetilde{B}=\Pen(B,\frac{\mu+\epsilon}{2})$. Then every coarse
    $\mu$-connected component of $M\sm\widetilde{B}$ is contained in
    some connected component of $M\sm B$.
		
  \end{enumerate}
\end{lemma}

\begin{proof}
  \eqref{i: mu-connected subsets of graphs} Let $U$ be a coarse
  $N$-connected component of $M\sm\widetilde{B}$. For $x,y\in U$,
  there are points $x=z_0,z_1,\dots,z_k=y$ in $U$ such that
  $d(z_{l-1},z_l)\le N$ for all $l\in\{1,\dots,k\}$. Then there are
  points, $z_{l-1}=u^l_0,u^l_1,\dots,u^l_{p_l}=z_l$, in $M$ such that
  $p_l\le N$ and $d(u^l_{q-1},u^l_q)=1$ for all
  $q\in\{1,\dots,p_l\}$. Observe that
  $u^l_q\in\Pen(U,\lceil\frac{N-1}{2}\rceil)\subset M\sm B$. So $x$ is
  coarsely $1$-connected to $y$ in $M\sm B$. It follows that $U$ is a
  coarsely $1$-connected subset of $M\sm B$, and therefore it is
  contained in some coarsely $1$-connected component of $M\sm B$.
	
  \eqref{i: mu-connected subsets of Riem mfds} For $x,y\in
  U\in\UU_{\mu,\widetilde{B}}$, there are points
  $x=z_0,z_1,\dots,z_k=y$ in $U$ such that $d(z_{l-1},z_l)\le \mu$ for
  all $l\in\{1,\dots,k\}$. Thus there is a smooth path $\alpha_l$ from
  $z_{l-1}$ to $z_l$ with length $<\mu+\epsilon$. Each $\alpha_l$ is a
  path in $\Pen(U,\frac{\mu+\epsilon}{2})\subset M\sm B$, and so the
  product path $\alpha_1\cdots \alpha_k$ joins $x$ to $y$ in $M\sm
  B$. Thus $U$ is a connected subset of $M\sm B$, and hence it is
  contained in some connected component of $M\sm B$.
\end{proof} 

\begin{prop}\label{p: ends of a graph}
  If $M$ is the metric space of vertices of a connected graph $G$ that
  has finitely many edges abutting on each vertex, then
  $\EE(G)\equiv\EE_1(M)\approx\EE_\infty(M)$, canonically.
\end{prop}

\begin{proof}	
  The definitions of $\EE_1(M)$ and $\EE(G)$ are canonically
  equivalent so it has to be shown that
  $\EE_1(M)\approx\EE_\infty(M)$, canonically, which will be a
  consequence of showing that $\theta_{1,N}$ is a homeomorphism for
  each $N\in\Z^+$. For every $B\in\BB$, let
  $\widetilde{B}=\Pen(B,\lceil\frac{N-1}{2}\rceil)$. By Lemma~\ref{l:
    mu-connected subsets of graphs and Riem mfds}-\eqref{i:
    mu-connected subsets of graphs}, a map
  $\xi_{N,B}:\UU_{N,\widetilde{B}}\to\UU_{1,B}$ is determined by
  $\xi_{N,B}(U)\supset U$. Like in the case of~\eqref{CD}, it can be
  easily checked that, for $B\subset C$ in $\BB$, the diagrams
  \begin{gather*}
    \begin{CD}
      \UU_{N,\widetilde{C}} @>{\xi_{N,C}}>> \UU_{1,C} \\
      @VV{\eta_{N,\widetilde{B},\widetilde{C}}}V @VV{\eta_{N,B,C}}V \\
      \UU_{N,\widetilde{B}} @>{\xi_{N,B}}>> \UU_{1,B}
    \end{CD}
    \qquad
    \begin{CD}
      \UU_{N,\widetilde{B}} @>{\xi_{N,B}}>> \UU_{1,B} \\
      @| @V{\theta_{1,N,B}}VV \\
      \UU_{N,\widetilde{B}} @>{\eta_{N,B,\widetilde{B}}}>> \UU_{N,B}
    \end{CD}
    \qquad
    \begin{CD}
      \UU_{1,\widetilde{B}} @>{\theta_{1,N,\widetilde{B}}}>> \UU_{N,\widetilde{B}} \\
      @| @V{\xi_{N,B}}VV \\
      \UU_{1,\widetilde{B}} @>{\eta_{1,B,\widetilde{B}}}>> \UU_{1,B}
    \end{CD}
  \end{gather*}
  are commutative. So the maps $\xi_{N,B}$
  induce a continuous map $\xi_N:\EE_N(M)\to\EE_1(M)$ that is inverse
  of $\theta_{1,N}$. Thus $\theta_{1,N}$ is a homeomorphism.
\end{proof}

\begin{prop}\label{p: coarse ends of Riemannian mfds}
  If $M$ is a connected complete Riemannian manifold, then
  $\EE(M)\approx\EE_\mu(M)\approx\EE_\infty(M)$, canonically, for all
  $\mu>0$.
\end{prop}

\begin{proof}
  The proof is similar to that of Proposition~\ref{p: ends of a
    graph}, using Lemma~\ref{l: mu-connected subsets of graphs and
    Riem mfds}-\eqref{i: mu-connected subsets of Riem mfds}.
\end{proof}

\section{Functoriality of the space of coarse ends}\label{s: functoriality of coarse ends}

Let $f:M\to M'$ be a coarse map; in particular, it satisfies the
condition of uniform expansiveness with some mapping $s_r$, which can
be assumed to be non-decreasing. We have $f^{-1}(B')\in\BB(M)$ for all
$B'\in\BB(M')$ because $f$ is metric proper. For $x,y\in
U\in\UU_{\mu,f^{-1}(B')}$, there are points $x=z_0,z_1,\dots,z_k=y$ in
$U$ so that $d(z_{l-1},z_l)\le\mu$ for $l\in\{1,\dots,k\}$. Since
$d'(f(z_l),f(z_{l-1}))\le s_\mu$, we get that $f(U)$ is a coarsely
$s_\mu$-connected subset of $M'\sm B'$. Hence $f(U)$ is contained in
some coarse $s_\mu$-connected component $U'$ of $M'\sm B'$. Moreover
$U'$ is unbounded; otherwise, $f^{-1}(f(U))$ is bounded because $f$ is
metrically proper, obtaining that $U$ is bounded, a
contradiction. Thus there is a map
$f_{\mu,B'}:\UU_{\mu,f^{-1}(B')}(M)\to\UU_{s_\mu,B'}(M')$ determined
by $f_{\mu,B'}(U)\supset f(U)$. Since $B\subset f^{-1}(f(B))$ for all
$B\in\BB(M)$, and $f(B)\in\BB(M')$ by the uniform expansiveness of
$f$, the set $\{\,f^{-1}(B')\mid B'\in\BB(M')\,\}$ is cofinal in
$\BB$. Hence the maps $f_{\mu,B'}$ induce a continuous map
$f_\mu:\EE_\mu(M)\to\EE_{s_\mu}(M')$, determined by the condition
$\eta_{\mu,B'}\circ f_\mu=f_{\mu,B'}\circ \eta_{\mu,f^{-1}(B')}$ for
all $B'\in\BB(M')$. Thus
\begin{equation}\label{f_mu(e)(B')}
  f_\mu(\bfe)(B')=\eta_{\mu,B'}\circ f_\mu(\bfe)=f_{\mu,B'}\circ \eta_{\mu,f^{-1}(B')}(\bfe)\\
  =f_{\mu,B'}(\bfe(f^{-1}(B')))
\end{equation} 
for all $B'\in\BB(M')$. As in the case of~\eqref{CD}, it is easy to
check that the diagram
$$
\begin{CD}
  \UU_{\nu,f^{-1}(B')}(M) @>{f_\nu}>> \UU_{s_\nu,B'}(M') \\
  @A{\theta_{\mu,\nu,f^{-1}(B')}}AA @AA{\theta_{s_\mu,s_\nu,B'}}A \\
  \UU_{\mu,f^{-1}(B')}(M) @>{f_\mu}>> \UU_{s_\mu,B'}(M')
\end{CD}
$$
is commutative for $0<\mu<\nu$, and thus the diagram
$$
\begin{CD}
  \EE_\nu(M) @>{f_\nu}>> \EE_{s_\nu}(M') \phantom{\;.} \\
  @A{\theta_{\mu,\nu}}AA @AA{\theta_{s_\mu,s_\nu}}A \\
  \EE_\mu(M) @>{f_\mu}>> \EE_{s_\mu}(M')\;.
\end{CD}
$$
is also commutative.  So the maps $f_\mu$ induce a continuous map
$f_\infty:\EE_\infty(M)\to\EE_\infty(M')$.

\begin{lemma}\label{l: f_infty is independent of s_r}
  $f_\infty$ is independent of the choice of $s_r$.
\end{lemma}

\begin{proof}
  Let $\bar s_r\ge s_r$ for each $r\ge0$, and let
  \begin{gather*}
    \bar f_{\mu,B'}:\UU_{\mu,f^{-1}(B')}(M)\to\UU_{\bar s_\mu,B'}(M')\;,\\
    \bar f_\mu:\EE_\mu(M)\to\EE_{\bar s_\mu}(M')\;,\quad \bar
    f_\infty:\EE_\infty(M)\to\EE_\infty(M')
  \end{gather*} 
  be the maps induced by $f$ and $\bar s_r$. For $\mu>0$ and
  $B'\in\BB(M')$, the commutativity of the diagram
$$
\begin{CD}
  \UU_{\mu,f^{-1}(B')}(M) @>{\bar f_{\mu,B'}}>> \UU_{\bar s_\mu,B'}(M') \\
  @| @AA{\theta_{s_\mu,\bar s_\mu,B'}}A \\
  \UU_{\mu,f^{-1}(B')}(M) @>{f_{\mu,B'}}>> \UU_{s_\mu,B'}(M')
\end{CD}
$$
follows like in the case of~\eqref{CD}, obtaining the commutativity of
$$
\begin{CD}
  \EE_\mu(M) @>{\bar f_\mu}>> \EE_{\bar s_\mu}(M') \phantom{\;.} \\
  @| @AA{\theta_{s_\mu,\bar s_\mu}}A \\
  \EE_\mu(M) @>{f_\mu}>> \EE_{s_\mu}(M')\;.
\end{CD}
$$
Therefore $\bar f_\infty=f_\infty$.
\end{proof}

\begin{prop}\label{p: f_infty = g_infty}
  If $g:M\to M'$ is another coarse map close to $f$, then
  $g_\infty=f_\infty$.
\end{prop}

\begin{proof}
  Take some $c\ge0$ such that $f$ and $g$ are $c$-close. We can
  suppose that $g$ also satisfies the condition of uniform
  expansiveness with $s_r$. Let $\mu>0$, $B\in\BB(M)$, $B'\in\BB(M')$,
  $U\in\UU_{\mu,B}(M)$, $V\in\UU_{\mu,f^{-1}(B')}(M)$ and
  $W\in\UU_{\mu,g^{-1}(B')}(M)$ such that $f^{-1}(B')\cup
  g^{-1}(B')\subset B$ and $U\subset V\cap W$. Therefore 
  $\eta_{\mu,f^{-1}(B'),B}(U)=V$ and
  $\eta_{\mu,g^{-1}(B'),B}(U)=W$. We know that $f(V)$ and $g(W)$ are
  coarsely $s_\mu$-connected subsets of $M'\sm B'$. By Remark~\ref{r:
    mu-connected}-\eqref{i: A cup B is mu-conn} and since
  $d'(f(x),g(x))\le c$ for all $x\in U$, it follows that $f(V)\cup
  g(W)$ is a coarsely $\bar s_\mu$-connected subset of $M'\sm B'$,
  where $\bar s_\mu=\max\{s_\mu,c\}$. Hence
$f(V)$ and $g(W)$ are
  contained in the same coarse $\bar s_\mu$-connected component of
  $M'\sm B'$. This shows that the diagram
$$
\begin{CD}
  \UU_{\mu,f^{-1}(B')}(M) @>{f_{\mu,B'}}>> \UU_{s_\mu,B'}(M')
  @>{\theta_{s_\mu,\bar s_\mu,B'}}>> \UU_{\bar s_\mu,B'}(M') \\
  @A{\eta_{\mu,f^{-1}(B'),B}}AA && @AA{\theta_{s_\mu,\bar s_\mu,B'}}A \\
  \UU_{\mu,B}(M) @>{\eta_{\mu,g^{-1}(B'),B}}>> \UU_{\mu,g^{-1}(B')}(M)
  @>{g_{\mu,B'}}>> \UU_{s_\mu,B'}(M')
\end{CD}
$$
is commutative. So the diagram
$$
\begin{CD}
  \EE_{s_\mu}(M') @>{\theta_{s_\mu,\bar s_\mu}}>> \EE_{\bar s_\mu}(M') \\
  @A{f_\mu}AA @AA{\theta_{s_\mu,\bar s_\mu}}A \\
  \EE_\mu(M) @>{g_\mu}>> \EE_{s_\mu}(M')
\end{CD}
$$
is also commutative, and $f_\infty=g_\infty$ obtains.
\end{proof}

\begin{prop}\label{p: functor End_infty}
  A functor $\EE_\infty$ of the metric coarse category to the category
  of continuous maps between topological spaces is defined by
  $\EE_\infty([M])=\EE_\infty(M)$ and $\EE_\infty([f])=f_\infty$.
\end{prop}

\begin{proof}
  Obviously, $(\id_M)_\mu=\id_{\EE_\mu(M)}$ for all $\mu>0$, and
  therefore $(\id_M)_\infty=\id_{\EE_\infty(M)}$. Suppose that
  $f':M'\to M''$ satisfies the condition of uniform expansiveness with
  a mapping $s'_r$. Then $f'f$ satisfies the condition of uniform
  expansiveness with the mapping $s'_{s_r}$, and, like in the case
  of~\eqref{CD}, we get that the composite
$$
\begin{CD}
  \UU_{\mu,f^{-1}({f'}^{-1}(B''))}(M) @>{f_{\mu,{f'}^{-1}(B'')}}>>
  \UU_{s_\mu,{f'}^{-1}(B'')}(M') @>{f'_{s_\mu,B''}}>>
  \UU_{s'_{s_\mu},B''}(M'')
\end{CD}
$$
equals $(f' f)_{\mu,B''}$ for all $\mu>0$ and $B''\in\BB(M'')$. So the composite
$$
\begin{CD}
  \EE_\mu(M) @>{f_\mu}>> \EE_{s_\mu}(M') @>{f'_{s_\mu}}>>
  \EE_{s'_{s_\mu}}(M'')
\end{CD}
$$
equals $(f' f)_\mu$ for all $\mu>0$, obtaining $f'_\infty f_\infty=(f'
f)_\infty$. This shows that $M\mapsto\EE_\infty(M)$ and $f\mapsto
f_\infty$ defines a functor of the category of coarse maps between
metric spaces to the category of continuous maps between topological
spaces. Then $\EE_\infty(M)$ depends only on $[M]$. Furthermore
$f_\infty$ depends only on $[f]$ by Lemma~\ref{l: f_infty is
  independent of s_r} and Proposition~\ref{p: f_infty = g_infty}.
\end{proof}

\begin{rem}\label{r: functor End_infty}
  Continuing with the ideas of Remark~\ref{r: ends of coarse spaces},
  the functor $\EE_\infty$ can be obviously extended to the whole
  coarse category.
\end{rem}

\begin{cor}\label{c: End_infty and coarse equivalences}
  If $f:M\to M'$ is a coarse equivalence, then $f_\infty$ is a
  homeomorphism.
\end{cor}

\begin{cor}\label{c: End_infty and coarse quasi-isometries}
The spaces of ends of two coarsely quasi-isometric metric spaces are
homeomorphic.
\end{cor}

\section{Coarse end space of a class of metric spaces}\label{s: coarse end sp of a class of metric sps}

The following result extends a well known theorem for finitely
generated groups (Example~\ref{ex: group})
\cite[Theorem~13.5.7]{Geoghegan2008}.

\begin{thm}\label{t: coarse ends}
  Assume that $M$ is of coarse bounded geometry, coarsely quasi-convex
  and coarsely quasi-symmetric. Then, either $|\EE_\infty(M)|\le2$, or
  $\EE_\infty(M)$ is a Cantor space.
\end{thm}

\begin{proof}
  By Corollaries~\ref{c: coarse bounded geometry and coarse
    quasi-isometries} and~\ref{c: End_infty and coarse
    quasi-isometries}, Propositions~\ref{p:coarsely quasi-symmetric 2}
  and~\ref{p: coarsely quasi-convex}, and Example~\ref{ex: coarse
    bounded geometry}, we can assume that $M$ is the metric space of
  vertices of a connected graph of finite type. Thus
  $\EE_1(M)\approx\EE_\infty(M)$ by Proposition~\ref{p: ends of a
    graph}. Suppose that $|\EE_1(M)|\ge3$, and let us prove that
  $\EE_1(M)$ is a Cantor space. Take three distinct elements
  $\bfe_k\in\EE_1(M)$, $k\in\{1,2,3\}$. Since $\EE_1(M)$ is Hausdorff,
  totally disconnected and compact (Propositions~\ref{p: Ends_infty(X)
    is compact} and~\ref{p: ends of a graph}), it is enough to prove
  that $|\NN_1(B,U)|\ge2$ for any basic open set $\NN_1(B,U)$ of
  $\EE_1(M)$ (Section~\ref{s: coarse ends}). We can assume that $B$
  is coarsely $1$-connected and the sets $\bfe_k(B)$ are distinct.
	
  By Propositions~\ref{p: large scale Lipschitz extensions}
  and~\ref{p: any large scale Lipschitz map is rough}, there is a
  transitive set $\TT$ of equi-rough transformations of $M$, and
  let $(s_r,c)$ be a common rough equivalence distortion of all maps
  in $\TT$. We can suppose that $s_r\uparrow\infty$ as $r\to\infty$,
  and $s_n,c\in\N$ for all $n\in\N$. Let $N=\max\{s_{s_1},c\}\ge1$ and
  $\widetilde{B}=\Pen(B,\lceil\frac{N-1}{2}\rceil)$.
	
  \begin{claim}\label{cl: different ends}
    For each $f\in\TT$, the sets
    $f_1(\bfe_k)(\Pen(f(\widetilde{B}),c))$ are distinct.
  \end{claim}
	
  Given any $f\in\TT$, let $g:M\to M$ be a rough equivalence with
  rough distortion $s_r$ such that $g f$ and $f g$ are $c$-close to
  $\id_M$. Let $C=\Pen(\widetilde{B},c)$. We have
  $(gf)^{-1}(\widetilde{B})\subset C$ because, if
  $x\in(gf)^{-1}(\widetilde{B})$, then $gf(x)\in\widetilde{B}$ and
  $d(x,gf(x))\le c$. Then the diagram
$$
\begin{CD}
  \UU_{1,\widetilde{B}} @>{\eta_{1,B,\widetilde{B}}}>> \UU_{1,B} \\
  @| @AA{\xi_{N,B}}A && \\
  \UU_{1,\widetilde{B}} @>{\theta_{1,N,\widetilde{B}}}>>
  \UU_{N,\widetilde{B}}
  @<{\theta_{s_{s_1},N,\widetilde{B}}}<< \UU_{s_{s_1},\widetilde{B}} \\
  @A{\eta_{1,\widetilde{B},C}}AA && @AA{g_{1,\widetilde{B}}}A \\
  \UU_{1,C} @>{\eta_{1,(gf)^{-1}(\widetilde{B}),C}}>>
  \UU_{1,(gf)^{-1}(\widetilde{B})}(M)
  @>{f_{1,g^{-1}(\widetilde{B})}}>> \UU_{s_1,g^{-1}(\widetilde{B})}
			\end{CD}
$$
is commutative according to the proofs of Propositions~\ref{p: ends of
  a graph},~\ref{p: f_infty = g_infty} and~\ref{p: functor
  End_infty}. Hence, by~\eqref{eta_mu,B,C(e(B))},
\begin{multline*}
  \xi_{N,B}\circ\theta_{s_{s_1},N,\widetilde{B}}\circ
  g_{1,\widetilde{B}}\circ f_{1,g^{-1}(\widetilde{B})}\circ
  \eta_{1,(gf)^{-1}(\widetilde{B}),C}(\bfe_k(C))\\
  \begin{aligned}
  &=\xi_{N,B}\circ \theta_{1,N,\widetilde{B}}\circ
  \eta_{1,\widetilde{B},C}(\bfe_k(C))
  =\xi_{N,B}\circ \theta_{1,N,\widetilde{B}}(\bfe_k(\widetilde{B}))\\
  &=\eta_{1,B,\widetilde{B}}(\bfe_k(\widetilde{B})) =\bfe_k(B)\;,
  \end{aligned}
\end{multline*}
which are distinct sets for $k\in\{1,2,3\}$. Thus the sets
\begin{align*}
  f_{1,g^{-1}(\widetilde{B})}\circ
  \eta_{1,(gf)^{-1}(\widetilde{B}),C}(\bfe_k(C))
  &=f_{1,(gf)^{-1}(\widetilde{B})}(\bfe_k((gf)^{-1}(\widetilde{B})))\\
  &=f_1(\bfe_k)(\bfe_k(g^{-1}(\widetilde{B})))
\end{align*}
are also distinct, where we have used~\eqref{eta_mu,B,C(e(B))}
and~\eqref{f_mu(e)(B')}. On the other hand,
$g^{-1}(\widetilde{B})\subset\Pen(f(\widetilde{B}),c)$ because, if
$x\in g^{-1}(\widetilde{B})$, then $fg(x)\in f(\widetilde{B})$ and
$d(fg(x),x)\le c$. Then Claim~\ref{cl: different ends} follows because
$$
f_1(\bfe_k)(\bfe_k(g^{-1}(\widetilde{B})))
=\eta_{1,g^{-1}(\widetilde{B}),\Pen(f(\widetilde{B}),c)}\circ
f_1(\bfe_k) (\bfe_k(\Pen(f(\widetilde{B}),c)))\;.
$$
		
Take any integer $R\ge\diam B$.  Since $U$ is unbounded and $\TT$
transitive, we can fix some $f\in\TT$ such that there is some $x\in
U\cap f(\widetilde{B})$ with
\[
                d(x,B)>c+s_{R+N}+N+1\;.
\]
Since $\diam\widetilde{B}\le R+N$, we have $\diam f(\widetilde{B})\le
s_{R+N}$, obtaining
\begin{equation}\label{d(f(widetilde B),B)>c+N}
  d(f(\widetilde{B}),B)>c+N+1\;.
\end{equation} 
Let
\begin{gather*}
  B'=\Pen(f(\widetilde{B}),c)\;,\quad B''=\Pen(f(\widetilde{B}),c+\lceil\textstyle{\frac{N-1}{2}\rceil})\;,\\
  V_k=f_1(\bfe_k)(B')\in\UU_{s_1,B'}\;,\quad
  W_k=f_1(\bfe_k)(B'')\in\UU_{s_1,B''}\;.
\end{gather*}
Using the notation of the proof of Proposition~\ref{p: ends of a
  graph}, the composite
\[
\begin{CD}
  \UU_{s_1,B''} @>{\theta_{s_1,N,B''}}>> \UU_{N,B''} @>{\xi_{N,B'}}>>
  \UU_{1,B'}
\end{CD}
\]
is defined because $N\ge s_1$. So each $W_k$ is contained in some
$U'_k\in\UU_{1,B'}$. Moreover $U'_k\subset V_k$ by Remark~\ref{r:
  mu-connected}-\eqref{i: nu-conn for all nu ge mu} because
$W_k\subset V_k$; in particular, the sets $U'_k$ are disjoint from
each other by Claim~\ref{cl: different ends}.
	
\begin{claim}\label{cl: U meets U'_k}
  $U$ meets all sets $U'_k$.
\end{claim}
	
Since $B$ is coarsely $1$-connected, $\Pen(B'',1)$ is coarsely
$s_1$-connected. Moreover $\Pen(B'',1)$ is disjoint from $\widetilde
B$ by~\eqref{d(f(widetilde B),B)>c+N}. So $\Pen(B'',1)$ is contained
in some coarse $1$-connected component of $M\sm B$ by
Lemma~\ref{l: mu-connected subsets of graphs and Riem mfds}-\eqref{i:
  mu-connected subsets of graphs} and Remark~\ref{r:
  mu-connected}-\eqref{i: mu-conn comp are maximal} because $N\ge
s_1$. Since $f(B)$ meets $U$, we get $\Pen(B'',1)\subset U$. So $U$
meets every $W_k$ by Lemma~\ref{l: the mu-connected components of M
  setminus B are close to B}, and therefore it also meets every
$U'_k$, showing Claim~\ref{cl: U meets U'_k}.
			
	\begin{claim}\label{cl: at most one}
		$\widetilde{B}$ meets at most one of the sets $U'_k$.
	\end{claim}
	
	Indeed, suppose that $\widetilde{B}$ meets two of the sets
        $U'_k$, say $U'_1$ and $U'_2$. Since $U'_1,U'_2\in\UU_{1,B'}$
        and $\widetilde{B}$ is coarsely $1$-connected and disjoint
        from $B'$ by~\eqref{d(f(widetilde B),B)>c+N}, it follows that
        $U'_1\cup U'_2\cup\widetilde{B}$ is an unbounded coarsely
        $1$-connected subset of $M\sm B'$. Therefore $U'_1=U'_2$, 
        a contradiction, confirming Claim~\ref{cl: at most one}.
	
According to Claim~\ref{cl: at most one}, assume from now on that
$U'_1,U'_2 \subset M\sm\widetilde{B}$. Since these subsets are
coarsely $1$-connected, they are contained in coarse $1$-connected
components of $M\sm B$ (Remark~\ref{r: mu-connected}-\eqref{i: mu-conn
  comp are maximal}). Then $U'_1,U'_2\subset U$ by Claim~\ref{cl: U
  meets U'_k}. By Corollary~\ref{c: mu-connected components of U sm(A
  cup B)}, $U'_1,U'_2\in\UU_{1,B\cup B'}$, and, by Corollary~\ref{c:
  eta_mu,B is surjective}, there are $\bfe'_1,\bfe'_2\in\EE_1(M)$ such
that $\bfe'_k(B\cup B')=U'_k$ for $k\in\{1,2\}$. So
$\bfe'_1\ne\bfe'_2$ and $\bfe'_k(B)\supset\bfe'_k(B\cup B')=U'_k$,
obtaining $\bfe'_k(B)=U$ because $U'_k\subset U$. Thus
$\bfe'_1,\bfe'_2\in\NN_1(B,U)$, showing that $|\NN_1(B,U)|\ge2$.
\end{proof}

\chapter{Higson corona and asymptotic dimension}\label{c: Higson}

The Higson corona of any coarse space is very important in coarse geometry; roughly speaking, it contains all coarse information. Even though it is a very vast space, some of our main results state that the Higson corona of leaves has some nice properties, Theorems~\ref{t: semi weakly homogeneous, leaves} and~\ref{t: lim e is FF-saturated}--\ref{t: lim e is an FF-minimal set}. 

The asymptotic dimension is also recalled in this chapter. It is strongly related to the topological dimension of the Higson corona, as indicated in Chapter~\ref{c: intro}. Our main result Theorem~\ref{t: asdim leaves} is about the asymptotic dimension of leaves.

\section{Compactifications}\label{s: compactifications}

Recall that a \emph{compactification} \index{compactification} of a
topological space $X$ is a pair $(Y,h)$ consisting of a compact
Hausdorff\footnote{We only consider Hausdorff compactifications.}
space $Y$ and an embedding $h:X\to Y$ with dense image. The subspace
$Y\sm h(X)\subset Y$ is called the \emph{corona} \index{corona} of the
compactification. Usually, the notation is simplified by assuming that
$X\subset Y$ and $h$ is the inclusion map, which is omitted from the
notation; in particular, it will be simply said that $X$ is open in
$Y$ when $h$ is an open map. A typical notation is $\ol X$ for a
compactification and $\partial X$ for the corona, or $X^\gamma$ for
the compactification and $\gamma X$ for the corona, specially when
$\gamma$ refers to some kind of compactification of a class of spaces.

The space $X$ admits a compactification if and only if it is
Tychonov. Moreover $X$ is open in some compactification if and only if
it is locally compact and Hausdorff; in this case, $X$ is open in any
compactification.

Two compactifications of $X$, $(Y,h)$ and $(Y',h')$, are
\emph{equivalent} when there is a homeomorphism $\phi:Y'\to Y$ so that
$\phi h'=h$. The term ``compactification'' will refer to an
equivalence classes of compactifications. In this sense, the
set\footnote{All compactifications of $X$ are $\le X^\beta$, where
  $X^\beta$ is the Stone-\v{C}ech compactification. Thus we can assume
  that they are quotients of $X^\beta$, and therefore they form a
  set.} of compactifications has a partial order relation, ``$\le$'',
defined by declaring $(Y,h)\le(Y',h')$ if there is a continuous
$\pi:Y'\to Y$ so that $\pi h'=h$.

For a locally compact Hausdorff space $X$, let $C_b(X)$ denote the
commutative $C^*$ algebra of bounded $\C$-valued continuous functions
on $X$ with the supremum norm, and let $C_0(X)\subset C_b(X)$ be the
closed involutive ideal of continuous functions that vanish at
infinity\footnote{Recall that a function $f:X\to\C$ \emph{vanishes at
    infinity} when, for all $\epsilon>0$, there is a compact $K\subset
  X$ so that $|f(x)|<\epsilon$ for all $x\in X\sm K$.}. The closed
subalgebra of constant functions on $X$ may be canonicaly identified
to $\C$.

The Gelfand-Naimark theorem estates that the assignment $X\mapsto
C_b(X)$ defines a one-to-one correspondence between the (homeomorphism
classes of) compact Hausdorff spaces and the (isomorphism classes of)
commutative $C^*$ algebras with unit. The compact Hausdorff space
$\Delta(\AA)$ that corresponds to each unital commutative $C^*$
algebra $\AA$ is the space of characters $\AA\to \C$ with the topology
of pointwise convergence. For a compact Hausdorff space $X$, a
canonical homeomorphism $h:X\to\Delta(C_b(X))$ is given by evaluation:
$h(x)(f)=f(x)$ for all $x\in X$ and $f\in C_b(X)$.

More generally, when $X$ is a locally compact Hausdorff space
$X$, the assignment $(Y,h)\mapsto h^*C(Y)$ is a one-to-one
correspondence between (equivalence classes) of Hausdorff
compactifications of $X$ and unital closed involutive subalgebras of
$C_b(X)$ that generate the topology (in the sense that compact sets
can be separated from points by functions in the algebra). For each
subalgebra $\AA\subset C(X)$ of the above type, the corresponding
compactification is $(\Delta(\AA),h)$, where $h:X\to\Delta(\AA)$ is
defined by evaluation at each point of $X$, as before. For instance,
$C_b(X)$ correponds to the Stone-\v{C}ech compactification $X^\beta$,
and $\C+C_0(X)$ corresponds to the one-point compactification $X^*$.

\begin{example}\label{ex: compactification by coarse ends}
  Consider the notation of Section~\ref{s: ends}. If $M$ is of coarse
  bounded geometry and coarsely connected, then $\EE_\infty(M)$ is
  compact (Proposition~\ref{p: Ends_infty(X) is compact}). If moreover
  $M$ is proper, then $\EE_\infty(M)$ is the corona of a
  compactification of $M$, which can be seen as follows. To show this
  property, let us prove first that $\EE_\mu(M)$ is the corona of a
  compactification of $M$ for $\mu>0$ large enough, like the usual
  space of ends of a manifold or a graph. The space $\EE_\mu(M)$ is
  compact for $\mu>0$ large enough (Proposition~\ref{p: Ends_mu(X) is
    compact}). Then let $\ol M_\mu=M\cup\EE_\mu(M)$ with the topology
  so that the inclusion map $M\hookrightarrow\ol M_\mu$ is an open
  embedding, and a base of neighborhoods in $\ol M_\mu$ of any
  $\bfe\in\EE_\mu(M)$ is given by the sets $\bfe(B)$ for
  $B\in\BB(M)$. Using Corollary~\ref{c: finite number of mu-connected
    components}, it can be easily seen that $\ol M_\mu$ is compact,
  and it is obvious that $M$ is dense in $\ol M_\mu$. Now, for
  $\mu<\nu$ large enough, the continuous map
  $\theta_{\mu,\nu}:\EE_\mu(M)\to\EE_\nu(M)$ is surjective because it
  has dense image (Lemma~\ref{l: theta_mu,nu has dense image}) and
  these spaces are compact and Hausdorff. The combination of
  $\theta_{\mu,\nu}$ and $\id_M$, denoted by $\bar\theta_{\mu,\nu}:\ol
  M_\mu\to\ol M_\nu$, is continuous by~\eqref{theta_mu,nu(e)(B)} and
  the definition of $\theta_{\mu,\nu,B}$; thus $\ol M_\nu\le\ol
  M_\mu$. Moreover $\{\ol M_\mu,\bar\theta_{\mu,\nu}\}$ is a direct
  system of topological spaces because so is
  $\{\EE_\mu(M),\theta_{\mu,\nu}\}$. Since $\EE_\infty(M)$ is the
  injective limit of $\{\EE_\mu(M),\theta_{\mu,\nu}\}$, it easily
  follows that the injective limit of $\{\ol
  M_\mu,\bar\theta_{\mu,\nu}\}$ is a compactification $\ol M_\infty$
  of $M$ with corona $\EE_\infty(M)$.
\end{example}

\begin{lemma}\label{l: ol X mapsto ol X'}
  Let $X^\gamma$ and ${X'}^\gamma$ be compactifications of locally
  compact Hausdorff spaces $X$ and $X'$, with coronas $\gamma X$ and
  $\gamma X'$, respectively. Let $\phi:X\to X'$ and $\psi:X'\to X$ be
  {\rm(}possibly non-continuous\/{\rm)} maps such that:
		\begin{enumerate}[{\rm(}i\/{\rm)}]
			
			\item\label{i: phi^gamma} $\phi$ and $\psi$ have extensions, $\phi^\gamma:X^\gamma\to{X'}^\gamma$ and $\psi^\gamma:{X'}^\gamma\to X^\gamma$, which are continuous at the points of $\gamma X$ and $\gamma X'$, respectively; and
			
			\item\label{i: gamma phi} $\phi^\gamma$ and $\psi^\gamma$ restrict to respective homeomorphisms, $\gamma\phi:\gamma X\to\gamma X'$ and $\gamma\psi:\gamma X'\to\gamma X$, which are inverse of each other.
			
		\end{enumerate}
	Then there is a bijection $\ol X\mapsto\ol{X'}$ between the set of compactifications $\ol X\le X^\gamma$, with coronas $\partial X$, and the set of compactifications $\ol{X'}\le{X'}^\gamma$, with coronas $\partial X'$, such that:
		\begin{enumerate}[{\rm(}a\/{\rm)}]
			
			\item\label{i: bar phi} $\phi$ and $\psi$ have extensions, $\bar\phi:\ol X\to\ol{X'}$ and $\bar\psi:\ol{X'}\to\ol X$, which are continuous at the points of $\partial X$ and $\partial X'$, respectively; and
			
			\item\label{i: partial phi} $\bar\phi$ and $\bar\psi$ restrict to respective homeomorphisms, $\partial\phi:\partial X\to\partial X'$ and $\partial\psi:\partial X'\to\partial X$, which are inverse of each other.
			
		\end{enumerate}
\end{lemma}

\begin{proof}
  Take a compactification $\ol X\le X^\gamma$, with corona $\partial
  X$. Thus $\id_M$ has a continuous extension $\pi:\gamma X\to\partial
  X$, which is an identification. Then the restriction
  $\pi:X^\gamma\to\partial X$ is also an identification. Let $R$ be
  the equivalence relation on $\gamma X$ whose equivalence classes are
  the fibers of $\pi$ (thus $\partial X\equiv\gamma X/R$), let $R'$ be
  the equivalence relation on $\gamma X'$ that corresponds to $R$ via
  $\gamma\phi$, and let $\partial X'=\gamma X'/R'$. Then
  $\partial\phi$ induces a homeomorphism
  $\partial'\phi:\partial'X\to\partial'Y$. Extend $R'$ to
  ${X'}^\gamma$ so that each point of $X'$ is only equivalent to
  itself, and let $\ol{X'}={X'}^\gamma/R'$. Let
  $\pi':{X'}^\gamma\to\ol{X'}$ be the quotient map, whose restriction
  $\pi':\gamma X'\to\partial X'$ is also the quotient map. Using that
  $X'$ is open in $\ol{X'}$, it easily follows that the restriction of
  $\pi':X'\to\ol{X'}$ is an embeddding with dense image. In this way,
  $\ol{X'}$ becomes a compactification of $X'$, with corona $\partial
  X'$, satisfying $\ol{X'}\le{X'}^\gamma$. This defines the stated
  mapping $\ol X\mapsto\ol{X'}$.
	
  The above notation will be kept for the rest of the
  proof. Extend $R$
  to $X^\gamma$ so that each point of $X$ is only equivalent to
  itself. Then, like in the case of $R'$, the space $X^\gamma/R$ is a
  compactification of $X$. Moreover the canonical map
  $X^\gamma/R\to\ol X$ is a continuous bijection between compact
  Hausdorff spaces, and thus it is a homeomorphism, showing that the
  compactifications $X^\gamma/R$ and $\ol X$ are equivalent. Since $R$
  corresponds to $R'$ via $\gamma\psi=(\gamma\phi)^{-1}$, it follows
  that $\ol X$ corresponds to $\ol{X'}$ by the mapping defined by
  $\psi$ in the same way. Hence the mapping defined by $\psi$ is left
  inverse of the mapping defyned by $\phi$. Reversing also the roles
  played by $\phi$ and $\psi$, it follows that the mapping $\ol
  X\mapsto\ol{X'}$ of the statement is bijective.
	
  Since $\phi^\gamma$ is compatible with $R$ and $R'$, it induces a
  map $\bar\phi:\ol X\to\ol{X'}$. Let us show that $\bar\phi$ is
  continuous at every $\bfe\in\partial X$. Let $V'$ be a neighborhood
  of $\bfe':=\bar\phi(\bfe)$ in $\ol{X'}$. The set $\widetilde
  V':={\pi'}^{-1}(V')$ is an $R'$-saturated neighborhood of
  ${\pi'}^{-1}(\bfe')$ in ${X'}^\gamma$. So, by~\eqref{i: phi^gamma},
  $\widetilde V:=(\phi^\gamma)^{-1}(\widetilde
  V')=\pi^{-1}(\bar\phi^{-1}(V'))$ is an $R$-saturated neighborhood of
  $\widetilde
  V:=(\phi^\gamma)^{-1}({\pi'}^{-1}(\bfe'))=\pi^{-1}(\bar\phi^{-1}(\bfe'))$
  in $X^\gamma$. Hence $\pi(\widetilde V)=\bar\phi^{-1}(V')$ is a
  neighborhood of $\bar\phi^{-1}(\bfe')$ in $\ol X$. Similarly, $\psi$
  induces a map $\bar\psi:\ol{X'}\to\ol X$, which is continuous at the
  points of $\partial X'$. This completes the proof of~\eqref{i: bar
    phi}.
	
  Property~\eqref{i: partial phi} follows directly from~\eqref{i:
    gamma phi} because $\partial\phi$ and $\partial\psi$ are induced
  by $\gamma\phi$ and $\gamma\psi$, respectively.
\end{proof}

\section{Higson compactification}\label{s: Higson}

Suppose that the metric space $M$ is proper. For $R>0$, the \emph{$R$-variation} a function $f:M\to \C$ is the function ${\sf V}_R f:M\to\C$ be given by
	\[
		{\sf V}_R f(x)= \sup\{\,|f(x)-f(y)| \mid d(x,y)<R\,\}\;.
	\]
It is said that $f:M\to\C$ is a \emph{Higson function} \index{Higson function} if it is bounded and ${\sf V}_R f$ vanishes at infinity for all $R>0$. 
The continuos Higson functions on $M$ form a unital closed involutive subalgebra $C_\nu(X)\subset C_b(X)$ that generates the topology. The compactification of $M$ that corresponds to $C_\nu(X)$ is called the \emph{Higson compactification}, \index{Higson compactification} and $\nu M=M^\nu\sm M$ is called the \emph{Higson corona} \index{Higson corona} of $M$.

\begin{rem}\label{i: Higson compactification of proper coarse spaces}
	\begin{enumerate}[(i)]
	
		\item The construction of the Higson corona can be extended to the case where $M$ is not proper in the following way. The (possibly non-continuous) Higson functions form a unital closed involutive subalgebra $\BB_\nu(M)$ of the commutative $C^*$ algebra of $\C$-valued bounded functions on $M$ with the supremum norm. Now, it is said that a function $f:M\to\C$ \emph{vanishes at infinity} if, for all bounded subset $B\subset M$, there is some $r>0$ such that $|f|<\epsilon$ on $M\sm B$. The functions vanishing at infinity form a closed involutive ideal $\BB_0\subset\BB_\nu(M)$. Then the Higson corona $\nu M$ is the compact Hausdorff space that corresponds to the unital commutative $C^*$ algebra $\BB_\nu(M)/\BB_0(M)$; this $C^*$ algebra is isomorphic to $C_\nu(M)/C_0(M)\cong C(\nu M)$ when $M$ is proper \cite[Lemma~2.40]{Roe2003}.
		
		\item In fact, the Higson compactification can be defined for arbitrary proper coarse spaces, and the Higson corona can be defined for all coarse spaces \cite[Section~2.3]{Roe2003}.

	\end{enumerate}
\end{rem}

\begin{prop}[\'Alvarez-Candel {\cite[Corollary~4.14, Proposition~4.15 and Theorem~4.16]{AlvarezCandel2011}}]
\label{p: Higson}
	The following properties hold for maps $\phi,\psi:M\to M'$ between proper metric spaces:
  		\begin{enumerate}[{\rm(}i\/{\rm)}]

			\item\label{i: phi^nu} $\phi$ is coarse if and only if it has an extension $\phi^\nu:M^\nu\to{M'}^\nu$ that is continuous at the points of $\nu M$ and such that $\phi^\nu(\nu M)\subset\nu M'$. In particular,  $\phi^\nu$ restricts to a continuous map $\nu\phi:\nu M\to\nu M'$.

			\item\label{i: phi is close to psi} $\phi$ is close to $\psi$ if and only if the extensions $\phi^\nu$ and $\psi^\nu$, given by~{\rm(}\ref{i: phi^nu}\/{\rm)}, are equal on $\nu M$.
		
			\item\label{i: phi is a coarse equivalence} $\phi$ is a coarse equivalence if and only if it satisfies the conditions of~{\rm(}\ref{i: phi^nu}\/{\rm)} and $\nu\phi:\nu M\to\nu M'$ is a homeomorphism.
		
		\end{enumerate}
\end{prop}

\begin{prop}\label{p: nu phi is an embedding} 
	If $\phi:M\to M'$ is rough map, then $\nu\phi:\nu M\to\nu M'$ is an embedding whose image is $\Cl_{{M'}^\nu}(\phi(M))\cap\nu M'$.
\end{prop}

\begin{proof}
	By Corollary~\ref{c: rough = coarse embedding} and Proposition~\ref{p: Higson}-\eqref{i: phi is a coarse equivalence}, it can be assumed that $M$ is a metric subspace of $M'$, and $\phi$ is the inclusion map $M\hookrightarrow M'$. Since the Higson coronas are compact Hausdorff metric spaces and $\nu\phi$ is continuous (Proposition~\ref{p: Higson}-\eqref{i: phi^nu}), it is enough to prove that $\nu\phi$ is injective. 
	
	Let $e_0\ne e_1$ in $\nu M$. Take open subsets $V_0,V_1\subset M^\nu$ such that $e_i\in V_i$ ($i\in\{0,1\}$) and $\Cl_{M^\nu}(V_0)\cap\Cl_{M^\nu}(V_1)=\emptyset$. Fix any $x_0\in M\sm\Cl_{M^\nu}(V_0\cup V_1)$, and let $B_R=B_M(x_0,R)$ for each $R\ge0$.
	
	\begin{claim}\label{cl: d(V_0 sm B_R, V_1 sm B_R) to infty}
		$d(V_0\sm B_R,V_1\sm B_R)\to\infty$ as $R\to\infty$.
	\end{claim}
	
	There is a function $F\in C(M^\nu)$ such that $F(V_i)=i$, and let $f=F|_M\in C_\nu(M)$. If Claim~\ref{cl: d(V_0 sm B_R, V_1 sm B_R) to infty} were false, there would be some $r>0$ and sequences $x_{i,k}\in V_i\sm B_k$ so that $d(x_{0,k},x_{1,k})\le r$ for all $k$, obtaining the contradiction $1=f(x_{0,k})-f(x_{1,k})\to0$ as $k\to\infty$ because $f\in C_\nu(M)$.
	
	The mapping $R\mapsto d(V_0\sm B_R,V_1\sm B_R)$ is non-decreasing and upper semi continuous. It may not be continuous but, using Claim~\ref{cl: d(V_0 sm B_R, V_1 sm B_R) to infty}, it easily follows that there is a smooth function $\rho:[0,\infty)\to\R^+$ such that $\rho(R)\le d(V_0\sm B_R,V_1\sm B_R)$, $0\le\rho'\le1$, and $\rho(R)\to\infty$ as $R\to\infty$; in particular, $\rho(R)\le\rho(R+r)\le\rho(R)+r$ for all $R,r\ge0$.
	
	Now let $f':M'\sm B_1\to\C$ be defined by
		\[
			f'(x)=
				\begin{cases}
					\frac{\rho(d'(x,x_0))-d'(x,V_1)}{\rho(d'(x,x_0))} & \text{if $d'(x,V_1)\le\rho(d'(x,x_0))$}\\
					0 & \text{otherwise}\;.
				\end{cases}
		\]
	Note that $f'$ is continuous, non-negative and bounded, and $f'(V_i)=i$. Take some $r>0$ and $x,y\in M'$ such that $d'(x,y)<r$. For $R=d'(x,x_0)$ and $D=d'(x,V_1)$, we have $R-r\le d'(y,x_0)\le R+r$ and $D-r\le d'(y,V_1)\le D+r$, obtaining
		\[
			\rho(R)-r\le\rho(R-r)\le\rho(d'(y,x_0))\le\rho(R+r)\le\rho(R)+r\;.
		\]
	For the sake of simplicity, let $\rho=\rho(R)$ for this particular $R$. If $D+2r\le\rho$, we get		
		\begin{gather*}
			f'(x)-f'(y)=\frac{\rho-D}{\rho}-\frac{\rho-r-(D+r)}{\rho+r}
			=\frac{\rho-D}{\rho}-\frac{\rho(R)-D+2r}{\rho+r}\to0\;,\\
			f'(y)-f'(x)\le\frac{\rho+r-(D-r)}{\rho-r}-\frac{\rho-D}{\rho}
			=\frac{\rho-D+2r}{\rho-r}-\frac{\rho-D}{\rho}\to0\;,
		\end{gather*}
	as $\rho\to\infty$ (and therefore as $R\to\infty$). If $D+2r\ge\rho$, we get		
		\begin{gather*}
			f'(x)=\frac{\rho-D}{\rho}\le\frac{2r}{\rho}\to0\;,\\
			f'(y)\le\frac{\rho+r-(D-r)}{\rho-r}=\frac{\rho-D+2r}{\rho-r}\le\frac{4r}{\rho-r}\to0\;,
		\end{gather*}
	as $\rho\to\infty$. So $f'\in C_\nu(M')$, and therefore $f'$ has an extension $F'\in C({M'}^\nu)$. Also, $f:=f'|_M\in C_\nu(M)$, whose continuous extension to $M^\nu$ is $F:=(\phi^\nu)^*F'$. Since $f(V_i)=i$, we have $F(\Cl_{M^\nu}(V_i))=i$, and therefore $i=F(e_i)=F'(\phi^\nu(e_i))=F'(\nu\phi(e_i))$, obtaining $\nu\phi(e_0)\ne\nu\phi(e_1)$.
\end{proof}

\begin{cor}\label{c: phi^nu is an embedding}  
	If $\phi:M\to M'$ is a continuous rough map, then $\phi^\nu:M^\nu\to{M'}^\nu$ is an embedding whose image is $\Cl_{{M'}^\nu}(\phi(M))$.
\end{cor}

\begin{proof}
	By Propositions~\ref{p: Higson}-\eqref{i: phi^nu} and~\ref{p: nu phi is an embedding}, $\phi^\nu$ is an injective continuous map between compact Hausdorff spaces.
\end{proof}

\begin{prop}[\'Alvarez-Candel {\cite[Proposition~4.12]{AlvarezCandel2011}}]
\label{p: If W contains balls of arbitrarily large radius}
  If an open subset $W\subset M$ contains balls of arbitrarily large radius, then
  	\begin{equation}\label{Int_M^nu(Cl_M^nu(W)) cap nu M ne emptyset}
		\Int_{M^\nu}(\Cl_{M^\nu}(W))\cap\nu M\ne\emptyset\;.
	\end{equation}
\end{prop}

\begin{prop}\label{p: W_r}
  	For all $W\subset M$ and $r>0$,
		\[
			\Int_{M^\nu}(\Cl_{M^\nu}(W))\cap\nu M
			=\Int_{M^\nu}(\Cl_{M^\nu}(W_r))\cap\nu M\;,
		\]
	where
		\begin{equation}\label{W_r}
			W_r=\{\,x\in W\mid d(x,M\sm W)>r\,\}\;.
		\end{equation}
\end{prop}

\begin{proof}
  	The inclusion ``$\supset$'' of the statement is obvious.
  
  	Let us prove the inclusion ``$\subset$'' of the statement. For the sake of simplicity, given $W\subset M$ and $r>0$, let
		\[
			V=\Int_{M^\nu}(\Cl_{M^\nu}(W))\;,\quad V_r=\Int_{M^\nu}(\Cl_{M^\nu}(W_r))\;.
		\]
	For any $e\in V\cap\nu M$, there is a continuous function $F:M^\nu\to [0,1]$ such that $F(e)=0$ and $F(M^\nu\sm V)=1$. Since $f=F|_M\in C_\nu(M)$, there is a compact $K\subset M$ such that $|{\sf V}_{r+1}f|<1/2$ on $M\sm K$. 
	
	\begin{claim}\label{cl: f > 1/2}
		$f>1/2$ on $M\sm(\Cl_{M^\nu}(W_r)\cup K)$.
	\end{claim}
	
	For all $u\in M\sm(\Cl_{M^\nu}(W_r)\cup K)$, there is some $v\in M\sm W$ so that $d(v,w)\le r<r+1$. Thus $|f(u)-1|=|f(u)-f(v)|<1/2$, obtaining $f(u)>1/2$, which shows Claim~\ref{cl: f > 1/2}.
	
	\begin{claim}\label{cl: M^nu sm (Cl_M^nu(W_r) cup K)}
		We have
			\[
				M^\nu\sm(\Cl_{M^\nu}(W_r)\cup K)\subset\Cl_{M^\nu}(M\sm(\Cl_{M^\nu}(W_r)\cup K))\;.
			\]
	\end{claim}
	
	For any open neighborhood $O$ of a point $e'\in M^\nu\sm(\Cl_{M^\nu}(W_r)\cup K)$, we get
		\[
			\emptyset\ne(M^\nu\sm(\Cl_{M^\nu}(W_r)\cup K))\cap M\cap O=(M\sm(\Cl_{M^\nu}(W_r)\cup K))\cap O
		\]
	because $M^\nu\sm(\Cl_{M^\nu}(W_r)\cup K)$ is open in $M^\nu$, and $M$ is open and dense in $M^\nu$. So $e'\in\Cl_{M^\nu}(M\sm(\Cl_{M^\nu}(W_r)\cup K))$, showing Claim~\ref{cl: M^nu sm (Cl_M^nu(W_r) cup K)}.
	
	From Claims~\ref{cl: f > 1/2} and~\ref{cl: M^nu sm (Cl_M^nu(W_r) cup K)}, it follows that $F\ge1/2$ on $M^\nu\sm(\Cl_{M^\nu}(W_r)\cup K)$ by the continuity of $F$. Hence
		\[
			F^{-1}([0,1/2))\subset\Int_{M^\nu}(\Cl_{M^\nu}(W_r)\cup K)
		\]
	because $F^{-1}([0,1/2))$ is open in $M^\nu$. So
		\[
			e\in F^{-1}([0,1/2))\cap\nu M\subset\Int_{M^\nu}(\Cl_{M^\nu}(W_r)\cup K)\cap\nu M=V_r\cap\nu M\;.
		\]
\end{proof}

The following corollary states the reciprocal of Proposition~\ref{p: If W contains balls of arbitrarily large radius}.

\begin{cor}\label{c: then W contains balls of arbitrarily large radius}
	If~\eqref{Int_M^nu(Cl_M^nu(W)) cap nu M ne emptyset} holds for some open subset $W\subset M$, then $W$ contains balls of arbitrarily large radius.
\end{cor}

\begin{proof}
	By Proposition~\ref{p: W_r}, with the notation~\eqref{W_r}, we get $W_r\ne\emptyset$ for all $r>0$. So $B(x,r)\subset W$ for any $x\in W_r$.
\end{proof}

\begin{cor}\label{c: Int_M^nu(Cl_M^nu(W)) cap nu M is dense}
	If~\eqref{Int_M^nu(Cl_M^nu(W)) cap nu M ne emptyset} holds for some open subset $W\subset M$, then
		\[
			\Cl_{\nu M}(\Int_{M^\nu}(\Cl_{M^\nu}(W))\cap\nu M)=\Cl_{M^\nu}(W)\cap\nu M\;.
		\]
\end{cor}

\begin{proof}
	There is a canonical identity $\nu W\equiv\Cl_{M^\nu}(W)\cap\nu M$ by Corollary~\ref{c: phi^nu is an embedding}. Any open set of $\nu W$ is of the form $V_0\cap\nu W$ for some open subset $V_0\subset W^\nu$. If $V_0\cap\nu W\ne\emptyset$, then there is some open subset $V_1\subset W^\nu$ such that $V_1\cap\nu W\ne\emptyset$ and $\Cl_{M^\nu}(V_1)\subset V_0$. By Corollary~\ref{c: then W contains balls of arbitrarily large radius}, $V_1\cap W$ contains balls of arbitrarily large radii by Corollary~\ref{c: then W contains balls of arbitrarily large radius}. So
		\[
			\emptyset\ne\Int_{M^\nu}(\Cl_{M^\nu}(V_1\cap W))\cap\nu M
			\subset\Int_{M^\nu}(\Cl_{M^\nu}(W))\cap\nu M
		\]
	by Proposition~\ref{p: If W contains balls of arbitrarily large radius}. On the other hand,
		\[
			\Int_{M^\nu}(\Cl_{M^\nu}(V_1\cap W))\cap\nu M
			\subset V_0\cap\Cl_{M^\nu}(W)\cap\nu M=V_0\cap\nu W\;.
		\]
	Hence $V_0\cap\nu W$ meets $\Int_{M^\nu}(\Cl_{M^\nu}(W))\cap\nu M$.
\end{proof}

\begin{prop}\label{p: VV'}
	Let $\phi:M\to M'$ be an $(s_r,c)$-rough equivalence. Given a compactification $\ol M\le M^\nu$, according to Lemma~\ref{l: ol X mapsto ol X'}, $\ol{M'}\le{M'}^\nu$ be the corresponding compactification, and let $\bar\phi:\ol X\to\ol{X'}$ be the map induced by $\phi$, whose restriction to the coronas is denoted by $\partial\phi:\partial M\to\partial M'$. Let $\bfe\in\partial M$, and let $\VV$ be a base of neighborhoods of $e$ in $\ol M$. Then
		\[
			\VV'=\{\,\Cl_{\ol{M'}}(\Pen_{M'}(\phi(V\cap M),c))\mid V\in\VV\,\}
		\]
	is a base of neighborhoods of $\partial\phi(e)$ in $\ol{M'}$.
\end{prop}

\begin{proof}
	 Let us consider first the case where $\ol M=M^\nu$, and thus $\ol{M'}={M'}^\nu$, $\bar\phi=\phi^\nu$ and $\partial\phi=\nu\phi$. 
	 
	 Let us prove that any element $\Cl_{{M'}^\nu}(\Pen_{M'}(\phi(V),c))$ of $\VV'$, defined by some $V\in\VV$, is a neighborhood of $e':=\nu\phi(e)$. Let $\psi:M'\to M$ be another $s_r$-rough equivalence such that $\psi\phi$ and $\phi\psi$ are $c$-close to $\id_M$ and $\id_{M'}$, respectively. By Proposition~\ref{p: Higson}-\eqref{i: phi^nu}, there is some open neighborhood $V'$ of $e'$ in ${M'}^\nu$ such that $\psi^\nu(V')\subset V$. For any $x\in V'\cap M'$, we have $d'(x,\phi\psi(x))\le c$ and $\psi(x)\in V$. So $V'\cap M'\subset\Pen_{M'}(\phi(V),c)$, giving
	 	\[
			\Cl_{{M'}^\nu}(V'\cap M')\subset\Cl_{{M'}^\nu}(\Pen_{M'}(\phi(V),c))\;.
		\]
	But $V'\subset\Cl_{{M'}^\nu}(V'\cap M')$ because $V'$ is open in ${M'}^\nu$ and $M'$ is open and dense in ${M'}^\nu$. So $\Cl_{{M'}^\nu}(\Pen_{M'}(\phi(V),c))$ is an open neighborhood of $e'$.
	 	
	Now, let us prove that any neighborhood $V'$ of $e'$ in ${M'}^\nu$ contains some element of $\VV'$. Take another neighborhood $V'_0$ of $e'$ in ${M'}^\nu$ with $\Cl_{{M'}^\nu}(V'_0)\subset V'$. Let $W'=V_0\cap M'$ and, for any $r>c$, let $W'_r$ be defined by~\eqref{W_r}. By~\eqref{ol Pen_M(S,r) subset Pen_M(S,s)} and since $\Pen_{M'}(W'_r,r)\subset W'$, we get $\Pen_{M'}(\Cl_{M'}(W'_r),c)\subset W'$. By Proposition~\ref{p: W_r}, the set $V'_{0,r}:=\Int_{{M'}^\nu}(\Cl_{{M'}^\nu}(W'_r))$ is another neighborhood of $e'$ in ${M'}^\nu$. By Proposition~\ref{p: Higson}-\eqref{i: phi^nu}, there is some $V\in\VV$ such that $\Cl_{{M'}^\nu}(\phi^\nu(V))\subset V'_{0,r}$. Hence
		\begin{multline*}
			\Cl_{{M'}^\nu}(\Pen_{M'}(\phi(V\cap M),c))\subset\Cl_{{M'}^\nu}(\Pen_{M'}(V'_{0,r}\cap M',c))\\
			\subset\Cl_{{M'}^\nu}(\Pen_{M'}(\Cl_{M'}(W'_r),c))\subset\Cl_{{M'}^\nu}(W')
			\subset\Cl_{{M'}^\nu}(V'_0)\subset V'\;.
		\end{multline*}
This completes the proof in the case $\ol M=M^\nu$.

Now consider the general case. Let $\pi:M^\nu\to\ol M$ and $\pi':{M'}^\nu\to\ol{M'}$ be continuous extensions of $\id_M$ and $\id_{M'}$, respectively. Given $e$ and $\VV$ like in the statement, the sets $\widetilde V:=\pi^{-1}(V)$, for $V\in\VV$, form a base $\widetilde{\VV}$ of neighborhoods of any $\widetilde e\in\pi^{-1}(e)$ in $M^\nu$. So, by the above case, it is easy to see that the sets $\widetilde V':=\Cl_{{M'}^\nu}(\Pen_{M'}(\phi(\widetilde V\cap M),c))$, for $\widetilde V\in\widetilde\VV$, form a base of neighborhoods of $\nu\phi(\pi^{-1}(e))={\pi'}^{-1}(\partial\phi(e))$ in $M^\nu$. Since the sets $\widetilde V'$ are saturated by the fibers of $\pi'$, it follows that the sets $\pi'(\widetilde V')=\Cl_{\ol{M'}}(\Pen_{M'}(\phi(V\cap M),c))$, for $V\in\VV$, form a base of neighborhoods of $\partial\phi(e)$.
\end{proof}

\section{Asymptotic dimension}\label{s: as dim}

Let $\VV$ be a cover of a space $X$. The \emph{multiplicity} \index{multiplicity} of a $\VV$ is the least $n\in\N\cup\{\infty\}$ such that there are at most $n$ elements of $\VV$ meeting at any point of $X$. It is said that $\VV$ \emph{refines} another cover $\WW$ of $X$ if every element of $\VV$ is contained in some element of $\WW$. Recall that the \emph{Lebesgue covering dimension} \index{Lebesgue covering dimension} of $X$ is the least $n\in\N\cup\{\infty\}$ such that every open cover of $X$ is refined by a cover with multiplicity $\le n+1$.

A family $\VV$ of subsets of the metric space $M$ is called \emph{uniformly bounded} \index{uniformly bounded} if there is some $D>0$ such that $\diam V\le D$ for all $V\in\VV$. The following definition is somehow dual to the definition of Lebesgue covering dimension.

\begin{defn}\label{d: asdim}
The \emph{asymptotic dimension} \index{asymptotic dimension} of $M$, denoted\footnote{The original notation of Gromov \cite{Gromov1993} is $\operatorname{as\,dim}_+M$.} $\asdim M$, is the least $n\in\N\cup\{\infty\}$ such that any uniformly bounded open cover of $X$ refines some uniformly bounded open cover with multiplicity $\le n+1$.
\end{defn}

\begin{rem}
	The asymptotic dimension was introduced by Gromov \cite{Gromov1993} in a different way. We follow the survey \cite{BellDranishnikov2011} by Bell and Dranishnikov, which contains many relevant examples and results about the asymptotic dimension.
\end{rem}

\begin{prop}[See e.g.\ {\cite[Theorem~1]{BellDranishnikov2011}}]\label{p: asdim}
For $n\in\N$, we have $\asdim M\le n$ if and only if, for any $R>0$, there exist uniformly bounded families $\VV_0,\dots,\VV_n$ of subsets of $M$ such that $\bigcup_{i=0}^n\VV_i$ is a cover of $M$, and $d(V,V')>R$ for $V\ne V'$ in any $\UU_i$.
\end{prop}

\begin{prop}[See e.g.\ {\cite[Proposition~2]{BellDranishnikov2011}}]
\label{p: corsely equivalent implies the same asdim}
	Corsely equivalent metric spaces have the same asymptotic dimension.
\end{prop}

\begin{examples}
  We have $\asdim T\le 1$ for any tree $T$
  \cite{BellDranishnikov2011}, $\asdim \mathbb{H}^n = n$ for the
  hyperbolic space $\mathbb{H}^n$ \cite{Gromov1993}, and $\asdim
  \R^n=n$ for Euclidean space $\R^n$
  \cite{DranishnikovKeeslingUspenskij1998},
  \cite{BellDranishnikov2011}. For any finitely generated group
  $\Gamma$, we have $\asdim\Gamma=0$ if and only if $\Gamma$ is finite
  \cite{BellDranishnikov2011}.
\end{examples}

\chapter{Pseudogroups}\label{c: pseudogroups}

This chapter mainly recalls basic notions and results on pseudogroups,
and fixes the notation. Most of these preliminaries can be seen in
\cite{Haefliger1985}, \cite{Haefliger1988}, \cite{Haefliger2002} and
\cite{AlvarezCandel2009}. Some new results are also proved.

\section{Pseudogroups}\label{s: pseudogroups}

A collection, \(\HH\), of homeomorphisms between open subsets of a
topological space $Z$ is called a \emph{pseudogroup} \index{pseudogroup} of local
transformations of $Z$ (or simply a pseudogroup on \(Z\)) if
$\id_Z\in\HH$, and $\HH$ is closed under composition (wherever
defined\footnote{Composite of partial maps}), inversion, restriction
(to open subsets) and combination (or union) of maps. A subset
$E\subset \HH$ of the pseudogroup \(\HH\) is called \emph{symmetric}
when $h^{-1}\in E$ for all $h\in E$, and the pseudogroup $\HH$ is said
to be \emph{generated} by $E$ if every element of $\HH$ can be
obtained from $E$ by using the above pseudogroup operations. The
\emph{restriction} of $\HH$ to an open subset $U$ of $Z$ is the
pseudogroup on \(U\) given by
\[
\HH|_U=\{\,h\in\HH\mid\dom h\cup\im h\subset U\,\}\;.
\]
Let $\HH'$ be another pseudogroup on a space $Z'$. Then
$\HH\times\HH'$ denotes the pseudogroup on $Z\times Z'$ generated by
the maps $h\times h'$ with $h\in\HH$ and $h'\in\HH'$.

A pseudogroup on a space is an obvious generalization of a group
acting on a space via homeomorphisms, and so all basic concepts from
the theory of group actions can be generalized to pseudogroups. For
instance, the \emph{orbit} \index{orbit} (or \emph{$\HH$-orbit}, or \emph{trajectory}) \index{trajectory} of each $x\in Z$
is the set
$$
\HH(x)=\{\,h(x)\mid h\in\HH,\ x\in\dom h\,\}\;.
$$
The orbits of \(\HH\) form a partition of $Z$. The corresponding
quotient space (the orbit space) is denoted by $Z/\HH$.

\begin{defn}[{Haefliger \cite{Haefliger1985},
    \cite{Haefliger1988}}]\label{d: pseudogroup equivalence}
  An \emph{\'etale morphism} \index{\'etale morphism} $\Phi:\HH\to\HH'$ is a maximal
 set of homeomorphisms of open subsets of $Z$ to open subsets
  of $Z'$ such that:
  \begin{itemize}

  \item if $\phi\in\Phi$, $h\in\HH$ and $h'\in\HH'$, then $h'\phi
    h\in\Phi$;

  \item the sources of elements of $\Phi$ cover $Z$; and,

  \item if $\phi,\psi\in\Phi$, then $\psi\phi^{-1}\in\HH'$.

  \end{itemize}
  An \'etale morphism $\Phi:\HH\to\HH'$ is called an \emph{equivalence} \index{equivalence} if  $\Phi^{-1}=\{\,\phi^{-1}\mid
  \phi\in\Phi\,\}$ is also an \'etale morphism $\HH'\to\HH$, which is
  called the \emph{inverse} of $\Phi$. An \'etale morphism
  $\Phi:\HH\to\HH'$ is \emph{generated} by a subset
  $\Phi_0\subset\Phi$ if all the elements of $\Phi$ can be obtained by
  combination of composites $h'\phi h$ with $h\in\HH$, $\phi\in\Phi_0$
  and $h'\in\HH'$. The \emph{composite} of two \'etale morphisms,
  $\Phi:\HH\to\HH'$ and $\Psi:\HH'\to\HH''$, is the \'etale morphism
  $\Psi\Phi:\HH\to\HH''$ generated by $\{\,\psi\phi\mid\phi\in\Phi,\
  \psi\in\Psi\,\}$.
\end{defn}

An \'etale morphism $\Phi:\HH\to\HH'$ clearly induces a continuous map
$\bar\Phi:Z/\HH\to Z'/\HH'$, which is a homeomorphism if $\Phi$ is an
equivalence. If $\HH$ and $\HH'$ are equivalent, then they should be
considered as the same generalized dynamical system. Thus the
properties of pseudogroups that are invariant by equivalences are
especially relevant.

\begin{example}\label{ex: pseudogroup equivalence}
  If $U\subset Z$ is an open subset that meets every $\HH$-orbit, then
  the inclusion map $U\hookrightarrow Z$ generates an equivalence
  $\HH|_U\to\HH$.
\end{example}

\begin{rem}\label{r: pseudogroup equivalence}
  Example~\ref{ex: pseudogroup equivalence} can be used to describe
  any pseudogroup equivalence $\Phi:\HH\to\HH'$ as follows. Let
  $\HH''$ be the pseudogroup on $Z''=Z\sqcup Z'$ generated by
  $\HH\cup\HH'\cup\Phi$, and let $\Psi:\HH\to\HH''$ and
  $\Psi':\HH'\to\HH''$ be the equivalences generated by
  $Z\hookrightarrow Z''$ and $Z'\hookrightarrow Z''$,
  respectively. Then $\Phi={\Psi'}^{-1}\Psi$.
\end{rem}

The \emph{germ groupoid} \index{germ groupoid} of $\HH$ is the topological groupoid of
germs of maps in $\HH$ at all points of their domains, with the
operation induced by the composite of partial maps and the \'etale
topology. Its subspace of units can be canonically identified with
$Z$. For each $x\in Z$, the group of elements of this groupoid whose
source and range is $x$ will be called the \emph{germ group} \index{germ group} of $\HH$
at $x$. The germ groups at points in the same orbit are isomorphic by
conjugation in the germ groupoid. Thus the \emph{germ group} of each
orbit is defined up to isomorphisms. By Remark~\ref{r: pseudogroup
  equivalence}, it follows that, under pseudogroup equivalences,
corresponding orbits have isomorphic germ groups.

Let $\HH$ be a pseudogroup on a locally compact space $Z$. Then the
orbit space $Z/\HH$ is compact if and only if $Z$ has a relatively
compact open subset that meets every $\HH$-orbit. The following is a
stronger compactness condition on $\HH$.

\begin{defn}[{Haefliger \cite{Haefliger1985}}]\label{d:compactly generated}
  A pseudogroup, $\HH$, is \emph{compactly generated} \index{compactly generated} if there is a
  relatively compact open set $U$ in $Z$, meeting each orbit, such
  that $\HH|_U$ has a finite symmetric set of generators, $E$, so that
  each $g\in E$ has an extension $\bar g\in\HH$ with $\ol{\dom
    g}\subset\dom\bar g$. In this case, $E$ is called a \emph{system of
    compact generation} of $\HH$ on $U$.
\end{defn}

It was observed in \cite{Haefliger1985} that this notion is invariant
by equivalences, and that the relatively compact open set $U$ meeting
each orbit can be chosen arbitrarily.

\section{Coarse quasi-isometry type of orbits}\label{s: coarse q.i. type of orbits}

Let $\HH$ be a pseudogroup on a space $Z$, and let $E$ be a symmetric set of generators
of $\HH$. Each $\HH$-orbit $\OO$ is the set of vertices of a connected graph, defined by attaching an edge to
vertices $x,y\in\OO$ whenever $y=h(x)$ for some $h\in E$ with
$x\in\dom h$. This connected graph structure induces a metric $d_E$ on
$\OO$ according to Section~\ref{s: graphs}. For $x\in\OO$,
$S\subset\OO$ and $r\ge0$, the open and closed $r$-balls of center $x$
in $(\OO,d_E)$ are denoted by $B_E(x,r)$ and $\ol B_E(x,r)$, and the
$r$-penumbra of $S$ is denoted by $\Pen_E(S,r)$.

We focus on the following case. Suppose that $Z$ is locally compact and $\HH$ is compactly generated. Let $U\subset Z$ be a relatively compact open subset that meets all orbits, let $E$ be a symmetric system of compact generation of $\HH$ on $U$, and let $\GG=\HH|_U$. Then we consider the metric $d_E$ on the $\GG$-orbits. Even under these conditions, the coarse quasi-isometry type of the $\GG$-orbits may depend on the choice of $E$ \cite[Section~6]{AlvarezCandel2009}. In \cite{AlvarezCandel2009}, this problem is solved by introducing the following additional condition.

\begin{defn}[\'Alvarez-Candel {\cite[Definition~4.2]{AlvarezCandel2009}}]
  \label{d:recurrent finite symmetric family of generators}
  $E$ is called \emph{recurrent} \index{recurrent} if there exists a relatively compact
  open subset $V\subset U$ whose intersections with all $\GG$-orbits
  are equi-nets in the $\GG$-orbits with $d_E$.
\end{defn}

According to the following result, the role played by $V$ in
Definition~\ref{d:recurrent finite symmetric family of generators} can
actually be played by any relatively compact open subset of $U$ that
meets all $\GG$-orbits.

\begin{prop}[\'Alvarez-Candel {\cite[Lemma~4.3]{AlvarezCandel2009}}]
\label{p: recurrent finite symmetric family of generators}
If $E$ is recurrent and $W\subset U$ is an open subset that meets
every $\GG$-orbit, then the intersections of all $\GG$-orbits with $W$
are equi-nets in the $\GG$-orbits with $d_E$.
\end{prop}

The following result guarantees the existence of recurrent systems of
compact generation.

\begin{prop}[\'Alvarez-Candel {\cite[Corollary~4.5]{AlvarezCandel2009}}]
\label{p:recurrent systems of compact generation}
There exists a recurrent system $E$ of compact generation of $\HH$ on
$U$ such that the extension $\bar g\in\HH$ of each $g\in E$ with
$\overline{\dom g}\subset\dom\bar g$ can be chosen so that
$\ol{E}=\{\,\bar g\mid g\in E\,\}$ is also a recurrent symmetric system of
compact generation on some relatively compact open subset $U'\subset
Z$ containing $\overline{U}$.
\end{prop}

Let $\HH'$ be another compactly generated pseudogroup on a locally
compact space $Z'$, let $U'$ be a relatively compact open subset of $Z'$
that meets all $\HH'$-orbits, let $E'$ be recurrent symmetric system of
compact generation for $\HH'$ on $U'$, and let $\GG'=\HH'|_{U'}$.

\begin{thm}[\'Alvarez-Candel {\cite[Theorem~4.6]{AlvarezCandel2009}}]
\label{t: equi-coarsely quasi-isometric orbits}
With the above notation, suppose that there exists an equivalence
$\HH\to\HH'$, and consider the induced equivalence $\GG\to\GG'$ and
homeomorphism $U/\GG\to U'/\GG'$. Then the $\GG$-orbits, endowed with
$d_E$, are equi-coarsely quasi-isometric to the corresponding
$\GG'$-orbits, endowed with $d_{E'}$.
\end{thm}

Theorem~\ref{t: equi-coarsely quasi-isometric orbits} implies the
invariance of the coarse quasi-isometry type of the orbits by
equivalences when appropriate representatives of pseudogroups and
generators are chosen. The following result gives a more explicit
relation between $d_E$ and $d_{E'}$ in a particular case.

\begin{prop}\label{p: d_E and the restrictions of d_E' are equi-Lipschitz equivalent}
  Suppose that $Z=Z'$, $\HH=\HH'$ and $U\subset U'$. Then $d_E$ and
  the restrictions of $d_{E'}$ to the $\GG$-orbits are equi-Lipschitz
  equivalent.
\end{prop}

\begin{proof}
  If $U=U'$, this was established in
  \cite[Corollary~4.9]{AlvarezCandel2009}. Assume, by enlarging \(E\)
  if necessary, that $E\subset E'$. Then $d_{E'}(x,y)\le d_E(x,y)$ for
  all $x,y\in U$ in the same $\GG$-orbit.
	
  On the other hand, by Proposition~\ref{p: recurrent finite symmetric
    family of generators}, there is some $R\in\N$ such that $\OO'\cap
  U$ is an $R$-net in $(\OO',d_{E'})$ for every $\GG'$-orbit $\OO'$. Let
  $F$ be the family composites, $f$, of at most $R$ maps in $E'$ such
  that $\im f\subset U$. Then $U'=\bigcup_{f\in F}\dom f$. Let
  $$
  G=\{\,f_2^{-1}g'f_1\mid f_1,f_2\in F\cup\{\id_U\},\ g'\in E'\,\}\;.
  $$
  Then $G$ is symmetric because $E'$ is symmetric. Moreover every
  $g\in G$ has an extension $\bar g\in\HH$ with $\ol{\dom
    g}\subset\dom \bar g$ because the elements of $E'$ have such type
  of extensions. It follows from the definition of \(G\) that
  $E\subset G$ because $E\subset E'$, and therefore $G$ is a recurrent
  symmetric system of compact generation of $\HH$ on $U$. By
  \cite[Corollary~4.9]{AlvarezCandel2009} (the case $U=U'$), there is
  $C\ge1$ such that $d_E(x,y)\le C\,d_G(x,y)$ for every $x,y\in U$ in
  the same $\GG$-orbit. If $x\ne y$ are in the same \(\GG\)-orbit and have
  $d_{E'}(x,y)=m\ge1$, then there are $g'_1,\dots,g'_m\in E'$ such
  that $y=g'_m\cdots g'_1(x)$, and  $f_0,\dots,f_m\in F$
  such that $f_1=f_m=\id_U$ and $g'_k\cdots g'_1(x)\in\dom f_k$ for
  all $k\in\{1,\dots,n-1\}$. Thus $g_k=f_kg'_kf_{k-1}\in G$ for all
  $k\in\{1,\dots,m\}$ and $y=g_m\cdots g_1(x)$, obtaining $d_G(x,y)\le
  m$. So $d_E(x,y)\le C\,d_{E'}(x,y)$.
\end{proof}

\begin{rem}
  In the case $U=U'$, without assuming that $E'$ is recurrent,
  \cite[Corollary~4.9]{AlvarezCandel2009} in fact states that $E'$ is
  recurrent if and only if $d_E$ and $d_{E'}$ are
  equi-Lipschitz equivalent on the $\GG$-orbits.
\end{rem}

\section{A version of local Reeb stability}\label{s: Reeb}

Let $\HH$ be a compactly generated pseudogroup on a locally
compact space $Z$, let $U$ be a relatively compact open subset of $Z$
that meets all $\HH$-orbits, let $E$ be recurrent symmetric system of
compact generation of $\HH$ on $U$, and let $\GG=\HH|_U$.

The following notation will be used.  For each $m\in\Z^+$, let $E^m$
denote the $m$-fold Cartesian product $E\times\dots\times E$, and set
$E^0=\{\id_Z\}$. For every $g=(g_1,\dots,g_m)\in E^m$, its \emph{domain} is the set $\dom g=\dom(g_1\cdots g_m)$, which may be
empty. Moreover let $g(x)=g_1\cdots g_m(x)$ for every $x\in\dom g$,
and let $g^{-1}=\left(g_m^{-1},\dots,g_1^{-1}\right)$. For another
$n\in\Z^+$ and $h=(h_1,\dots,h_n)\in E^n$, let
$gh=(g_1,\dots,g_m,h_1,\dots,h_n)\in E^{m+n}$. Finally, for
$r\in\Z^+$, let $E^{\le r}=\bigcup_{m=1}^rE^m$.

By Proposition~\ref{p:recurrent systems of compact generation}, each
\(g\in E\) has an extension \(\bar g\) such that $\overline{\dom
  g}\subset\dom\bar g$, and the collection $\ol{E}=\{\,\bar g\mid g\in
E\,\}$ is a recurrent symmetric system of compact generation on some relatively compact
open subset $U'\subset Z$ with $\overline{U}\subset U'$. Let
$\GG'=\HH|_{U'}$. For $m\in\Z^+$ and $g=(g_1,\dots,g_m)\in E^m$, let
$\bar g=(\bar g_1,\dots,\bar g_m)\in\ol{E}^m$. There is some
$C\in\Z^+$ such that, for all $x,y\in U$ in the same $\GG$-orbit,
	\begin{equation}\label{d_ol E(x,y) le d_E(x,y) le C d_ol E(x,y)}
  		d_{\ol E}(x,y)\le d_E(x,y)\le C\,d_{\ol E}(x,y)\;.
	\end{equation}
In~\eqref{d_ol E(x,y) le d_E(x,y) le C d_ol E(x,y)} above, the first
inequality holds because $\bar g$ is an extension of the corresponding
$g\in E$, and the second inequality follows from Proposition~\ref{p:
  d_E and the restrictions of d_E' are equi-Lipschitz equivalent}. On
the other hand, by Proposition~\ref{p: recurrent finite symmetric
  family of generators}, there is some $R\in\N$ so that $\OO'\cap U$ is
an $R$-net in $(\OO',d_{\ol{E}})$ for every $\GG'$-orbit $\OO'$.

Let $U_0$ be the set of points in $U$ where $\GG$ has trivial germ
groups. The following is a coarsely quasi-isometric pseudogroup
version of the Reeb local stability around points in $U_0$, which will
play a very important role in the present work.
 
\begin{prop}\label{p: orbit Reeb with x in U_0}
  For every $r\in\Z^+$ and $x\in U_0$, there exists an open
  neighborhood $V(x,r)$ of $x$ in $U$ such that:
  \begin{enumerate}[{\rm(}i\/{\rm)}]
  
  \item\label{i: V(x,r') subset V(x,r) if r'>r} $V(x,r')\subset V(x,r)$ if $r'>r$;
  
  \item\label{i: V(x,r) subset dom g} $V(x,r)\subset\dom g$ for all
    $g\in E^{\le r}$ with $x\in\dom g$;
  
  \item\label{i: phi_x,y,r} for every $y\in V(x,r)$, a map
    $\phi_{x,y,r}:\ol B_E(x,r)\to\ol B_E(y,r)$ is determined by the
    condition $\phi_{x,y,r}g(x)=g(y)$ for all $g\in E^{\le r}$ with
    $x\in\dom g$;
  
  \item\label{i: phi_x,y,r is non-expanding} $\phi_{x,y,r}$ is non-expanding with respect to $d_E$;

  \item\label{i: phi_x,y,r are equi-bi-Lipschitz}
    $\phi_{x,y,r}$ is $C$-bi-Lipschitz with respect to $d_E$; and,
  
  \item\label{i: ol B_E(y,r/C) cap im phi_x,y,r is a 2CR-net} if
    $r\ge CR$, then $\ol
    B_E(y,r/C)\cap\im\phi_{x,y,r}$ is a $2CR$-net in $(\ol
    B_E(y,r/C),d_E)$.

  \end{enumerate}
\end{prop}

\begin{proof}
  Since all the sets $E^m$ and $\ol E^m$ are finite, and since
  $\overline{\dom g}\subset\dom\bar g$ for each $g\in E^m$, there is,
  for every $r\in\Z^+$ and $x\in U_0$, a largest open
  neighborhood $V(x,r)$ of $x$ in $U$ 
  satisfying~\eqref{i: V(x,r) subset dom g} and the following properties:
  \begin{enumerate}[(a)]
		
  \item\label{i: V(x,r) subset dom g and g(y)=y} if $g(x)=x$ for some $g\in E^{\le4r}$ and $x\in\dom g$, then $V(x,r)\subset\dom g$ and $g(y)=y$ for all $y\in V(x,r)$;
			
  \item\label{i: x in dom bar g} if $V(x,r)$ meets the domain of some
    $g\in E^{\le4r}$, then $x\in\dom\bar g$;
  
  \item\label{i: bar g(x)=x} if $g(y)=y$ for some $g\in E^{\le4r}$ and
    $y\in V(x,r)\cap\dom g$, then $\bar g(x)=x$;
			
  \item\label{i: V(x,r) subset dom bar g} $V(x,r)\subset\dom\bar g$
    for all $\bar g\in\ol E^{\le r}$ with $x\in\dom\bar g$; and,
			
  \item\label{i: V(x,r) subset dom bar g and bar g(y)=y} if
    $x\in\dom\bar g$ and $\bar g(x)=x$ for some $\bar g\in\ol
    E^{\le4r}$, then $V(x,r)\subset\dom\bar g$ and $\bar g(y)=y$ for
    all $y\in V(x,r)$.
			
  \end{enumerate}
  We also get~\eqref{i: V(x,r') subset V(x,r) if r'>r} since $V(x,r)$ is maximal.
  
  Fix $x\in U_0$, $r\in\Z^+$ and $y\in V(x,r)$. Each point in $\ol
  B_E(x,r)$ is of the form $g(x)$ for some $g\in E^{\le r}$ with
  $x\in\dom g$, and thus, by~\eqref{i: V(x,r) subset dom g}, $y\in\dom
  g$. Suppose that $g(x)=h(x)$ for another $h\in E^{\le r}$ with
  $x\in\dom h$. Then $h^{-1}g\in E^{\le2r}$, $x\in\dom h^{-1}g$ and
  $h^{-1}g(x)=x$. By~\eqref{i: V(x,r) subset dom g and g(y)=y}, it
  follows that $y\in\dom h^{-1} g$ and $h^{-1} g(y)=y$; i.e.,
  $g(y)=h(y)$. Therefore the assignment $g(x)\mapsto g(y)$ for $g\in
  E^{\le r}$ defines a map $\phi_{x,y,r}:\ol B_E(x,r)\to\ol B_E(y,r)$,
  which is the statement of~\eqref{i: phi_x,y,r}. For the sake of
  simplicity, $\phi_{x,y,r}$ will be simply denoted by $\phi$ in the
  rest of the proof.
  
  For every $z,z'\in\ol B_E(x,r)$ there are $g,g'\in E^{\le r}$ whose
  domains contain $x$ and such that $g(x)=z$ and $g'(x)=z'$. If
  $m=d_E(z,z')\le2r$, there is $h\in E^m$ so that $z\in\dom h$ and
  $h(z')=z$. Thus $x\in\dom g^{-1}h g'$ and $g^{-1} h g'(x)=x$ with
  $g^{-1} h g'\in E^{\le4r}$. It follows from~\eqref{i: V(x,r) subset dom g and
    g(y)=y} that $y\in\dom g^{-1}hg'$ and $g^{-1}hg'(y)=y$. Hence
  $g'(y)\in\dom h$ and $hg'(y)=g(y)$, giving $d_E(g(y),g'(y))\le m$,
  which shows~\eqref{i: phi_x,y,r is non-expanding}.

  Let $z,z'\in\ol B_E(x,r)$, and let $g,g'\in E^{\le r}$ be such
  that $z=g(x)$ and $z'=g(x')$. Thus $y\in\dom g\cap\dom g'$,
  $\phi(z)=g(y)$ and $\phi(z')=g'(y)$. If
  $m=d_E(g(y),g'(y))\le2r$, then there is $h\in E^m$ so that
  $g'(y)\in\dom h$ and $h g'(y)=g(y)$. Hence $g^{-1} h g'\in E^{\le4r}$,
  $y\in\dom g^{-1} h g'$ and $g^{-1} h g'(y)=y$. By~\eqref{i: x in dom
    bar g} and~\eqref{i: bar g(x)=x}, it follows that $x\in\dom \bar
  g^{-1}\bar h\bar{g'}$ and $\bar g^{-1}\bar h\bar{g'}(x)=x$, and thus
  $\bar h(z)=z'$ with $\bar h\in\ol{E}^m$. Therefore
  $$
  d_E(z,z')\le C\,d_{\ol{E}}(z,z')\le C\,d_E(\phi(z),\phi(z'))\;,
  $$
  by~\eqref{d_ol E(x,y) le d_E(x,y) le C d_ol E(x,y)}. This
  shows~\eqref{i: phi_x,y,r are equi-bi-Lipschitz}
  (considering~\eqref{i: phi_x,y,r is non-expanding}).

  Observe that~\eqref{i: x in dom bar g} and~\eqref{i: phi_x,y,r is
    non-expanding} only require~\eqref{i: V(x,r) subset dom g}
  and~\eqref{i: V(x,r) subset dom g and g(y)=y}. Thus, by~\eqref{i:
    V(x,r) subset dom bar g} and~\eqref{i: V(x,r) subset dom bar g and
    bar g(y)=y}, using $\overline E$ instead of $E$, for each \(y\in
  V(x,r)\) there is a map $\bar\phi_{x,y,r}:\ol B_{\ol E}(x,r)\to\ol
  B_{\ol E}(y,r)$, which is determined by the condition that
  $\bar\phi_{x,y,r}\bar g(x)=\bar g(y)$ for all $\bar g\in\ol E^{\le r}$ with $x\in\dom\ol
  g$. Moreover $\bar\phi_{x,y,r}$ is non-expanding with respect to
  $d_{\ol E}$. Like $\phi$, we will simply use the notation $\bar\phi$ 
  for $\bar\phi_{x,y,r}$. Note that $\bar\phi=\phi$ on
  $\ol B_{\ol{E}}(x,r/C)\cap U$, which is contained in $\ol B_E(x,r)$
  by~\eqref{d_ol E(x,y) le d_E(x,y) le C d_ol E(x,y)}. Thus
  \begin{equation}\label{e:reeb2}
    \bar\phi(\ol B_{\ol{E}}(x,r/C)\cap U)\subset\im\phi\;. 
  \end{equation}
  On the other hand, for each $z\in\ol B_E(y,r/C)$, there is $g\in
  E^{\le\lfloor r/C\rfloor}$ such that $y\in\dom g$ and
  $z=g(y)$. Hence $x\in\dom\bar g$ by~\eqref{i: x in dom bar g}, $\bar
  g(x)\in\ol B_{\ol{E}}(x,r/C)$ because $\bar g\in\ol E^{\le\lfloor
    r/C\rfloor}$, and $\bar\phi\bar g(x)=\bar g(y)=z$
  by~\eqref{i: bar g(x)=x}. So
  \begin{equation}\label{e:reeb3}
    \ol B_E(y,r/C)\subset\bar\phi(\ol B_{\ol{E}}(x,r/C))\;. 
  \end{equation}

  Assume that $r\ge CR$. Then $\ol
  B_{\ol{E}}(x,r/C)\cap U$ is a $2R$-net in 
  $(\ol B_{\ol{E}}(x,r/C),d_{\ol{E}})$ by Lemma~\ref{l:penumbra, graph}. So
  $\bar\phi(\ol B_{\ol{E}}(x,r/C)\cap U)$ is a $2R$-net in
  $(\bar\phi(\ol B_{\ol{E}}(x,r/C)),d_{\ol{E}})$ because
  $\bar\phi$ is non-expanding.  Hence $\ol
  B_E(y,r/C)\cap\im\phi$ is a $2R$-net in $(\ol B_E(y,r/C),d_{\ol{E}})$
   by~\eqref{e:reeb2} and~\eqref{e:reeb3}, and
  therefore it is a $2CR$-net in $(\ol B_E(y,r/C),d_E)$
  by~\eqref{d_ol E(x,y) le d_E(x,y) le C d_ol E(x,y)},
  showing~\eqref{i: ol B_E(y,r/C) cap im phi_x,y,r is a 2CR-net}.
\end{proof}

\begin{rem}\label{r: orbit Reeb with x in U_0}
  	Observe the following in Proposition~\ref{p: orbit Reeb with x in U_0}:
	\begin{enumerate}[(i)]
	
		\item\label{i: phi_x,y,r'|_ol B_E(x,r) = phi_x,y,r} According to~\eqref{i: V(x,r') subset V(x,r) if r'>r},~\eqref{i: V(x,r) subset dom g} and~\eqref{i: phi_x,y,r}, $\phi_{x,y,r'}|_{\ol B_E(x,r)}=\phi_{x,y,r}$ if $r<r'$ and $y\in V(x,r')$.
		
		\item\label{i: phi_x,z,s = phi_y,z,s circ phi_x,y,r} By~\eqref{i: V(x,r) subset dom g},~\eqref{i: phi_x,y,r} and~\eqref{i: phi_x,y,r is non-expanding}, it follows that, if $s\ge r>0$, $x\in U_0$, $y\in V(x,r)\cap U_0$ and $z\in V(y,s)\cap V(x,r)$, then $\im\phi_{x,y,r}\subset\ol B_E(y,s)$ and $\phi_{x,z,s}=\phi_{y,z,s}\circ\phi_{x,y,r}$ on $\ol B_E(x,r)$.
  
  	\end{enumerate}
\end{rem}

\begin{rem}\label{r: U_E}
Let $U_E$ be the complement in $U$ of the $\GG$-saturation of the
union of boundaries in $U$ of the domains of the maps in $E$. Such a
$U_E$ is a dense $G_\delta$ set, and thus so is $U_0\cap U_E$. For all $x\in U_0\cap U_E$, we can choose 
the open neighbourhoods $V(x,r)$ of Proposition~\ref{p: orbit Reeb with x in U_0} satisfying the conditions~\eqref{i: V(x,r) subset dom g} and~\eqref{i: V(x,r) subset dom g and g(y)=y} of Proposition~\ref{p: orbit Reeb with x in U_0} and its proof, and moreover so that:
\begin{itemize}
  
\item if $V(x,r)$ meets the domain of some $g\in E^{\le r}$, then
  $x\in\dom g$; and,
  
\item if $g(y)=y$ for some $g\in E^{\le4r}$ and $y\in V(x,r)\cap\dom g$,
  then $x\in\dom g$ and $g(x)=x$.

\end{itemize}
Then, arguing like in the proof of Proposition~\ref{p: orbit Reeb with
  x in U_0}, it is easy to prove that the maps $\phi_{x,y,r}$ are
isometric bijections for all $r>0$, $x\in U_0\cap U_E$ and $y\in
V(x,r)$.
\end{rem}

The following weaker version of Proposition~\ref{p: orbit Reeb with x in U_0} is valid for all points of $U$.

\begin{prop}\label{p: orbit Reeb with x in U}
  	For every $r\in\Z^+$ and $x\in U$, there exists an open neighborhood $W(x,r)$ of $x$ in $U$ such that:
  		\begin{enumerate}[{\rm(}i\/{\rm)}]
  
  			\item\label{i: W(x,r') subset W(x,r) if r'>r} $W(x,r')\subset W(x,r)$ if $r'>r$;
			
			\item\label{i: x in dom bar g for g in E^le r} $x\in\dom\bar g$ for all $g\in E^{\le r}$ and $y\in W(x,r)\cap\dom g$;
  
  			\item\label{i: xi_y,x,r} for every $y\in W(x,r)$, a map $\xi_{y,x,r}:\ol B_E(y,r)\to\ol B_{\ol E}(x,r)$ is determined by the condition $\xi_{y,x,r}g(y)=\bar g(x)$ for all $g\in E^{\le r}$ with $y\in\dom g$;
  
  			\item\label{i: xi_x,y,r is C-Lipschitz} $\xi_{y,x,r}$ is $C$-Lipschitz with respect to $d_E$; and
  
  			\item\label{i: ol B_E(x,r) subset im xi_y,x,r} $\ol B_E(x,r)\subset\im\xi_{y,x,r}$.
			
  		\end{enumerate}
\end{prop}

\begin{proof}
  	Like in the proof of Proposition~\ref{p: orbit Reeb with x in U_0}, for every $r\in\Z^+$ and $x\in U$, there is a largest open neighborhood $W(x,r)$ of $x$ in $U$ satisfying the following properties:
  		\begin{enumerate}[(a)]
		
  			\item\label{i: W(x,r) subset dom g} $W(x,r)\subset\dom g$ for all $g\in E^{\le r}$ with $x\in\dom g$;
			
			\item\label{i: x in dom bar g for g in E^le 4r} $x\in\dom\bar g$ for all $g\in E^{\le4r}$ and $y\in W(x,r)\cap\dom g$; and
			
			\item\label{i: g(x)=x} if $g(y)=y$ for some $g\in E^{\le4r}$ and $y\in W(x,r)\cap\dom g$, then $\bar g(x)=x$.
			
  		\end{enumerate}
	Property~\eqref{i: V(x,r') subset V(x,r) if r'>r} is also satisfied because $W(x,r)$ is maximal.
  
  	Fix $x\in U$, $r\in\Z^+$ and $y\in W(x,r)$. Property~\eqref{i: x in dom bar g for g in E^le 4r} is stronger than~\eqref{i: x in dom bar g for g in E^le r}.  Each point in $\ol B_E(x,r)$ is of the form $g(x)$ for some $g\in E^{\le r}$ with $y\in\dom g$, and therefore $x\in\dom\bar g$ by~\eqref{i: W(x,r) subset dom g}. Suppose that $g(y)=h(y)$ for another $h\in E^{\le r}$ with $y\in\dom h$. Then $h^{-1}g\in E^{\le2r}$, $y\in\dom h^{-1}g$ and $h^{-1}g(y)=y$. By~\eqref{i: g(x)=x}, we get $x\in\dom\ol{h^{-1}g}=\bar h^{-1}\bar g$ and $\bar h^{-1}\bar g(x)=x$, obtaining $\bar g(x)=\bar h(x)$. Therefore a map $\xi_{y,x,r}:\ol B_E(y,r)\to\ol B_E(x,r)$ is defined by $g(y)\mapsto\bar g(x)$ for $g\in E^{\le r}$, giving~\eqref{i: xi_y,x,r}.
  
  	For every $z,z'\in\ol B_E(y,r)$ there are $g,g'\in E^{\le r}$ whose domains contain $y$ and such that $g(y)=z$ and $g'(y)=z'$. If $m=d_E(z,z')\le2r$, there is $h\in E^m$ so that $z\in\dom h$ and $h(z')=z$. Thus $y\in\dom g^{-1}h g'$ and $g^{-1} h g'(y)=y$ with $g^{-1} h g'\in E^{\le4r}$. It follows from~\eqref{i: x in dom bar g for g in E^le 4r} and~\eqref{i: g(x)=x} that $x\in\dom\ol{g^{-1}hg'}=\bar g^{-1}\bar h\bar{g'}$ and $\bar g^{-1}\bar h\bar{g'}(x)=x$. Hence $\bar{g'}(x)\in\dom\bar h$ and $\bar h\bar{g'}(x)=\bar g(x)$, giving $d_{\ol E}(\bar g(x),\bar{g'}(x))\le m$. So
		\[
			d_E(\bar g(x),\bar{g'}(x))\le C\,d_{\ol E}(\bar g(x),\bar{g'}(x))\le C\,d_E(z,z')
		\]
	by~\eqref{d_ol E(x,y) le d_E(x,y) le C d_ol E(x,y)}, showing~\eqref{i: xi_x,y,r is C-Lipschitz}.

  	Let $z\in\ol B_E(x,r)$, and let $g\in E^{\le r}$ be such that $z=g(x)$. Thus $y\in\dom g$ by~\eqref{i: W(x,r) subset dom g}, and $\xi_{y,x,r}g(y)=\bar g(x)=g(x)$, obtaining~\eqref{i: ol B_E(x,r) subset im xi_y,x,r}.
\end{proof}

\begin{rem}\label{r: orbit Reeb with x in U}
	In Proposition~\ref{p: orbit Reeb with x in U}, note the following:
		\begin{enumerate}[(i)]
  		\item\label{i: xi_y,x,r'|_ol B_E(y,r) = xi_y,x,r} By~\eqref{i: W(x,r') subset W(x,r) if r'>r},~\eqref{i: x in dom bar g for g in E^le r} and~\eqref{i: xi_y,x,r}, $\xi_{y,x,r'}|_{\ol B_E(y,r)}=\xi_{y,x,r}$ if $r<r'$ and $y\in W(x,r')$.
	
		\item\label{i: xi_y,x,r phi_x,y,r = id} By~\eqref{i: V(x,r) subset dom g},~\eqref{i: phi_x,y,r} and Proposition~\ref{p: orbit Reeb with x in U}-\eqref{i: x in dom bar g for g in E^le r},\eqref{i: xi_y,x,r}, $\xi_{y,x,r}\phi_{x,y,r}=\id$ on $\ol B_E(x,r)$ if $x\in U_0$ and $y\in V(x,r)\cap W(x,r)$.
	
	\end{enumerate}
\end{rem}

\begin{prop}\label{p: x_i to x} 
	For any convergent sequence in $U$, $x_i\to x$, and all $r\in\Z^+$, 
			\[
				\ol B_E(x,r)\subset\bigcap_i\Cl_U\left(\bigcup_{j\ge i}\ol B_E(x_i,r)\right)\subset\ol B_E(x,Cr)\;.
			\]
\end{prop}

\begin{proof}
	 Like in the proof of Proposition~\ref{p: orbit Reeb with x in U_0}, for every $r\in\Z^+$ and $x\in U$, there is a largest open neighbourhood $P$ of $x$ in $U$ such that the following properties hold for all $g\in E^{\le r}$:
	 	\begin{enumerate}[(a)]
		
			\item\label{i: P subset dom g} if $x\in\dom g$, then $P\subset\dom g$; and
			
			\item\label{i: y in P cap dom g}  if $y\in P\cap\dom g$, then $x\in\dom\bar g$.
		
		\end{enumerate}
	We can assume that $x_i\in P$ for all $i$. 
	
	The first inclusion of the statement can be proved as follows. Any element of $\ol B_E(x,r)$ is of the form $g(x)$ for some $g\in E^{\le r}$. Then $g(x_i)\in\ol B_E(x_i,r)$ for all $i$ by~\eqref{i: P subset dom g}, and $g(x_i)\to g(x)$ as $i\to\infty$.
	
	Now let us prove the second inclusion. Consider a convergent sequence in $U$, $z_k\to z$, such that $z_k\in\ol B_E(x_{i_k},r)$ for indices $i_k\to\infty$. Thus there are elements $g_k\in E^{\le r}$ such that $x_{i_k}\in\dom g_k$ and $g_k(x_{i_k})=z_k$. Since $E^{\le r}$ is finite, by passing to a subsequence of $z_k$ if needed, we can assume that all maps $g_k$ are equal, and therefore they will be denoted by $g$. Then $x\in\dom\bar g$ by~\eqref{i: y in P cap dom g}, and $z_k=\bar g(x_{i_k})\to\bar g(x)$ as $k\to\infty$. Thus $z=\bar g(x)\in\ol B_{\ol E}(x,r)\subset\ol B_E(x,Cr)$ by~\eqref{d_ol E(x,y) le d_E(x,y) le C d_ol E(x,y)}.
\end{proof}

\section{Topological dynamics}

\subsection{Preliminaries on Baire category}\label{ss: prelim Baire}

We recall some terminology and results about subsets of a topological
space that are relevant to topological dynamics. Good references for
all this and related material are \cite{Kechris1995}, \cite{GottschalkHedlund1955}, \cite{Auslander1988}.

\begin{defn}\label{d: residual etc} A subset $A$ of a topological
  space $X$ is called:
	\begin{itemize}
			
		\item \emph{residual}\footnote{The term \emph{comeager} is also used.} \index{residual} if $A$ contains a countable intersection of open dense subsets;
			
		\item \emph{nowhere dense} \index{nowhere dense} if its closure $\overline{A}$ has empty interior;
			
		\item \emph{meager} \index{meager} if $A$ is a countable union of nowhere dense sets (i.e., $X\sm A$ is residual);
			
		\item \emph{Borel} \index{Borel} if $A$ is a member of the $\sigma$-algebra generated by the open subsets of $X$; and
			
		\item \emph{Baire}\footnote{It is also said that $A$ satisfies the \emph{Baire property}.} \index{Baire} if the symmetric difference\footnote{Recall that the \emph{symmetric difference} of the sets $A,B\subset X$ is the set $A\triangle B=(A\sm B)\cup(B\sm A)$.} $A\triangle U$ is meager for some open $U\subset X$.
			
       \end{itemize}
\end{defn}        

The Baire sets of a topological space also form a $\sigma$-algebra:
the smallest one containing all the open sets and all the meager sets;
in particular, every Borel set is a Baire set. A topological space in
which every residual subset is dense is called a \emph{Baire
  space}. Any open subspace of a Baire space is a Baire space. 
  The Baire category theorem states that every completely
metrizable space and every locally compact Hausdorff space is a Baire
space~\cite[Theorem~8.4]{Kechris1995}.

The Kuratowski-Ulam theorem is the topological analog to Fubini's theorem.

\begin{thm}[Kuratowski-Ulam; see e.g.\ {\cite[Theorem~8.41]{Kechris1995}}]\label{t: KU}
  	Let $X$ and $Y$ be second countable spaces, let $A\subset X\times Y$ be a Baire subset, and let $A_x=\{\,y\in X\mid(x,y)\in A\,\}$ for each $x\in X$. Then the following properties hold:
		\begin{enumerate}[{\rm(}i\/{\rm)}]

			\item\label{i: A_x is Baire} $A_x$ is Baire for residually many $x\in X$.			

			\item\label{i: A is meager iff A_x is meager for a residually many x} $A$ is meager {\rm(}respectively, residual\/{\rm)} if and only if $A_x$ is meager {\rm(}respectively, residual\/{\rm)} for residually many $x\in X$.

		\end{enumerate}
\end{thm} 

A topological space is called \emph{Polish} \index{Polish} if it is separable and
completely metrizable; in particular, it is a Baire space. A subspace of a Polish space is Polish if and only if it is a $G_\delta$ \cite[Theorem~3.11]{Kechris1995}. A locally
compact space is Polish if and only if it is Hausdorff and second
countable \cite[Theorem~5.3]{Kechris1995}.

\subsection{Saturated sets}

Let $\HH$ be a pseudogroup on a  space $Z$.  A subset of $Z$ is
said to be \emph{$\mathcal{H}$-saturated} (or \emph{saturated}) \index{saturated} if
it is a union of orbits of $\mathcal{H}$. The \emph{saturation} \index{saturation} of a
subset $A\subset Z$, denoted by $\mathcal{H}(A)$, is the union of all
orbits that meet $A$; i.e.,
	\begin{equation}\label{HH(A)}
		\HH(A)=\bigcup_hh(A\cap\dom h)\;,
	\end{equation}
where $h$ runs in $\HH$. If a property $P$ is satisfied by the $\HH$-orbits in a residual (respectively, meager) saturated subset of $Z$, then it will be said that $P$ is satisfied by \emph{residually} \index{residually many} (respectively, \emph{meagerly}) \emph{many} \index{meagerly many} $\HH$-orbits.

\begin{lemma}\label{HH(A) open}
	Let $A\subset B\subset Z$. If $A$ is open, dense or residual in $B$, then $\HH(A)$ is open, dense or residual in $\HH(B)$, respectively.
\end{lemma}

\begin{proof}
  By~\eqref{HH(A)}, $\HH(A)$ is open (respectively, dense) in $\HH(B)$
  if $A$ is open (respectively, dense) in $B$. It follows directly from this
   that $\HH(A)$ is residual in $\HH(B)$ if $A$ is residual in
  $B$.
\end{proof}

\begin{lemma}\label{l: saturation when HH is countably generated}
  	If $\HH$ is countably generated, then the saturation of a Borel, Baire or meager subset of $Z$ is Borel, Baire or meager, respectively.
\end{lemma}

\begin{proof}
  Let $E$ be a countable symmetric set of generators of $\HH$, and let
  $S$ be the countable set of all composites of elements of $E$,
  wherever defined. Then the result follows because~\eqref{HH(A)} still
  holds if $h$ runs only in $S$.
\end{proof}

\begin{rem}\label{r: R_HH}
  Let $R_\HH\subset Z\times Z$ be the relation set of the equivalence
  relation ``being in the same $\HH$-orbit.'' Assume that $\HH$ is
  countably generated. With the notation of the proof of Lemma~\ref{l:
    saturation when HH is countably generated}, $R_\HH$ equals the
  union of the graphs of maps in $S$, which are easily seen to be
  $F_\sigma$-sets. Hence $R_\HH$ is an $F_\sigma$ set because $S$ is
  countable.
\end{rem}

A pseudogroup, $\HH$, is called \emph{transitive} \index{transitive} (respectively, \emph{minimal}) \index{minimal} if it has a dense orbit (respectively, every orbit is
dense). An ($\HH$-)\emph{minimal set} \index{minimal set} is a non-empty, saturated and closed
subset of $Z$ which is minimal among the sets with these
properties. The following result is well known. 

\begin{prop}\label{p:dense orbits}
  If $Z$ is second countable, then the union of all the $\HH$-orbits
  that are dense is a $G_\delta$-set. In particular, this set is
  residual if and only if $\HH$ is transitive.
\end{prop}

\begin{proof}
  Let $\{U_n\}$ be a countable base for the topology of $Z$. Then the
  union of all the dense orbits is equal to the intersection of the
  saturations $\HH(U_n)$, which are open in $Z$.
\end{proof}

\begin{cor}
  If $Z$ is second countable and $\HH$ is transitive, then the union
  of all the proper minimal sets is meager.
\end{cor}

\begin{prop}\label{p: there is some minimal set}
  If $Z$ is locally compact space and $Z/\HH$ is compact, then any
  nonempty $\HH$-invariant closed subset of $Z$ contains some minimal
  set.
\end{prop}

\begin{proof}
  By Zorn's lemma, it is enough to prove that
  any family of $\HH$-invariant nonempty closed subsets $Y_i\subset Z$
  has nonempty intersection. By the hypothesis, there is a relatively
  compact open subset $U\subset Z$ that meets all $\HH$-orbits. Then
  $\ol U\cap Y_i$ is a nonempty compact subset, obtaining
  $\emptyset\ne\ol U\cap\bigcap_iY_i\subset\bigcap_iY_i$.
\end{proof}

The next result is the topological zero-one law, a topological version
of ergodicity.

\begin{thm}\label{t: 01law for pseudogroups}
  If \(Z\) is a Baire space and $\HH$ is a transitive pseudogroup on 
  \(Z\), then each saturated Baire subset of $Z$ is
  either meager or residual.
\end{thm}

\begin{proof}
  Let $A$ be a saturated Baire subset of $Z$.  There is an open set
  $U$ such that $A \triangle U$ is a meager set. If $A$ is not meager,
  then $U$ is non-empty and $U\sm A$ is meager. Thus, if $A$ is
  neither meager nor residual, then there are non-empty open subsets
  $U$ and $V$ so that $A\cap U$ is residual in $U$ and $V\sm A$ is
  residual in $V$. Hence $\HH(A\cap U)$ and $\HH(V\sm A)$ are residual
  in $\HH(U)$ and $\HH(V)$, respectively (Lemma~\ref{HH(A) open}). Since there is a dense orbit, the
  non-empty open sets $\HH(U)$ and $\HH(V)$ intersect in a non-empty
  set, and thus $\HH(A\cap U)$ meets $\HH(V\sm A)$ because $\HH(U)$
  and $\HH(V)$ are Baire spaces. But $\HH(A\cap
  U)\subset A$ and $\HH(V\sm A)\subset Z\sm A$, a contradiction.
\end{proof}

The following theorem is of basic importance for the contents of this work.

\begin{thm}[Hector~\cite{Hector1977a}, Epstein-Millet-Tischler~\cite{EpsteinMillettTischler1977}]
\label{t: Hector,E-M-T for pseudogroups}
If $\HH$ is a countably generated pseudogroup on a space \(Z\), then the union of orbits with trivial germ
groups is a dense $G_\delta$ subset of $Z$, hence Borel and residual.
\end{thm}

\begin{cor}\label{c: the union of dense orbits with trivial germ groups is residual}
  If $Z$ is second countable, and $\HH$ is transitive and countably
  generated, then the union of dense orbits with trivial germ groups
  is a residual subset of $Z$.
\end{cor}

\begin{proof}
  This is a direct consequence of Proposition~\ref{p:dense orbits} and
  Theorem~\ref{t: Hector,E-M-T for pseudogroups}.
\end{proof}

\subsection{The property of being recurrent on Baire sets}

The following shows that the condition of being recurrent on systems
of compact generation also holds in a Baire sense. It is a
topological-coarsely quasi-isometric version of Ghys' ``Proposition
fondamentale'' \cite[p.~402]{Ghys1995} for pseudogroups.

\begin{thm}\label{t: Ghys for pseudogroups}
  Let $\HH$ be a compactly generated minimal pseudogroup of local
  transformations of a locally compact space $Z$, let $U$ be a
  relatively compact open subset of $Z$ that meets all $\HH$-orbits,
  let $E$ be a recurrent symmetric system of compact generation of
  $\HH$ on $U$, and let $\GG=\HH|_U$. Then, for any Baire subset $B$
  of $U$,
  \begin{enumerate}[{\rm(}i\/{\rm)}]
		
  \item either $\GG(B)$ is meager;
			
  \item or else the intersections of residually many $\GG$-orbits with
    $B$ are equi-nets in those $\GG$-orbits with $d_E$.
			
  \end{enumerate}
\end{thm}

\begin{proof}
  Suppose $\GG(B)$ is not meager in $U$. So $B$ is not meager in $U$
  by Lemma~\ref{l: saturation when HH is countably generated}. Then
  there is a non-empty open subset $V$ of $U$ such that $V\sm B$ is
  meager in $V$. Hence $\GG(V\sm B)$ is meager in $U$ by Lemma~\ref{l:
    saturation when HH is countably generated}, and thus
  $Y=U\sm\GG(V\sm B)$ is residual in $U$ and $\GG$-saturated. Since
  $\GG$ is minimal and $E$ recurrent, there is some $R>0$ such that
  $\OO\cap V$ is an $R$-net in $(\OO,d_E)$ for any $\GG$-orbit $\OO$
  (Proposition~\ref{p: recurrent finite symmetric family of
    generators}); but $\OO\cap V\subset B$ if $\OO\subset Y$, and the
  result follows.
\end{proof}

\subsection{Pseudogroups versus group actions}

As said in Section~\ref{s: pseudogroups}, a pseudogroup on a space is a generalization of a group
acting on a space via homeomorphisms. With a slight change of the
topology of the space acted on, the converse is also true, in the
following sense.

\begin{thm}\label{theorem91}
  Let \(\HH\) be a countable generated pseudogroup of local homeomorphisms of a Polish
  space \(Z\). Then there is a Polish space $Z'$, with the same underlying set as $Z$, and a
    pseudogroup $\HH'$ on $Z'$ such that $\HH'$ has the same
    orbits as $\HH$, $\HH\subset\HH'$, and $\HH'$ is equivalent to the pseudogroup
    generated by a countable group $G$ of homeomorphisms on another Polish space. Moreover, for each symmetric set $E$ of generators of $\HH$, there is a symmetric set $F$ of generators of $G$, with the same cardinality as $E$, such that the $\HH$-orbits with $d_E$ are equi-coarsely quasi-isometric to the corresponding $G$-orbits with $d_F$.
\end{thm}

\begin{proof}
  Let \(W=Z\sqcup Z=Z\times\{0,1\}\). For \(i\in\{0,1\}\), let $\iota_i:Z\to W$ be given by $\iota_i(z)=(z,i)$, and let \(Z_i=\iota_i(Z)\). Let \(E\) be a countable symmetric set of generators for \(\HH\) containing $\id_Z$. For each $h\in E$, let \(g_h\) be the Borel measurable bijection of \(W\) given by $g_h\iota_0(z)=\iota_1h(z)$ if $z\in\dom h$, $g_h\iota_1(z)=\iota_0h^{-1}(z)$ if $z\in\im h$, and $g_h\iota_i(z)=\iota_i(z)$ otherwise; in particular, $g_{\id_Z}\iota_0(z)=\iota_1(z)$ and $g_{\id_Z}\iota_1(z)=\iota_0(z)$ for all $z\in Z$. Then \(F=\{\,g_h\mid h\in E\,\}\) generates a countable group \(G\) of Borel measurable bijections of \(W\). Note that $F$ has the same cardinality as $E$, and is symmetric because every $g_h$ is of order~$2$.
  
  By \cite[Theorem 5.2.1]{BeckerKechris1996}, there is a Polish topology $\tau^*$ on $W$ so that $G$ consists of homeomorphisms of the corresponding space $W^*$, and inducing the same Borel $\sigma$-algebra as the original topology of $W$. Writing \(E=\{h_n\mid n\in\N\}\), by \cite[Theorem~5.1.11]{BeckerKechris1996}, there are Polish topologies $\tau_n$ on $W$ such that $\tau^*\subset\tau_0\subset\tau_1\subset\cdots$, $\iota_0(\dom h_n)\in\tau_n$, and $G$ consists of homeomorphisms of the corresponding spaces. By \cite[Theorem~5.1.3-(b)]{BeckerKechris1996}, the topology $\tau'$ generated by $\bigcup_n\tau_n$ is also Polish. Since $\bigcup_n\tau_n$ is a base of $\tau'$, the maps in $G$ are homeomorphisms of the space $W'$ defined with $\tau'$. Let $\GG$ be the pseudogroup generated by $G$ on $W'$. 
  
  Since $Z_0,Z_1\in\tau'$ and \(g_{\id_Z}\) restricts to a homeomorphism \(Z_0\to Z_1\), there is a Polish topology on $Z$ so that the corresponding space $Z'$ satisfies $Z'\sqcup Z'=W'$. Thus the maps \(\iota_i:Z'\to W'\) are open embeddings, the sets $Z_i$ meet all $G$-orbits, and there is a unique pseudogroup $\HH'$ on $Z'$ so that any $\iota_i$ generates an equivalence $\HH'\to\GG$. 
  
  For each $h\in E$, the restriction $g_{\id_Z}g_h:\iota_0(\dom h)\to\iota_0(\im h)$ is in $\HH'$ and corresponds to $h$ via $\iota_0$. Thus $E\subset\HH'$, and therefore $\HH\subset\HH'$; in particular, the $\HH$-orbits are contained in $\HH'$-orbits. 
  
  Suppose that $\OO\subset\OO'$ for orbits $\OO$ of $\HH$ and $\OO'$ of $\HH'$. For different points, $z\in\OO$ and $z'\in\OO'$, let $k=d_F(\iota_0(z),\iota_0(z'))\ge1$, and take $h_1,\dots,h_k\in E$ so that $g_{h_k}\cdots g_{h_1}\iota_0(z)=\iota_0(z')$. Then $z\in\dom(h_k\cdots h_1)$ and $h_k^{\epsilon_k}\cdots h_1^{\epsilon_1}(z)=z'$ for some choice of $\epsilon_1,\dots,\epsilon_k\in\{\pm1\}$, obtaining that $z'\in\OO$ and $d_E(z,z')\le k$. In particular, this shows that $\OO=\OO'$. Now let $l=d_E(z,z')\ge1$, and take $\bar h_1,\dots,\bar h_l\in E$ such that $z\in\dom(\bar h_l\cdots\bar h_1)$ and $\bar h_l\cdots\bar h_1(z)=z'$. Then, either $g_{\bar h_l^{\delta_l}}\cdots g_{\bar h_1^{\delta_1}}\iota_0(z)=\iota_0(z')$, or $g_{\id_Z}g_{\bar h_l^{\delta_l}}\cdots g_{\bar h_1^{\delta_1}}\iota_0(z)=\iota_0(z')$, for some choices of $\delta_1,\dots,\delta_l\in\{\pm1\}$, obtaining $k\le l+1$. On the other hand, the intersections of the $G$-orbits with $Z_0$ are $1$-nets in those $G$-orbits with $d_F$ because $g_{\id_Z}(Z_1)=Z_0$. So $\iota_0$ defines equi-corse quasi-isometries of the $\HH$-orbits with $d_E$ to the corresponding $G$-orbits with $d_F$.
\end{proof}

\begin{rem}\begin{enumerate}[(i)]
  \item The doubling of the space is required because the domain and
    image of a local homeomorphisms \(h\in E\) may have non-empty
    intersection. However, if \(\HH\) is the representative of the holonomy 
    pseudogroup of a foliated space generated by the transition mappings of a
    regular foliated atlas (Section~\ref{s: foliated space}), then that doubling is not
    necessary because each transition mapping has disjoint image and
    domain.
  \item Theorem~\ref{theorem91} is related to the result of
    Feldman and Moore~\cite[Theorem~1]{FeldmanMoore1977:I} stating the following :
    \textit{If \(R\) countable equivalence relation on a standard
      Borel space \(X\), then there is a countable group \(G\) of
      Borel automorphisms of \(X\) so that \(R\) is the orbit
      equivalence relation induced by \(G\)}.
\end{enumerate}
\end{rem}

\chapter{Generic coarse geometry of orbits}\label{c: orbits}

The following notation will be used in the whole of this chapter. 

\begin{hyp}\label{h: hyp part III}
A quintuple \((Z,\HH,U,\GG,E)\) is required to satisfy the following conditions:
 \begin{itemize}
\item  $Z$ is a locally compact Polish space, 
\item  $\HH$ is a compactly generated pseudogroup of local transformations of $Z$,
\item $U$ is a relatively compact open subset of $Z$ that meets all $\HH$-orbits,
\item $\GG$ denotes the restriction of $\HH$ to $U$, and
\item $E$ is a recurrent symmetric system of compact generation of $\HH$
on $U$.
\end{itemize}
\end{hyp}

All metric concepts in the $\GG$-orbits will be considered
with respect to the metric $d_E$ induced by $E$. Thus the subindex
``$E$'' will be deleted from the notation of the metric, open balls,
closed balls, and penumbras. Let $U_0$ denote the union of $\GG$-orbits
with trivial germ groups, which is a dense $G_\delta$-subset of $U$ by
Theorem~\ref{t: Hector,E-M-T for pseudogroups}, and therefore $U_0$ is a Polish subspace (Section~\ref{ss: prelim Baire}). Moreover let
$U_{0,\text{\rm d}}$ be the union of dense orbits in $U_0$, which is a
residual subset of $Z$ if $\HH$ is transitive (Corollary~\ref{c: the
  union of dense orbits with trivial germ groups is residual}). For
each $A\subset U_0$, let $\Cl_0(A)$ and $\Int_0(A)$ denote its closure
and interior in $U_0$, respectively; the same notation will be used
for the closure and interior in $U_0\times U_0$. Assume that
$|E|\ge2$, otherwise the $\GG$-orbits have at most two elements.

\section{Coarsely quasi-isometric orbits}\label{s: coarsely quasi-isometric orbits}

For $K\in\N$ and $C\in\Z^+$, let $Y(K,C)$ be the set of pairs
$(x,y)\in U_0\times U_0$ such that there is a $(K,C)$-coarse
quasi-isometry $f:A\to B$ of $\GG(x)$ to $\GG(y)$ with $x\in A$, $y\in
B$ and $f(x)=y$. Notice that $Y(K,C)\subset Y(K',C')$ if $K\le K'$ and
$C\le C'$.

\begin{lemma}\label{l:Y(K,C)}
  For $K\in\N$ and $C\in\Z^+$, there is some $K'\in\N$ and
  $C'\in\Z^+$, depending only on $K$ and $C$, such that
  $\Cl_0(Y(K,C))\subset Y(K',C')$.
\end{lemma}

\begin{proof}
  Let $(x,y)\in\Cl_0(Y(K,C))$. For each $r\in\Z^+$, consider the open
  neighborhoods $V(x,r)$ and $V(y,r)$ of $x$ and $y$ in $U$ given by
  Proposition~\ref{p: orbit Reeb with x in U_0}. Then there is some
  $$
  			(x_n,y_n)\in Y(K,C)\cap(V(x,n)\times V(y,Cn))
  $$
  for each $n\in\Z^+$. According to Proposition~\ref{p: orbit Reeb
    with x in U_0}, by taking $K$ and $C$ large enough, it can be
  assumed that all $\phi_n:=\phi_{x,x_n,n}$ and
  $\psi_n:=\phi_{y,y_n,Cn}$ are non-expanding equi-bi-Lipschitz maps
  with bi-Lipschitz distortion $C$, and the sets $\ol
  B(x_n,n/C)\cap\im\phi_n$ and $\ol B(y_n,n)\cap\im\psi_n$ are
  $K$-nets in $\ol B(x_n,n/C)$ and $\ol B(y_n,n)$, respectively. In
  particular, the restriction
  $$
  \phi_n:\phi_n^{-1}(\ol B(x_n,n/C))\to \ol B(x_n,n/C)\cap\im\phi_n
  $$
  is a $(K,C)$-coarse quasi-isometry of $\phi_n^{-1}(\ol B(x_n,n/C))$
  to $\ol B(x_n,n/C)$.
  
  On the other hand, for each $n$, there is a $(K,C)$-coarse
  quasi-isometry $f_n:A_n\to B_n$ of $\GG(x_n)$ to $\GG(y_n)$, and so
  that $x_n\in A_n$, $y_n\in B_n$ and $f_n(x_n)=y_n$.  Each set $\ol
  B(x_n,n/C)\cap A_n$ is a $2K$-net in $\ol B(x_n,n/C)$; this holds
  for $\lfloor n/C\rfloor>K$ by Lemma~\ref{l:penumbra, graph}, and for
  $\lfloor n/C\rfloor\le K$ because $x_n\in \ol B(x_n,n/C)\cap
  A_n$. Also, note that
  $$
  			f_n(\ol B(x_n,n/C)\cap A_n)\subset \ol B(y_n,n)\;,
  $$
  and thus each $f_n(\ol B(x_n,n/C)\cap A_n)$ is a $2K$-net in its
  $2K$-penumbra $P_n$ in $\ol B(y_n,n)$, obtaining that each
  restriction
  $$
  			f_n:\ol B(x_n,n/C)\cap A_n\to f_n(\ol B(x_n,n/C)\cap A_n)
  $$
  is a $(2K,C)$-coarse quasi-isometry of $\ol B(x_n,n/C)$ to
  $P_n$. Moreover, because each $\ol B(y_n,n)\cap\im\psi_n$ is a
  $K$-net in $\ol B(y_n,n)$, each $P_n\cap\im\psi_n$ is a $2K$-net in
  $P_n$ by Lemma~\ref{l:penumbra, graph}. So the restriction
  $\psi_n^{-1}:P_n\cap\im\psi_n\to\psi_n^{-1}(P_n)$ is a
  $(2K,C)$-coarse quasi-isometry of $P_n$ to $\psi_n^{-1}(P_n)$. It
  follows from Proposition~\ref{p: coarse composite} that, for some
  $K'\ge0$ and $C'\ge1$, depending only on $K$ and $C$, there is a
  $(K',C')$-coarse quasi-isometry $g_n$ of $\phi_n^{-1}(\ol
  B(x_n,n/C))$ to $\psi_n^{-1}(P_n)$; this $g_n$ is a coarse composite
  of the above three coarse quasi-isometries. Since
  \begin{gather*}
    x\in\phi_n^{-1}(\ol B(x_n,n/C))\;,\quad\phi_n(x)=x_n\in \ol B(x_n,n/C)\cap A_n\;,\\
    f_n(x_n)=y_n\in P_n\cap\im\psi_n\;,\quad\psi_n^{-1}(y_n)=y\;,
  \end{gather*}
  Proposition~\ref{p: coarse composite} also guarantees that $g_n$ can
  be chosen so that $x\in\dom g_n$, $y\in\im g_n$ and $g_n(x)=y$.
  
  Observe that
  $$
  			\ol B(x,n/C)\subset\phi_n^{-1}(\ol B(x_n,n/C))
  $$
  because $\phi_n$ is non-expanding and $\phi_n(x)=x_n$. Therefore the
  sequence of finite sets $\phi_n^{-1}(\ol B(x_n,n/C))$ is exhausting
  in $\GG(x)$. On the other hand,
  $$
  \ol B(y_n,n/C^2)\cap B_n\subset f_n(\ol B(x_n,n/C)\cap A_n)
  $$
  since $f_n:A_n\to B_n$ is a $C$-bi-Lipschitz bijection such that
  $x_n\in A_n$, $y_n\in B_n$ and $f_n(x_n)=y_n$.  Furthermore $\ol
  B(y_n,n/C^2)\cap B_n$ is a $2K$-net in $\ol B(y_n,n/C^2)$, which
  holds for $\lfloor n/C^2\rfloor>K$ by Lemma~\ref{l:penumbra, graph},
  and for $\lfloor n/C^2\rfloor\le K$ because $y_n\in \ol
  B(y_n,n/C^2)\cap B_n$. Therefore
  $$
  P_n\supset \ol B(y_n,n)\cap\Pen(\ol B(y_n,n/C^2)\cap B_n,2K) \supset
  \ol B(y_n,n/C^2)\;,
  $$
  	giving
  $$
  \ol B(y,n/C^2)\subset\psi_n^{-1}(\ol
  B(y_n,n/C^2))\subset\psi_n^{-1}(P_n)
  $$
  since $\psi_n$ is non-expanding and $\psi_n(y)=y_n$. So the sequence
  of finite sets $\psi_n^{-1}(P_n)$ is exhausting in $\GG(y)$.
  
  By applying Proposition~\ref{p: Arzela-Ascoli for coarse
    quasi-isometries} to the sequence of coarse quasi-isometries
  $g_n$, and using Remark~\ref{r:coarse quasi-isometry}, it follows
  that there is a $(K',C')$-coarse quasi-isometry $g$ of $\GG(x)$ to
  $\GG(y)$ such that $x\in\dom g$ and $g(x)=y$; i.e., $(x,y)\in
  Y(K',C')$.
\end{proof}

\begin{rem}
If the statement of Lemma~\ref{l:Y(K,C)} were restricted to the 
  residual union of orbits which do not meet the boundaries in $U$ of
  the domains of maps in $E$, then it would be an easy consequence of
  Proposition~\ref{p: Arzela-Ascoli for coarse quasi-isometries} and
  Remark~\ref{r: U_E}.
\end{rem}

Let $Y$ be the set of points $(x,y)\in U_0\times U_0$ such that
$\GG(x)$ is coarsely quasi-isometric to $\GG(y)$.

\begin{lemma}\label{l:Y}
  $Y=\bigcup_{K,C=1}^\infty Y(K,C)$.
\end{lemma}

\begin{proof}
  This equality follows from Corollary~\ref{c: coarse quasi-isometry,
    x_0 mapsto x'_0}.
\end{proof}

\begin{cor}\label{c:Y}
  $Y=\bigcup_{K,C=1}^\infty\Cl_0(Y(K,C))$; in particular, $Y$ is an
  $F_\sigma$-subset of $U_0\times U_0$.
\end{cor}

\begin{proof}
  This is elementary by Lemmas~\ref{l:Y(K,C)} and~\ref{l:Y}.
\end{proof}

\begin{cor}\label{c:either Y(K,C) has non-empty interior in U_0 times U_0 or Y is meager}
  Either $\Int_0(Y(K,C))\ne\emptyset$ for some $K,C\in\Z^+$,
  or else $Y$ is a meager subset of $U_0\times U_0$.
\end{cor}

\begin{proof}
  If $\Int_0(Y(K,C))=\emptyset$ for all $K,C\in\Z^+$, then
  $\Int_0(\Cl_0(Y(K,C)))=\emptyset$ for all $K,C\in\Z^+$ by
  Lemma~\ref{l:Y(K,C)}, and thus $Y$ is meager in $U_0\times U_0$ by
  Corollary~\ref{c:Y}.
\end{proof}

\begin{thm}\label{t: coarsely q.i. orbits 1}
  Let \((Z,\HH, U,\GG,E)\) satisfy Hypothesis~\ref{h: hyp part III}. The equivalence
  relation ``$x\sim y$ if and only if the orbits $\GG(x)$ and $\GG(y)$
  are coarsely quasi-isometric'' has a Borel relation set in
  $U_0\times U_0$, and has a Baire relation set in $U\times U$; in
  particular, it has Borel equivalence classes in $U_0$, and Baire
  equivalence classes in $U$.
\end{thm}

\begin{proof}
  This follows from Corollary~\ref{c:Y} and Theorem~\ref{t: Hector,E-M-T for pseudogroups}.
\end{proof}

\begin{thm}\label{t: coarsely q.i. orbits 2}
  Let \((Z,\HH, U,\GG,E)\) satisfy Hypothesis~\ref{h: hyp part III}. Suppose that $\HH$ is
  transitive. Then:
  \begin{enumerate}[{\rm(}i\/{\rm)}]
		
  \item\label{i: all GG-orbits in U_0,d are equi-coarsely
      quasi-isometric to each other} either all $\GG$-orbits in
    $U_{0,\text{\rm d}}$ are equi-coarsely quasi-isometric to each
    other;
			
  \item\label{i: there are uncountably many coarse quasi-isometry
      types of GG-orbits} or else every $\GG$-orbit is coarsely
    quasi-isometric to meagerly many $\GG$-orbits; in particular,
    there are uncountably many coarse quasi-isometry types of
    $\GG$-orbits in $U_0$ in this case.
			
  \end{enumerate}
\end{thm}

\begin{proof}
  Suppose that $\Int_0(Y(K,C))\ne\emptyset$ for some
  $K,C\in\Z^+$. Then all dense $\GG\times\GG$-orbits in $U_0\times
  U_0$ meet $Y(K,C)$, and thus all $\GG$-orbits in $U_{0,\text{\rm
      d}}$ are equi-coarsely quasi-isometric to each other.
  
  If $\Int_0(Y(K,C))=\emptyset$ for all $K,C\in\Z^+$, then $Y$ is
  meager in $U_0\times U_0$ by Corollary~\ref{c:either Y(K,C) has
    non-empty interior in U_0 times U_0 or Y is meager}.  It follows
  from Theorem~\ref{t: KU} that there is a residual subset $A\subset U_0$
  such that $Y_x=\{\,y\in U_0\mid (x,y)\in Y\,\}$ is meager in
  $U$ for all $x\in A$. But each $Y_x$ is the union of orbits in $U_0$
  which are coarsely quasi-isometric to $\GG(x)$.  Finally, $A$ can be
  assumed to be $\GG$-saturated since $Y$ is $\GG\times\GG$-saturated.
\end{proof}

\begin{thm}\label{t: coarsely q.i. orbits 3}
  Let \((Z,\HH, U,\GG,E)\) satisfy Hypothesis~\ref{h: hyp part III}. Then the following properties hold:
  \begin{enumerate}[{\rm(}i\/{\rm)}]
  
  \item\label{i: if there is a coarsely quasi-symmetric orbit in
      U_0,d, then alternative (i) holds} Suppose that $\HH$ is
    transitive. If there is a coarsely quasi-symmetric $\GG$-orbit in
    $U_{0,\text{\rm d}}$, then the alternative~{\rm(}\ref{i: all
      GG-orbits in U_0,d are equi-coarsely quasi-isometric to each
      other}\/{\rm)} of Theorem~\ref{t: coarsely q.i. orbits 2} holds.
  
  \item\label{i: if alternative (i) holds, then all orbits in U_0 are
      equi-coarsely quasi-symmetric} Suppose that $\HH$ is minimal. If
    the alternative~{\rm(}\ref{i: all GG-orbits in U_0,d are
      equi-coarsely quasi-isometric to each other}\/{\rm)} of
    Theorem~\ref{t: coarsely q.i. orbits 2} holds, then all
    $\GG$-orbits in $U_0$ are equi-coarsely quasi-symmetric.

  \end{enumerate}
\end{thm}

\begin{proof}
  Suppose that $\HH$ is transitive and that there is a coarsely
  quasi-symmetric $\GG$-orbit $\OO$ in $U_{0,\text{\rm d}}$. Then
  $\OO\times\OO\subset Y(K,C)$ for some $K,C\in\Z^+$, giving
  $U_0\times U_0=Y(K',C')$ for $K',C'\in\Z^+$ by Lemma~\ref{l:Y(K,C)},
  which means that all $\GG$-orbits in $U_0$ are equi-coarsely
  quasi-isometric to each other.
  
  Assume that $\HH$ is minimal and that all $\GG$-orbits in $U_0$ are
  coarsely quasi-isometric to each other. This means that $Y=U_0\times
  U_0$. Hence $\Int_0(Y(K,C))\ne\emptyset$ for some $K,C\in\Z^+$ by
  Corollary~\ref{c:either Y(K,C) has non-empty interior in U_0 times
    U_0 or Y is meager}; i.e., there are some non-empty open subsets
  $V$ and $W$ of $U_0$ so that $V\times W\subset Y(K,C)$. But, by
  Proposition~\ref{p: recurrent finite symmetric family of generators}
  and since $\GG$ is minimal, the intersections $\OO\cap V$ and
  $\OO\cap W$ are equi-nets in the $\GG$-orbits $\OO$ in $U_0$.
  So the $\GG$-orbits in $U_0$ are equi-coarsely quasi-symmetric 
  by Lemma~\ref{l: coarsely quasi-symmetric}.
\end{proof}

The following result follows from Theorems~\ref{t: coarsely
  q.i. orbits 3} and~\ref{t: coarse ends}.

\begin{cor}\label{c: coarse end space of orbits in U_0}
  Let \((Z,\HH, U,\GG,E)\) satisfies Hypothesis~\ref{h: hyp part III}. Suppose that
  $\GG$ is minimal and satisfies the alternative~{\rm(}\ref{i: all
    GG-orbits in U_0,d are equi-coarsely quasi-isometric to each
    other}\/{\rm)} of Theorem~\ref{t: coarsely q.i. orbits 2}. Then all
  $\GG$-orbits in $U_0$ have zero, one, two or a Cantor space of
  coarse ends, simultaneously.
\end{cor}

\section{Growth of the orbits}\label{s: growth, orbits}

\subsection{Orbits with the same growth type}

Since the $\GG$-orbits are equi-quasi-lattices in themselves (Example~\ref{ex: coarse bounded geometry}-\eqref{i: graph of coarse bd geom}), the
growth type of $\GG(x)$ is represented by the mapping $r\mapsto
v(x,r)=|\ol B(x,r)|$ for all $x\in U$. For $a,b\in\Z^+$ and $c\in\N$,
let
	$$
		Y(a,b,c)=\{\,(x,y)\in U_0\times U_0\mid v(x,r)\leq a\,v(y,br)\ \forall r\ge c\,\}\;.
	$$
Note that 
	$$
		a\le a',\ b\le b',\ c\le c'\Longrightarrow Y(a,b,c)\subset Y(a',b',c')\;.
	$$

\begin{lemma}\label{l:Y(a,b,c)}
  For all $a,b\in\Z^+$ and $c\in\N$, there are some integers $a'\ge a$ and $b'\ge
  b$ such that $\Cl_0(Y(a,b,c))\subset Y(a',b',c)$.
\end{lemma}

\begin{proof}
  Consider the notation of Proposition~\ref{p: orbit Reeb with x in
    U_0}. Let $(x,y)\in\Cl_0(Y(a,b,c))$. For any integer $r\ge c$,
  take a pair
		$$
			(x',y')\in Y(a,b,c)\cap(V(x,r)\times V(y,Cbr))\;.
		$$
	Then, with the notation given by~\eqref{Lambda_K,r},
		$$
			v(x,r)\le v(x',r)\le a\,v(y',br)\le a\,\Lambda_{|E|,2RC}\,v(y,Cbr)
		$$
                since $\phi_{x,x',r}$ is injective
                (Proposition~\ref{p: orbit Reeb with x in
                  U_0}-\eqref{i: phi_x,y,r are equi-bi-Lipschitz}) and
                $\ol B(y',br)\cap\im\phi_{y,y',Cbr}$ is a $2RC$-net in
                $\ol B(y',br)$ (Proposition~\ref{p: orbit Reeb with x
                  in U_0}-\eqref{i: ol B_E(y,r/C) cap im phi_x,y,r is
                  a 2CR-net}), and because $|\ol
                B(z,2RC)|\le\Lambda_{|E|,2RC}$ for all $z\in\ol
                B(y',br)$ by~\eqref{|ol B(x,r)|}. Hence $(x,y)\in
                Y(a\Lambda_{|E|,2RC},Cb,c)$.
\end{proof}

Note that $Y=\bigcup_{a,b,c=1}^\infty Y(a,b,c)$ is the set of points
$(x,y)\in U_0\times U_0$ such that the growth type of $\GG(x)$ is
dominated by the growth type of $\GG(y)$. Let $\tau:U\times U\to
U\times U$ be the homeomorphism given by $\tau(x,y)=(y,x)$

\begin{thm}\label{t:orbit growth 1}
  Let \((Z,\HH, U,\GG,E)\) satisfy Hypothesis~\ref{h: hyp part III}. The equivalence
  relation ``$x\sim y$ if and only if the orbits $\GG(x)$ and $\GG(y)$
  have the same growth type'' has a Borel relation set in $U_0\times
  U_0$, and has a Baire relation set in $U\times U$; in particular, it
  has Borel equivalence classes in $U_0$, and Baire equivalence
  classes in $U$.
\end{thm}

\begin{proof}
  By Lemma~\ref{l:Y(a,b,c)}, $Y$ is an $F_\sigma$-subset of $U_0\times
  U_0$, and thus so is $Y_\tau:=Y\cap \tau(Y)$. But $Y_\tau$ is the
  relation set in $U_0\times U_0$ of the statement, and the result
  follows because $U_0$ is residual in $U$ (Theorem~\ref{t:
    Hector,E-M-T for pseudogroups}).
\end{proof}

\begin{thm}\label{t:orbit growth 2}
  Let \((Z,\HH, U,\GG,E)\) satisfy Hypothesis~\ref{h: hyp part III}. If $\GG$ is transitive,
  then:
  \begin{enumerate}[{\rm(}i\/{\rm)}]
		
  \item\label{i: all orbits in U_0,d have equi-equivalent growth}
    either all $\GG$-orbits in $U_{0,\text{\rm d}}$ have
    equi-equivalent growth;
			
  \item\label{i: uncountably many growth types of orbits in U_0} or
    else the growth type of every $\GG$-orbit is comparable with the growth type
    of meagerly many $\GG$-orbits; in particular, there are
    uncountably many growth types of $\GG$-orbits in $U_0$ in this
    second case.
			
  \end{enumerate}
\end{thm}

\begin{proof}
  If $\Int_0(Y(a,b,c))\ne\emptyset$ for some $a,b,c\in\Z^+$, then all
  dense $\GG\times\GG$-orbits in $U_0\times U_0$ meet $Y(a,b,c)$. It
  follows that all $\GG$-orbits in $U_{0,\text{\rm d}}$ have
  equi-equivalent growth.
  
  On the other hand, if $\Int_0(Y(a,b,c))=\emptyset$ for all
  $a,b,c\in\Z^+$, then $Y$ is a meager subset of $U_0\times U_0$ by
  Lemma~\ref{l:Y(a,b,c)}, and thus it is meager in $U\times U$ too. So
  $Y^\tau:=Y\cup \tau(Y)$ is meager in $U\times U$ as well. It follows
  from Theorem~\ref{t: KU} that there is a residual subset $A\subset
  U_0$ such that $Y^\tau_x=\{\,y\in U_0\mid (x,y)\in Y^\tau\,\}$ is
  meager for all $x\in A$.  But each $Y^\tau_x$ is the union of
  $\GG$-orbits in $U_0$ whose growth type is comparable with the
  growth type of $\GG(x)$.  Obviously, it can be assumed that $A$ is
  saturated.
\end{proof}

\begin{thm}\label{t:orbit growth 3}
Let \((Z,\HH, U,\GG,E)\) satisfy Hypothesis~\ref{h: hyp part III}. Then the following properties hold:
  \begin{enumerate}[{\rm(}i\/{\rm)}]

  \item\label{i: if there is a growth symmetric orbit in U_0,d, then
      the alternative (i) holds} Suppose that $\HH$ is transitive. If
    there is a growth symmetric $\GG$-orbit in $U_{0,\text{\rm d}}$,
    then the alternative~{\rm(}\ref{i: all orbits in U_0,d have
      equi-equivalent growth}\/{\rm)} of Theorem~\ref{t:orbit growth 2}
    holds.
                          
  \item\label{i: if the alternative (i) holds, then all orbits in U_0
      are equi-growth symmetric} Suppose that $\HH$ is minimal. If the
    alternative~{\rm(}\ref{i: all orbits in U_0,d have equi-equivalent
      growth}\/{\rm)} of Theorem~\ref{t:orbit growth 2} holds, then all
    $\GG$-orbits in $U_0$ are equi-growth symmetric.

  \end{enumerate}
\end{thm}

\begin{proof}
  Suppose that $\GG$ is transitive and that $\GG(x)$ is growth
  symmetric for some $x\in U_{0,\text{\rm d}}$. Then
  $\GG(x)\times\GG(x)\subset Y(a,b,c)$ for some $a,b\ge1$ and $c\ge0$,
  giving $U_0\times U_0=Y(a',b',c)$ for some $a',b'\ge1$ by
  Lemma~\ref{l:Y(a,b,c)}.
  
  Now, assume that $\GG$ is minimal and that all $\GG$-orbits in $U_0$
  have equi-equivalent growth. This means that
  $\Int_0(Y(a,b,c))\ne\emptyset$ for some $a,b,c\in\Z^+$ according to
  the proof of Theorem~\ref{t:orbit growth 2}; i.e., there are
  non-empty open subsets $V$ and $W$ of $U_0$ such that $V\times
  W\subset Y(a,b,c)$. Since $\GG$ is minimal, the intersections
  $\OO\cap V$ and $\OO\cap W$ are equi-nets in the $\GG$-orbits $\OO$
  in $U_0$ by Proposition~\ref{p: recurrent finite symmetric family of
    generators}. So the $\GG$-orbits in $U_0$ is equi-growth symmetric
  by Remark~\ref{r: growth symmetry}-\eqref{i: equi-domination gives
    growth symmetry}.
\end{proof}

\subsection{Some growth classes of the orbits}

\begin{thm}\label{t: orbit liminf ...}
  	Let \((Z,\HH,U,\GG,E)\) satisfy Hypothesis~\ref{h: hyp part III} and suppose that $\GG$ is transitive. Then there are $a_1,a_3\in[1,\infty]$, $a_2,a_4\in[0,\infty)$ and $p\ge1$ such that
 		\begin{alignat*}{2}
    			\limsup_{r\to\infty}\frac{\log v(x,r)}{\log r}&=a_1\;,&\qquad
    			a_2\le\liminf_{r\to\infty}\frac{\log v(x,r)}{r}&\le pa_2\;,\\
			\liminf_{r\to\infty}\frac{\log v(x,r)}{\log r}&=a_3\;,&\qquad
    			\limsup_{r\to\infty}\frac{\log v(x,r)}{r}&=a_4
		\end{alignat*}
	for residually many points $x$ in $U$. Moreover
		\[
			\liminf_{r\to\infty}\frac{\log v(x,r)}{\log r}\ge a_3\;,\qquad
    			\limsup_{r\to\infty}\frac{\log v(x,r)}{r}\le a_4
		\]
	for all $x\in U_{0,\text{\rm d}}$.
\end{thm}

\begin{proof}
    	For $a,b>0$, let $Y_1(a,b)$, $Y_2(a,b)$, $Y_3(a,b)$ and $Y_4(a,b)$ be the sets of points $x\in U_0$ that satisfy the following respective conditions:
  		\begin{alignat*}{2}
    			\sup_{r\ge b}\frac{\log v(x,r)}{\log r}&\leq a\;,&\qquad
    			\inf_{r\ge b}\frac{\log v(x,r)}{r}&\ge a\;,\\
    			\inf_{r\ge b}\frac{\log v(x,r)}{\log r}&<a\;,&\qquad 
			\sup_{r\ge b}\frac{\log v(x,r)}{r}&>a\;.
  		\end{alignat*}
 	Then
  		\begin{equation}\label{Y_1(a,b) subset Y_1(a',b'), ...}
    			a\le a'\ \&\ b\le b'\Longrightarrow\left\{
      				\begin{alignedat}{2}
        					Y_1(a,b)&\subset Y_1(a',b')\;,&\quad Y_2(a',b)&\subset Y_2(a,b')\;,\\
        					Y_3(a,b')&\subset Y_3(a',b)\;,&\quad Y_4(a',b')&\subset Y_4(a,b)\;,
      				\end{alignedat}\right.
  		\end{equation}

  	\begin{claim}\label{cl: Y,Z,Y',Z'}
    		The following properties hold:
    			\begin{enumerate}[{\rm(}i\/{\rm)}]

    				\item\label{i: Y_1(a,b) is closed} $Y_1(a,b)$ is closed in $U_0$ for all $a,b>0$.
  
    				\item\label{i: Cl_0(Y_2(a,b)) subset Y_2(a/p,b')}  There is some $p\ge1$ such that,  for all $a,b>0$, there is some $b'\ge b$ so that $\Cl_0(Y_2(a,b))\subset Y_2(a/p,b')$.
  
    				\item\label{i: bigcap_bY_3(a,b) subset bigcap_a'>a Int_0(bigcap_bY_3(a',b))} $\bigcap_bY_3(a,b)\subset\bigcap_{a'>a}\Int_0(\bigcap_bY_3(a',b))$ for all $a>0$.

    				\item\label{i: Y_4(a,b) is open} $Y_4(a,b)$ is open in $U_0$ for all $a,b>0$.

    			\end{enumerate}
  	\end{claim}

  	Let $x\in\Cl_0(Y_1(a,b))$. With the notation of Proposition~\ref{p: orbit Reeb with x in U_0}, for each $r\ge b$, if $y\in Y_1(a,b)\cap V(x,r)$, then
  		\[
  			\frac{\log v(x,r)}{\log r}=\frac{\log v(x,\lfloor r\rfloor)}{\log r}
			\le\frac{\log v(y,\lfloor r\rfloor)}{\log r}=\frac{\log v(y,r)}{\log r}\le a
  		\]
  	because $\phi_{x,y,\lfloor r\rfloor}$ is injective by Proposition~\ref{p: orbit Reeb with x in U_0}-\eqref{i: phi_x,y,r are equi-bi-Lipschitz}. Therefore $x\in Y_1(a,b)$, confirming Claim~\ref{cl: Y,Z,Y',Z'}-\eqref{i: Y_1(a,b) is closed}.

  	Let $x\in\Cl_0(Y_2(a,b))$, and let $\Lambda:=\Lambda_{|E|,2RC}$ be defined by~\eqref{Lambda_K,r}. If
		\[
			r\ge b':=\max\left\{Cb,CR,\frac{2C\,\log\Lambda}{a}\right\}
		\]
	and $y\in Y_2(a,b)\cap V(x,r)$, then
  		\begin{multline*}
  			\frac{\log v(x,r)}{r}=\frac{\log v(x,\lfloor r\rfloor)}{r}
			\ge\frac{\log(v(y,\lfloor r\rfloor/C)/\Lambda)}{r}
			=\frac{\log v(y,\lfloor r\rfloor/C)-\log\Lambda}{r}\\
			=\frac{\log v(y,r/C)-\log\Lambda}{r}
			\ge\frac{\log v(y,r/C)}{r}-\frac{a}{2C}
			\ge\frac{a}{C}-\frac{a}{2C}=\frac{a}{2C}\;,
  		\end{multline*}
  	using that $\ol B(y,\lfloor r\rfloor/C)\cap\im\phi_{x,y,\lfloor r\rfloor}$ is a $2RC$-net in $\ol B(y,\lfloor r\rfloor/C)$ according to Proposition~\ref{p: orbit Reeb with x in U_0}-\eqref{i: ol B_E(y,r/C) cap im phi_x,y,r is a 2CR-net}, and $|\ol B(z,2RC)|\le \Lambda$ for all $z\in\ol B(y,r/C)$ by~\eqref{|ol B(x,r)|}. Thus $x\in Y_2(a/p,b')$ for $p=2C$, which confirms Claim~\ref{cl: Y,Z,Y',Z'}-\eqref{i: Cl_0(Y_2(a,b)) subset Y_2(a/p,b')}.

  	Let $x\in\bigcap_bY_3(a,b)$. Given $a'>a$, for any choice of $\alpha\in(a/a',1)$, let
		\[
			b\ge\max\left\{CR,C^\frac{1}{1-\alpha},C\Lambda^{\frac{1}{a'-a/\alpha}}\right\}\;,
		\]
	where $\Lambda$ is defined like in the proof of Claim~\ref{cl: Y,Z,Y',Z'}-\eqref{i: Cl_0(Y_2(a,b)) subset Y_2(a/p,b')}. Then $x\in Y_3(a,b)$, which means that there is some $r\ge b$ so that $\log v(x,r)/\log r<a$. If $y\in V(x,Cr)$, then
  		\begin{multline*}
  			\frac{\log v(y,r/C)}{\log(r/C)}=\frac{\log v(y,\lfloor r\rfloor/C)}{\log(r/C)}
			\le\frac{\log(v(x,\lfloor r\rfloor)\,\Lambda)}{\log(r/C)}
			=\frac{\log(v(x,r)\,\Lambda)}{\log(r/C)}\\
			\le\frac{\log v(x,r)+\log\Lambda}{\log r-\log C}
			<\frac{a}{1-\frac{\log C}{\log r}}+\frac{\log\Lambda}{\log r-\log C}
			<\frac{a}{\alpha}+a'-\frac{a}{\alpha}=a'\;,
  		\end{multline*}
  	like in the proof of Claim~\ref{cl: Y,Z,Y',Z'}-\eqref{i: Cl_0(Y_2(a,b)) subset Y_2(a/p,b')}. It follows that $V(x,Cr)\cap U_0\subset Y_3(a',b)$. Since this holds for all $b$ large enough, we get $V(x,Cr)\cap U_0\subset \bigcap_bY_3(a',b)$ by~\eqref{Y_1(a,b) subset Y_1(a',b'), ...}, and thus $x\in\Int_0(\bigcap_bY_3(a',b))$, showing Claim~\ref{cl: Y,Z,Y',Z'}-\eqref{i: bigcap_bY_3(a,b) subset bigcap_a'>a Int_0(bigcap_bY_3(a',b))}. 

  	For any $x\in Y_4(a,b)$, there is some integer $r\ge b$ such that $\log v(x,r)/r>a$. So $\log v(y,r)/r>a$ for any $y\in V(x,r)$ since $\phi_{x,y,r}$ is injective (Proposition~\ref{p: orbit Reeb with x in U_0}-\eqref{i: phi_x,y,r are equi-bi-Lipschitz}), giving $V(x,r)\cap U_0\subset Y_4(a,b)$. Therefore $Y_4(a,b)$ is open in $U_0$. This confirms Claim~\ref{cl: Y,Z,Y',Z'}-\eqref{i: Y_4(a,b) is open}.
  
  	For $a\in[0,\infty]$, let 
  		\begin{alignat*}{2}
    			Y_1(a)&=\bigcap_{\alpha>a}\bigcup_bY_1(\alpha,b)\;,&\qquad
    			Y_2(a)&=\bigcap_{\alpha<a}\bigcup_bY_2(\alpha,b)\;,\\
    			Y_3(a)&=\bigcap_{\alpha>a}\bigcap_bY_3(\alpha,b)\;,&\qquad
    			Y_4(a)&=\bigcap_{\alpha<a}\bigcap_bY_4(\alpha,b)\;.
  		\end{alignat*}
	It is clear that these are the sets of points $x\in U_0$ that respectively satisfy
  		\begin{alignat*}{2}
    			\limsup_{r\to\infty}\frac{\log v(x,r)}{\log r}\leq&a\;,&\quad
    			\liminf_{r\to\infty}\frac{\log v(x,r)}{r}&\geq a\;,\\
    			\liminf_{r\to\infty}\frac{\log v(x,r)}{\log r}\leq&a\;,&\quad
    			\limsup_{r\to\infty}\frac{\log v(x,r)}{r}&\ge a\;.
  		\end{alignat*}
	Observe also that
  		\begin{equation}\label{Y_1(a) subset Y_1(a'), ...}
		 	a\le a'\Longrightarrow\left\{
		 		\begin{alignedat}{2}
					Y_1(a)&\subset Y_1(a')\;,&\quad Y_2(a)&\supset Y_2(a')\;,\\ 
  					Y_3(a)&\subset Y_3(a')\;,&\quad Y_4(a)&\supset Y_4(a')\;.
				\end{alignedat}\right.
  		\end{equation}
  	We get the same sets $Y_1(a)$, $Y_2(a)$, $Y_3(a)$ and $Y_4(a)$ above if the condition that \(a',b\in\Q\) is added in their definitions. So, by Claim~\ref{cl: Y,Z,Y',Z'}-\eqref{i: Y_1(a,b) is closed},\eqref{i: bigcap_bY_3(a,b) subset bigcap_a'>a Int_0(bigcap_bY_3(a',b))},\eqref{i: Y_4(a,b) is open}, the sets $Y_1(a)$, $Y_3(a)$ and $Y_4(a)$ are Borel in $U_0$, and therefore they are Baire subsets of $U$. However the same kind of argument, using Claim~\ref{cl: Y,Z,Y',Z'}-\eqref{i: Cl_0(Y_2(a,b)) subset Y_2(a/p,b')}, does not apply to $Y_2(a)$. Thus consider also the set
		\[
			Y'_2(a)=\bigcap_{\alpha<a}\bigcup_b\Cl_0(Y_2(a,b))\;,
		\]
	which are Borel in $U_0$ and Baire in $U$. Obviously,
		\begin{equation}\label{Y'_2(a) supset Y'_2(a')}
		 	a\le a'\Longrightarrow Y'_2(a)\supset Y'_2(a')\;,
		\end{equation}
	and, by Claim~\ref{cl: Y,Z,Y',Z'}-\eqref{i: Cl_0(Y_2(a,b)) subset Y_2(a/p,b')},
		\begin{equation}\label{Y'_2(pa) subset Y_2(a)}
			Y'_2(pa)\subset Y_2(a)\;.
		\end{equation}
	The sets $Y_1(a)$ and $Y_3(a)$ are $\GG$-saturated for all $a\in[0,\infty]$ by Remark~\ref{r: bounded geometry and growth}-\eqref{i: limsup_r to infty frac log v_Gamma(x,r) log r}, and we have
		\begin{equation}\label{GG(Y_2(qa)) subset Y_2(a), ...}
			\GG(Y_2(qa))\subset Y_2(a)\;,\qquad\GG(Y_4(qa))\subset Y_4(a)
		\end{equation}
	for all $a\in[0,\infty]$ and $q>1$ by Remark~\ref{r: bounded geometry and growth}-\eqref{i: liminf_r to infty frac log v_Gamma(x,r) r} and Example~\ref{ex: coarse bounded geometry}-\eqref{i: graph of coarse bd geom} since $\Lambda_{|E|,0}=1$ (see~\eqref{Lambda_K,r}). So, by~\eqref{Y'_2(pa) subset Y_2(a)},
		\begin{equation}\label{GG(Y'_2(pqa)) subset Y_2(a)}
			\GG(Y'_2(pqa))\subset Y_2(a)\;.
		\end{equation}
	
	\begin{claim}\label{cl: Y_1(a), ... is either residual or meager}
		Each of the sets $Y_1(a)$, $\GG(Y'_2(a))$, $Y_3(a)$ and $\GG(Y_4(a))$ is either residual or meager in $U$.
	\end{claim}
	
	This assertion is a consequence of Theorem~\ref{t: 01law for pseudogroups} because the stated sets are Baire and $\GG$-saturated in $U$.

    	We obviously have
  		\begin{equation}\label{Y_1(infty)=Y_2(0)=Y_3(infty)=Y_4(0)=U_0}
  			Y_1(\infty)=Y_2(0)=Y_3(\infty)=Y_4(0)=U_0\;.
  		\end{equation}
  	Then let
        \begin{align*}
          a_1&=\inf\{\,a\in[0,\infty]\mid \text{$Y_1(a)$ is residual in $U$}\,\}\;,\\
          a_2&=\sup\{\,a\in[0,\infty]\mid \text{$Y_2(a)$ is residual in $U$}\,\}\;,\\
          a_3&=\inf\{\,a\in[0,\infty]\mid \text{$Y_3(a)$ is residual in $U$}\,\}\;,\\
          a_4&=\sup\{\,a\in[0,\infty]\mid \text{$Y_4(a)$ is residual in $U$}\,\}\;.
  		\end{align*}
                Since the growth type of an infinite connected graph
                of finite type is at least linear and at most
                exponential, the sets $Y_1(a)$ and $Y_3(a)$ are the
                union of finite orbits in $U_0$ for all $0\leq a <1$,
                and $Y_2(\infty)=Y_4(\infty)=\emptyset$. So
                $a_1,a_3\ge1$ by Corollary~\ref{c: the union of dense
                  orbits with trivial germ groups is residual}, and
                $a_2,a_4<\infty$. By~\eqref{Y_1(a) subset Y_1(a'),
                  ...}, the sets $Y_1(a)$, $Y_2(a)$, $Y_3(a)$ or
                $Y_4(a)$ are residual in $U$ if $a>a_1$, $a<a_2$,
                $a>a_3$ or $a<a_4$, respectively. Hence,
                by~\eqref{Y_1(infty)=Y_2(0)=Y_3(infty)=Y_4(0)=U_0},
                and because
  		\begin{alignat*}{2}
    			Y_1(a_1)&=\bigcap_{a_1<a<\infty}Y_1(a)&\quad\text{if}\quad a_1&<\infty\;,\\
    			Y_2(a_2)&=\bigcap_{0<a<a_2}Y_2(a)&\quad\text{if}\quad a_2&>0\;,\\
    			Y_3(a_3)&=\bigcap_{a_1<a<\infty}Y_3(a)&\quad\text{if}\quad a_3&<\infty\;,\\
    			Y_4(a_4)&=\bigcap_{0<a<a_4}Y_4(a)&\quad\text{if}\quad a_4&>0\;,
  		\end{alignat*}
                where $a$ can be taken in $\Q$, we get that
                $Y_1(a_1)$, $Y_2(a_2)$, $Y_3(a_3)$ and $Y_4(a_4)$ are
                residual in $U$. On the other hand, by~\eqref{Y_1(a)
                  subset Y_1(a'), ...},~\eqref{Y'_2(a) supset
                  Y'_2(a')},~\eqref{GG(Y_2(qa)) subset Y_2(a),
                  ...},~\eqref{GG(Y'_2(pqa)) subset Y_2(a)} and
                Claim~\ref{cl: Y_1(a), ... is either residual or
                  meager}, the sets $Y_1(a)$, $Y_2(a)$, $Y_3(a)$ or
                $Y_4(a)$ are meager in $U$ if $a<a_1$, $a>pa_2$,
                $a<a_3$ or $a>a_4$, respectively. So the following
                unions are meager in $U$ because they do not change if
                $a$ is taken in $\Q$:
  		\[
  			\bigcup_{0\leq a<a_1}Y_1(a)\;,\qquad\bigcup_{pa_2<a<\infty}Y_2(a)\;,\qquad
  			\bigcup_{0\leq a<a_3}Y_3(a)\;,\qquad\bigcup_{a_4<a<\infty}Y_4(a)\;.
  		\]
                Thus the following sets are residual in $U$:
		\begin{gather*}
			Y_1(a_1)\sm\bigcup_{0\leq a<a_1}Y_1(a)\;,\qquad
			Y_2(a_2)\sm\bigcup_{pa_2<a<\infty}Y_2(a)\;,\\
			Y_3(a_3)\sm\bigcup_{0\leq a<a_3/p'}Y_3(a)\;,\qquad
			Y_4(a_4)\sm\bigcup_{a_4<a<\infty}Y_4(a)\;.
		\end{gather*}
                These are the sets described by the first group of
                equalities and inequalities of the statement.
	
                Now, let us prove the last two inequalities of the
                statement on $U_{0,\text{\rm d}}$. By Claim~\ref{cl:
                  Y,Z,Y',Z'}-\eqref{i: bigcap_bY_3(a,b) subset
                  bigcap_a'>a Int_0(bigcap_bY_3(a',b))},
		\[
			Y_3(a)=\bigcap_{a'>a}\Int_0(\bigcap_bY_3(a,b))
		\]
	for all $a\in[0,\infty]$. So
		\[
			U_0\sm Y_3(a)=U_0\sm\bigcap_{a'>a}\Int_0(\bigcap_bY_3(a,b))
			=\bigcup_{a'>a}(U_0\sm\Int_0(\bigcap_bY_3(a,b)))\;.
		\]
                Here, we can take $a'$ in $\Q$, and thus this
                expression is a countable union of closed subsets of
                $U_0$. Moreover we know that $U_0\sm Y_3(a)$ is
                residual in $U_0$ for $a<a_3$, and therefore there is
                some $b$ such that $U_0\sm\Int_0(\bigcap_bY_3(a,b))$
                has nonempty interior, obtaining that any $\GG$-orbit
                in $U_{0,\text{\rm d}}$ meets $U_0\sm Y_3(a)$. Hence
                $U_{0,\text{\rm d}}\subset U_0\sm Y_3(a)$ for all
                $a<a_3$ because $Y_3(a)$ is $\GG$-saturated, obtaining
                that
                $U_{0,\text{\rm d}}\subset
                U_0\sm\bigcup_{a<a_3}Y_3(a)$.
	
	Finally,
		\[
			U_0\sm Y_4(a)=U_0\sm\bigcap_{a'<a}\bigcap_bY_4(a',b)
			=\bigcup_{a'<a}\bigcup_b(U_0\sm Y_4(a',b))\;,
		\]
	where $a'$ and $b$ can be taken in $\Q$. Then, by Claim~\ref{cl: Y,Z,Y',Z'}-\eqref{i: Y_4(a,b) is open}, this expression is a countable union of closed subsets of $U_0$. Furthermore $U_0\sm Y_4(a)$ is residual for each $a>a_4$, obtaining that $U_0\sm Y_4(a',b)$ has nonempty interior in $U_0$ for some $a'<a$ and $b$. So any $\GG$-orbit in $U_{0,\text{\rm d}}$ meets $U_0\sm Y_4(a)$ for all $a>a_4$, and therefore
		\[
			U_{0,\text{\rm d}}\subset\GG(U_0\sm Y_4(a))\subset U_0\sm\GG(Y_4(a'))
			\subset U_0\sm Y_4(a')
		\]
	if $a'>a>a_4$ by~\eqref{GG(Y_2(qa)) subset Y_2(a), ...}. Thus $U_{0,\text{\rm d}}\subset U_0\sm\bigcup_{a>a_4}Y_4(a)$.	
\end{proof}

\begin{cor}\label{c: orbit polynomial exponential growth}
  	Let \((Z,\HH,U,\GG,E)\) satisfy Hypothesis~\ref{h: hyp part III} and suppose that $\GG$ is transitive. Then each of the following sets is either meager or residual in $U$:
  		\begin{enumerate}[{\rm(}i\/{\rm)}]

  			\item\label{i: the union of orbits in U_0 with polynomial growth} the union of $\GG$-orbits in $U_0$ with polynomial growth;

  			\item\label{i: the union of orbits in U_0 with exponential growth} the union of $\GG$-orbits in $U_0$ with exponential growth;
    
   			\item\label{i: the union of orbits in U_0 with quasi-polynomial growth} the union of $\GG$-orbits in $U_0$ with quasi-polynomial growth;

  			\item\label{i: the union of orbits in U_0 with quasi-exponential growth} the union of $\GG$-orbits in $U_0$ with quasi-exponential growth; and

  			\item\label{i: the union of orbits in U_0 with pseudo-quasi-polynomial growth}  the union of $\GG$-orbits in $U_0$ with pseudo-quasi-polynomial growth.

  		\end{enumerate}
  	Moreover, 
		\begin{enumerate}[{\rm(}a\/{\rm)}]
		
			\item if the set~{\rm(}\ref{i: the union of orbits in U_0 with quasi-polynomial growth}\/{\rm)} is residual in $U$, then it contains $X_{0,\text{\rm d}}$; and,
			
			\item if one of the sets~{\rm(}\ref{i: the union of orbits in U_0 with quasi-exponential growth}\/{\rm)} or~{\rm(}\ref{i: the union of orbits in U_0 with pseudo-quasi-polynomial growth}\/{\rm)} is meager in $U$, then it does not meet $U_{0,\text{\rm d}}$.
			
		\end{enumerate}
\end{cor}

\begin{rem}
	The properties used in Corollary~\ref{c: orbit polynomial exponential growth} depend only on the growth type of the $\GG$-orbits by Remark~\ref{r: bounded geometry and growth}-\eqref{i: polynomial, ..., metric space}.
\end{rem}

\section{Amenable orbits}

Like in the above section, since the $\GG$-orbits are
equi-quasi-lattices in themselves, that they are (equi-) amenable
means that they are (equi-) F{\o}lner.

\begin{thm}\label{t: orbit folner}
Let \((Z,\HH,U,\GG,E)\) satisfy Hypothesis~\ref{h: hyp part III}. Then the following properties hold:
  \begin{enumerate}[{\rm(}i\/{\rm)}]
		
  \item\label{i: if some orbit in U_0 is amenable, then all orbits in
      U_0,d are equi-amenable} If $\GG$ is transitive and some
    $\GG$-orbit in $U_0$ is amenable, then all $\GG$-orbits in
    $U_{0,\text{\rm d}}$ are equi-amenable.
			
  \item\label{i: if some orbit in U_0 is amenable, then all orbits in
      U_0 are jointly amenably symmetric} If $\GG$ is minimal and some
    $\GG$-orbit in $U_0$ is amenable, then all $\GG$-orbits in $U_0$
    are jointly amenably symmetric.
			
  \end{enumerate}
\end{thm}
 
\begin{proof}
  Consider the notation of Proposition~\ref{p: orbit Reeb with x in
    U_0}.  Suppose that $\GG(x)$ is F{\o}lner for some $x\in U_0$. Let
  $S_n$ be a F{\o}lner sequence for $\GG(x)$, and let $r\in\N$. For
  each $n$, let $u_n\ge C(r+4CR)$ be an  integer such that
  \begin{equation}\label{S_n subset ol B(x,s_n)}
    S_n\subset\ol B(x,u_n/C-r-4CR)\;.
  \end{equation}
  For $y\in V(x,u_n)$, let $\phi_n$ denote $\phi_{x,y,u_n}$, and let
  $S_{y,n}=\Pen(\phi_n(S_n),2CR)$. By~\eqref{S_n subset ol B(x,s_n)}
  and Proposition~\ref{p: orbit Reeb with x in U_0}-\eqref{i:
    phi_x,y,r is non-expanding},
  \begin{equation}\label{S_y,n subset ol B(y,2CR+s_n)}
    S_{y,n}\subset\ol B(y,u_n/C-r-2CR)\;.
  \end{equation}
  Moreover, by Proposition~\ref{p: orbit Reeb with x in U_0}-\eqref{i:
    phi_x,y,r are equi-bi-Lipschitz},
  \begin{equation}\label{|S_y,n|}
    |S_n|\le|S_{y,n}|\;.
  \end{equation}

\begin{claim}\label{cl: partial_r S_n,y}
  $\partial_r S_{y,n}\subset\Pen(\phi_n(\partial_{C(r+4CR)}S_n),2CR)$.
\end{claim}

Observe that the right hand side of this inclusion is well defined
by~\eqref{S_n subset ol B(x,s_n)}. To prove this inclusion, take first
any $z\in\partial_rS_{y,n}\sm S_{y,n}$. Then $d(z,\phi_n(S_n))>2CR$
and there is some $z_0\in S_{n,y}$ such that $d(z,z_0)\le r$. Thus
there is some $z'_0\in\phi_n(S_n)$ so that $d(z_0,z'_0)\le 2CR$,
obtaining $d(z,z'_0)\le 2CR+r$ by the triangle
inequality. By~\eqref{S_y,n subset ol B(y,2CR+s_n)},
$$
d(y,z)\le d(y,z_0)+d(z_0,z)\le\frac{u_n}{C}-r-2CR+r=\frac{u_n}{C}-2CR\;.
$$
By Proposition~\ref{p: orbit Reeb with x in U_0}-\eqref{i: ol
  B_E(y,r/C) cap im phi_x,y,r is a 2CR-net}, it follows that there is
some $z'\in\ol B(y,u_n/C)\cap\im\phi_n$ such that $d(z,z')\le
2CR$. Then $d(z'_0,z')\le4CR+r$ by the triangle inequality. We have
$$
d(z',\phi_n(S_n))\ge d(z,\phi_n(S_n))-d(z',z)>2CR-2CR=0\;;
$$
i.e., $z'\notin\phi_n(S_n)$.  Let $\bar z'_0=\phi_n^{-1}(z'_0)\in S_n$
and $\bar z'=\phi_n^{-1}(z')\in\GG(x)\sm S_n$. By Proposition~\ref{p:
  orbit Reeb with x in U_0}-\eqref{i: phi_x,y,r are
  equi-bi-Lipschitz},
$$
d(\bar z'_0,\bar z')\le C\,d(z'_0,z')\le C(4CR+r)\;.
$$
Thus 
$$
\bar z'\in\Pen(S_n,C(4CR+r))\sm S_n\subset\partial_{C(4CR+r)}S_n\;.
$$
So $z'\in\phi_n(\partial_{C(4CR+r)}S_n)$, and therefore $z\in\Pen(\phi_n(\partial_{C(4CR+r)}S_n),2CR)$.

Now take any $z\in\partial_rS_{y,n}\cap S_{y,n}$. Then there are
points $z_0\in\phi_n(S_n)$ and $z_1\in\GG(y)\sm S_{y,n}$ such that
$d(z,z_0)\le 2CR$ and $d(z,z_1)\le r$, obtaining $d(z_0,z_1)\le 2CR+r$
by the triangle inequality. We have $d(z_1,\phi_n(S_n))>2CR$ because
$z_1\notin S_{y,n}$. By~\eqref{S_y,n subset ol B(y,2CR+s_n)},
$$
d(y,z_1)\le d(y,z_0)+d(z_0,z_1)\le\frac{u_n}{C}-r-2CR+2CR+r=\frac{u_n}{C}\;.
$$
By Proposition~\ref{p: orbit Reeb with x in U_0}-\eqref{i: ol
  B_E(y,r/C) cap im phi_x,y,r is a 2CR-net}, it follows that there is
some $z'_1\in\ol B(y,u_n/C)\cap\im\phi_n$ such that $d(z_1,z'_1)\le
2CR$. Then $d(z_0,z'_1)\le4CR+r$ by the triangle inequality. We have
$$
d(z'_1,\phi_n(S_n))\ge d(z_1,\phi_n(S_n))-d(z'_1,z_1)>2CR-2CR=0\;;
$$
i.e., $z'_1\notin\phi_n(S_n)$.  Let $\bar z_0=\phi_n^{-1}(z'_0)\in
S_n$ and $\bar z'_1=\phi_n^{-1}(z'_1)\in\GG(x)\sm S_n$. By
Proposition~\ref{p: orbit Reeb with x in U_0}-\eqref{i: phi_x,y,r are
  equi-bi-Lipschitz},
$$
d(\bar z_0,\bar z'_1)\le C\,d(z_0,z'_1)\le C(4CR+r)\;.
$$
Thus 
$$
\bar z_0\in\Pen(\GG(x)\sm S_n,C(4CR+r))\cap S_n\subset\partial_{C(4CR+r)}S_n\;.
$$
So $z_0\in\phi_n(\partial_{C(4CR+r)}S_n)$, and therefore
$z\in\Pen(\phi_n(\partial_{C(4CR+r)}S_n),2CR)$. This completes the
proof of Claim~\ref{cl: partial_r S_n,y}.

By Claim~\ref{cl: partial_r S_n,y} and Proposition~\ref{p: orbit Reeb
  with x in U_0}-\eqref{i: phi_x,y,r are equi-bi-Lipschitz},
\begin{multline}\label{|partial_rS_y,n|}
|\partial_rS_{y,n}|\le|\Pen(\phi_n(\partial_{C(r+4CR)}S_n),2CR)|\\
\le\Lambda_{|E|,2RC}\,|\phi_n(\partial_{C(r+4CR)}S_n)|
=\Lambda_{|E|,2RC}\,|\partial_{C(r+4CR)}S_n|\;;
\end{multline}
in particular, by~\eqref{|partial_rS|},
\begin{equation}\label{|partial S_y,n|}
|\partial S_{y,n}|\le\Lambda_{|E|,2RC}\,|\partial_{C(1+4CR)}S_n|
\le\Lambda_{|E|,2RC}\,\Lambda_{|E|,C(1+4CR)-1}\,|\partial S_n|\;.
\end{equation}

Assume that $\GG$ is transitive. For every dense $\GG$-orbit $\OO$ and
all $n$, there is some $y_n\in\OO\cap
V(x,u_n)$. By~\eqref{|S_y,n|},~\eqref{|partial S_y,n|}
and~\eqref{|partial_rS|}, the sets $S_{y_n,n}$ form a F{\o}lner
sequence in $\OO$, and therefore $\OO$ is amenable.

Assume that every dense $\GG$-orbit is unbounded, for otherwise $\GG$
would have only one orbit. Then, for every dense $\GG$-orbit $\OO$,
write
$$
\OO\cap V(x,u_n)=\{\,y(\OO,m,n)\mid m\in\N\,\}\;,
$$
and let $S_{\OO,m,n}=S_{y(\OO,m,n),n}$. 

Suppose that $\OO\subset U_0$. Given $m$, $n$ and $t\in\N$, and any
$\GG$-orbit $\OO'$ in $U_{0,\text{\rm d}}$, let $v_n=u_n+C(2CR+t)$ and
$$
\NN_{\OO,\OO',m,n,t}=\{\,m'\in\N\mid y(\OO',m',n)\in
V(y(\OO,m,n),v_n)\cap V(x,u_n)\,\}\;,
$$
which is a nonempty set. Take a F{\o}lner sequence $X_n$ of $\OO$ such
that $S_{\OO,m,n}\subset X_n\subset\Pen(S_{\OO,m,n},t)$. Thus,
by~\eqref{S_y,n subset ol B(y,2CR+s_n)},
\begin{align*}
  X_n&\subset\ol B(y(\OO,m,n),u_n/C-r-2CR+t)\\
  &=\ol B(y(\OO,m,n),v_n/C-r-4CR)\;.
\end{align*}
For the sake of simplicity, given $n$ and any $m'\in\NN_{\OO',m,n}$,
write $y=y(\OO,m,n)$ and $y'=y(\OO',m',n)$. Let
$$
Y_{\OO,\OO',m',n}=\Pen(\phi_{y,y',v_n}(X_n),2CR)\;.
$$
By Remark~\ref{r: orbit Reeb with x in U_0}-\eqref{i: phi_x,z,s = phi_y,z,s circ phi_x,y,r},
\begin{multline*}
  S_{\OO',m',n}=\Pen(\phi_{x,y',u_n}(S_n),2CR)=\Pen(\phi_{y,y',v_n}\circ\phi_{x,y,u_n}(S_n),2CR)\\
  \subset\Pen(\phi_{y,y',v_n}(S_{\OO,m,n}),2CR)
  \subset\Pen(\phi_{y,y',v_n}(X_n),2CR)=Y_{\OO,\OO',m',n}\;,
\end{multline*}
and, by~\eqref{Pen(S,r+s), graph}, Proposition~\ref{p: orbit Reeb with
  x in U_0}-\eqref{i: phi_x,y,r is non-expanding} and Remark~\ref{r:
  orbit Reeb with x in U_0},
\begin{multline*}
  Y_{\OO,\OO',m',n}=\Pen(\phi_{y,y',v_n}(X_n),2CR)
  \subset\Pen(\phi_{y,y',v_n}(\Pen(S_{\OO,m,n},t)),2CR)\\
  =\Pen(\phi_{y,y',v_n}(\Pen(\phi_{x,y,u_n}(S_n),2CR+t)),2CR)\\
  \subset\Pen(\phi_{y,y',v_n}\circ\phi_{x,y,u_n}(S_n),4CR+t)\\
  =\Pen(\phi_{x,y',u_n}(S_n),4CR+t)=\Pen(S_{\OO',m',n},2CR+t)\;.
\end{multline*}
Furthermore, applying~\eqref{|S_y,n|},~\eqref{|partial_rS_y,n|} and~\eqref{|partial S_y,n|} to $X_n$ and $Y_{\OO,\OO',m',n}$, we get
\begin{align*}
  |\partial_rY_{\OO,\OO',m',n}|/|Y_{\OO,\OO',m',n}|
  &\le\Lambda_{|E|,2RC}\,|\partial_{C(r+4CR)}X_n|/|X_n|\;,\\
  |\partial Y_{\OO,\OO',m',n}|/|Y_{\OO,\OO',m',n}|
  &\le\Lambda_{|E|,2RC}\,\Lambda_{|E|,C(1+4CR)-1}\,|\partial X_n|/|X_n|\;.
\end{align*}
This shows that the dense $\GG$-orbits in $U_0$ are equi-amenable.

Next assume that $\GG$ is minimal. Then, by Proposition~\ref{p:
  recurrent finite symmetric family of generators}, for every $\OO$
and $n$, there is some $L_n\in\N$ so that $\OO\cap V(x,u_n)$ is an
$L_n$-net of $\OO$ for any $\GG$-orbit $\OO$. Then
$\bigcup_mS_{\OO,m,n}$ is an $(L_n+u_n/C-r-2CR)$-net in $\OO$
by~\eqref{S_y,n subset ol B(y,2CR+s_n)}. Similarly, for every $\OO$,
$m$, $n$ and $t\in\N$, there is some $L_{\OO,m,n,t}\in\N$ so that
$\OO'\cap V(y(\OO,m,n),v_n)\cap V(x,u_n)$ is an $L_{\OO,m,n,t}$-net of
$\OO'$ for any $\GG$-orbit $\OO'$. Then
$\bigcup_{m'\in\NN_{\OO,\OO',m,n,t}}S_{\OO',m',n,t}$ is an
$(L_{\OO,m,n,t}+u_n/C-r-2CR)$-net in $\OO'$ by~\eqref{S_y,n subset ol
  B(y,2CR+s_n)}. This shows that, when \(\GG\) is minimal, all
$\GG$-orbits in $U_0$ are jointly amenably symmetric.
\end{proof}

\section{Asymptotic dimension of the orbits}\label{s: asdim orbits}

For $r\in\Z^+\cup\{\infty\}$ and $R,D\in\Z^+$ and $n\in\N$, let
$Y(r,R,D,n)$ be the set of elements $x\in U_0$ for which there exists
families $\VV_0,\dots,\VV_n$ of subsets of the ball\footnote{Recall
  that $B(x,\infty)=\GG(x)$.} $B(x,r)$ such that:
\begin{enumerate}[(a)]

\item\label{i: diam V le D} $\diam V\le D$ for all $V\in\VV_i$;

\item\label{i: d(V,V') ge R} $d(V,V')\ge R$ if $V\ne V'$ in $\VV_i$;
  and
  
\item\label{i: bigcup_i=0^n VV_i covers B(x,r)} $\bigcup_{i=0}^n\VV_i$
  covers $B(x,r)$.

\end{enumerate}
We have
	\[
		r\ge r',\ R\ge R',\ D\le D'\;\Longrightarrow\;Y(r,R,D,n)\subset Y(r',R',D',n)\;;
	\]
in particular, 
	\[
		Y(R,D,n):=Y(\infty,R,D,n)\subset Y(r,R,D,n)
	\]
        for all $r\in\Z^+$. To see the above inclusion, for a family
        $\VV$ of subsets of $B(x,r')$, consider the family
        $\VV|_{B(x,r)}$ of intersections of the elements of $\VV$ with
        $B(x,r)$. Note that each set $Y(R,D,n)$ is
        saturated. Moreover, by Proposition~\ref{p: asdim},
	\begin{equation}\label{bigcap_R bigcup_D Y(R,D,n)}
		\bigcap_R \bigcup_D Y(R,D,n)=\{\,x\in U_0\mid\asdim\GG(x)\le n\,\}\;.
	\end{equation}

\begin{lemma}\label{l: Y(R,D,n)=bigcap_rY(r,R,D,n)}
	$Y(R,D,n)=\bigcap_rY(r,R,D,n)$.
\end{lemma}

\begin{proof}
  Let $x\in\bigcap_rY(r,R,D,n)$. Construct a graph with vertices the
  elements $(r,\VV_0,\dots,\VV_n)$, where $r\in\Z^+$ and
  $\VV_0,\dots,\VV_n$ are families of subsets of $B(x,r)$
  satisfying~(\ref{i: diam V le D})--(\ref{i: bigcup_i=0^n VV_i covers
    B(x,r)}) with these $x$ and $r$, and having an edge from a vertex
  $(r,\VV_0,\dots,\VV_n)$ to another vertex $(r+1,\WW_0,\dots,\WW_n)$
  if and only if $\WW_i|_{B(x,r)}=\VV_i$ for all
  $i\in\{0,\dots,n\}$. This graph is locally finite because $B(x,r)$
  is finite for all $r$. On the other hand, the fact that
  $x\in\bigcap_rY(r,R,D,n)$ implies that this graph has arbitrarily
  large rays. Therefore there is a sequence $r_k\uparrow\infty$, and,
  for each $k$, there are families $\VV_{k,0},\dots,\VV_{k,n}$ of
  subsets of $B(x, r_k)$ satisfying~(\ref{i: diam V le D})--(\ref{i:
    bigcup_i=0^n VV_i covers B(x,r)}) with $x$ and $r_k$, and such
  that, whenever $k<l$, $\VV_{k,i} = \VV_{l,i}|_{B(x,r_l)}$ for all
  $i\in\{0,\dots,n\}$.  Let $\VV_i$ be the family of unions
  $\bigcup_kV_k$ for increasing sequences of sets,
  $V_0\subset V_1\subset\cdots$, with $V_k\in\VV_{k,i}$ for all
  $k$. It is easy to verify that the families $\VV_0,\dots,\VV_n$
  satisfy~(\ref{i: diam V le D})--(\ref{i: bigcup_i=0^n VV_i covers
    B(x,r)}) with $x$ and $r=\infty$ (on $\GG(x)$). Hence
  $x\in Y(R,D,n)$.
\end{proof}

\begin{lemma}\label{l: Cl_0(Y(R,D,n)) subset Y(R,D',n)}
  For all $R,D\in\Z^+$ and $n\in\N$, there is some integer $D'\ge D$
  so that $\Cl_0(Y(R,D,n))\subset Y(R,D',n)$.
\end{lemma}

\begin{proof}
  With the notation of Proposition~\ref{p: orbit Reeb with x in U_0},
  for any $x\in\Cl_0(Y(R,D,n))$, there are points
  $x_k\in Y(R,D,n)\cap V(x,k)$ for all $k\in\Z^+$ such that
  $x=\lim_kx_k$. According to Proposition~\ref{p: orbit Reeb with x in
    U_0}, for some $C\in\Z^+$, independent of $x$, the maps
  $\phi_k:=\phi_{x,x_k,k}:B(x,k)\to B(x_k,k)$ are non-expanding and
  equi-bi-Lipschitz with bi-Lipschitz distortion $C$ for all $k$.  For
  each $k$, take families $\VV_{k,0},\dots,\VV_{k,n}$ of subsets of
  $B(x_k, k)$ satisfying~(\ref{i: diam V le D})--(\ref{i: bigcup_i=0^n
    VV_i covers B(x,r)}) with $x_k$ and $r=k$. For $k\ge k_0$, let
		\[
			\WW_{k,i}=\{\,\phi_k^{-1}(V)\mid V\in\VV_{k,i}\,\}
		\]
	for each $i\in\{0,\dots,n\}$. Obviously, $\bigcup_{i=0}^n\WW_{k,i}$ covers $B(x,k)$. Let $W=\phi_k^{-1}(V)\in\WW_{k,i}$ with $V\in\VV_{k,i}$. For $w,w'\in W$,
		\[
			d(w,w')\le C\,d(\phi_k(w),\phi_k(w'))\le CD\;,
		\]
	showing that $\diam W\le CD$. Take a different set $W'=\phi_k^{-1}(V')\in\WW_{k,i}$ for $V'\ne V$ in $\VV_{k,i}$. For $z\in W$ and $z'\in W'$,
		\[
			d(z,z')\ge d(\phi_k(z),\phi_k(z'))\ge R\;,
		\]
	obtaining that $d(W,W')\ge R$. So $x\in Y(k,R,CD,n)$ for all $k$, and therefore $x\in Y(R,CD,n)$ by Lemma~\ref{l: Y(R,D,n)=bigcap_rY(r,R,D,n)}.
\end{proof}

\begin{cor}\label{c: bigcup_DY(D,R,n) is F_delta}
	For all $R\in\Z^+$ and $n\in\N$, $\bigcup_DY(D,R,n)$ is an $F_\sigma$ subset of $U_0$.
\end{cor}

\begin{thm}\label{t: asdim orbits} 
  Let \((Z,\HH,U,\GG,E)\) satisfy Hypothesis~\ref{h: hyp part III}. If $\GG$ is transitive, then residually many $\GG$-orbits have the same asymptotic dimension.
\end{thm}

\begin{proof}
	The key step of the proof is the following assertion.
	
	\begin{claim}\label{cl: bigcup_DY(D,R,n)}
          For each $n\in\N$, the set $\bigcap_R\bigcup_DY(D,R,n)$ is
          either residual or meager in $U$.
	\end{claim}
	
	If $\Int_0(\bigcup_DY(D,R,n))\ne\emptyset$ for all $R$, then
        $\bigcup_DY(D,R,n)$ is residual in $U_0$, and therefore it is
        also residual in $U$, because this set is saturated and $\GG$
        is transitive. Hence $\bigcap_R\bigcup_DY(D,R,n)$ is also
        residual in $U$.
	
	If $\Int_0(\bigcup_DY(D,R_0,n))=\emptyset$ for some
        $R_0\in\Z^+$, then $\bigcup_DY(D,R_0,n)$ is meager in $U_0$ by
        Corollary~\ref{c: bigcup_DY(D,R,n) is F_delta}, and therefore
        it is also meager in $U$, completing the proof of
        Claim~\ref{cl: bigcup_DY(D,R,n)}.
	
	Assume that $\bigcap_R\bigcup_DY(D,R,n)$ is residual in $U$
        for some $n\in\N$, and let $n_0$ be the least $n$ satisfying
        this property. By~\eqref{bigcap_R bigcup_D Y(R,D,n)}, $n_0$ is
        the asymptotic dimension of any $\GG$-orbit in the
        $\GG$-saturated set
		\begin{equation}\label{asdim = n_0}
			\bigcap_R\bigcup_DY(D,R,n_0)\sm\bigcup_{n=0}^{n_0-1}\bigcap_R\bigcup_DY(D,R,n)\;,
		\end{equation}
	which is residual by Claim~\ref{cl: bigcup_DY(D,R,n)}.
	
	Finally, suppose that there is no $n\in\N$ so that
        $\bigcap_R\bigcup_DY(D,R,n)$ is residual in $U$. Hence
        $\bigcap_R\bigcup_DY(D,R,n)$ is meager in $U$ for all $n$ by
        Claim~\ref{cl: bigcup_DY(D,R,n)}, obtaining that
		\begin{equation}\label{asdim = infty}
			U_0\sm\bigcup_{n=0}^\infty\bigcap_R\bigcup_DY(D,R,n)
		\end{equation}
	is residual in $U$. Moreover every $\GG$-orbit in this $\GG$-saturated set is of infinite asymptotic dimension by~\eqref{bigcap_R bigcup_D Y(R,D,n)}.
\end{proof}

\section{Highson corona of the orbits}\label{s: Higson, orbits}

\subsection{Limit sets}\label{ss: limit sets, orbits}

Let $\OO$ be an infinite $\GG$-orbit, and $\ol\OO$ a compactification of $\OO$ with corona $\partial\OO$.

\begin{defn}\label{d: limit set, orbits}
  The \emph{limit set} \index{limit set} of any subset
  $\Sigma\subset\partial\OO$, denoted by $\lim_\Sigma\OO$, is the
  subset $\bigcap_V\Cl_U(V\cap\OO)$ of $U$, where $V$ runs in the
  collection of neighborhoods of $\Sigma$ in $\ol\OO$. If
  $\Sigma=\{\bfe\}$, then the notation $\lim_\bfe\OO$ is used for
  $\lim_\Sigma\OO$.
\end{defn}

Take another compactification $\ol\OO'\le\ol\OO$ of $\OO$ with corona
$\partial'\OO$. Thus there is a continuous extension
$\pi:\ol\OO\to\ol\OO'$ of $\id_\OO$. The restriction
$\pi:\partial\OO\to\partial'\OO$ clearly satisfies
$\lim_\bfe\OO\subset\lim_{\pi(\bfe)}\OO$ for all
$\bfe\in\partial\OO$. Thus, roughly speaking, smaller
compactifications of the orbits induce larger limit sets.

Each limit set is a closed subset of $U$, which may or may not be $\GG$-saturated. The following examples will serve as illustration of this fact.

\begin{examples}\label{ex: lim, orbits}
	\begin{enumerate}[(i)]

		\item\label{i: limit of one-point compactification, orbits} The corona of the one-point compactification $\OO^*$ is a singleton, and the corresponding limit set is the standard limit set of the orbit $\OO$, which of course $\GG$-saturated.
  
		\item\label{i: limits of coarse ends, orbits} Consider the compactification of $\OO$ whose corona is its coarse end space. The limit set of $\OO$ at any of its coarse ends is $\GG$-saturated.
		
		\item\label{i: alpha- and omega-limits} As a particular case of~\eqref{i: limits of coarse ends, orbits}, if $Z=U$ is compact, $\HH=\GG$ is the pseudogroup generated by a homeomorphism $h$ of $Z$, and $E=\{h^{\pm1}\}$, then $\OO$ is isometric to $\Z$, whose coarse end space consists of two points. The corresponding limit sets are the usual $\alpha$- and $\omega$-limits of $\OO$, which are $\GG$-saturated.
	
		\item\label{i: lim x = x, orbits} Consider the set
			\[
				\ol\OO=\OO\sqcup\Cl_Z(\OO)=(\OO\times\{0\})\cup(\Cl_Z(\OO)\times\{1\})
			\]
		with the topology determined as follows: each point of $\OO\times\{0\}$ is isolated in $\ol\OO$, and, a basic neighborhood of a point in $(z,1)\in\Cl_Z(\OO)\times\{1\}$ in $\ol\OO$ is of the form $(V\cap\OO)\sqcup V$, where $V$ is any neighborhood of $z$ in $\Cl_Z(U)$. Observe that $\ol\OO$ is a compact Hausdorff space, and $\OO\equiv\OO\times\{0\}$ is open and dense in $\ol\OO$; thus $\ol\OO$ is a compactification of $\OO$. In terms of algebras of functions, $\ol\OO$ corresponds to the algebra of $\C$-valued functions on $\OO$ that admit a continuous extension to $\Cl_Z(\OO)$.  The corona of this compactification is $\partial\OO=\Cl_Z(\OO)\times\{1\}\equiv\Cl_Z(\OO)$. Moreover, for each $z\in\partial\OO$, it is easy to see that $\lim_z\OO=\{z\}$ if $z\in U$, and $\lim_z\OO=\emptyset$ if $z\not\in U$. Thus $\lim_z\OO$ may not be $\GG$-saturated.
  
		\item\label{i: Stone-Cech, orbits} For the Stone-\v{C}ech compactification $\OO^\beta$, the limit set of $\OO$ at any point in its corona $\beta\OO$ is either a singleton or empty by~(\ref{i: lim x = x, orbits}) since $\OO^\beta$ is the maximum of the compactifications.
		
		\item\label{i: Higson corona, orbit} For any compactification $\ol\OO\le\OO^\nu$ with corona $\partial\OO$, it will be shown that the limit sets of $\OO$ at points in $\partial\OO$ are $\GG$-saturated (Theorem~\ref{t: lim e if GG-saturated}).
		
		\item\label{i: ideal bd of a hyperbolic orbit} As a particular case of~\eqref{i: Higson corona, orbit}, if $(\OO,d_E)$ is hyperbolic (in the sense of Gromov), we can consider its compactification whose corona is the ideal boundary. Then the limit sets of $\OO$ at points in its ideal boundary are $\GG$-saturated.
  
	\end{enumerate}
\end{examples}

\begin{lemma}\label{l: x_i to x}
	Let $x\in\lim_\Sigma\OO$ for some $\Sigma\subset\partial\OO$. If $\ol\OO\le\OO^\nu$, $V$ is a neighborhood of $\Sigma$ in $\ol\OO$, and $S_1\subset S_2\subset\cdots$ is an increasing sequence of bounded subsets of $V\cap\OO$, then there is a sequence $x_i\to x$ in $U$ such that
		\[
			\ol B(x_i,i)\subset V\cap\OO\;,\quad d(x_i,\{x_1,\dots,x_{i-1}\}\cup S_i)>3i\;.
		\]
\end{lemma}

\begin{proof}
  	Since $\OO\le\OO^\nu$, there is a continuous extension $\pi:\OO^\nu\to\ol\OO$ of $\id_\OO$. The sets $\widetilde V:=\pi^{-1}(V)$, for neighborhoods $V$ of $\Sigma$ in $\ol\OO$, form a base of neighborhoods of $\widetilde\Sigma:=\pi^{-1}(\Sigma)$ in $\OO^\nu$, and we have $\widetilde V\cap\OO=V\cap\OO$, obtaining also $\lim_{\widetilde\Sigma}\OO=\lim_{\Sigma}\OO$. So it is enough to consider only the case where $\ol\OO=\OO^\nu$.
	
	Assuming $\ol\OO=\OO^\nu$, let $W=V\cap\OO$, and, for $i\in\Z^+$, let
  		\[
			W_i=\{\,y\in W\mid d(y,\OO\sm W)>i\,\}\;,\quad V_i=\Int_{\OO^\nu}(\Cl_{M^\nu}(W_i))\;.
		\]
	Observe that $W_i=V_i\cap\OO$ because $\OO$ is a discrete space. By Proposition~\ref{p: W_r}, $V_i$ is a neighborhood of $\Sigma$ in $\OO^\nu$, and therefore $x\in\Cl_U(W_i)$. Then, given a countable base $\{P_i\}$ of open neighborhoods of $x$ in $U$, we have $P_i\cap W_k\ne\emptyset$ for all $i$ and $k$. 
	
	The elements $x_i$ are defined by induction on $i$. Let
		\[
			l_1=\max\{\,d(y,\OO\sm W)\mid y\in S_1\,\}\;.
		\]
	We can choose any element $x_1\in P_1\cap W_{3+l_1}$. Then, for all $y\in S_1$,
		\[
			d(x_1,y)\ge d(x_1,\OO\sm W)-d(y,\OO\sm W)>3+l_1-d(y,\OO\sm W)\ge3\;.
		\]
	Now let $i>1$ and assume that $x_j$ is defined for all $j<i$ satisfying $x_j\in P_j\cap W_j$ and $d(x_j,\{x_1,\dots,x_{j-1}\}\cup S_j)>3j$. Let 
		\[
			l_i=\max\{\,d(y,\OO\sm W)\mid y\in\{x_1,\dots,x_{i-1}\}\cup S_i\,\}\;,
		\]
	and choose any point $x_i\in P_i\cap W_{3i+l_i}$. For all $y\in\{x_1,\dots,x_{i-1}\}\cup S_i$,
		\[
			d(x_i,y)\ge d(x_i,\OO\sm W)-d(y,\OO\sm W)>3i+l_i-d(y,\OO\sm W)\ge3i\;.
		\]
	Moreover $\ol B(x_i,i)\subset W$ because $x_i\in W_{3i+l_i}\subset W_i$.
\end{proof}

\begin{thm}\label{t: lim e if GG-saturated}
  	Let \((Z,\HH,U,\GG,E)\) satisfy Hypothesis~\ref{h: hyp part III}, and let $\ol\OO$ be a compactification of a $\GG$-orbit $\OO$, with corona $\partial\OO$. If $\ol\OO\le\OO^\nu$, then $\lim_\bfe\OO$ is $\GG$-saturated and nonempty for all $\bfe\in\partial\OO$.
\end{thm}

\begin{proof}
  	Let $x\in\lim_\bfe\OO$ for some $\GG$-orbit $\OO$ and $\bfe\in\partial\OO$. Take a sequence $x_i\to x$ satisfying the conditions of Lemma~\ref{l: x_i to x} with any neighborhood $V$ of $\bfe$ in $\ol\OO$ and $S_i=\emptyset$ for all $i$. Then, by Proposition~\ref{p: x_i to x}, for any $r\in\Z^+$,
		\[
			\ol B(x,r)\subset\bigcap_{i>r}\Cl_U\left(\bigcup_{j\ge i}\ol B(x_j,r)\right)\subset\Cl_U(V\cap\OO)\;.
		\]
	Since $V$ and $r$ are arbitrary, it follows that $\GG(x)\subset\lim_\bfe\OO$. Hence $\lim_\bfe\OO$ is saturated.
	
	Let $U_0$ be a relatively compact open subset of $U$ that meets all $\GG$-orbits. Since, for any open neighborhood $V$ of $\bfe$ in $\ol\OO$, the set $V\cap\OO$ contains balls of arbitrarily large radius, we have $U_0\cap V\ne\emptyset$ by Proposition~\ref{p: recurrent finite symmetric family of generators}, and therefore $\Cl_U(V\cap\OO)\cap\Cl_U(U_0)$ is a nonempty compact set. It follows that 
		\[
			\lim_\bfe\OO\cap\Cl_U(U_0)=\bigcap_V\Cl_U(V\cap\OO)\cap\Cl_U(U_0)\ne\emptyset\;,
		\]
	showing that $\lim_\bfe\OO\ne\emptyset$.
\end{proof}

\begin{defn}\label{d: Higson recurrent orbit}
  	A $\GG$-orbit $\OO$ is said to be \emph{Higson recurrent} \index{Higson recurrent} if $\lim_\bfe\OO=U$ for all $\bfe\in\nu\OO$.
\end{defn}

\begin{rem}\label{r: Higson recurrent orbit}
	Every Higson recurrent $\GG$-orbit is obviously dense in $U$.
\end{rem}

\begin{thm}\label{t: Higson recurrent orbits}
  	Let \((Z,\HH,U,\GG,E)\) satisfy Hypothesis~\ref{h: hyp part III}. A $\GG$-orbit is Higson recurrent if and only if $\GG$ is minimal.
\end{thm}

\begin{proof} 
	Let $\OO$ be a Higson recurrent $\GG$-orbit, and let $Y$ be a $\GG$-minimal set (Proposition~\ref{p: there is some minimal set}). Since $\OO$ is dense (Remark~\ref{r: Higson recurrent orbit}), there is a convergent sequence in $U$, $x_i\to x$, with $x_i\in\OO$ and $x\in Y$. Let $P_1\supset P_2\supset\cdots$ be a nested sequence of open neighborhoods of $Y$ in $U$ such that $\bigcap_k\Cl_U(P_k)=Y$. By Proposition~\ref{p: x_i to x}, for each $k\in\Z^+$, there is some index $i_k$ such that $\ol B(x_{i_k},k)\subset P_k$. Hence
		\[
			\bigcap_l\Cl_U\left(\bigcup_{k\ge l}\ol B(x_{i_k},k)\right)
			\subset\bigcap_l\Cl_U(P_l)=Y\;.
		\]
	By Proposition~\ref{p: If W contains balls of arbitrarily large radius}, $W=\bigcup_kB(x_{i_k},k)\subset\OO$ satisfies that $V=\Int_{\OO^\nu}(\Cl_{\OO^\nu}(W))$ is an open neighborhood of some $\bfe\in\nu\OO$ in $\OO^\nu$. By Proposition~\ref{p: W_r}, the set $V_l=V\sm\bigcup_{k=1}^l\ol B(x_{i_k},k)$ is another open neighborhood of $\bfe$ in $\OO^\nu$. Since $\OO$ is Higson recurrent, we get
		\[
			U=\lim_\bfe\OO\subset\bigcap_l\Cl_U(V_l\cap\OO)
			=\bigcap_l\Cl_U\left(\bigcup_{k\ge l}\ol B(x_{i_k},k)\right)\subset Y\;.
		\]
	Thus $U$ is the only $\GG$-minimal set; i.e., $\GG$ is minimal.
  
  	Now, assume that $\GG$ is minimal. Then $\lim_\bfe\OO=U$ for all $\GG$-orbit $\OO$ and $\bfe\in\OO^\nu$ because $\lim_\bfe\OO$ is a $\GG$-saturated non-empty closed subset of $U$ (Theorem~\ref{t: lim e if GG-saturated}).
\end{proof}

For each $\GG$-minimal set $Y$ and any $\GG$-orbit $\OO$, let $\nu_Y\OO=\{\,\bfe\in\nu\OO\mid\lim_\bfe\OO=Y\,\}$.

\begin{thm}\label{t: lim e is a GG-minimal set}
	Let \((Z,\HH,U,\GG,E)\) satisfy Hypothesis~\ref{h: hyp part III}. For any $\GG$-orbit $\OO$, the set $\bigcup_Y\Int_{\nu\OO}(\nu_Y\OO)$, where $Y$ runs in the family of $\GG$-minimal sets, is dense in $\nu\OO$.
\end{thm}

\begin{proof}
	Let $\bfe\in\nu\OO$ for some $\GG$-orbit $\OO$, and let $V$ be a neighborhood of $\bfe$ in $\OO^\nu$. Take another open neighborhood $V'$ of $\bfe$ in $\OO^\nu$ such that $\Cl_{\OO^\nu}(V')\subset V$. By Proposition~\ref{p: there is some minimal set} and Theorem~\ref{t: lim e if GG-saturated}, $\lim_\bfe\OO$ contains a $\GG$-minimal set $Y$. Like in the proof of Theorem~\ref{t: Higson recurrent orbits}, we can find a subset $W\subset V'\cap\OO$ so that, for $V''=\Int_{\OO^\nu}(\Cl_{\OO^\nu}(W))$, any $\bfe'\in V''\cap\nu\OO$ satisfies $\lim_{\bfe'}\OO\subset Y$, and therefore $\lim_{\bfe'}\OO=Y$ because $Y$ is a minimal set. Thus $V''\cap\nu\OO\subset\nu_Y\OO$ and $V''\subset\Cl_{\OO^\nu}(V')\subset V$.
\end{proof}

Suppose that \((Z',\HH',U',\GG',E')\) also satisfies Hypothesis~\ref{h: hyp part III}, and that there is an equivalence $\HH\to\HH'$, which induces an equivalence $\GG\to\GG'$, an homeomorphism $U/\GG\to U'/\GG'$, and a bijection between the families of $\GG$- and $\GG'$-saturated sets. For each $\GG$-orbit $\OO$, let $\OO'$ denote the corresponding $\GG'$-orbit. By Theorem~\ref{t: equi-coarsely quasi-isometric orbits}, there are equi-coarse quasi-isometries of the metric spaces $(\OO,d_E)$ to the corresponding metric spaces $(\OO',d_{E'})$. By Propositions~\ref{p: large scale Lipschitz extensions},~\ref{p: any large scale Lipschitz map is rough} and~\ref{p: Higson}, these equi-coarse quasi-isometries induce maps between the corresponding Higson compactifications, $\OO^\nu\to{\OO'}^\nu$, that are continuous at the points of the Higson coronas, and restrict to homeomorphisms between the corresponding Higson coronas, $\nu\OO\to\nu\OO'$. In fact, the maps $\OO^\nu\to{\OO'}^\nu$ are continuous at all points because the orbits are discrete metric spaces.

\begin{prop}\label{p: lim_GG e corresponds to lim_GG' e'}
	For corresponding orbits, $\OO$ of $\GG$ and $\OO'$ of $\GG'$, and corresponding points $\bfe\in\nu\OO$ and $\bfe'\in\nu\OO'$, the $\GG$-saturated set $\lim_\bfe\OO$ corresponds to the $\GG'$-saturated set $\lim_{\bfe'}\OO'$.
\end{prop}

\begin{proof}
	By Remark~\ref{r: pseudogroup equivalence}, it is enough to consider the case where $Z'=Z$, $\HH'=\HH$, $\Cl_Z(U)\subset U'$, and $E'=\ol E=\{\,\bar g\mid g\in E\,\}$, where \(\bar g\) is and extension of each $g\in E$ with $\Cl_Z(\dom g)\subset\dom\bar g$ (like in Section~\ref{s: coarse q.i. type of orbits}). Then $\OO=\OO'\cap U$ and $\OO'=\GG'(\OO)$. Moreover the above coarse quasi-isometry of $(\OO,d_E)$ to $(\OO',d_{\ol E})$ is the inclusion map $\OO\hookrightarrow\OO'$, which is a $C$-bi-Lipschitz map whose image is an $R$-net with respect to the metrics $d_E$ and $d_{\ol E}$ (see Section~\ref{s: coarse q.i. type of orbits}). It induces the above continuous map $\OO^\nu\to{\OO'}^\nu$, which is an embedding in this case (Corollary~\ref{c: phi^nu is an embedding}). Thus we will consider $\OO^\nu$ as a subspace of ${\OO'}^\nu$ with $\nu\OO=\nu\OO'$; in particular, $\bfe=\bfe'$. 
	
	Let $V'$ be an arbitrary open neighborhood of $\bfe$ in ${\OO'}^\nu$, and therefore $V=V'\cap\OO^\nu$ is an arbitrary open neighborhood of $\bfe$ in $\OO^\nu$. We have $V\cap\OO=V'\cap\OO'\cap U$. So
		\[
			\Cl_U(V\cap\OO)=\Cl_U(V'\cap\OO'\cap U)=\Cl_{U'}(V'\cap\OO')\cap U\;,
		\]
	where the inclusion ``$\supset$'' of last equality holds because $U$ is open in $U'$. It follows that
		\[
			\lim_\bfe\OO=\bigcap_V\Cl_U(V\cap\OO)
			=\bigcap_{V'}\Cl_{U'}(V'\cap\OO')\cap U=U\cap\lim_\bfe\OO'\;.
		\]
\end{proof}

\subsection{Semi weak homogeneity of the Higson corona}

\begin{defn}\label{d: semi weakly homogeneous}
	A topological space $X$ is called \emph{semi weakly homogeneous} \index{semi weakly homogeneous} if, for all nonempty open subsets $V,V'\subset X$, there are homeomorphic nonempty open subsets, $\Omega\subset V$ and $\Omega'\subset V'$.
\end{defn}

\begin{thm}\label{t: semi weakly homogeneous, orbits}
	Let \((Z,\HH,U,\GG,E)\) satisfy Hypothesis~\ref{h: hyp part III}. If $\GG$ is minimal, then the space $\bigsqcup_\OO\nu\OO$, with $\OO$ running in the set of all $\GG$-orbits in $U_0$, is semi weakly homogeneous.
\end{thm}

\begin{proof}
	Let $\OO$ and $\OO'$ be $\GG$-orbits in $U_0$, and let $\bfe\in\nu\OO$ and $\bfe'\in\nu\OO'$. Given open neighborhoods, $V$ of $\bfe$ in $\OO^\nu$ and $V'$ of $\bfe'$ in ${\OO'}^\nu$, take other open neighborhoods, $V_0$ of $\bfe$ in $\OO^\nu$ and $V'_0$ of $\bfe'$ in ${\OO'}^\nu$, such that $\Cl_{\OO^\nu}(V_0)\subset V$ and $\Cl_{{\OO'}^\nu}(V'_0)\subset V'$. By Corollary~\ref{c: then W contains balls of arbitrarily large radius}, there is a sequence $x_k$ in $\OO$ such that $\ol B(x_k,k)\subset V_0$ and $d(x_l,x_k)>3k$ if $l<k$. We have $\lim_{\bfe'}\OO'=U$ by Theorem~\ref{t: Higson recurrent orbits}. Using Lemma~\ref{l: x_i to x} by induction on $k$, it follows that there are convergent sequences in $U$, $x'_{k,i}\to x_k$, such that $B(x'_{k,i},i)\subset V'_0\cap\OO'$, and $d(x'_{l,i},x'_{k,j})>3i$ if $l<k$ and $j\le i$, or $l=k$ and $j<i$. With the notation of Proposition~\ref{p: orbit Reeb with x in U_0}, for all $k\in\Z^+$, there is some increasing sequence of indices $i_k$ such that $i_k\ge k$ and $x'_{k,i_k}\in V(x_k,k)$. 
	
	According to Proposition~\ref{p: orbit Reeb with x in U_0}-\eqref{i: phi_x,y,r are equi-bi-Lipschitz},\eqref{i: ol B_E(y,r/C) cap im phi_x,y,r is a 2CR-net}, the restrictions
		\[
			\phi_k:=\phi_{x_k,x'_{i_k},k}:B_k:=\phi_k^{-1}(\ol B(x'_{k,i_k},k/C))\to B'_k:=\ol B(x'_{k,i_k},k/C)\;,
		\]
	for $k\ge CR$, form a family of equi-coarse equivalences. Moreover $B_k\subset\ol B(x_k,k)$ and $B'_k\subset\ol B(x'_{k,i_k},k)\subset\ol B(x'_{k,i_k},i_k)$, obtaining $d(B_k,B_l)>k$ and $d(B'_k,B'_l)>i_k\ge k$ if $l<k$. Then, by Proposition~\ref{p: combining equi-rough equivalences}, the combination of the maps $\phi_k$ is coarse equivalence $\phi:B:=\bigcup_{k\ge k_0}B_k\to B':=\bigcup_{k\ge k_0}B'_k$. Thus $\nu\phi:\nu B\to\nu B'$ is a homeomorphism by Proposition~\ref{p: Higson}-\eqref{i: phi is a coarse equivalence}. We have canonical identities $\nu B\equiv\Cl_{\OO^\nu}(B)\cap\nu\OO$ and $\nu B'\equiv\Cl_{{\OO'}^\nu}(B')\cap\nu\OO'$ (Corollary~\ref{c: phi^nu is an embedding}). Let $V_1=\Int_{\OO^\nu}(\Cl_{\OO^\nu}(B))$ and $V'_1=\Int_{{\OO'}^\nu}(\Cl_{{\OO'}^\nu}(B'))$. The open subsets $\Omega_1:=V_1\cap\nu\OO\subset\nu\OO$ and $\Omega'_1:=V'_1\cap\nu\OO'\subset\nu\OO'$ are nonempty by Proposition~\ref{p: If W contains balls of arbitrarily large radius} since $B_k\supset\ol B(x_k,k/C)$ because $\phi_k$ is non-expanding (Proposition~\ref{p: orbit Reeb with x in U_0}-\eqref{i: phi_x,y,r is non-expanding}). By Corollary~\ref{c: Int_M^nu(Cl_M^nu(W)) cap nu M is dense}, the sets $\Omega_1$ and $\Omega_1'$ are also dense in $\nu B$ and $\nu B'$, respectively.  Hence $\Omega:=\Omega_1\cap(\nu\phi)^{-1}(\Omega'_1)$ and $\Omega':=\nu\phi(\Omega_1)\cap\Omega'_1$ are open dense subsets of $\Omega_1$ and $\Omega'_1$, respectively, and therefore $\Omega$ and $\Omega'$ are nonempty and open in $\nu M$. Moreover $\nu\phi$ restricts to a homeomorphism $\Omega\to\Omega'$. Finally, 
		\[
			\Cl_{\nu\OO}(\Omega_1)=\Cl_{\OO^\nu}(V_1)\cap\nu\OO
			\subset\Cl_{\OO^\nu}(B)\cap\nu\OO
			\subset\Cl_{\OO^\nu}(V_0)\cap\nu\OO\subset V\cap\nu\OO\;,
		\]
	and, similarly, $\Cl_{\nu\OO'}(\Omega'_1)\subset V'\cap\nu\OO'$.
\end{proof}

\section{Measure theoretic versions}\label{s: measure, orbit}

Let $\mu$ be a Borel measure on $Z$. The \emph{measure class} \index{measure class} $[\mu]$
on $Z$ is the set of Borel measures on $Z$ with the same sets of
zero measure as $\mu$, and therefore also the same sets of full
measure (the complements of the sets of zero measure). If a measurable
$A\subset Z$ is of full measure, then $[\mu]$ is said to be \emph{supported} in $A$. The \emph{product measure class}
$[\mu]\times[\mu]$ on $Z\times Z$ is the measure class represented by
the product measure $\mu\times\mu$. 

Let \(\mu\) a measure on \(Z\). The measure class $[\mu]$ is $\HH$-invariant if
$\mu(B)=0\Rightarrow\mu(h(B))=0$, for all $h\in\HH$ and every
measurable $B\subset\dom h$. An $\HH$-invariant measure class $[\mu]$
is said to be \emph{ergodic} \index{ergodic} (or \emph{$\HH$-ergodic}) when it
consists of ergodic measures; i.e., any $\HH$-saturated measurable set
is either of zero measure, or of full measure for every measure in
\([\mu]\). In this case, $[\mu]\times[\mu]$ is also $\HH$-invariant
and $\HH$-ergodic. 

\begin{lemma}\label{l: Borel in a G_delta set}
  Let $X$ be a topological space, and $A\subset X$ a $G_\delta$
  subset. Then every Borel subset of $A$ is Borel in $X$.
\end{lemma}

\begin{proof}
  It is enough to consider the case of an open subset $B\subset A$,
  that is, $B=A\cap V$ for some open $V\subset X$. Because \(A\) is
  \(G_\delta\) in \(X\), it can be expressed as $A=\bigcap_nU_n$, for
  countable many open sets $U_n\subset X$. Therefore
  $B=\bigcap_n(U_n\cap V)$ is a $G_\delta$ subset of $X$.
\end{proof}

\begin{hyp}\label{h: [mu]}
	A sextuple \((Z,\HH,U,\GG,E,[\mu])\) is required to satisfy the following conditions:
 		\begin{itemize}
			
			\item \((Z,\HH,U,\GG,E)\) satisfy Hypothesis~\ref{h: hyp part III}, and
			
			\item  $[\mu]$ is a $\GG$-invariant measure class on $U$ such that $U_0$ has full measure.
		
		\end{itemize}
\end{hyp}

\begin{thm}\label{t: coarsely q.i. orbits 2, measure}
  Let \((Z,\HH,U,\GG,E,[\mu])\) satisfy Hypothesis~\ref{h: [mu]}. If $[\mu]$ is $\GG$-ergodic, then:
 \begin{enumerate}[{\rm(}i\/{\rm)}]
    
  \item\label{i: almost all orbits in U_0 are coarsely quasi-isometric
      to each other} either $[\mu]$-almost all $\GG$-orbits are coarsely
    quasi-isometric to $[\mu]$-almost all $\GG$-orbits;
    
  \item\label{i: almost all orbits in U_0 are coarsely quasi-isometric
      to almost no orbit} or else $[\mu]$-almost all $\GG$-orbits are coarsely
    quasi-isometric to $[\mu]$-almost no $\GG$-orbit.
			
  \end{enumerate}
\end{thm}

\begin{proof}
  Consider the notation of Section~\ref{s: coarsely quasi-isometric
    orbits}. In particular, $Y$ is a $\GG\times\GG$-saturated Borel
  subset of $U_0\times U_0$ by Lemma~\ref{l:Y(K,C)}, and therefore it
  is also a Borel subset of $U\times U$ by Lemma~\ref{l: Borel in a
    G_delta set}. By the $\GG\times\GG$-ergodicity of the product
  measure class on $U\times U$, either $Y$ is of full
  measure, or $Y$ is of zero measure.
  
  Suppose $Y$ is of full measure. By Fubini's theorem, the set
  $Y_x=\{\,y\in U_0\mid (x,y)\in Y\,\}$ is of full measure for almost
  all $x\in U$ and \eqref{i: almost all orbits in U_0 are
    coarsely quasi-isometric to each other} obtains. 
    
  Now assume that $Y$ is of measure
  zero. It follows from Fubini's theorem that $Y_x$ is of zero measure for almost all $x\in
  U$, and \eqref{i: almost all orbits in U_0 are coarsely
    quasi-isometric to almost no orbit} obtains.
\end{proof}

\begin{thm}\label{t:orbit growth 2, measure}
  Let \((Z,\HH,U,\GG,E,[\mu])\) satisfy Hypothesis~\ref{h: [mu]}. If $[\mu]$ is $\GG$-ergodic, then:
  \begin{enumerate}[{\rm(}i\/{\rm)}]
		
  \item\label{i: almost all orbits have the same growth} either $[\mu]$-almost
    all $\GG$-orbits have the same growth type;
			
  \item\label{i: almost all orbits have not comparable growth} or else
    the growth type of $[\mu]$-almost all $\GG$-orbits are comparable with the
    growth type of $[\mu]$-almost no $\GG$-orbit.
			
  \end{enumerate}
\end{thm}

\begin{proof}
  With the notation of Section~\ref{s: growth, orbits}, $Y$ is a
  $\GG\times\GG$-saturated Borel subset of $U_0\times U_0$ by
  Lemma~\ref{l:Y(a,b,c)}, and therefore it is also a Borel subset of
  $U\times U$ by Lemma~\ref{l: Borel in a G_delta set}. Like in the
  proof of Theorem~\ref{t: coarsely q.i. orbits 2, measure}, it
  follows that, either $Y$ is of full measure, or $Y$ is of zero
  measure. Hence, either $Y_\tau$ is of full measure, or
  $Y^\tau$ is of zero measure. Then the result follows with
  the arguments of the proof of Theorem~\ref{t: coarsely q.i. orbits
    2, measure}, using $Y_\tau$ to get~\eqref{i: almost all orbits have
    the same growth}, and $Y^\tau$ to get~\eqref{i: almost all orbits
    have not comparable growth}.
\end{proof}

\begin{thm}\label{t: orbit liminf ..., measure}
  Let \((Z,\HH,U,\GG,E,[\mu])\) satisfy Hypothesis~\ref{h: [mu]}. If $[\mu]$ is $\GG$-ergodic, then the equalities and inequalities of Theorem~\ref{t: orbit liminf ...} hold $[\mu]$-almost everywhere with some $a_1,a_3\in[1,\infty]$, $a_2,a_4\in[0,\infty)$ and $p\ge1$.
\end{thm}

\begin{proof}
  Consider the notation of Theorem~\ref{t: orbit liminf ...}. Its
  proof shows that the sets $Y(a)$, $\GG(Y'_2(a))$, $Y_3(a)$ and $\GG(Y_4(a))$
  are $\GG$-invariant and Borel in $U_0$ for all
  $a\in[0,\infty]$. Thus they are also Borel in $U$ by Lemma~\ref{l:
    Borel in a G_delta set}. By ergodicity, each of these sets are
  either of zero measure or of full measure. Then the result follows
  easily from the definition of these sets, and using~\eqref{Y_1(a) subset Y_1(a'), ...},~\eqref{Y'_2(a) supset Y'_2(a')},~\eqref{GG(Y_2(qa)) subset Y_2(a), ...} and~\eqref{GG(Y'_2(pqa)) subset Y_2(a)} like in the proof of Theorem~\ref{t: orbit liminf ...}.
\end{proof}

\begin{cor}\label{t: orbit polynomial exponential growth, measure}
  Let \((Z,\HH,U,\GG,E,[\mu])\) satisfy Hypothesis~\ref{h: [mu]}. If $[\mu]$ is $\GG$-ergodic, then any of the sets of Corollary~\ref{c: orbit polynomial exponential growth} is either of zero $[\mu]$-measure or of
  full $[\mu]$-measure.
\end{cor}

\begin{thm}\label{t: asdim orbits, measure}
  Let \((Z,\HH,U,\GG,E,[\mu])\) satisfy Hypothesis~\ref{h: [mu]}. If $[\mu]$ is $\GG$-ergodic, then $[\mu]$-almost all $\GG$-orbits have the same asymptotic dimension.
\end{thm}

\begin{proof}
  	Consider the notation of Theorem~\ref{t: asdim orbits}. By Lemma~\ref{c: bigcup_DY(D,R,n) is F_delta}, each $\GG$-saturated set $\bigcap_R\bigcup_DY(D,R,n)$ Borel. So, by ergodicity, this set is, either of full $[\mu]$-measure, or of zero $[\mu]$-measure. 
  
  	Suppose that $\bigcap_R\bigcup_DY(D,R,n)$ is of full $[\mu]$-measure for some $n$, and let $n_0$ be the least $n$ satisfying this property. Like in the proof of Theorem~\ref{t: asdim orbits}, $n_0$ is the asymptotic dimension of any $\GG$-orbit in the $\GG$-saturated set~\eqref{asdim = n_0}, which is of full $[\mu]$-measure.
	
	If $\bigcap_R\bigcup_DY(D,R,n)$ is of zero $[\mu]$-measure for all $n$, then, like in the proof of Theorem~\ref{t: asdim orbits}, the $\GG$-saturated set~\eqref{asdim = infty} is of full $[\mu]$-measure, and consists of orbits with infinite asymptotic dimension.
\end{proof}

\chapter{Generic coarse geometry of leaves}\label{c: leaves}

This chapter is devoted to recall preliminaries needed about foliated
spaces, fixing the notation, so that the main theorems follow directly
from their pseudogroup versions. Introductions to foliated spaces,
with many examples, are given in \cite{MooreSchochet1988},
\cite[Chapter~11]{CandelConlon2000-I},
\cite[Part~1]{CandelConlon2003-II} and \cite{Ghys1999}.

\section{Foliated spaces}\label{s: foliated space}

Let $Z$ be a space and let $U$ be an open set
in $\R^n\times Z$ ($n\in\N$), with coordinates $(x,z)$. For $m\in\N$,
a map $f:U\to \R^p$ ($p\in\N$) is (\emph{smooth} or \emph{differentiable}) of \emph{class $C^m$} if its partial derivatives
up to order $m$ with respect to $x$ exist and are continuous on
$U$. If $f$ is of class $C^m$ for all $m$, then it is called (\emph{smooth} or \emph{differentiable}) of \emph{class $C^\infty$}.

Let $Z'$ be another space, and let
$h:U\to\R^p\times Z'$ be a map of the form $h(x,z)=(h_1(x,z),h_2(z))$,
for maps $h_1:U\to\R^p$ and $h_2:\pr_2(U)\to Z'$. It
will be said that $h$ is of \emph{class $C^m$} if $h_1$ is of
class $C^m$ and $h_2$ is continuous.

For $m\in\N\cup\{\infty\}$ and $n\in\N$, a \index{foliated structure} \emph{foliated
  structure}\footnote{The term \emph{lamination} is also used,
  specially when $X$ is a subspace of a manifold. The term \emph{foliation} is used when the spaces $Z_i$ are open subsets of
  some Euclidean space, and therefore $X$ is a manifold. The condition
  to be of \emph{class $C^m$} for a foliation $\FF$ also requires
  that the maps $\phi_j\phi_i^{-1}$ are of class $C^m$.} $\FF$ of \emph{class
  $C^m$} and \emph{dimension} $\dim\FF=n$ on a space $X$ is defined by a collection $\UU=\{U_i,\phi_i\}$,
where $\{U_i\}$ is an open covering of $X$, and each $\phi_i$ is a
homeomorphism $U_i\to B_i\times Z_i$, for a locally compact Polish
space $Z_i$ and an open ball $B_i$ in $\R^n$, such that the
coordinate changes $\phi_j\phi_i^{-1}:\phi_i(U_i\cap
U_j)\to\phi_j(U_i\cap U_j)$ are locally $C^m$ maps of the form
\begin{equation}\label{changes of foliated coordinates}
  \phi_j\phi_i^{-1}(x,z) = (x_{ij}(x,z),h_{ij}(z))\;.
\end{equation} 
Each $(U_i,\phi_i)$ is called a \emph{foliated chart} \index{foliated chart} or \emph{flow
  box}, \index{flow box} the sets $\phi_i^{-1}(B_i\times \{z\})$ ($z\in Z_i$) are
called \emph{plaques} \index{plaque} (or \emph{$\UU$-plaques}), and the collection
$\UU$ is called a \emph{foliated atlas} \index{foliated atlas} (of \emph{class $C^m$}). Two
$C^m$ foliated atlases on $X$ define the same \emph{$C^m$ foliated
  structure} if their union is a $C^m$ foliated atlas. If we consider
foliated atlases so that the sets $Z_i$ are open in some fixed space, then $\FF$ can be also described as a maximal
foliated atlas of class $C^m$. The term \emph{foliated space} \index{foliated space} (of
\emph{class $C^m$}) is used for $X\equiv(X,\FF)$. If no reference to
the class $C^m$ is indicated, then it is understood that $X$ is a
$C^0$ (or \emph{topological}) foliated space. The restriction of
$\FF$ to some open subset $U\subset X$ is the foliated structure
$\FF|_U$ on $U$ defined by the charts of $\FF$ whose domains are
contained in $U$.

A map between foliated spaces is called \emph{foliated} \index{foliated map} if it maps
leaves to leaves. A foliated map between $C^m$ foliated spaces is said
to be of \emph{class $C^m$} if its local representations in terms of 
foliated charts are of class $C^m$.

The foliated structure of a space $X$ induces a locally Euclidean
topology on $X$, the basic open sets being the
plaques of all of its foliated charts, which is finer than the original topology. The connected
components of $X$ in this topology are called \emph{leaves} \index{leaf} (or
\emph{$\FF$-leaves}). Each leaf becomes a connected manifold of
dimension $n$ and of class $C^m$ with the differential structure
canonically induced by $\FF$. The leaf which contains each point $x\in
X$ will usually be denoted by $L_x$. The leaves of $\FF$ form a
partition of $X$ that determines the (topological) foliated
structure. The corresponding quotient space, called \emph{leaf space}, 
is denoted by $X/\FF$. It is said that $\FF$ is
\emph{transitive} \index{transitive} (respectively, \emph{minimal}) \index{minimal} when some leaf is
dense (respectively, all leaves are dense) in $X$.

Many concepts of manifold theory readily extend to foliated spaces. In
particular, if $\FF$ is of class $C^m$ with $m\ge1$, there is a vector
bundle $T\FF$ over $X$ whose fiber at each point $x\in X$ is the
tangent space $T_xL_x$. Observe that $T\FF$ is a foliated space of
class $C^{m-1}$ with leaves $TL$ for leaves $L$ of $X$, and any
section of $TF$ is foliated. The same applies to any bundle naturally
associated to $T\FF$. Then we can consider a $C^{m-1}$ Riemannian
structure\footnote{This means a section of the associated bundle over
  $X$ of positive definite symmetric bilinear forms on the fibers of
  $T\FF$, which is $C^{m-1}$ as foliated map.} on $T\FF$, which will
be called a (\emph{leafwise}) \emph{Riemannian metric} on $X$. A $C^m$ foliated space with a $C^{m-1}$ Riemannian metric is called a $C^m$ \emph{Riemannian foliated space}; the reference to $C^m$ is omitted if $m=\infty$.

From now on, it will be assumed that $X$ is locally compact and Polish; i.e., 
the spaces $Z_i$ are locally compact and Polish.

\begin{defn}[See \cite{HectorHirsch1986-A}, \cite{CandelConlon2000-I},
  \cite{Godbillon1991}]\label{d: regular atlas}
  A foliated atlas $\UU=\{U_i,\phi_i\}$ of $\FF$ is called \emph{regular} \index{regular foliated atlas} if:
  \begin{enumerate}[(i)]
	
  \item\label{i: (widetilde U_i,tilde phi_i)} for all $i$, there is a
    foliated chart $(\widetilde U_i,\tilde\phi_i)$ of $\FF$ so that
    $\ol{U_i}\subset\widetilde U_i$ and $\tilde\phi_i|_{U_i}=\phi_i$;
		
  \item\label{i: U_i is locally finite} the cover $\{U_i\}$ of $X$ is
    locally finite; and,
		
  \item\label{i: any plaque of (U_i,phi_i) meets at most the closure
      of one plaque of (U_j,phi_j)} for all $i$ and $j$, the closure
    of every plaque of $(U_i,\phi_i)$ meets at most the closure of one
    plaque of $(U_j,\phi_j)$.
	
  \end{enumerate}
\end{defn}

Since $X$ is Polish and locally compact, every foliated atlas
$\UU=\{U_i,\phi_i\}$ of $\XF$ is \emph{refined} by a regular atlas
$\VV=\{V_\alpha,\psi_\alpha\}$ in the sense that, for each $\alpha$,
there is some index $i(\alpha)$ so that $\ol{V_\alpha}\subset
U_{i(\alpha)}$ and $\phi_{i(\alpha)}$ extends $\psi_\alpha$.

Let $\UU=\{U_i,\phi_i\}$ be a regular foliated foliated atlas of $\FF$
with $\phi_i:U_i\to B_i\times Z_i$, and let $p_i:U_i\to Z_i$ denote
the composition of $\phi_i$ with the second factor projection
$B_i\times Z_i\to Z_i$. Then the form~\eqref{changes of foliated
  coordinates} of the changes of coordinates holds globally, with
$x_{ij}:Z_{ij}=\phi_j(U_i\cap U_j) \to \R^n$ and $h_{ij}:Z_{ij}\to
Z_{ji}$. Each map $h_{ij}$ is determined by the condition
$p_j=h_{ij}p_j$ on $U_i\cap U_j$. They satisfy the cocycle property
$h_{jk}h_{ij}=h_{ik}$ on $U_i\cap U_j\cap U_k$ for all $i$, $j$ and
$k$. The family $\{U_i,p_i,h_{ij}\}$ is called a \emph{defining
  cocycle} \index{defining cocycle} of $\FF$ \cite{Haefliger1988}, \cite{Haefliger2002}.

The maps $h_{ij}$ generate a pseudogroup $\HH$ of local
transformations of $Z=\bigsqcup_iZ_i$, which is called the
representative of the \textit{holonomy pseudogroup} \index{holonomy pseudogroup} of $\FF$ induced
by $\UU$ (or by $\{U_i,p_i,h_{ij}\}$). This $\HH$ is independent of $\UU$ up to pseudogroup
equivalences. Let $E=\{h_{ij}\}$, which is a symmetric family of
generators of $\HH$. There is a canonical homemorphism between the
leaf space and the orbit space, $X/\FF\to Z/\HH$, given by
$L\mapsto\HH(p_i(x))$ if $x\in L\cap U_i$.

By fixing any $b_i\in B_i$, each $Z_i$ can be considered as a subset
of $X$, called a \emph{local transversal}, \index{local transversal} via
	$$
		\begin{CD}
			Z_i\equiv\{b_i\}\times Z_i\subset B_i\times Z_i @>{\phi_i^{-1}}>> U_i\;.
		\end{CD}
	$$
        It can be assumed that all of these local transversals are
        mutually disjoint, and thus $Z$ becomes embedded in
        $X$; then it is called a \emph{complete transversal} \index{complete transversal} in the
        sense that it meets all leaves and is locally given by local
        transversals. Each $\HH$-orbit injects into the corresponding
        $\FF$-leaf in this way.

The \emph{holonomy groups} \index{holonomy group} of the leaves are the germ groups of the
corresponding orbits. The leaves with trivial holonomy groups are
called \emph{leaves without holonomy}, \index{without holonomy!leaf} and they correspond to orbits
with trivial germ groups. Then the union $X_0$ of leaves with trivial
holonomy groups is a dense $G_\delta$ saturated subset of $X$, hence
Borel and residual (by Theorem~\ref{t: Hector,E-M-T for pseudogroups},
or see directly \cite{Hector1977a}
and~\cite{EpsteinMillettTischler1977}). When $\FF$ is transitive, the
union $X_{0,\text{\rm d}}$ of dense leaves with trivial holonomy
groups is a residual subset of $X$ (by Corollary~\ref{c: the union of
  dense orbits with trivial germ groups is residual}). If $X=X_0$, then it is said that $X$ is a \emph{foliated space without holonomy}. \index{without holonomy!foliated space}

By the regularity of $\UU$, we can consider the foliated atlas
$\widetilde{\UU}=\{\widetilde U_i,\tilde\phi_i\}$ given by
Definition~\ref{d: regular atlas}-\eqref{i: (widetilde U_i,tilde
  phi_i)}, where $\tilde\phi_i:\widetilde U_i\to \widetilde
B_i\times\widetilde Z_i$. By refining $\widetilde\UU$ if necessary, we
can assume that it is also regular. Thus it also induces a
representative $\widetilde\HH$ of the holonomy pseudogroup on
$\widetilde Z=\bigsqcup_i\widetilde Z_i$, a symmetric set of
generators $\widetilde E=\{\tilde h_{ij}\}$ given by~\eqref{changes of
  foliated coordinates}, and a defining cocycle $\{\widetilde
U_i,\tilde p_i,\tilde h_{ij}\}$. Observe that $Z$ is an open subset of
$\widetilde Z$ that meets all $\widetilde\HH$-orbits,
$\widetilde\HH|_Z=\HH$, each $\tilde h_{ij}$ extends $h_{ij}$, and
$\ol{\dom h_{ij}}\subset\dom\tilde h_{ij}$.

Let $(X',\FF')$ be another locally compact Polish foliated space of
class $C^m$ and dimension $n'$. Then $\FF\times\FF'$ denotes the
foliated structure on $X\times X'$ with leaves $L\times L'$, where $L$
and $L'$ are leaves of $\FF$ and $\FF'$, respectively. Let
$\UU'=\{U'_\alpha,\phi'_\alpha\}$ be a foliated atlas of $\FF'$, where
$\phi'_\alpha:U'_\alpha\to B'_\alpha\times Z'_\alpha$. For all
foliated charts $(U_i,\phi_i)\in\UU$ and
$(U'_\alpha,\phi'_\alpha)\in\UU'$, we get a foliated chart $(U_i\times
U'_\alpha,\psi_{i\alpha})$ of $\FF\times\FF'$, where $\psi_{i\alpha}$
is the composite
$$
\begin{CD}
  U_i\times U'_\alpha @>{\phi_i\times\phi'_\alpha}>> B_i\times
  Z_i\times B'_\alpha\times Z'_\alpha @>{\xi_{i\alpha}}>>
  B_{i\alpha}\times Z_i\times Z'_\alpha\;,
\end{CD}
$$
where $B_{i\alpha}$ is an open ball in $\R^{n+n'}$ and 
$$
\xi_{i\alpha}(x_i,y_i,x_j,y_j)=(\zeta_{i\alpha}(x_i,x_j),y_i,y_j)
$$
for some homeomorphism $\zeta_{i\alpha}:B_i\times B'_\alpha\to
B_{i\alpha}$. The collection $\VV=\{U_i\times
U'_\alpha,\psi_{i\alpha}\}$ is a foliated atlas of $\FF\times\FF'$. We
can assume that the maps $\zeta_{i\alpha}$ are $C^m$ diffeomorphisms,
and therefore $\FF\times\FF'$ becomes of class $C^m$ with $\VV$.

Suppose that $\UU'$ is regular; in particular, it satisfies
Definition~\ref{d: regular atlas}-\eqref{i: (widetilde U_i,tilde
  phi_i)} with charts $(\widetilde U'_\alpha,\tilde\phi'_\alpha)$,
where $\tilde\phi'_\alpha:\widetilde U'_\alpha\to\widetilde
B'_\alpha\times\widetilde Z'_\alpha$. We can assume that each
$\zeta_{i\alpha}$ extends to a homeomorphism
$\tilde\zeta_{i\alpha}:\widetilde B_i\times\widetilde
B'_\alpha\to\widetilde B_{i\alpha}$ for some open balls $\widetilde
B_{i\alpha}$ containing $\ol{B_{i\alpha}}$. As before, using
$\tilde\phi_i$, $\tilde\phi'_\alpha$ and $\tilde\zeta_{i\alpha}$, we
get a foliated chart $(\widetilde U_i\times\widetilde
U'_\alpha,\tilde\psi_{i\alpha})$ of $\FF\times\FF'$, which extends
$\psi_{i\alpha}$. This shows that $\VV$ satisfies Definition~\ref{d:
  regular atlas}-\eqref{i: (widetilde U_i,tilde phi_i)}. The other
conditions of Definition~\ref{d: regular atlas} are obviously
satisfied, and therefore $\VV$ is regular. Observe that, if $\HH'$ is
the representative of the holonomy pseudogroup of $\FF'$ induced by
$\UU'$, then $\HH\times\HH'$ is the representative of the holonomy
pseudogroup of $\FF\times\FF'$ induced by $\VV$.

\section{Saturated sets}\label{s: FF-saturated sets}

Consider the notation of Section~\ref{s: foliated space}. A subset of
$X$ is called \emph{saturated} \index{saturated} (or \emph{$\FF$-saturated}) if it is
a union of leaves. The \emph{saturation} \index{saturation} (or \emph{$\FF$-saturation}) of a subset $A\subset X$ is the union $\FF(A)$
of leaves that meet $A$. The canonical homeomorphism $X/\FF\approx
Z/\HH$ gives a canonical bijection between the families of
$\FF$-saturated subsets of $X$ and $\HH$-saturated subsets of $Z$,
which preserves the properties of being open, closed, $G_\delta$,
$F_\sigma$, Borel, Baire, dense, residual or meager.

\begin{rem}\label{r: R_FF}
  Note that each leaf is an $F_\sigma$ subset of $X$. Moreover the
  relation set $R_\FF\subset X\times X$ of the equivalence relation
  ``being in the same $\FF$-leaf'' is an $F_\sigma$ subset, which
  follows from Remark~\ref{r: R_HH} since
  $$
  R_\FF=\bigcup_{i,j}(p_i\times
  p_j)^{-1}(R_\HH\cap(Z_i\times Z_j))\;,
  $$
  where each $p_i\times p_j$ is a trivial fiber bundle with $\sigma$-compact fibers.
\end{rem}

The relation between saturations in $X$ and $Z$ is the following:
\begin{equation}\label{FF(A)}
  \FF(A)=\bigcup_{i,j}p_j^{-1}(\HH(p_i(A\cap U_i))\cap Z_j)
\end{equation}
for any $A\subset X$. Thus $\FF(A)$ is open if $A$ is open, which is
well known. However the behavior of the saturation is worse in
foliated spaces than in pseudogroups: the saturation of a meager,
Borel or Baire set may not be meager, Borel or Baire,
respectively. But we have the following result.

\begin{lemma}\label{l: FF(A) is residual in FF(V)}
  Let $A\subset V\subset X$, where $V$ is open in $X$. The following
  properties hold:
  \begin{enumerate}[{\rm(}i\/{\rm)}]
	
  \item\label{i: FF(A) is residual in FF(V)} If $A$ is residual in
    $V$, then $\FF(A)$ is residual in $\FF(V)$.
		
  \item\label{i: FF(A) is meager in FF(V)} If $A$ is meager and
    $\FF|_V$-saturated in $V$, then $\FF(A)$ is meager in $\FF(V)$.
		
  \end{enumerate}
\end{lemma}

\begin{proof}
  Property~\eqref{i: FF(A) is residual in FF(V)} follows
  from~\eqref{FF(A)}, Lemma~\ref{HH(A) open} and Theorem~\ref{t: KU}
  (applied to the trivial fiber bundles $p_i$).
	
  To prove~\eqref{i: FF(A) is meager in FF(V)}, we can assume that $V$
  is some $U_i$. In this case, we have the following simplification
  of~\eqref{FF(A)}:
  $$
  \FF(A)=\bigcup_jp_j^{-1}(\HH(p_i(A))\cap Z_j)\;.
  $$
  By Theorem~\ref{t: KU}, $p_i(A)$ is meager in $Z_i$. So
  $\HH(p_i(A))$ is meager in $Z$ by Lemma~\ref{l: saturation when HH
    is countably generated}. Therefore each $p_j^{-1}(\HH(p_i(A))\cap
  Z_j)$ is meager in $U_j$ by Theorem~\ref{t: KU}, obtaining that
  $\FF(A)$ is meager.
\end{proof}

\section{Coarse quasi-isometry type of the leaves}\label{s: coarse q.i. type of leaves}

Consider the notation of Sections~\ref{s: foliated space} and~\ref{s:
  FF-saturated sets}, and suppose from now on that $X$ is compact\footnote{Several
  concepts of this section do not need compactness of $X$ to be
  defined.}. Let $R_\FF\subset X\times X$ be the subset of pairs of points in the same
leaf (the relation set of the partition into leaves). For $(x,y)\in
R_\FF$, let $d_\UU(x,y)$ be the minimum number of $\UU$-plaques whose union is connected and contains $\{x,y\}$. This defines a map\footnote{The same definition gives a map $d_\UU:X\times X\to[0,\infty]$ with the analogous properties so that $R_\FF$ is the set with finite values.} $d_\UU:R_\FF\to[0,\infty)$, which is upper semi-continuous,
symmetric and satisfies the triangle inequality, but it is not 
a metric\footnote{This $d_\UU$ is a coarse
  metric on the leaves in the sense of Hurder \cite{Hurder1994}.} on the leaves because $d_\UU\ge1$. To
solve this problem, consider the function $d^*_\UU:R_\FF\to[0,\infty)$ given by
\[
d^*_\UU(x,y)=
\begin{cases}
  d_\UU(x,y) & \text{if $x\ne y$}\\
  0 & \text{if $x=y$}\;.
\end{cases}
\]
The map $d^*_\UU$ is symmetric, satisfies the triangle inequality and
its zero set is the diagonal $\Delta_X\subset R_\FF$, and therefore
its restriction to each leaf is a metric. However $d^*_\UU$ is upper 
semi-continuous only on $R_\FF\sm\Delta_X$. For each $x\in X$, $S\subset L_x$ and $r\ge0$, the
notation $B_\UU(x,r)$, $\ol B_\UU(x,r)$ and $\Pen_\UU(S,r)$ will be
used for the corresponding open and closed balls, and penumbra in
$(L_x,d^*_\UU)$.
  
A \emph{plaque chain} \index{plaque chain} (or \emph{$\UU$-plaque chain}) is a finite sequence of $\UU$-plaques such that each pair of consecutive plaques has nonempty intersection. For $(x,y)\in R_\FF$, the value $d_\UU(x,y)$ equals the least $k\in\Z^+$ such that there is a plaque chain\footnote{For a leaf $L$, and plaques $P,Q\subset L$ (respectively, $x,y\in L$), let $\bar d_\UU(P,Q)$ (respectively, $\bar d_\UU(x,y)$) be the least $k\in\N$ such that there is a plaque chain $(P_0,\dots,P_k)$ with $P_0=P$ and $P_k=Q$. This defines a metric $\bar d_\UU$ on the set of plaques in $L$, which can be identified to the metric $d_E$ on the corresponding $\HH$-orbit $\OO$, where plaques in $L$ are identified to the points in $\OO$ via the maps $p_i$. However the map $\bar d_\UU=d_\UU-1:L\times L\to[0,\infty)$, defined in this way, does not satisfy the triangle inequality.} $(P_1,\dots,P_k)$ with $x\in P_1$ and $y\in P_2$. 

If $X$ is $C^1$, we can also pick up any Riemannian metric $g$ on $X$
and take the corresponding Riemannian distance $d_g$ on the leaves,
which defines an upper semicontinuous map $d_g:R_\FF\to[0,\infty)$.

Since $X$ is compact, $\UU$ is finite by Definition~\ref{d: regular
  atlas}-\eqref{i: U_i is locally finite}. Moreover $\ol{U_i}$ is compact, and therefore every $Z_i$ has
compact closure in $\widetilde Z_i$, So $Z$ has compact closure in
$\widetilde Z$. By the observations of Section~\ref{s: foliated
  space}, it follows that $\widetilde\HH$ is compactly generated and
$E$ is a symmetric system of compact generation of $\widetilde\HH$ on
$Z$.

\begin{lemma}\label{l: E is recurrent}
  $E$ is recurrent.
\end{lemma}

\begin{proof}
  There is an open cover $\{V_i\}$ of $X$ such that $\ol{V_i}\subset
  U_i$ for all $i$. Then $W_i=p_i(V_i)$ is open with compact closure
  in $Z_i$. Thus $W=\bigcup_iW_i$ is open with compact closure in $Z$.
	
  Take any $z\in Z$, which is in some $Z_i$. Then there is some $x\in
  U_i$ such that $p_i(x)=z$. Since $\{V_j\}$ covers $X$, there is also
  some $V_j$ containing $x$. So $h_{ij}(z)=p_j(x)\in W_j$. This shows
  that $\HH(z)\cap W$ is a $1$-net in $\HH(z)$ with $d_E$.
\end{proof}

By Lemma~\ref{l: E is recurrent} and Theorem~\ref{t: equi-coarsely
  quasi-isometric orbits}, $Z$, $\HH$ and $E$ satisfy the conditions
to determine a coarse quasi-isometry type on the orbits that is
``equi-invariant'' by pseudogroup equivalences. 

The following result is well known, at least in the case of foliations
of manifolds. For the reader's convenience, its proof is indicated
by its relevance in this work, and because some subtleties show
up in the case of foliated spaces.

\begin{prop}\label{p: independence of the metrics on the leaves}
	The following properties hold:
		\begin{enumerate}[{\rm(}i\/{\rm)}]
		
			\item\label{i: d_UU and d_VV are equi-Lipschitz equivalent} For any other regular foliated atlases $\VV$ of $\FF$, $d^*_\UU$ and $d^*_\VV$ are equi-Lipschitz equivalent on the leaves.
			
			\item\label{i: d_UU and d_g are equi-large scale Lipschitz equivalent} Suppose that $\FF$ is $C^3$. Then, for any $C^2$ Riemannian metric $g$ on $X$, $d^*_\UU$ and $d_g$ are equi-large scale Lipschitz equivalent on the leaves.
			
			\item\label{i: the leaves are equi-coarsely quasi-isometric to the corresponding orbits} The leaves with $d^*_\UU$ are equi-coarsely quasi-isometric to the corresponding $\HH$-orbits with $d_E$.
		
		\end{enumerate}
\end{prop}

\begin{proof}
	Let $R_i\subset R_\FF$ be the relation set defined by the restriction $\FF|_{U_i}$. Note that $\bigcup_iR_i$ is an open neighborhood of $\Delta_X$ in $R_\FF$. By Definition~\ref{d: regular atlas}-\eqref{i: (widetilde U_i,tilde phi_i)},\eqref{i: U_i is locally finite} and the compactness of $X$, it follows that each $R_i$ has compact closure in $R_\FF$. Hence, by the upper semicontinuity of $d_\VV$ and since $\UU$ is finite, there is some $C>0$ such that $\sup d^*_\VV(R_i)\le\sup d_\VV(R_i)\le C$ for all $i$. This means that the $d^*_\VV$-diameter of all plaques of $\UU$ is $\le C$, obtaining that $d^*_\VV\le Cd^*_\UU$ on $R_\FF$. This proves~\eqref{i: d_UU and d_VV are equi-Lipschitz equivalent}. 
	
	Similarly, by the upper semicontinuity of $d_g$, we get $\sup d_g(R_i)\le K$ for some $K>0$, obtaining\footnote{This inequality only requires $\FF$ to be $C^1$.} $d_g\le Kd^*_\UU$ on $R_\FF$. 
	
	Consider the disjoint union of the leaves as a Riemannian manifold with $g$.
	
	\begin{claim}\label{cl: positive injectivity radius}
		The disjoint union of the leaves has a positive injectivity radius.
	\end{claim}
	
	For each $i$, take relatively compact open subsets, $Z'_i\subset Z_i$ and $B'_i\subset B_i$ such that the sets $U'_i=\phi^{-1}(B'_i\times Z'_i)$ cover $X$. Since $\UU$ is finite, it is enough to prove that, for all $i$, there is some $C_i>0$ such that the injectivity radius of $L_x$ at every $x\in U'_i$ is $\inj_g(x)\ge C_i$. Let $\psi:\widetilde B_i\to N$ be a $C^3$ open embedding into a closed $n$-manifold, and set $V=\psi(\widetilde B_i)$ and $W=N\sm\psi(\ol{B_i})$. Let $\{\lambda,\mu\}$ be a $C^3$ partition of unity of $N$ subordinated to its open covering $\{V,W\}$. For all $z\in\ol{Z'_i}$, let $g_z$ be the Riemannian metric on $V$ that corresponds to $g$ by the $C^3$ diffeomorphism 
		$$
			\begin{CD}
				\tilde\phi_i^{-1}(\{z\}\times\widetilde B_i) 
				@>{\tilde\phi_i}>> \{z\}\times\widetilde B_i
				\equiv\widetilde B_i
				@>{\psi}>> V\;.
			\end{CD}
		$$
	 Pick up any $C^2$ Riemannian metric $h$ on $N$. Then the metrics $h_z=\lambda g_z+\mu h$ ($z\in\ol{Z'_i}$) form o compact family of $C^2$ Riemannian metrics on $N$ with the $C^2$ topology. By the continuous dependence of the injectivity radius on the Riemannian metric with respect to the $C^2$ topology on closed manifolds \cite{Ehrlich1974}, \cite{Sakai1983}, there is some $K>0$ such that $\inj_{h_x}(y)\ge K$ for all $y\in N$. Let $K'>0$ denote the infimum of the $h_z$-distance between $\psi(B'_i)$ and $W$, with $z$ running in $\ol{Z'_i}$, and set $C_i=\min\{K,K'\}>0$. Since $h_x=g_x$ on $\psi(B_i)$, we get $\inj_g(x)\ge C_i$ for all $x\in U'_i$. This completes the proof of Claim~\ref{cl: positive injectivity radius}.
	
	By Claim~\ref{cl: positive injectivity radius}, the continuous dependence of the geodesic flow on the metric with respect to the $C^1$ topology \cite[Lemma~1.5]{Sakai1983},  and because $X$ is compact, it easily follows that there is some $\epsilon>0$ such that 
		\begin{equation}\label{d_g^-1([0,epsilon)) subset bigcup_iR_i}
			d_g^{-1}([0,\epsilon))\subset\bigcup_iR_i\;.
		\end{equation}
                Take points $x$ and $y$ in a leaf $L$. Suppose first
                that $d_g(x,y)\ge\epsilon/2$, and let $\gamma(t)$
                ($0\le t\le1$) be a minimizing geodesic in $L$ with
                $\gamma(0)=x$ and $\gamma(1)=y$ ($L$ is a complete
                Riemannian manifold by Claim~\ref{cl: positive
                  injectivity radius}). Take a partition
                $0=t_0<t_1<\dots<t_k=1$ such that the length of
                $\gamma|_{[t_{l-1},t_l]}$ is in
                $[\epsilon/2,\epsilon)$ for all
                $l\in\{1,\dots,k\}$. Observe that $d_g(x,y)\ge
                k\epsilon/2$. By~\eqref{d_g^-1([0,epsilon)) subset
                  bigcup_iR_i}, for each $l\in\{1,\dots,k\}$, there is
                some index $i_l$ such that
                $(\gamma(t_{l-1}),\gamma(t_l))\in R_{i_l}$, obtaining
		$$
			d^*_\UU(x,y)=d_\UU(x,y)\le k\le\frac{2}{\epsilon}\,d_g(x,y)\;.
		$$
Now, assume that $d_g(x,y)<\epsilon/2$. Then $(x,y)$ is in some $R_i$ by~\eqref{d_g^-1([0,epsilon)) subset bigcup_iR_i}, giving $d_\UU(x,y)=1$, and therefore $d^*_\UU(x,y)\le1$. This shows that $d^*_\UU\le \frac{2}{\epsilon}d_g+1$ on $R_\FF$, obtaining~\eqref{i: d_UU and d_g are equi-large scale Lipschitz equivalent}. 

Let $L$ be an arbitrary $\FF$-leaf, and $\OO$ the corresponding
$\HH$-orbit. Consider $Z$ as a subset of $X$ via the embedding of $Z$
into $X$ defined by an appropriate choice of the points $b_i$
(Section~\ref{s: foliated space}). In this way, $\OO$ becomes a
$1$-net in $L$ with $d^*_\UU$. On the other hand, the $\UU$-plaques can
be identified to the points of $Z$ via the maps $p_i$. Moreover, given
two different $\UU$-plaques, $P$ of $(U_i,\phi_i)$ and $Q$ of
$(U_j,\phi_j)$, we have $P\cap Q\ne\emptyset$ if and only if
$p_i(P)\in\dom h_{ij}$ and $h_{ij} p_i(P)=p_j(Q)$. Then, by using plaque chains, it easily follows that
$d_E\le d^*_\UU\le d_E+1$ on the subset $\OO\subset L$. Thus the inclusion map
$(\OO,d_E)\to(L,d^*_\UU)$ is a $(1,1)$-large scale bi-Lipschitz, and its image $\OO$ is a $1$-net of $(L,d^*_\UU)$. This gives~\eqref{i: the leaves are equi-coarsely quasi-isometric to the corresponding orbits} by Lemma~\ref{l: large scale bi-Lipschitz and the image is c-net} and Proposition~\ref{p: restrictions of large scale Lipschitz maps}.
\end{proof}

\begin{rem}\label{r: (OO,d_E) to (L,d^*_UU)}
	The inclusion maps $(\OO,d_E)\to(L,d^*_\UU)$ of the proof of Proposition~\ref{p: independence of the metrics on the leaves}-\eqref{i: the leaves are equi-coarsely quasi-isometric to the corresponding orbits} are $(1,3,1)$-large scale Lipschitz equivalences according to Lemma~\ref{l: large scale bi-Lipschitz and the image is c-net}.
\end{rem}

By Proposition~\ref{p: independence of the metrics on the
  leaves}-\eqref{i: d_UU and d_VV are equi-Lipschitz equivalent},
when $X$ is compact, the $\FF$-leaves have a well determined coarse
quasi-isometry type of metrics, represented by $d^*_\UU$ for any regular
atlas $\UU$. Moreover, by Propositions~\ref{p: independence of the
  metrics on the leaves}-\eqref{i: d_UU and d_g are equi-large scale
  Lipschitz equivalent} and~\ref{p: restrictions of large scale
  Lipschitz maps}, if $X$ is $C^3$, the quasi-isometry type of metrics
on the leaves can be also represented by $d_g$ for any $C^2$
Riemannian metric $g$ on $X$. 

All coarsely quasi-isometric invariants considered in the Chapter~\ref{c: intro} make sense for the leaves with the above coarse quasi-isometry type. Via the identification of $\UU$-plaques to points of $Z$ (indicated in the proof of Theorem~\ref{p: independence of the metrics on the leaves}-\eqref{i: the leaves are equi-coarsely quasi-isometric to the corresponding orbits}), the definitions of growth and amenability of the $\FF$-leaves given in Chapter~\ref{c: intro}, using $\UU$-plaques, correspond to the definitions of growth and amenability as metric spaces with $d^*_\UU$, and also to the growth and amenability of the $\HH$-orbits with $d_E$.

Using the properties indicated in Sections~\ref{s: foliated space} and~\ref{s: FF-saturated sets}, and Proposition~\ref{p: independence of the metrics on the leaves}-\eqref{i: the leaves are equi-coarsely quasi-isometric to the corresponding orbits}, we get Theorems~\ref{t: coarsely q.i. leaves 2}--\ref{t: coarsely q.i. leaves 3} and~\ref{t:growth 2}--\ref{t:liminf ...},~\ref{t: amenable} and~\ref{t: asdim leaves} from Theorems~\ref{t: coarsely q.i. orbits 2},~\ref{t: coarsely q.i. orbits 3},~\ref{t:orbit growth 2}--\ref{t: orbit liminf ...},~\ref{t: orbit folner} and~\ref{t: asdim orbits}, applied to $\HH$, because all of those theorems deal with (equi-) coarse quasi-isometric invariants. Similarly, Corollary~\ref{c: end space of leaves in X_0} can be obtained from Corollary~\ref{c: coarse end space of orbits in U_0}.

The subsets of $X\times X$ and $Z\times Z$ considered in
Theorems~\ref{t: coarsely q.i. leaves 1} and~\ref{t: coarsely
  q.i. orbits 1} are obviously saturated by $\FF\times\FF$ and
$\HH\times\HH$, respectively, and moreover they correspond one
another by the canonical homeomorphism between the leaf space of
$\FF\times\FF$ and the orbit space of $\HH\times\HH$. By
Corollary~\ref{c: growth type and coarse quasi-isometries}, the
subsets of $X\times X$ and $Z\times Z$ considered in
Theorems~\ref{t:growth 1} and~\ref{t:orbit growth 1} are also
saturated by $\FF\times\FF$ and $\HH\times\HH$, respectively, and
correspond one another. Hence, using the properties indicated in
Sections~\ref{s: foliated space} and~\ref{s: FF-saturated sets}, and Proposition~\ref{p: independence
  of the metrics on the leaves}-\eqref{i: the leaves are equi-coarsely
  quasi-isometric to the corresponding orbits}, we also get
Theorems~\ref{t: coarsely q.i. leaves 1} and~\ref{t:growth 1} from
Theorems~\ref{t: coarsely q.i. orbits 1} and~\ref{t:orbit growth 1}.

\begin{proof}[First proof of Theorem~\ref{t: Ghys' ``Proposition fondamentale''}]
  Let $B$ be a Baire subset of $X$ such that the $\FF$-saturation
  $\FF(B)$ is not meager. For each $n\in\N$, let
  $B_n=\bigcup_L\Pen_\UU(L\cap B,n)$, where $L$ runs in the family of
  $\FF$-leaves. Since $B=B_0\subset B_1\subset\cdots$ and
  $\bigcup_nB_n=\FF(B)$ is not meager, there is some $N\in\N$ such
  that $B_N$ is not meager. Then there is some nonempty open subset
  $V\subset X$ such that $B_N\cap V$ is residual in $V$. By refining
  $\UU$ if necessary, we can assume that this holds with $V=U_i$ for
  some $(U_i,\phi_i)\in\UU$. By Theorem~\ref{t: KU}-\eqref{i: A is
    meager iff A_x is meager for a residually many x}, $C=p_i(U_i\cap
  B_N)$ is residual in $Z_i$. Hence $\HH(C)\cap Z_i$ is also residual
  in $Z_i$, and therefore $\HH(C)$ is not meager in $Z$. By Theorem~\ref{t: Ghys for pseudogroups}, it
  follows that there is some $\HH$-saturated residual subset $Y\subset Z$ and
  some $R>0$ such that $\OO\cap C$ is an $R$-net in $(\OO,d_E)$
  for all orbit $\OO\subset Y$. 
  Let $Y'$ denote the $\FF$-saturated residual subset of $X$
  that corresponds to $Y$. Let $\OO$ be an $\HH$-orbit in $Y$ and $L$
  the corresponding $\FF$-leaf in $Y'$. Consider the embedding of $Z$
  into $X$ for an appropriate choice of the points $b_i$
  (Section~\ref{s: foliated space}). In this way, according to the
  proof of Proposition~\ref{p: independence of the metrics on the
  leaves}-\eqref{i: the leaves are equi-coarsely quasi-isometric to
  the corresponding orbits}, $\OO\cap C$ is an $(R+1)$-net in 
  $(\OO,d^*_\UU)$, and $\OO$ a $1$-net into $(L,d^*_\UU)$. 
  Thus $\OO\cap C$ becomes an $(R+2)$-net in $L$. Since
	\[
		\OO\cap C\subset\Pen_\UU(L\cap B_N,1)\subset\Pen_\UU(L\cap B,N+1)\;,
	\]
  we get that $L\cap B$ is an $(R+3+N)$-net in $(L,d^*_\UU)$.
\end{proof}

\begin{proof}[Second proof of Theorem~\ref{t: Ghys' ``Proposition
    fondamentale''}]
  Consider first the case where $B$ is open and non-empty. Then the
  following argument shows that the intersection of all leaves with
  $B$ are equi-nets in the leaves. If this were wrong, then $B$ would
  not cut a sequence of balls $B_\UU(x_n,r_n)$ in the leaves with
  $r_n\to\infty$. Note that $\bigcap_m\ol{\bigcup_{n\ge
      m}B_\UU(x_n,r_n)}$ is saturated because $r_n\to\infty$. Thus
  $X=\bigcap_m\ol{\bigcup_{n\ge m}B_\UU(x_n,r_n)}$ by the minimality
  of $\FF$, obtaining that the nonempty open set $B$ cuts infinitely
  many balls $B_\UU(x_n,r_n)$, a contradiction.
	
  In the general case, with the notation of the first proof, there is
  some $N\in\N$ and a nonempty open subset $V\subset X$ such that
  $B_N\cap V$ is residual in $V$. By increasing $N$ and reducing $V$ if necessary, we
  can also suppose that $B_N\cap V$ is $\FF|_V$-saturated. Thus $V\sm
  B_N$ is meager and $\FF|_V$-saturated. So $\FF(V\sm B_N)$ is meager
  by Lemma~\ref{l: FF(A) is residual in FF(V)}-\eqref{i: FF(A) is
    meager in FF(V)}, and therefore the saturated set $Y=X\sm\FF(V\sm
  B_N)$ is residual. Since $L\cap B_N\cap V=L\cap V$ for any leaf $L$
  in $Y$, it follows that the intersections of all leaves in $Y$ with
  $B_N$ are equi-nets in those leaves. Then the same property is
  satisfied with $B$ like in the first proof.
\end{proof}

\section{Higson corona of the leaves}\label{s: Higson, leaves}

\subsection{Higson compactification}\label{ss: Higson, leaves}

Let $L$ be a leaf of $\FF$. To simplify the notation, let $L_\UU=(L,d^*_\UU)$, whose Higson compactification is denoted by $L^\nu_\UU$. If $\VV$ is another regular foliated atlas of $\FF$, then the identity map $L_\UU\to L_\VV$ is a large scale bi-Lipschitz bijection (Proposition~\ref{p: independence of the metrics on the leaves}-\eqref{i: d_UU and d_VV are equi-Lipschitz equivalent}). Therefore it induces a map $L^\nu_\UU\to L^\nu_\VV$, which is continuous at the points of $\nu L$ and restricts to a homeomorphism between the corresponding coronas (Proposition~\ref{p: Higson}-\eqref{i: phi^nu},\eqref{i: phi is a coarse equivalence}). Thus the corona of $L^\nu_\UU$ will be simply denoted by $\nu L$. The notation $L^\nu$ will be used for the underlying set of $L^\nu_\UU$ equipped with the coarsest topology so that the identity map $L^\nu\to L^\nu_\UU$ is continuous and the inclusion map $L\to L^\nu$ is an open embedding. The space $L^\nu$ is a compactification of $L$, called its \emph{Higson compactification}, \index{Higson compactification} whose corona is $\nu L$. If $\FF$ is $C^3$ and $g$ is a $C^2$ Riemannian metric on $X$, then the Higson compactification of $(L,d_g)$ is $L^\nu$.  

Let $\OO$ be the $\HH$-orbit that corresponds to $L$, equipped with $d_E$. Then $\OO$ becomes a subspace of both $L_\UU$ and $L$ with the injection $Z\to X$ given in Section~\ref{s: foliated space}. According to the proof of Proposition~\ref{p: independence of the metrics on the leaves}-\eqref{i: the leaves are equi-coarsely quasi-isometric to the corresponding orbits}, the inclusion map $\OO\to L_\UU$ is $(1,1)$-large scale bi-Lipschitz and its image is a $1$-net.  Thus $\OO\hookrightarrow L$ induces an embedding $\OO^\nu\to L^\nu_\UU$ (Corollary~\ref{c: phi^nu is an embedding}), which restricts to a homeomorphism $\nu\OO\to\nu L$ (Remark~\ref{r: (OO,d_E) to (L,d^*_UU)} and Proposition~\ref{p: Higson}-\eqref{i: phi is a coarse equivalence}). If the above map $\OO^\nu\to L^\nu_\UU$ is considered as map $\OO\to L^\nu$, then it is also an embedding because $\OO$ is subspace $L$. In this way, $\OO^\nu$ becomes a subspace of both $L^\nu_\UU$ and $L^\nu$, with $\nu\OO=\nu L$. Hence Theorem~\ref{t: semi weakly homogeneous, leaves} is a direct consequence of Theorem~\ref{t: semi weakly homogeneous, orbits}.

Using Lemma~\ref{l: ol X mapsto ol X'}, it easily follows that, for any compactification $\ol L\le L^\nu$, there are unique compactifications, $\ol L_\UU\le L^\nu_\UU$ and $\ol\OO\le\OO^\nu_\UU$, such that the identity map $L\to L_\UU$ and inclusion map $\OO\to L^\nu_\UU$ have continuous extensions $\ol L\to\ol L_\UU$ and $\ol\OO\to\ol L_\UU$ that restrict to the identity map on $\partial L=\partial\OO$. In fact, $\ol\OO=\Cl_{\ol L_\UU}(\OO)=\Cl_{\ol L}(\OO)$.

Consider the pseudogroup $\widetilde\HH$ on $\widetilde Z=\bigsqcup_i\widetilde Z_i$, with symmetric set of generators $\widetilde E=\{\tilde h_{ij}\}$, induced by the foliated atlas $\widetilde\UU$ (Section~\ref{s: foliated space}). By refining $\widetilde\UU$ if necessary, we can assume that this foliated atlas is also regular. Then $\widetilde E$ is also a recurrent system of compact generation on $\widetilde Z$ of another compactly generated pseudogrup on a larger space (Lemma~\ref{l: E is recurrent}). The points $b_i$ also define a map $\widetilde Z\to X$, which can be assumed to be the embedding. Thus we get the subspace inclusions $Z\subset\widetilde Z\subset X$. 

Let $\widetilde{\OO}$ be the $\widetilde\HH$-orbit that corresponds to $L$, equipped with $d_{\widetilde E}$. As before, there are unique compactification $\ol{\widetilde\OO}\le{\widetilde\OO}^\nu$,  with corona $\partial\widetilde\OO$,  where $\ol{\widetilde\OO}=\Cl_{\ol L}(\widetilde\OO)=\Cl_{\ol L_\UU}(\widetilde\OO)$ and $\partial\widetilde\OO=\partial L$. Furthermore, by Proposition~\ref{p: d_E and the restrictions of d_E' are equi-Lipschitz equivalent}, the inclusion map $\OO\to\widetilde{\OO}$ is bi-Lipschitz and its image is a net. Thus, as above, it induces an embedding $\ol\OO\to\ol{\widetilde\OO}$, which restricts to a homeomorphism $\partial\OO\to\partial\widetilde\OO$. In this way, we will consider $\ol\OO$ as a subspace of $\ol{\widetilde\OO}$, with $\partial\OO=\partial\widetilde\OO$; indeed, as explained before, $\ol\OO=\Cl_{\ol L}(\OO)\subset\Cl_{\ol L}(\widetilde\OO)=\ol{\widetilde\OO}$ with $\partial\OO=\partial L=\partial\widetilde\OO$.

\subsection{Limit sets}\label{ss: limit sets, leaves}

Let us continue with the notation of Section~\ref{ss: Higson, leaves}.

\begin{defn}
  	The \emph{limit set} \index{limit set} of $L$ at any $\bfe\in\partial L$, denoted by\footnote{When $\partial L$ is the end space of $L$ (Example~\ref{ex: lim, leaves}-\eqref{i: lim ends, leaves}), it is standard to use the term $\bfe$-limit of $L$ and the notation $\bfe-\lim L$. We prefer the stated terminology and notation because this concept represents the limit of the inclusion map $L\hookrightarrow X$ at $\bfe$ (a formalization of a point at the ``infinity'' of $L$).} $\lim_\bfe L$, is the subset $\bigcap_V \Cl_X (V\cap L)$ of $X$, where $V$ runs in the collection of neighborhoods of $\bfe$ in $\ol L$.
\end{defn}

Like in the case of pseudogroups (Section~\ref{ss: limit sets, orbits}), higher compactifications of the leaves induce smaller limit sets. Moreover $\lim\bfe$ is closed and nonempty\footnote{If the compactness assumption on $X$ is removed, then $\lim\bfe$ can be defined as well, but it may be empty.}, which may not be $\FF$-saturated. The following examples are foliated versions of Examples~\ref{ex: lim, orbits}.

\begin{examples}\label{ex: lim, leaves}
	\begin{enumerate}[(i)]

		\item\label{i: lim, L^*} For the one-point compactification $L^*$, the limit set of $L$ at the unique point in the corona is the standard limit set of $L$, which is saturated.

		\item\label{i: lim ends, leaves} For the compactification of $L$ by the end space, we get the standard limit set of $L$ at any end of $L$, which is also a saturated set.
		
		\item\label{i: lim x = x, L} Consider the set
			\[
				\ol L=L\sqcup\Cl_X(L)=(L\times\{0\})\cup(\Cl_X(L)\times\{1\})
			\]
		with the topology determined as follows: $L\equiv L\times\{0\}\hookrightarrow\ol L$ is an open embedding of the leaf, and, a basic neighborhood of a point in $(x,1)\in\Cl_X(L)\times\{1\}$ in $\ol L$ is of the form $(V\cap L)\sqcup V$, where $V$ is any neighborhood of $x$ in $\Cl_X(U)$. This $\ol L$ is a compactification of $L\equiv L\times\{0\}$. In terms of algebras of functions, $\ol L$ corresponds to the algebra of $\C$-valued functions on $L$ that admit a continuous extension to $\Cl_X(L)$.  The corona of $\ol L$ is $\partial L=\Cl_X(L)\times\{1\}\equiv\Cl_X(L)$. Moreover, for each $x\in\partial L$, it is easy to see that $\lim_xL=\{x\}$, which is not saturated if $\dim\FF>0$. (An analytic application of this compactification is given in \cite{Candel2003}.)
  
		\item\label{i: lim, L^beta} For the Stone-\v{C}ech compactification $L^\beta$,  the limit set of $L$ at any point in the corona $L^\beta$ is a singleton by~\eqref{i: lim x = x, L}.
  
  		\item\label{i: lim, ol L le L^nu} If $\ol L\le L^\nu$, it will be shown that $\lim_\bfe L$ is $\FF$-saturated for all $\bfe\in\partial L$ (Theorems~\ref{t: lim e if GG-saturated} and~\ref{t: lim_eL = FF(lim_eOO)}).
		
		\item\label{i: ideal bd of a hyperbolic orbit bis} As a particular case of~\eqref{i: lim, ol L le L^nu}, suppose that $\FF$ is $C^\infty$ and $g$ is a $C^\infty$ Riemannian metric on $X$ so that $(L,g)$ has negative curvature. Then we can consider the compactification of $L$ whose corona is the ideal boundary. The limit sets of $L$ at points in its ideal boundary are $\FF$-saturated.
  
\end{enumerate}
\end{examples}

For $S\subset X$ and $r\ge0$, the \index{penumbra} \emph{penumbra}\footnote{We can consider $d_\UU$ as a ``metric with possible infinite'' values on $X$, defining $d_\UU(x,y)=\infty$ when $L_x\ne L_y$. Then this definition of penumbra is the direct extension of the above one to ``metrics with possible infinite values.''} of $S$ of \emph{radius} $r$ is the set
	\[
		\Pen_\UU(S,r)=\bigcup_{x\in S}\ol B_\UU(x,r)\;.
	\]
Observe that $\Pen_\UU(S,1)$ is the union of the $\UU$-plaques that meet $S$.

\begin{lemma}\label{l: lim_eL}
	For all $\bfe\in\partial L$,
		\[
			\Pen_\UU(\lim_\bfe\OO,1)\subset\lim_\bfe L
			\subset\Pen_{\widetilde\UU}(\Pen_{\widetilde E}(\lim_\bfe\widetilde\OO,1),1)\;.
		\]
\end{lemma}

\begin{proof}
	Since the inclusion map $\OO\to L_\UU$ is $(1,1)$-large scale bi-Lipschitz and its image is a $1$-net, it is a $(1,3,1)$-large scale Lipschitz equivalence (Remark~\ref{r: (OO,d_E) to (L,d^*_UU)}). Thus, by Proposition~\ref{p: VV'}, given a base $\VV$ of neighborhoods of $\bar e$ in $\OO^\nu$, the sets $V':=\Cl_{L^\nu_\UU}(\Pen_\UU(V\cap\OO,1))$, with $V\in\VV$, form a base $\VV'$ of neighborhoods of $\bfe$ in $L^\nu$. Since $\OO$ is a net in $L_\UU$ and each $\UU$-plaque is closed in $L_\UU$, it easily follows that $\Pen_{L_\UU}(V\cap\OO,1)$ is closed in $L_\UU$, and therefore
		\begin{equation}\label{V' cap L}
			V'\cap L=\Pen_\UU(V\cap\OO,1)\;.
		\end{equation}
	
\begin{claim}\label{cl: Cl_X(V cap L) subset ...}
	 For all $V\in\VV$,
		\[
			\Cl_X(V'\cap L)\subset\Pen_{\widetilde\UU}(\Cl_{\widetilde Z}(\Pen_E(V\cap\OO,1)),1)\;.
		\]
\end{claim}
	
	So any $x\in\Cl_X(V'\cap L)$ is the limit in $X$ of some sequence $x_k\in V'\cap L$. By~\eqref{V' cap L}, there are sequences of points $y_k\in V\cap\OO$, and indices $i_k$ and $j_k$, such that $x_k\in U_{i_k}$, $y_k\in U_{i_k}\cap\OO\cap Z_{j_k}$, and $p_{i_k}(x_k)=p_{i_k}(y_k):=z_k\in Z_{i_k}$ for all $k$. Then $z_k=h_{j_ki_k}(y_k)\in\Pen_E(V\cap\OO,1)$. Since $\UU$ is finite, each $U_i$ is relatively compact in $\widetilde U_i$, and each $Z_i$ is a relatively compact in $\widetilde Z_i$, by passing to a subsequence if necessary, we can assume that there is an index $i$, and points $y\in\widetilde U_i$ and $z\in\widetilde Z_i$ such that $i_k=i$ for all $k$, and $y_k\to y$ in $\widetilde U_i$ and $z_k\to z$ in $\widetilde Z_i$ as $k\to\infty$. Thus $y\in\Cl_{\widetilde Z}(\Pen_E(V\cap\OO,1))$. Moreover $x\in\widetilde U_i$ and $\tilde p_i(x)=\lim_k\tilde p_i(x_k)=\lim_kz_k=z$, obtaining that $d^*_{\widetilde\UU}(x,y)\le1$. This shows Claim~\ref{cl: Cl_X(V cap L) subset ...}.
	
	\begin{claim}\label{cl: ... subset Cl_X(V cap L)}
	 For all $V\in\VV$,
		\[
			\Pen_\UU(\Cl_Z(V\cap\OO),1)\subset\Cl_X(V'\cap L)\;.
		\]
	\end{claim}
	
	For $x\in\Pen_\UU(\Cl_Z(V\cap\OO),1)$, there is some $y\in\Cl_Z(V\cap\OO)$ and some index $i$ such that $x,y\in U_i$ and $p_i(x)=p_i(y)=:z\in Z_i$. Moreover, for some $j$ so that $y\in Z_j$, there is some sequence $y_k\in V\cap\OO\cap U_i\cap Z_j$ such that $y_k\to y$ in $Z_j$ as $k\to\infty$. Since $p_i(x)=z=p_i(y)=\lim_kp_i(y_k)$, there is some sequence $x_k\in U_i$ such that $p_i(x_k)=p_i(y_k)$ and $x_k\to x$ as $k\to\infty$. Thus $d_\UU(x_k,y_k)=1$, giving $x_k\in V'\cap L$ by~\eqref{V' cap L}, and therefore $x\in\Cl_X(V'\cap L)$. This proves Claim~\ref{cl: ... subset Cl_X(V cap L)}.

	Now the result follows from Claims~\ref{cl: Cl_X(V cap L) subset ...} and~\ref{cl: ... subset Cl_X(V cap L)} by taking intersections with $V$ running in $\VV$.
\end{proof}

The following is a more explicit version of Theorem~\ref{t: lim e is an FF-minimal set}.

\begin{thm}\label{t: lim_eL = FF(lim_eOO)}
	Let $\XF$ be a compact Polish foliated space, let $\ol L$ be a compactification of an $\FF$-leaf $L$, with corona $\partial L$, let $\HH$ be the representative of the holonomy pseudogroup induced by a regular foliated atlas, and let $\OO$ be the $\HH$-orbit that corresponds to $L$. Suppose that $\ol L\le L^\nu$, and let $\ol\OO\le\OO^\nu$ be the compactification of $\OO$ that corresponds to $\ol L$, with corona $\partial\OO=\partial L$. Then $\lim_\bfe L=\FF(\lim_\bfe\OO)$ for all $\bfe\in\partial L$.
\end{thm}
	
\begin{proof}
	From Lemma~\ref{l: lim_eL}, Theorem~\ref{t: lim e if GG-saturated} and Proposition~\ref{p: lim_GG e corresponds to lim_GG' e'}, we get
		\[
			\FF(\lim_\bfe\OO)\subset\lim_\bfe L
			\subset\FF(\lim_\bfe\widetilde\OO)
			=\FF(\widetilde\HH(\lim_\bfe\OO))=\FF(\lim_\bfe\OO)\;.
		\]
\end{proof}

Now, Theorems~\ref{t: Higson recurrent leaves} and~\ref{t: lim e is an FF-minimal set} follow from Theorems~\ref{t: Higson recurrent orbits},~\ref{t: lim e is a GG-minimal set} and~\ref{t: lim_eL = FF(lim_eOO)}.

\section{Algebraic asymptotic invariants}

Let $\tp_\sigma$ denote the category of $\sigma$-compact topological spaces
and surjective continuous maps, let $\A$ be a category with limits, and
let $F:\tp_\sigma \to \A$ be a functor which is continuous in the
following sense:
For each object $X$ in $\tp_\sigma$ and each increasing sequence
$K_n$ of compact subspaces of $X$, $F(\bigcup _n K_n) = \lim_n
F(K_n)$, where the limit is injective or projective according to
whether $F$ is covariant or contravariant.

If $F$ is a continuous functor, there is an associated functor $F^\infty$, the
limit of $F$ at infinity, which is defined as follows. If $X$ is an
object in $\tp_\sigma$, then
$$F^\infty(X)=\lim_K F(X\sm K),$$
and if $f:X\to Y$ is surjective, then
$$F^\infty(f)=\lim_K F(X\sm K \to Y\sm K),$$
where $K$ denotes the family of compact subsets of $X$.

\begin{thm}
  Let $F$ be a continuous functor on $\tp_\sigma$ with values in the
  category of vector spaces over a field. Let $X$ be a transitive
  foliated space with no holonomy. Then the map $\dim:X\to
  \N\cup\{\infty\}$ which assigns to $x\in X$ the dimension of
  $F^\infty(L_x)$ is Borel, and is constant on a residual set of leaves.
\end{thm}
\begin{proof}
  Assume that the functor $F$ is covariant. If $x\in X$, let $B(x,r)$
  denote the ball of radius $r$ in the leaf containing $x$.  For
  positive integers \(n\), \(m\), and $i<j<k<l$, let $Y_n(i,j,k,l,m)$
  denote the set of points $x\in X$ for which there are compact
  domains
\[\ol B(x,k) \subset \Omega_{x,k,m} \subset \ol B(x,k+1/m)\;, \quad
\ol B(x,l) \subset \Omega_{x,l,m} \subset \ol B(x,l+1/m)\;,\]
and  open domains
\[\ol B(x,j-1/m) \subset U_{x,j,m} \subset \ol B(x,j)\;, \quad
\ol B(x,i-1/m) \subset U_{x,i,m} \subset \ol B(x,i)\;,\]
such that the map
\[F\left(\Omega_{x,k,m}\sm U_{x,j,m} \to \Omega_{x,l,m}\sm U_{x,i,m}\right)\]
  has image of dimension $\ge n$. The following assertion follows from the local Reeb stability for foliated spaces \cite[Proposition~11.4.8]{CandelConlon2000-I}.

\begin{claim}\label{cl: Y_n(i,j,k,l,m)}
  The sets $Y_n(i,j,k,l,m)$ are closed in $X$.
\end{claim}
  
For each leaf $L$ of $X$ and each point $x\in L$, we have that
  \[F^\infty(L) = \lim_{\leftarrow}{}_i \lim_{\rightarrow}{}_{k>i}
  F\left(\ol B(x,k)\sm B(x,i)\right)\;,\]
  and so 
  the set of points $x\in X$ where $\dim F^\infty(L_x)\ge n$ is
  \[Y_n = \bigcup_i \bigcap_{j>i}\bigcup_{k>j}
  \bigcap_{l>k}\bigcap_m Y_n(i,j,k,l,m)\;.\]
  Hence the function $\dim:X\to \N\cup\{ \infty\}$,
  which is constant along the leaves, has the property that
  $\dim^{-1}[n, \infty]$ is Borel for each $n\in \N$ by Claim~\ref{cl: Y_n(i,j,k,l,m)}, and so
  $\dim^{-1}\{n\}$ is a Borel saturated subset of $X$ for each $n\in
  \N$. The transitivity hypothesis on $X$ in turn implies that the
  function $\dim$ is constant on a residual saturated subset of $X$.
\end{proof}

\section{Versions with quasi-invariant currents}\label{s: q-i currents}

Let \(\XF\) be a foliated space of class \(C^2\). Then the space of
densities on \(\XF\) is a foliated space of class \(C^1\), and a
bundle over \(\XF\). A \emph{current}, \(m\), on \(\XF\) is a positive linear
functional on the space of compactly supported densities on \(X\). It
is called a \emph{quasi-invariant current} \index{quasi-invariant current} if on each foliation chart
\(\phi:U\to B\times Z\), the current \(\phi_*m\) on \(B\times Z\)
admits a disintegration of the form
\[ \phi_* m= \int_Z \lambda_z \cdot \mu_Z(z)\;,\]
where \(\mu_Z \) is a measure on the transversal \(Z\) and, for
\(\mu_Z\)-almost all \(z\in Z\),  \(\lambda_z\)
is equivalent to the Lebesgue current on the plaque \(B\times \{
z\}\).
The measures \(\{\mu_Z\}\) on the transversals are quasi-invariant
under the holonomy pseudogroup of \(\XF\), so they define
quasi-invariant measure class for the holonomy pseudogroup of \(\XF\).

\begin{examples}
\begin{enumerate}[{\rm(}i\/{\rm)}]
\item A transverse invariant measure combined with the current of
  integration along the leaves defines a quasi-invariant current on
  the foliated space.

\item If the space \(X\) is a \(C^1\) manifold so that \(\XF\) is the
  foliated space having leaves the connected components of \(X\), then
  the Lebesgue current of integration is a quasi-invariant measure on
  \(\XF\).
\end{enumerate}
\end{examples}

These two examples are certainly not common among foliated spaces. A
foliated space is rarely a manifold, and the generic foliated space lacks an
invariant measure. Nevertheless, quasi-invariant currents do exist on
any compact foliated space.

Suppose that $\XF$ is differentiable of class $C^2$, let $g$ be a
$C^1$ leafwise Riemannian metric on $X$, and let $\Delta$ denote the
Laplacian defined by $g$ on the leaves, mapping $C^2$ functions on $X$
to continuous functions.  A Borel measure $m$ on $X$ is called
\emph{harmonic} \index{harmonic measure} if $m(\Delta f)=0$ for all
$C^2$ function $f$ on $X$. The Riemannian metric determines a volume
density on \(\XF\) and induces an identification of densities on
\(\XF\) with functions on \(X\). Therefore a harmonic measure \(m\)
gives rise to a quasi-invariant current which has the following local
structure: on any foliated chart \(U\cong B\times Z\), \(m\) admits a
disintegration of the form
\[ m(f) = \int_Z \int_{B\times \{z\}}h(x,z) f(x,z) \mu(z)\]
where \(h(\cdot, z)\) is a positive harmonic fuction on \(B\times
\{z\}\), for \(\mu\)-almost all \(z\in Z\).

Harmonic measures were introduced by Garnett~\cite{Garnett1983}, who
proved that they satisfy several relevant properties (see also
\cite{Candel2003}, \cite[Chapter~2]{CandelConlon2003-II}). For
instance, if \(X\) is compact, there exists some harmonic probability
measure on $X$. A harmonic measure on $X$ is called \emph{ergodic}
\index{ergodic} (or \emph{$\FF$-ergodic}) if every $\FF$-saturated set
is either of zero measure or of full measure. Any $\HH$-invariant
measure on $Z$ (a transverse invariant measure of $\FF$) induces a
harmonic measure on $X$. On the other hand, any harmonic measure $m$
on $X$ induces an $\HH$-invariant measure class $[\nu]$ on $Z$ so that
the $\FF$-saturated sets of $m$-zero measure correspond to
$\HH$-saturated sets of $[\nu]$-zero measure, and $[\nu]$ is
$\HH$-ergodic if and only $\mu$ is $\FF$-ergodic.

Hamonic measures have good recurrence properties. A leaf in a foliated
space is called a \emph{wandering leaf} \index{wandering leaf} if it is proper and not compact.  The
\emph{wandering set} \index{wandering set} of a foliated space is the union of all its wandering leaves, and
the \emph{nonwandering set} \index{nonwandering set} is its complement. Both sets are
saturated Borel sets.

A theorem of Garnett~\cite{Garnett1983} states that the wandering set
of a foliated space has full measure  with respect to any harmonic measure.

For a foliated space, \(X\), the subset of leaves without holonomy,
\(X_0\subset X\), may consist entirely of wandering leaves and thus be
of measure zero for any harmonic measure on \(X\) . This is for
example the case of proper foliations (a well studied class of
foliations of codimension one), where the nonwandering set consists of
only compact leaves.  There are also large classes of minimal,
foliated spaces where the set \(X_0\) has full measure. One such class
is when \(X\) arises as a Markov exceptional minimal set of a
codimension one foliation \cite{CantwellConlon1988}; for these,
there are at most countably many leaves with holonomy, so \(X\sm
X_0\) has measure zero (with respect to any quasi-invariant transverse
measure class) because its intersection with any foliation chart is a
countable set of plaques.

The following results follow directly from Theorems~\ref{t: coarsely
  q.i. orbits 2, measure}--\ref{t: asdim orbits, measure}.

\begin{thm}\label{t: coarsely q.i. orbits 2, qi-measure}
  Let $\XF$ be a compact Polish foliated space. With respect to an
  ergodic quasi-invariant current on $X$ for which $X\sm X_0$
  has zero measure, either
  \begin{enumerate}[{\rm(}i\/{\rm)}]
		
  \item\label{i: almost all leaves are coarsely quasi-isometric to
      almost all leaves} almost all $\FF$-leaves are coarsely
    quasi-isometric to almost all $\FF$-leaves; or else
			
  \item\label{i: almost all leaves are coarsely quasi-isometric to
      almost no leaf}  almost all $\FF$-leaves are coarsely
    quasi-isometric to almost no $\FF$-leaf.
			
  \end{enumerate}
\end{thm}

\begin{thm}\label{t:orbit growth 2, qi-measure}
  Let $\XF$ be a compact Polish foliated space. With respect to an
  ergodic quasi-invariant current on $X$  for which  $X\sm X_0$ has
  zero measure, either 
  \begin{enumerate}[{\rm(}i\/{\rm)}]
		
  \item\label{i: almost all orbits have the same growth, qi-measure}
    almost all $\FF$-leaves have the same growth type; or else
			
  \item\label{i: almost all orbits have not comparable growth,
      qi-measure} the growth type of almost all $\FF$-leaves are
    comparable with the growth type of almost no
    $\FF$-leaf.
			
  \end{enumerate}
\end{thm}

\begin{thm}\label{t: orbit liminf ..., qi-measure}
  Let $\XF$ be a compact Polish foliated space. With respect to an
  ergodic quasi-invariant current on $X$ for which $X\sm X_0$
  has zero measure, the equalities and inequalities of
  Theorem~\ref{t:liminf ...} are satisfied almost everywhere with some
  $a_1,a_3\in[1,\infty]$, $a_2,a_4\in[0,\infty)$ and $p\ge1$.
\end{thm}

\begin{cor}\label{c: orbit polynomial exponential growth, qi-measure}
  Let $\XF$ be a compact Polish foliated space. With respect to an
  ergodic quasi-invariant current on $X$ for which $X\sm X_0$ 
  has zero measure, any of the sets of Corollary~\ref{c: polynomial
    exponential growth} is either of zero measure or of full measure.
\end{cor}

\begin{thm}\label{t: asdim leaves, qi-measure}
  Let $\XF$ be a compact Polish foliated space. With respect to an
  ergodic harmonic measure on $X$, for which $X\sm X_0$ has zero
  measure, almost all leaves have the same asymptotic dimension.
\end{thm}

\section{There is no measure theoretic version of recurrence}\label{s:
  no measure theoretic recurrence}

We show that there is no measure theoretic version of Theorem~\ref{t:
  Ghys' ``Proposition fondamentale''}. It fails for the most simple
non-trivial minimal foliation: a minimal Kronecker flow on the
$2$-torus. Let us first prove the measure theoretic version of its 
pseudogroup counterpart (Theorem~\ref{t: Ghys for pseudogroups})
fails for the pseudogroup
generated by a rotation with dense orbits on the unit circle.

Let $h$ be a rotation of the unit circle $S^1\subset\C\equiv\R^2$ with
dense orbits. Consider the action of $\Z$ on $S^1$ induced by $h$ (the
action of each $n\in\Z$ is given by $h^n$). Consider the standard
Riemannian metric on $S^1$, and let $\Lambda$ be the corresponding
Riemannian measure. Thus the action is isometric and $\Lambda$ is
invariant. Moreover $\Lambda$ is ergodic \cite{Denjoy1932}.
	
Let $\HH$ be the minimal pseudogroup on $S^1$ generated by the above
action, which satisfies the conditions of Theorem~\ref{t: Ghys for
  pseudogroups}, taking $U=S^1$, $\GG=\HH$ and $E=\{h,h^{-1}\}$. For
each positive integer $n$, let $I_n$ be an open arc in $S^1$ with
$\Lambda(I_n)<\frac{1}{(2n+1)2^n}$. Then
$$
A=\bigcup_{n=1}^\infty\bigcup_{i=-n}^nh^i(I_n)
$$
is a Borel set with
$$
\Lambda(A)\le\sum_{n=1}^\infty\sum_{i=-n}^n\Lambda(h^i(I_n))
=\sum_{n=1}^\infty\frac{1}{2^n}=1<2\pi=\Lambda(S^1)\;.
$$
So its complement $B=S^1\sm A$ is a Borel set with $\Lambda(B)>0$, and
thus $\Lambda(\HH(B))=\Lambda(S^1)$ because $\Lambda$ is
ergodic. Nevertheless, every orbit $\OO$ meets each $I_n$ at some
point $x$, and thus
\begin{equation}\label{OO cap A supset ol B_E(x,n)}
  \OO\cap A\supset\{\,h^i(x)\mid-n\le i\le n\,\}=\ol B_E(x,n)\;.
\end{equation}
Hence $\OO\cap B$ is not a net in $(\OO,d_E)$ for any orbit
$\OO$.
	
The suspension (Section~\ref{s: suspensions}) of the above action produces a
minimal Kronecker flow on the $2$-torus. Equip $\R$ with the standard
Riemannian metric. Then the universal cover $\R\to S^1$, $t\mapsto
e^{2\pi ti}$, is a local isometry. Let $\widetilde X=\R\times S^1$,
$\tilde g$ the product Riemannian metric on $\widetilde X$, and
$\widetilde\FF$ the $C^\infty$ foliation on $\widetilde X$ whose
leaves are the fibers $\R\times\{x\}$ for $x\in S^1$. The Riemannian
measure $\tilde\mu$ of $\tilde g$ is harmonic for $\widetilde\FF$ with
respect to the restriction of $\tilde g$ to the leaves. Consider the
$C^\infty$ diagonal $\Z$-action on $\widetilde X$, given by
$n\cdot(t,x)=(n+r,h^n(x))$, which is isometric and preserves
$\widetilde\FF$. Moreover $X=\Z\backslash\widetilde X$ is a $C^\infty$
manifold so that the quotient map $p:\widetilde X\to X$ is a
$C^\infty$ covering map. Thus $\tilde g$ and $\widetilde\FF$ project
to a Riemannian metric $g$ and a $C^\infty$ foliation $\FF$ on $X$,
and the Riemannian measure $\mu$ of $g$ is harmonic for $\FF$ with
respect to the restriction of $g$ to the leaves. Note that $\HH$ is a
representative of the holonomy pseudogroup of $\FF$. Observe that the
restriction $p:[0,1/2]\times B\to p([0,1/2]\times B)=:B'$ is
bijective. So
$$
\mu(B')=\tilde\mu([0,1/2]\times B)=\frac{\Lambda(B)}{2}>0\;.
$$
On the other hand, suppose that there is an $\FF$-leaf $L$ so that
$L\cap B'$ is a $K$-net in $L$ for some $K\in\N$. Since $\HH$ is
minimal, there is some $x\in I_{K+1}$ such that
$L=p(\R\times\{x\})$. Then $L\cap B'=p((\R\times\{x\})\cap
p^{-1}(B'))$. Since the restriction $p:\widetilde L\to L$ is an
isometry with the restrictions of $\tilde g$ and $g$, it follows that
$$
(\R\times\{x\})\cap p^{-1}(B')=\bigcup_{h^{-i}(x)\in B}[i,i+1/2]
$$
is a $K$-net in $\R$. But, by~\eqref{OO cap A supset ol B_E(x,n)},
$h^i(x)\notin B$ for $-K-1\le i\le K+1$ because $x\in I_{K+1}$, and
therefore $0$ is not in the $K$-penumbra of $\bigcup_{h^{-i}(x)\in
  B}[i,i+1/2]$ in $\R$, a contradiction. Thus $L\cap B'$ is not a net
in $L$ for all $\FF$-leaf $L$.

\chapter{Examples and open problems}\label{c: examples}

This chapter is mainly devoted to illustrate our main results with examples. To begin with, we recall recall some basic constructions of foliated spaces. Then we recall a procedure that allows to realize any connected Riemannian manifold of bounded geometry as a leaf of a compact Riemannian foliated space without holonomy. Besides of its theoretical interest, it can be used as a practical way of produce examples. This realization relies on a version using graphs instead of Riemannian manifolds. Indeed graphs can be also use to produce foliated spaces in a more direct way, specially using Cayley graphs. Then we recall and provide concrete examples, where our main theorems are confirmed. For this purpose, we have to recall some more concepts from descriptive set theory and theory of levels. 

A few open problems are also included in the last section.

\section{Foliated spaces defined by suspensions}\label{s: suspensions}

Let $\pi:\widetilde L\to L$ be a regular $\Gamma$-covering of a closed Riemannian manifold, and consider an action of $\Gamma$ on a Polish space $Z$. Such a group $\Gamma$ is finitely generated. We get a diagonal action of $\Gamma$ on $\widetilde L\times Z$, defined by $\gamma\cdot(\tilde y,z)=(\gamma\cdot\tilde y,\gamma\cdot z)$. This diagonal action preserves the trivial foliation with leaves $\widetilde L\times\{z\}$ ($z\in Z$), and therefore it induces a foliated structure on the Polish space $X=\Gamma\backslash(\widetilde L\times Z)$, which is called a \emph{suspension foliated space}. \index{suspension foliated space} Moreover the first factor projection $\widetilde L\times Z\to\widetilde L$ induces a fiber bundle projection $X\to L$ with typical fiber $Z$, whose fibers are transverse to the leaves. The holonomy pseudogroup of $X$ can be represented by the pseudogroup generated by the action of $\Gamma$ on $Z$. Thus $X$ is transitive or minimal if so is the action of $\Gamma$ on $Z$. For every orbit $\Gamma\cdot z$ ($z\in Z$), the corresponding leaf $L$ of $X$ is the projection of $\widetilde L\times\{z\}$. Recall that $\Gamma\cdot z\equiv\Gamma/\Gamma_z$, where $\Gamma_z$ is the isotropy subgroup. 

If $Z$ is compact, then $X$ is also compact, and $L$ is coarsely quasi-isometric to $\Gamma\cdot z\equiv\Gamma/\Gamma_z$. So the leaves are quasi-isometric to $\Gamma$ if the isotropy groups are trivial. More generally, two leaves are coarsely quasi-isometric if the corresponding conjugacy classes of isotropy groups have commensurable representatives. Thus, in this setting, the relation of commensurability up to conjugation is an interesting refinement of the coarse quasi-isometry relation between leaves. 

Any action of a finitely generated group $\Gamma$ on a compact space $Z$ can be used to produce a compact suspension foliated space. If $\Gamma$ can be generated by $n$ elements ($n\in\Z^+$), then it is a quotient of the free group $F_n$ with $n$ generators. Thus any action of $\Gamma$ on a compact space $Z$ can be transformed into an $F_n$ action. On the other hand, the closed orientable surface $\Sigma_n$ of genus $n$ has an obvious $F_n$ regular covering $\widetilde\Sigma_n$. This gives rise to the suspension foliated space $X=F_n\backslash(\widetilde\Sigma_n\times Z)$. If $K\vartriangleleft F_2$ is the kernel of the projection $F_n\to\Gamma$, then $\widetilde\Sigma'_n=K\backslash\widetilde\Sigma_n$ is a regular $\Gamma$-covering of $\Sigma_n$, and $X\equiv\Gamma\backslash(\widetilde\Sigma'_n\times Z)$.

We can use suspension foliated spaces to easily produce examples of transitive foliated spaces whose leaves have different asymptotic dimension. A very simple one is given by any suspension of an action of $\Z^2$ on the circle, where one point is fixed, and the other points have trivial isotropy groups and dense orbits. We get one compact leaf, which has asymptotic dimension zero, and the other leaves are coarsely quasi-isometric to $\Z^2$, which has asymptotic dimension two \cite{BellDranishnikov2011}.

\section{Foliated spaces defined by locally free actions of Lie groups}\label{s: locally free actions}

Any locally free action of a connected Lie group $G$ on a Polish space $X$ defines a foliated structure on $X$ whose leaves are the orbits \cite[Theorem~11.3.9]{CandelConlon2000-I}. Many interesting examples of this type are given in \cite[Chapter~11]{CandelConlon2000-I}. If we equip $G$ with a right invariant Riemannian metric, then the leaf through every $x\in X$, $G\cdot x\equiv G/G_x$, can be equipped with the induced metric, obtaining that $X$ is a compact Riemannian foliated space. The isotropy groups $G_x$ are discrete in $G$. If they are trivial, then all leaves are isometric to $G$. 

Assume that $X$ is compact. In general, two leaves are coarsely-quasi-isometric if the corresponding conjugacy classes of isotropy groups have commensurable representatives; again, the commensurability up to conjugacy refines the coarse quasi-isometry relation.

\section{Inverse limits of covering spaces}\label{s: inverse limits}

A compact connected foliated space whose local trasversals are totally disconnected is called a \emph{matchbox manifold}. \index{matchbox manifold} The following is a typical construction of matchbox manifolds.

Let $M_0\leftarrow M_1\leftarrow\cdots$ be a tower of smooth non-trivial finite fold regular covering maps between closed smooth manifolds. Its inverse limit $X$ is a compact minimal smooth matchbox manifold, called \emph{McCord} or \emph{regular solenoid}, \index{solenoid!McCord} \index{solenoid!regular} which generalizes the usual solenoid, where every $M_i$ is the circle. McCord solenoids are just the homogeneous smooth matchbox manifolds, except for closed manifolds \cite[Theorem~1.2]{ClarkHurder2013}. By definition of homogeneity, all leaves of $X$ are coarsely quasi-isometric to each other (the alternative~\eqref{i: all leaves in X_0,d are equi-coarsely quasi-isometric  to each other} of Theorem~\ref{t: coarsely q.i. leaves 2}). The groups $\Gamma_i=\pi_1(M_i)$ form a nested sequence of normal subgroups, $\Gamma_0\vartriangleright\Gamma_1\vartriangleright\dots$, and the groups of deck transformations of the composites $M_0\leftarrow M_i$ form a sequence of homomorphisms between finite groups, $\{1\}\leftarrow\Gamma_0/\Gamma_1\leftarrow\Gamma_0/\Gamma_2\leftarrow\cdots$, whose inverse limit is a topological group $Z$ with a canonical action of $\Gamma_0$. Then $X$ can be identified to the suspension of this action, using the universal covering $\widetilde M\to M$. The underlying space of $Z$ is a Cantor space.

If the covering maps are not required to be regular, then the term \emph{weak solenoid} \index{solenoid!weak} is used. The weak solenoids are just the equicontinuous matchbox manifolds, except for closed manifolds \cite[Theorem~1.4]{ClarkHurder2013}. In this case, the Cantor space $Z$ has no induced group structure because the sets $\Gamma_0/\Gamma_i$ may not be groups. But we continue having an action of $\Gamma_0$ on $Z$, realizing $X$ as suspension. Under some additional conditions, the holonomy covers of all leaves are coarsely quasi-isometric to each other by equicontinuity \cite[Theorem~17.3]{AlvarezCandel2009}.

More general matchbox manifolds can be described with inverse limits if we allow branched coverings \cite{AlcaldeLozanoMacho2011}, \cite{LozanoRojo2013}.

\section{Bounded geometry and leaves}\label{s: bounded geometry}

Let $M=(M,g)$ be complete connected Riemannian manifold. As usual, let $\nabla$ denote its Levi-Civita connection, and $\Iso(M)$ its group of isometries. It is said that $M$ is \emph{non-periodic} \index{non-periodic!Riemannian manifold} (respectively, \emph{locally non-periodic}) \index{locally non-periodic!Riemannian manifold} if $\Iso(M)=\{\id_M\}$ (respectively, the canonical projection $M\to\Iso(M)\backslash M$ is a covering map). For any domain $\Omega$ in $M$ and every smooth tensor $T$ on $\Omega$, let $\|T\|_\Omega=\sup_\Omega|T|$. It is said that $M$ is \emph{limit-aperiodic} \index{limit-aperiodic!Riemannian manifold} if, for all sequences, $m_i\uparrow\infty$ in $\N$, of compact domains $\Omega'_i\subset\Omega_i\subset M$, of points $x_i\in\Omega'_i$ and $y_i\in\Omega_i$, and of $C^{m_i}$ pointed embeddings $\phi_{ij}:(\Omega_i,x_i)\to(\Omega_j,x_j)$ ($i\le j$) and $\psi_i:(\Omega'_i,x_i)\to(\Omega_i,y_i)$, such that
	\[
		\lim_id(x_i,\partial\Omega'_i))=\infty\;,\quad
		\lim_{i,j}\|\nabla^{m_i}(g-\phi_{ij}^*g)\|_{\Omega_i}
		=\lim_i\|\nabla^{m_i}(g-\psi_i^*g)\|_{\Omega'_i}=0\;,
	\]
we have
	\[
		\lim_i\max\{\,d(x,\psi_i(x))\mid x\in\Omega'_i\cap\ol B(x_i,r)\,\}=0
	\]
for some $r>0$ \cite[Definition~12.4]{AlvarezBarralCandel2016}. Finally, it is said that $M$ is \emph{repetitive} \index{repetitive!Riemannian manifold} if, for every compact domain $\Omega$ in $M$, and all $\epsilon>0$ and $m\in\N$, there is a family of $C^m$ embeddings $\phi_i:\Omega\to M$ such that $\bigcup_i\phi_i(\Omega)$ is a net in $M$ and $\|\nabla^m(g-\phi_i^*g)\|_\Omega<\epsilon$ for all $i$ \cite[Definition~12.6]{AlvarezBarralCandel2016}. 

It is obvious that any leaf $L$ of a compact Riemannian foliated space $X$ is of bounded geometry. If moreover $L$ is without holonomy and $X$ minimal, then $L$ is repetitive, which is a direct consequence of the local Reeb stability \cite[Proposition~11.4.8]{CandelConlon2000-I}. The converse statements are given by the following result.

\begin{thm}[{\cite{AlvarezBarral-colorings}; see also \cite[Theorem~1.5]{AlvarezBarralCandel2016}, \cite[Theorem~1.1]{AlvarezBarral2017}}]
\label{t: realization}
	Any (repetitive) Riemannian manifold $M$ of bounded geometry can be realized as a leaf of a (minimal) compact Riemannian foliated space $X$ without holonomy.
\end{thm}

Thus the general study the leaves of (minimal) compact Riemannian foliated spaces without holonomy is the study of (repetitive) Riemannian manifolds of bounded geometry.

Let us recall the construction of $X$ in Theorem~\ref{t: realization} because it is a source of examples of compact foliated spaces with prescribed leaves. To begin with, we recall a simpler construction that only works under additional conditions on the manifold. For any $n\in\N$, let $\MM_*(n)$ denote the set\footnote{This set is well defined by assuming that the underlying set of every $M$ is contained in a common set.} of isometry classes, $[M,x]$, of pointed complete connected Riemannian $n$-manifolds, $(M,x)$. A sequence $[M_i,x_i]\in\MM_*(n)$ is said to be \emph{$C^\infty$ convergent} \index{$C^\infty$ convergence!of pointed Riemannian manifolds} to $[M,x]\in\MM_*(n)$ if, for every compact domain $\Omega\subset M$ containing $x$, there are pointed $C^\infty$ embeddings $\phi_i:(\Omega,x)\to(M_i,x_i)$ for large enough $i$ such that
	\[
		\lim_i\|\nabla^m(\phi_i^*g_i-g)\|_\Omega=0
	\]
for all $m\in\N$ \cite[Chapter~10, Section~3.2]{Petersen1998}. This convergence on $\MM_*(n)$ defines a Polish topology \cite[Theorem~1.2]{AlvarezBarralCandel2016}, \cite[Appendix~A]{AbertBiringer-unimodular} (see also \cite[Chapter~10]{Petersen1998}, \cite{Lessa2015}). The corresponding Polish space is denoted by $\MM_*^\infty(n)$, and its closure operator by $\Cl_\infty$. There is a continuous injection of $\MM_*^\infty(n)$ into the Gromov space $\MM_*$ of isometry classes of pointed proper metric spaces \cite{Gromov1981}, \cite[Chapter~3]{Gromov1999}. For every complete connected Riemannian $n$-manifold $M$, there is a canonical continuous map $\iota_M:M\to\MM_*^\infty(n)$, given by $\iota_M(x)=[M,x]$, which induces a continuous injection $\bar{\iota}_M:\Iso(M)\backslash M\to\MM_*^\infty(n)$. The images of all possible maps $\iota_M$ form a partition $\FF_*(n)$ of $\MM_*^\infty(n)$; i.e., every set of this partition is defined by varying the distinguished point in a fixed Riemannian manifold. The non-periodic and locally non-periodic manifolds define subspaces $\MM_{*,\text{\rm np}}^\infty(n)\subset\MM_{*,\text{\rm lnp}}^\infty(n)\subset\MM_*^\infty(n)$, and let $\FF_{*,\text{\rm lnp}}(n)$ denote the restriction of $\FF_*(n)$ to $\MM_{*,\text{\rm lnp}}(n)$. Assume $n\ge2$, otherwise $\MM_*^\infty(n)$ is too simple. Then $\MM_{*,\text{\rm lnp}}^\infty(n)$ is open and dense in $\MM_*^\infty(n)$, $\FF_{*,\text{\rm lnp}}(n)$ underlies a canonical $C^\infty$ foliated structure $\FF_{*,\text{\rm lnp}}^\infty(n)$, and the $C^\infty$ foliated space $\MM_{*,\text{\rm lnp}}^\infty(n)\equiv(\MM_{*,\text{\rm lnp}}^\infty(n),\FF_{*,\text{\rm lnp}}^\infty(n))$ has a canonical Riemannian structure \cite[Theorem~1.3]{AlvarezBarralCandel2016}. Moreover the holonomy covers of the leaves are of the form $\iota_M:M\to\im\iota_M$, which are local isometries. Thus the union of leaves without holonomy is the subspace $\MM_{*,\text{\rm np}}^\infty(n)$. On the other hand, $\Cl_\infty(\im\iota_M)$ is compact if and only if $M$ is of bounded geometry \cite[Theorem~12.3]{AlvarezBarralCandel2016} (see also \cite{Cheeger1970}, \cite[Chapter~10, Sections~3 and~4]{Petersen1998}), and $M$ is limit-aperiodic (respectively, repetitive) if and only if $\Cl_\infty(\im\iota_M)\subset\MM_{*,\text{\rm np}}^\infty(n)$ (respectively, $\Cl_\infty(\im\iota_M)$ is minimal) \cite[Lemmas~12.5 and~12.7]{AlvarezBarralCandel2016}. Thus, if $M$ is of bounded geometry and limit-aperiodic, then it is isometric to a leaf of the compact Riemannian foliated space without holonomy, $\Cl_\infty(\im\iota_M)$, which is minimal if $M$ is also repetitive.

To avoid the requirement of limit-aperiodicity of $M$, this condition is achieved by adding extra structure on $M$, given by a distinguished function. Precisely, fix a separable Hilbert space $\E$, and consider pairs $(M,f)$, where $f\in C^\infty(M,\E)$, instead of just the simply connected Riemannian $n$-manifold $M$. An isomorphism of these objects is an isometry compatible with the distinguished functions. Then, proceeding as above, equivalence classes $[M,f,x]$ can be defined by using pointed isomorphisms. They form a set $\widehat\MM_*(n)$, where there is an obvious version of the \emph{$C^\infty$ convergence}. This convergence defines a Polish space $\widehat\MM_*^\infty(n)$ \cite[Theorem~1.3]{AlvarezBarral2017}, whose closure operator is denoted by $\widehat\Cl_\infty$.  There are also canonical maps $\hat\iota_{M,f}:M\to\widehat\MM_*(n)$, whose images form a natural partition $\widehat\FF_*(n)$. The concepts of being \emph{non-periodic},  \index{non-periodic!function} \emph{locally non-periodic}, \index{locally non-periodic!function} \emph{limit-aperiodic} \index{limit-aperiodic!function} or \emph{repetitive} \index{repetitive!function} have obvious versions for pairs $(M,f)$ (or simply for $f$), obtaining $\widehat\MM_{*,\text{\rm np}}^\infty(n)$ and $\widehat\MM_{*,\text{\rm lnp}}^\infty(n)\equiv(\widehat\MM_{*,\text{\rm lnp}}^\infty(n),\widehat\FF_{*,\text{\rm lnp}}^\infty(n))$ as above, satisfying analogous properties (without requiring $n\ge2$) \cite[Section~1]{AlvarezBarral2017}; in particular, $\widehat\MM_{*,\text{\rm lnp}}^\infty(n)$ is a Riemannian foliated space, whose subspace of leaves without holonomy is $\widehat\MM_{*,\text{\rm np}}^\infty(n)$. This foliated space is universal among Riemannian foliated spaces satisifying a property called covering-continuity \cite[Proposition~6.4]{AlvarezBarral2017}. Moreover $(M,f)$ (or simply $f$) is said to be of \emph{bounded geometry} \index{bounded geometry!function} if $M$ is of bounded geometry and $\|\nabla^mf\|_M<\infty$ for all $m\in\N$. This property means that $\widehat\Cl_\infty(\im\hat\iota_{M,f})$ is compact \cite[Claim~7.4]{AlvarezBarral2017}. Then Theorem~\ref{t: realization} follows with $X=\widehat{\Cl}_\infty(\im\hat\iota_{M,f})$, where $f$ is given by the following result.

\begin{prop}[{\'Alvarez-Barral \cite{AlvarezBarral-colorings}, see also \cite[Proposition~7.1]{AlvarezBarral2017}}]\label{p: there exists f}
	For any (repetitive) connected Riemannian manifold $M$ of bounded geometry, there is some (repetitive) limit-aperiodic function $f\in C^\infty(M,\E)$ of bounded geometry.
\end{prop}

The construction of $f$ in Proposition~\ref{p: there exists f} will be indicated in Section~\ref{s: there exists f}. It will be reduced to a graph version, which is indicated first in Section~\ref{s: GG_*}. 

For instance, Theorem~\ref{t: realization} can be applied to any complete connected hyperbolic manifold with a positive injectivity radius. It can be also applied to any connected Lie group with a left invariant metric. Some of them are not coarsely quasi-isometric to (the Cayley graph of) any finitely generated group \cite{ChaluleauPittet2001}, \cite{EskinFisherWhyte2012}, obtaining compact minimal Riemannian foliated spaces without holonomy, whose leaves are isometric to each other, but the leaves are not coarsely quasi-isometric to any finitely generated group.

The results of Chapter~\ref{c: intro} can be illustrated by applying Theorem~\ref{t: realization} to appropriate Riemannian manifolds. But it is simpler to construct graphs of finite type and the required properties. Then we can consider the corresponding subspaces in $\GG_*$, or in any of its variants, which can be transformed into foliated spaces by a standard procedure, called taking the boundary of a thickening. This will produce compact foliated spaces with all possible quasi-isometric types of leaves according to Proposition~\ref{p: coarsely quasi-convex}. This way of constructing examples is explained in Section~\ref{s: graph matchbox mfds}, and concrete examples are given in Section~\ref{s: examples of graph matchbox mfds}.

\section{Graph spaces}\label{s: GG_*}

We will only consider connected graphs $G$ with a countable set of vertices, all of them with finite degree. Moreover, unless otherwise stated, the graphs are simple in the sense that there are no loop edges, there is at most one edge between any pair of vertices, and the edges have no orientations. Identify $G$ with its vertex set, $G\equiv V(G)$, and let $E(G)$ denote its edge set. With the canonical graph metric, these vertex sets are proper metric spaces. Since a graph isomorphism is the same as an isometry between graphs, the pointed isomorfism classes of pointed graphs also form a subspace $\GG_*$ of the Gromov space $\MM_*$. Precisely, a local base at any $[G,v]\in\GG_*$ is given by the sets
	\begin{equation}\label{UU_G,v,R}
		\UU_{G,v,R}=\{\,[G',v']\in\GG_*\mid(\ol B_{G'}(v',R),v')\cong(\ol B_G(v,R),v)\,\}\quad(R>0)\;,
	\end{equation}
where isomorphisms of pointed graphs are used. A compatible ultra-metric $d_{\GG_*}$ on $\GG_*$ can be defined by
	\begin{equation}\label{d_GG_*([G,v], [G',v'])}
		d_{\GG_*}([G,v],[G',v'])=\exp(-\sup\{\,R>0\mid[G',v']\in\UU_{G,v,R}\,\})\;.
	\end{equation}
Note that $\GG_*$ is totally disconnected. Let $\Cl$ denote the closure operator in $\GG_*$. For every graph $G$, there is a canonical map $\iota_G:G\to\GG_*$, defined like $\iota_M$ in Section~\ref{s: bounded geometry}, obtaining a transitive partition of $\GG_*$ into graphs with possible loop edges (the image of every $\iota_G$ is a quotient graph of $G$ that may have loop edges). Precisely, two elements $\bfz,\bfz'\in\GG_*$ are called \emph{contiguous} if $\bfz=[G,v]$ and $\bfz'=[G,z']$ for some contiguous vertices $v$ and $v'$ in some graph $G$. The relation of contiguity on $\GG_*$ is also independent of the representatives. Thus $\GG_*$ is kind of a space foliated by graphs. Any transitive saturated subspace $Z=\Cl(\im\iota_G)\subset\GG_*$ can be called a \emph{graph space}. \index{graph space} This $Z$ is compact if and only if $G$ is of finite type.

Other versions of graph spaces can be considered as well, like a version of $\widehat\MM_*(n)$ in Section~\ref{s: bounded geometry}. To begin with, take (vertex) colorings of graphs with values in $\N$. This gives rise to the totally disconnected space $\widehat\GG_*$ of isomorphism classes of pointed colored graphs, where a local base at every $[G,\alpha,v]\in\widehat\GG_*$ is given by the sets $\widehat\UU_{G,\alpha,v,R}$ ($R>0$), defined like in~\eqref{UU_G,v,R} by using pointed colored graphs isomorphisms. Let $\widehat{\Cl}$ denote the closure operator in $\widehat\GG_*$. For every colored graph $(G,\alpha)$, there is a canonical map $\hat\iota_{G,\alpha}:G\to\widehat\GG_*$, defined like $\hat\iota_{M,f}$ in Section~\ref{s: bounded geometry}, obtaining a canonical transitive partition of $\widehat\GG_*$ into colored graphs with possible loop edges.

Now, take edge-colored graphs, $(G,\beta)$, where $\beta:E(G)\to\N$. They define a totally disconnected space $\GG'_*$, where a local base at every $[G,\beta,v]\in\GG'_*$ is given by the sets $\UU'_{G,\alpha,v,R}$ ($R>0$), defined like in~\eqref{UU_G,v,R} by using pointed edge-colored graph isomorphisms. A compatible metric can be defined like in~\eqref{d_GG_*([G,v], [G',v'])}. This space is equipped with a canonical transitive partition defined by the images of canonical maps $\iota'_{G,\beta}:G\to\GG'_*$. The closure operator of $\GG'_*$ is denoted by $\Cl'$.

Consider also partially directed graphs, $(G,\OO)$, where the partial direction $\OO$ assigns an orientation\footnote{Recall that an orientation of an edge can be understood as an order of its vertices, which can be written as an ordered pair of its vertices.} to the edges in some subset of $E(G)$. Their direction-preserving pointed isomorphism classes form a totally disconnected space $\GG_{*,+}$, where a local base at every $[G,\OO,v]\in\GG_{*,+}$ consists of the sets $\UU_{G,\OO,v,R}$ ($R>0$), defined like in~\eqref{UU_G,v,R} by using pointed direction-preserving graph isomorphisms. This space is equipped with a canonical transitive partition into the images of canonical maps $\iota_{G,\OO}:G\to\GG_{*,+}$. The closure operator of $\GG_{*,+}$ is denoted by $\Cl_+$. Note that $\GG_*$ is the saturated subspace of $\GG_{*,+}$ defined by the graphs with empty partial orientation.

We can also combine several of the above structures on graphs in an obvious way, giving rise to the totally disconnected spaces $\widehat\GG'_*$, $\widehat\GG_{*,+}$, $\GG'_{*,+}$ and $\widehat\GG'_{*,+}$, where the closure operators are denoted by $\widehat\Cl'_*$, $\widehat\Cl_+$, $\Cl'_+$ and $\widehat\Cl'_+$, and with canonical maps, $\hat\iota_{G,\alpha,\beta}:V(G)\to\widehat\GG'_*$, $\hat\iota_{G,\alpha,\OO}:G\to\widehat\GG_{*,+}$, $\iota'_{G,\beta,\OO}:G\to\GG'_{*,+}$ and $\hat\iota'_{G,\alpha,\beta,\OO}:G\to\widehat\GG'_{*,+}$, whose images define canonical transitive partitions of these spaces.

The concepts of being \emph{non-periodic}, \index{non-periodic!graph} \index{non-periodic!coloring} \emph{locally non-periodic}, \index{locally non-periodic!graph} \index{locally non-periodic!coloring} \emph{limit-aperiodic} \index{limit-aperiodic!graph} \index{limit-aperiodic!coloring} and \emph{repetitive} \index{repetitive!graph} \index{repetitive!coloring} have obvious versions for graphs with possible additional structures of the above type, using graph isomorphisms preserving those structures. The concept of bounded geometry is played by the condition of finite type in the case of graphs, with the additional condition of using finitely many colors in the case of vertex or edge colorings. Thus the following is a version of Proposition~\ref{p: there exists f} for graphs.

\begin{thm}[\'Alvarez-Barral \cite{AlvarezBarral-colorings}]\label{t: there exists alpha}
	If a (repetitive) connected graph $G$ has vertex degrees uniformly bounded by some $c\in\N$, then $G$ has a (repetitive) limit-aperiodic coloring by $c$ colors.
\end{thm}

In Theorem~\ref{t: there exists alpha}, the number of colors is optimal with this generality (consider the Cayley graph of $\Z$, or any complete finite graph). But indeed the existence of a (repetitive) limit-aperiodic coloring by finitely many colors would be enough for our purposes. The proof of Theorem~\ref{t: there exists alpha} is very involved to achieve the optimal number of colors. It would be much simpler if only any finite number of colors is required. There is a version of Theorem~\ref{t: there exists alpha} for edge colorings, which indeed can be obtained as a corollary \cite{AlvarezBarral-colorings}.

\section{Case of Cayley graphs}\label{s: TT}

As an example of Section~\ref{s: GG_*}, consider the Cayley graph $G(\Gamma,S)$ of any finitely generated group $\Gamma$ and a finite set of generators $S$, defined so that right translations are graph isomorphisms (Section~\ref{s: graphs}, Example~\ref{ex: group}). We also use $\Gamma$ to denote this graph. We can assume that $S\cap S^{-1}$ consists only of elements of order two, different from the identity. Then $\Gamma$ can be equipped with a right invariant $S$-valued edge coloring $\beta$ and a right invariant partial orientation $\OO$, defined as follows. For an edge $e$ joining $v,w\in\Gamma$:
	\begin{itemize}
	
		\item let $\beta(e)=a\in S$ if $av=w$ or $aw=v$; 
		
		\item declare that $e\in\dom\OO$ just when $a$ is not of order two; and, 
		
		\item in this case, define $\OO(e)=(v,w)$ if $av=w$, and $\OO(e)=(w,v)$ if $aw=v$.
	
	\end{itemize}

\begin{rem}\label{r: determining the edges with the same vertex}
	Note that the edges $e$ with a common vertex $v$ are determined by $\beta(e)$ and the position of $v$ in $\OO(e)$, if $e\in\dom\OO$. This property is also satisfied by any connected subgraph equipped with the restrictions of $\beta$ and $\OO$.
\end{rem}

The above edge coloring and partial orientation will be always considered on $\Gamma$, often without mentioning them. After choosing an injective map $S\to\N$ to consider $\beta$ with values in $\N$, we get a compact saturated subspace $\Cl'_+(\im_{\Gamma,\beta,\OO})\subset\GG'_{*,+}$. We may also consider the compact subspace of $\GG'_{*,+}$ defined by all connected subgraphs of $\Gamma$, or by a class of those subgraphs, like trees. And we may also add vertex colorings in any of these constructions. 

However, when dealing with connected subgraphs of $\Gamma$ and their vertex or edge colorings, it is interesting to modify the previous definitions by using only right translations instead of arbitrary isomorphisms. To begin with, the definitions of being non-periodic, limit-aperiodic or repetitive are modified by using this restriction on the type of isomorphisms; the terms \emph{$\Gamma$-non-periodic}, \index{$\Gamma$-non-periodic} \emph{$\Gamma$-limit-aperiodic} \index{$\Gamma$-limit-aperiodic} and \emph{$\Gamma$-repetitive} \index{$\Gamma$-repetitive} will be used for these versions. Being $\Gamma$-non-periodic or $\Gamma$-limit-aperiodic is weaker than being non-periodic or limit-aperiodic, respectively, and being $\Gamma$-repetitive is stronger than being repetitive. 

The following space, defined with the above point of view, is very practical to construct concrete examples. Let $\TT=\TT(\Gamma)$ be the compact totally disconnected space of pointed trees in $\Gamma$, up to right translations, with the topology described by local bases defined like in~\eqref{UU_G,v,R}, using only isomorphisms between balls given by right translations \cite{Ghys1999}, \cite{Blanc2001}, \cite{LozanoRojo2007}, \cite{LozanoRojo2008}, \cite{AlcaldeLozanoMacho2009}, \cite{Lukina2012}, and a corresponding ultra-metric can be defined like in~\eqref{d_GG_*([G,v], [G',v'])}. This definition can be simplified because, after using right translations, the distinguished point can be chosen to be the identity element $1$, which can be omitted from the terminology and notation. Thus $\TT$ can be described as the space of trees $T$ in $\Gamma$ containing $1$, and its topology can be described by a local base at every $T$ consisting of the sets
	\[
		\UU_{T,R}=\{\,T'\in\TT\mid\ol B_T(1,R)=\ol B_{T'}(1,R)\,\}\quad(R>0)\;.
	\]
Again, a compatible ultra-metric can be defined using the sets $\UU_{T,R}$ like in~\eqref{d_GG_*([G,v], [G',v'])}. Now, in the canonical transitive partition of $\TT$, the class of every $T\in\TT$ is $\{\,T\gamma^{-1}\mid \gamma\in T\,\}$, which can be identified with $T$ using the mapping $T\gamma^{-1}\mapsto\gamma$. The graph structure, edge coloring and partial direction of $T$ passes to the corresponding set of the canonical partition via this identity. The right action of $\Gamma$ on itself  induces a compactly generated pseudogroup $\HH$ on $\TT$: for $T\in\TT$ and $\gamma\in\Gamma$, $T$ is in the domain of the map in $\HH$ defined by $\gamma$ just when $T\cdot\gamma\in\TT$; i.e., $\gamma^{-1}\in T$.

If $S$ does not contain elements of order two, then any tree in $\TT(\Gamma)$ can be obviously considered as a tree in the free group $F_n$ with $n=|S|$ generators, obtaining a canonical embedding of $\TT(\Gamma)$ into $\TT(F_n)$, which is compatible with the corresponding pseudogroups, and canonical edge colorings and orientations (which are global in this case). Thus considering $\TT(F_n)$ is enough in many cases.

We will also use the obvious version $\widehat\TT(c)=\widehat\TT(\Gamma,c)$ of $\TT$, which consists of subtrees of $\Gamma$ containing $1$ equipped with vertex colorings by colors in $\{0,1,\dots,c\}$ ($c\in\Z^+$). It satisfies analogous properties.

Similarly, we can fix the graph $\Gamma$ and take arbitrary vertex colorings by colors in the set $\{0,\dots,c\}$, defining the compact coloring space $\{0,\dots,c\}^\Gamma$, with the Tychonoff topology, and a compatible ultra-metric defined like in~\eqref{d_GG_*([G,v], [G',v'])}. There is no need to indicate the fixed point because it can be assumed to be $1$ as before. Thus the canonical transitive partition is now given by the orbits of the transitive left action of $\Gamma$ on $\{0,\dots,c\}^\Gamma$ (induced by the right action of $\Gamma$ on itself by right translations). This coloring space with this action is called a \emph{Bernouilli shift}. \index{Bernouilli shift} For any coloring $\alpha\in\{0,\dots,c\}^\Gamma$, its orbit closure $\overline{\Gamma\cdot\alpha}$ with the restriction of the $\Gamma$-action is called a \emph{subshift}. \index{subshift} The left action of $\Gamma$ on $\overline{\Gamma\cdot\alpha}$ is free (respectively, minimal) just when $\alpha$ is $\Gamma$-limit-aperiodic (respectively, $\Gamma$-repetitive). Since the graph structure, edge coloring and partial direction of $\Gamma$ are right invariant, they can be projected to $\Gamma/\Gamma_\alpha\equiv\Gamma\cdot\alpha$. There always exist some $\Gamma$-limit-aperiodic coloring $\alpha$ by two colors, which indeed holds for any countable group \cite{GaoJacksonSeward2009}, \cite{AubrunBarbieriThomasse-aperiodic_subshifts}. Moreover we can assume that this $\alpha$ is repetitive, otherwise we can find a repetitive coloring in a minimal set of $\overline{\Gamma\cdot\alpha}$. Thus we will consider only the Bernouilli shift $\{0,1\}^\Gamma$.  

Bernuilli shifts play an important role in Dynamics. For instance, they are expansive, and indeed universally expansive: any expansive $\Gamma$-action on a compact metric space is an equivariant quotient of some closed invariant subspace of $\{0,\dots,c\}^\Gamma$, for $c$ large enough, and this quotient map is an equivariant homeomorphism in the totally disconnected case \cite[Proposition~2.6]{CoornaertPapadopoulos1993}. 

There is an orbit equivalent embedding of any Bernouilli shift $\{0,\dots,c\}^\Gamma$ into $\TT(F_n)$ for $n$ large enough \cite[Theorem 1.1]{LozanoLukina2013}. But this embedding is non-canonical and rather involved. 

With more generality, we can consider a semigroup $\Gamma$, like $\N$. Then the Bernouilli shift is given by the induced semigroup action of $\Gamma$ on $\{0,\dots,c\}^\Gamma$. The Bernouilli shifts defined by semigroups are used in the study automatic sequences (see e.g.\ \cite{AlloucheShallit2003}, \cite{PerrinPin2004}, \cite{PytheasFogg2002}). In the case of $\N$, the colorings in $\{0,\dots,c\}^\N$ can be considered as infinite words using the alphabet $\{0,\dots,c\}$; for instance, the Fibonacci and Thue-Morse words in $\{0,1\}^\N$ are very relevant.

\section{Construction of limit-aperiodic functions}\label{s: there exists f}

Let us indicate the proof of Proposition~\ref{p: there exists f} using Theorem~\ref{t: there exists alpha}. By the bounded geometry of $M$, there is some $0<r<\inj_M$ such that the following properties hold:
	\begin{enumerate}[(i)]
		
		\item\label{i: g_ij} For the normal parametrizations $\kappa_x:B_r:=B_{\R^n}(0,r)\to B_M(x,r)$ ($x\in M$), the corresponding metric coefficients, $g_{ij}$ and $g^{ij}$, as a family of $C^\infty$ functions on $B_r$ parametrized by $x$, $i$ and $j$, lie in a bounded subset of the Fr\'echet space $C^\infty(B_r)$ \cite[Theorem~A.1]{Schick1996}, \cite[Theorem 2.5]{Schick2001} (see also \cite[Proposition~2.4]{Roe1988I}, \cite{Eichhorn1991}). 
			
		\item\label{i: B_M(x_i,r)} There is some countable subset $\{\,x_i\mid i\in I\,\}\subset M$ and some $c\in\N$ such that the balls $B_M(x_i,r/2)$ cover $M$, and, for all $x\in M$, $B_M(x,r)$ meets at most $c$ balls $B_M(x_i,r)$ \cite[A1.2 and~A1.3]{Shubin1992}, \cite[Proposition~3.2]{Schick2001}. Let $\kappa_i=\kappa_{x_i}$.
			
	\end{enumerate}
Consider the graph $G$ with $V(G)=I$, and such that there is an edge connecting two different vertices, $i$ and $j$, if and only if $B_M(x_i,r)\cap B_M(x_j,r)\ne\emptyset$. By~\eqref{i: B_M(x_i,r)}, the vertex degrees are uniformly bounded by $c$. So there is a coloring of $G$ by $c+1$ colors so that adjacent vertices have different colors. This means that there is a partition of $I$ into finitely many sets, $I_1,\ldots, I_{c+1}$, such that $B_M(x_i,r)\cap B_M(x_j,r)=\emptyset$ for $i\in I_k$ and $j\in I_l$ with $k\ne l$. On the other hand, by Theorem~\ref{t: there exists alpha}, $G$ has a limit-aperiodic vertex coloring $\alpha:I\to\{1,\dots,c\}$. Let $\alpha_i=\alpha(x_i)$.

Let $S$ be an isometric copy in $\R^{n+1}$ of the standard $n$-dimensional sphere so that $0\in S$. Choose some radial\footnote{A function of the radius in polar coordinates.} function $\rho\in C^\infty(\R^n)$ such that $0\le\rho\le1$, $\rho(x)=1$ if $|x|\le r/2$ and $\rho(x)=0$ if $|x|\ge r$. Take also some $C^\infty$ map $\tau:\R^n\to\R^{n+1}$ that restricts to a diffeomorphism $B_r\to S\sm\{0\}$ and maps $\R^n\sm B_r$ to $0$. Let $\rho_i=\rho\circ \kappa_i^{-1}$ and $\tau_i=\tau\circ\kappa_i^{-1}$. For $k=1,\dots,c+1$, define $f^k:M\to\R^{n+2}$ by
	\[
		f^k(x)=
			\begin{cases}
					0 & \text{if $x\notin \bigcup_{i\in I_k}B_M(x_i,r)$}\\
					\left(\rho_i(x)\cdot\alpha_i,\rho_i(x)\cdot\tau_i(x)\right)
					& \text{if $x\in B_M(x_i,r)$ for some $i\in I_k$}\;.
				\end{cases}
	\]
Fix a linear injection $\R^{(c+1)(n+2)}\subset\E$. Then
	\[
		f=(f^1,\dots,f^{c+1}):M\to\R^{(c+1)(n+2)}\subset\E
	\]
is a $C^\infty$ immersion, and therefore it is locally non-periodic. Moreover $f$ is of bounded geometry and limit-aperiodic, as follows from~\eqref{i: g_ij}, and from the bounded geometry and limit-aperiodicity of $\alpha$.

If $M$ is repetitive, then this property can be easily used to choose the points $x_i$ so that the pair $(M,\{x_i\})$ is \emph{repetitive} in an obvious sense (as a Riemannian manifold with a distinguished subset). With this condition, $G$ is repetitive, and $\alpha$ can be assumed to be repetitive by Theorem~\ref{t: there exists alpha}. It follows that $f$ is also repetitive, showing Proposition~\ref{p: there exists f}. 

Smaller subspaces, $\widehat\MM_{*,\text{\rm imm}}^\infty(n)\subset\widehat\MM_{*,\text{\rm lnp}}^\infty(n)$ and $\widehat\MM_{*,\text{\rm emb}}^\infty(n)\subset\widehat\MM_{*,\text{\rm np}}^\infty(n)$, are defined by requiring the distinguished functions to be $C^\infty$ immersions or $C^\infty$ embeddings. It turns out that $\widehat\MM_{*,\text{\rm imm}}^\infty(n)$ is Polish and dense in $\widehat\MM_*^\infty(n)$ \cite[Theorem~1.4]{AlvarezBarral2017}. Thus we get a $C^\infty$ and Riemannian foliated subspace, $\widehat\MM_{*,\text{\rm imm}}(n)\equiv(\widehat\MM_{*,\text{\rm imm}}^\infty(n),\widehat\FF_{*,\text{\rm imm}}^\infty(n))$, where $\widehat\MM_{*,\text{\rm emb}}^\infty(n)$ is a union of leaves without holonomy. In fact, we can use the distinguished immersions to define its foliated charts more easily \cite[Theorem~1.4]{AlvarezBarral2017}. Then there is some $h\in C^\infty(M,\E)$ such that $\widehat{\Cl}_\infty(\im\hat\iota_{M,h})$ is a (minimal) compact subspace of $\widehat\MM_{*,\text{\rm emb}}^\infty(n)$. This slight sharpening of Proposition~\ref{p: there exists f} can be easily proved as follows. Let $f\in C^\infty(M,\E)$ be given by Proposition~\ref{p: there exists f}, inducing the foliated space $X=\widehat{\Cl}_\infty(\im\hat\iota_{M,f})$. Then there is a $C^\infty$ embedding $\tilde h:X\to\E$ \cite[Theorem~11.4.4]{CandelConlon2000-I}, and the function $h=\hat\iota_{M,f}^*\tilde h\in C^\infty(M,\E)$ satisfies the above property. However  $\widehat{\Cl}_\infty(\im\hat\iota_{M,h})\equiv X$ (no new foliated space is produced with this sharpening).

Distinguished subsets of Riemannian manifolds can be used instead of distinguished functions to construct a Riemannian foliated space, producing also compact Riemannian foliated spaces with a prescribed leaf \cite{AbertBiringer-unimodular}.

\section{Graph matchbox manifolds}\label{s: graph matchbox mfds}

We explain the method of taking the boundary of a thickening, which transforms graph spaces into Riemannian matchbox manifolds \cite{Ghys1999}, \cite{Blanc2001}, \cite{LozanoRojo2007}, \cite{LozanoRojo2008}, \cite{AlcaldeLozanoMacho2009}, \cite{Lukina2012}, \cite{LozanoLukina2013}.

Given $c\in\N$, for every $i=0,\dots,c$, let $\Sigma_i$ be the compact orientable connected surface of genus zero and $i$ boundary components; i.e., $\Sigma_i$ is the $2$-sphere with $i$ disjoint open disks taken out. For any subset $Q\subset\{1,\dots,c\}\times\{0,\pm1\}$ with $i$ elements, fix a bijective map $\sigma_Q:\pi_0(\partial\Sigma_i)\to Q$. Choose an orientation of every $\Sigma_i$, and consider the induced orientation on every $C\in\pi_0(\partial\Sigma_i)$. Equip every $\Sigma_i$ with a Riemannian metric so that every $C\in\pi_0(\partial\Sigma_i)$ has a compact collar neighborhood isometric to the flat cylinder $\R/\Z\times[0,1/2]$, and these collar neighborhoods are disjoint from each other. Fix a distinguished point $p_i$ in the complement of the union of these collar neighborhoods.

Now, consider edge-colored partially directed graphs, $(G,\beta,\OO)$, whose vertex degrees are uniformly bounded by $c$, with $\beta:E(G)\to\{1,\dots,c\}$, and such that the property of Remark~\ref{r: determining the edges with the same vertex} is satisfied. They define a compact transitive saturated subspace $\GG'_{*,+}(c)\subset\GG'_{*,+}$. For $[G,\beta,\OO,v]\in\GG'_{*,+}(c)$ and any connected subgraph $H\subset G$, the element $[H,\beta|_{E(H)},\OO|_{E(H)\cap\dom\OO},v]$ is simply denoted by $[H,\beta,\OO,v]$. Consider the finite subset 	\[
		\bfA=\{\,[G,\beta,\OO,v]\in\GG'_{*,+}(c)\mid\ol B_G(v,1)=G\,\}\subset\GG'_{*,+}(c)\;.
	\]
For any $\bfa=[G,\beta,\OO,v]\in\bfA$, let $i_{\bfa}=\deg_G(v)$, which depends only on $\bfa$. There is a continuous surjective map $\pi:\GG'_{*,+}(c)\to\bfA$ defined by $\pi([G,\beta,\OO,v])=[\ol B_G(v,1),\beta,\OO,v]$, whose fibers, $\ZZ_{\bfa}=\pi^{-1}(\bfa)$ ($\bfa\in\bfA$), form a finite partition of $\GG'_{*,+}(c)$ by clopen subsets. 

For every $\bfa\in\bfA$, let $\Sigma_{\bfa}$ be an isometric copy of $\Sigma_{i_{\bfa}}$ with the corresponding orientation, and let $p_{\bfa}\in\Sigma_{\bfa}$ be the point that corresponds to $p_{i_{\bfa}}$. For any edge $e$ of $G$ with vertex $v$, let
	\[
		\tau_{G,\OO,v}(e)=
			\begin{cases}
				0 & \text{if $e\notin\dom\OO$}\\
				1 & \text{if $e\in\dom\OO$ and $v$ is the first vertex in $\OO(e)$}\\
				-1 & \text{if $e\in\dom\OO$ and $v$ is the second vertex in $\OO(e)$}\;.
			\end{cases}
	\]
This defines a map $\tau_{G,\OO,v}:E(\ol B_G(v,1))\to\{0,\pm1\}$. The map
	\[
		(\beta,\tau_{G,\OO,v}):E(\ol B_G(v,1))\to\{1,\dots,c\}\times\{0,\pm1\}
	\]
is injective by the property of Remark~\ref{r: determining the edges with the same vertex}, and therefore its image, $Q_{\bfa}$, has $i_{\bfa}$ elements. Then, for every $(k,\epsilon)\in Q_{\bfa}$, let $C_{\bfa}^{k,\epsilon}$ be the connected component of $\partial\Sigma_{\bfa}$ that corresponds to the connected component $C$ of $\partial\Sigma_{i_{\bfa}}$ with $\sigma_{Q_{\bfa}}(C)=(k,\epsilon)$. 

For $\bfa,\bfb\in\bfA$, if $\bfz\in\ZZ_{\bfa}$ is contiguous in $\GG'_{*,+}$ to $\bfz'\in\ZZ_{\bfb}$, then there is a unique $(k,\epsilon)\in Q_{\bfa}$ with $(k,-\epsilon)\in Q_{\bfb}$. For $(k,\epsilon)\in Q_{\bfa}$ with $(k,-\epsilon)\in Q_{\bfb}$, fix an orientation reversing isometry $h_{\bfa,\bfb}^{k,\epsilon}:C_{\bfa}^{k,\epsilon}\to C_{\bfb}^{k,-\epsilon}$ so that $h_{\bfb,\bfa}^{k,-\epsilon}=(h_{\bfa,\bfb}^{k,\epsilon})^{-1}$. Then let $\MM\GG'_{*,+}(c)$ be the quotient of $\bigsqcup_{\bfa}\Sigma_{\bfa}\times\ZZ_{\bfa}$ by gluing every $(x,\bfz)\in C_{\bfa}^{k,\epsilon}\times\ZZ_{\bfa}$ with $(h_{\bfa,\bfb}^{k,\epsilon}(x),\bfz')\in C_{\bfb}^{k,-\epsilon}\times\ZZ_{\bfb}$ if $\bfz$ and $\bfz'$ are contiguous, $(k,\epsilon)\in Q_{\bfa}$ and $(k,-\epsilon)\in Q_{\bfb}$. The trivial Riemannian foliated structure on $\bigsqcup_{\bfa}\Sigma_{\bfa}\times\ZZ_{\bfa}$, defined by the fibers $\Sigma_{\bfa}\times\{\bfz\}$ ($\bfz\in\ZZ_{\bfa}$), can be projected to $\MM\GG'_{*,+}(c)$, which becomes a compact oriented Riemannian foliated space. There is an embedding $\GG'_{*,+}(c)\to\MM\GG'_{*,+}(c)$, assigning to every $\bfz\in\GG'_{*,+}(c)$ the projection to $\MM\GG'_{*,+}(c)$ of $(\bfz,p_{\bfa})\in\Sigma_{\bfa}\times\ZZ_{\bfa}$, where $\bfa=\pi(\bfz)$. This embedding realizes $\GG'_{*,+}(c)$ as a complete transversal of $\MM\GG'_{*,+}(c)$, and the canonical partition of $\GG'_{*,+}(c)$ is given by the orbits of the holonomy pseudogroup. Thus $\MM\GG'_{*,+}(c)$ is indeed a Riemannian matchbox manifold. Moreover every orbit closure $Z=\Cl'_+(\im\iota'_{G,\beta,\OO})\subset\GG'_{*,+}(c)$ determines a compact transitive saturated subspace $X=\MM Z\subset\MM\GG'_{*,+}(c)$, called a \emph{graph matchbox manifold}, \index{graph matchbox manifold} which is minimal just when $(G,\beta,\OO)$ is repetitive, and its leaves are without holonomy just when $(G,\beta,\OO)$ is limit-aperiodic.

This kind of construction applies as well to the other graph spaces $\TT=\TT(\Gamma)$, $\widehat\TT(c)=\widehat\TT(\Gamma,c)$ and $\{0,1\}^\Gamma$ of Section~\ref{s: TT}, defined by a finitely generated group $\Gamma$ and a finite set $S$ of generators such that $S\cap S^{-1}$ consists of elements of order two. We have to use their canonical partitions into graphs, equipped with edge colorings and partial directions satisfying the property of Remark~\ref{r: determining the edges with the same vertex}. Then we get compact transitive Riemannian oriented matchbox manifolds, $\MM\TT=\MM\TT(\Gamma)$, $\MM\widehat\TT(c)=\MM\widehat\TT(\Gamma,c)$ and $\MM\{0,1\}^\Gamma$, satisfying similar properties with respect to $\TT$, $\widehat\TT(c)$ and $\{0,1\}^\Gamma$. As before, we use the notation $X=\MM Z$ for the compact transitive saturated subspace of $\MM\TT$, $\MM\widehat\TT(c)$ or $\MM\{0,1\}^\Gamma$ that corresponds to any transitive compact subspace $Z$ of $\TT$, $\widehat\TT(c)$ or $\{0,1\}^\Gamma$. The term \emph{graph matchbox manifold} is also used for $X$ in the case of $\MM\TT$ and $\MM\widehat\TT(c)$, and the term \emph{subshift matchbox manifold} \index{subshift matchbox manifold} is used for $X$ in the case of $\MM\{0,1\}^\Gamma$.

Note that $\MM\{0,1\}^\Gamma$ is diffeomorphic to the suspension of the left $\Gamma$-action on $\{0,1\}^\Gamma$ using a $\Gamma$-covering of the closed oriented surface of genus two. Similarly, $\MM\TT$ can be considered as the ``suspension'' of its canonical compactly generated pseudogroup $\HH$.

\section{Concrete examples of graph matchbox manifolds}\label{s: examples of graph matchbox mfds}

We give some concrete examples of graph matchbox manifolds, which easily illustrate the theorems stated in Chapter~\ref{c: intro}.

\subsection{The Ghys-Kenyon matchbox manifold}\label{ss: Ghys-Kenyon matchbox mfd}
The \emph{Ghys-Kenyon tree} \index{Ghys-Kenyon!tree} \cite{Ghys1999} is the tree $T_\infty\in\TT(\Z^2)$ defined as the limit of a sequence in $\TT(\Z^2)$ whose first three terms are described in Figure~\ref{fig: Ghys-Kanyon}, where the thick vertex represents the identity element $0\in\Z^2$ (the distinguished vertex). These steps indicate the general procedure to construct the whole of $T_\infty$ by induction. The closure $Z$ of its class in $\TT(\Z^2)$ is the \emph{Ghys-Kenyon graph space}, \index{Ghys-Kenyon!graph space} and $X=\MM Z\subset\MM\TT(\Z^2)$ is called the \emph{Ghys-Kenyon matchbox manifold}. \index{Ghys-Kenyon!matchbox manifold} This example is relevant because, first, it shows that parabolic and hyperbolic Riemann surfaces can be leaves of the same minimal compact foliated space, and, second, it was the first example constructed with this useful method.

The tree $T_\infty$ is $\Z^2$-limit-aperiodic and $\Z^2$-repetitive, and therefore $X$ is minimal and its leaves have no holonomy. The graph $T_\infty$ has four ends, and the corresponding leaf in $X$ also has four ends. Hence, by Corollary~\ref{c: end space of leaves in X_0}, every leaf of $X$ is coarsely quasi-isometric to meagerly many leaves (the alternative~\eqref{i: there are uncountably many coarse quasi-isometry types of leaves} of Theorem~\ref{t: coarsely q.i. leaves 2}); in particular, there are uncountably many coarse quasi-isometry types. This also follows using Theorem~\ref{t: coarsely q.i. leaves 3}, since $T_\infty$ is not coarsely quasi-symmetric; indeed, if $v_n$ is an unbounded sequence of vertices in $T_\infty$, and $\phi_n$ is a pointed coarse quasi-isometry of $(T_\infty,0)$ to $(T_\infty,v_n)$ with coarse distortion $(K_n,C_n)$, then it is easy to check that the sequence $(K_n,C_n)$ is unbounded. 

\begin{figure}[h]
\centering
\includegraphics[width=8cm]{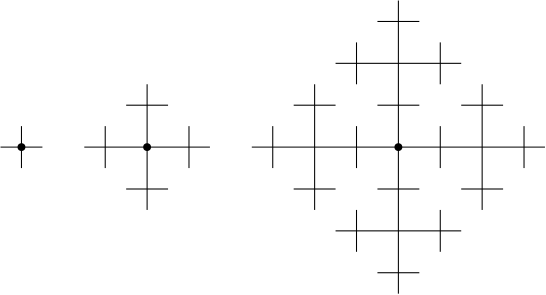}
\caption{Ghys-Kanyon tree $T_\infty$}
\label{fig: Ghys-Kanyon}
\end{figure}

A coding introduced by Alcalde, Lozano and Macho establishes a Borelian open bijection \cite[Proposition~3.3.2 and~3.4.1]{AlcaldeLozanoMacho2009}
	\[
		\Phi:\{0,1,2,3\}^\N\equiv\Z_4^\N\to\{\,T\in Z\sm\{T_\infty\}\mid\deg_T(0)\le2\,\}\;;
	\]
for instance, Figures~\ref{fig: 321} and~\ref{fig: 020} describe the processes to construct trees whose codings begin with $321$ and $020$, which clarifies the general procedure. Note that the saturation of $\im\Phi$ is $Z\sm\{T_\infty\}$. The classes of $\Phi(020202\dots)$ and $\Phi(131313\dots)$ are graphs with two ends, and all other classes in $Z\sm\{T_\infty\}$ are graphs with one end. For $\alpha,\beta\in\Z_4^\N$, the classes of $\Phi(\alpha)$ and $\Phi(\beta)$ are equal if and only if $\alpha$ and $\beta$ are eventually equal \cite[Proposition~3.2.1]{AlcaldeLozanoMacho2009}. On the other hand, it is easy to check that the classes of $\Phi(\alpha)$ and $\Phi(\beta)$ are coarsely quasi-isometric if and only if $\alpha$ and $\beta$ are eventually equal up to permutations of $\Z_4$ defined by orthogonal transformations of $\R^2$ that preserve the set $\{(1,0),(0,1),(-1,0),(0,-1)\}\equiv\Z_4$, where this identity is given by $(1,0)\mapsto0$, $(0,1)\mapsto1$, $(-1,0)\mapsto2$ and $(0,-1)\mapsto3$. This confirms that $X$ has uncountable many coarse-quasi-isometry types.

\begin{figure}[h]
\centering
\includegraphics[width=8cm]{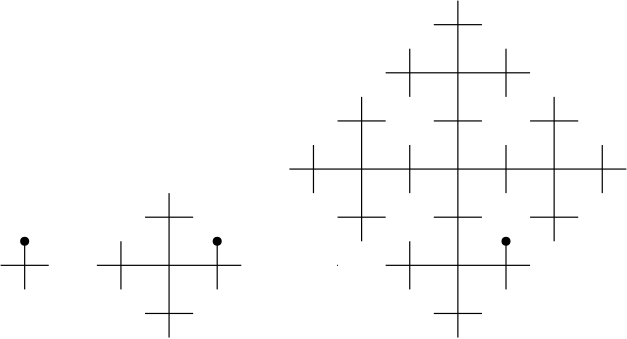}
\caption{Tree with coding $321\dots$}
\label{fig: 321}
\end{figure}

\begin{figure}[h]
\centering
\includegraphics[width=8cm]{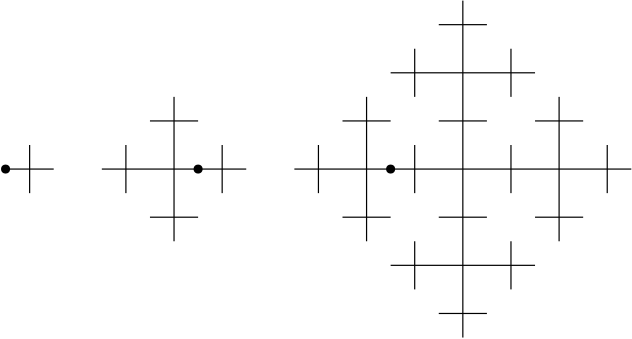}
\caption{Tree with coding $020\dots$}
\label{fig: 020}
\end{figure}

The growth of the leaves of $X$ is equi-equivalent and quadratic, which can be checked as follows. For all $T\in Z$, since the distances of $\Z^2$ and $T$ satisfy the relation $d_{\Z^2}\le d_T$ on $T$, we get $|\ol B_T(v,R)|\le|\ol B_{\Z^2}(v,R)|$ for all $R>0$ and $v\in T$. If $T$ is not canonically equivalent to $\Phi(aaa\dots)$ for any $a\in\Z_4$, then $T$ is a $1$-net in $\Z^2$, and therefore there is some $C>0$ such that $|B_{\Z^2}(v,R)|\le C|T\cap \ol B_{\Z^2}(v,R)|$. However the restriction of $d_{\Z^2}$ to $T$ is not Lipchitz equivalent to $d_T$. But we have $d_{\Z^2}=d_T$ on at least half of the points of $\ol B_{\Z^2}(v,R)$. Thus
	\[
		|\ol B_T(v,R)|\ge\frac{1}{2}|T\cap\ol B_{\Z^2}(v,R)|\ge\frac{1}{2C}|B_{\Z^2}(v,R)|\;.
	\]
On the other hand, all trees $\Phi(aaa\dots)$ ($a\in\Z_4$) are isometric to each other, as well as the trees in their classes. Thus, to check the remaining cases, we can assume $T=\Phi(000\dots)$. In this case, $T$ is a $1$-net in the quadrant
	\[
		Q=\{\,(v_1,v_2)\in\Z^2\mid|v_2|\le v_1\,\}
	\]
of $\Z^2$. Then we can repeat the above argument using $Q$ instead of $\Z^2$.

Many variants of the Ghys-Kennyon example were given by other authors \cite{Blanc2001}, \cite{LozanoRojo2007}, \cite{LozanoRojo2008}, \cite{Lukina2012}. The coarse quasi-isometry type of the leaves can be analyzed in all of them like in the Ghys-Kennyon example.

\subsection{An example with uncountably many growth types of leaves}\label{ss: ex with uncountably many growth types}
Let $F_3$ be the free group with three generators. We are going to construct a repetitive limit-aperiodic tree $T_\infty\in\TT(F_3)$ that is not growth symmetric. Once $T_\infty$ is constructed, let $Z$ denote the closure of the class of $T_\infty$ in $\TT(F_3)$, and let $X=\MM Z\subset\MM\TT(F_3)$. Then $X$ is minimal, its leaves have no holonomy, and the leaf that corresponds to $T_\infty$ is not growth symmetric. By Theorem~\ref{t:growth 3}, it follows that the growth type of each leaf of $X$ is comparable with the growth type of meagerly many leaves (the alternative~\eqref{i: there are uncountably many growth types of leaves in X_0,d} of Theorem~\ref{t:growth 2}); in particular, $X$ has uncountably many growth types of leaves.  

For the construction of $T_\infty$, we adapt the method of \cite{Blanc2001}\footnote{Be aware that thet Caley graph structures of \cite{Blanc2001} are left invariant, and ours are right invariant.}. Fix generators $a$, $b$ and $c$ of $F_3$. We identify the subgraphs of $F_3$ with their vertex set. Take strictly increasing sequences $r_n$ in $\Z^+$, and let $s_n=r_1\cdots r_n$ and $D_n=\ol B_{\langle b,c\rangle}(1,n)$. Consider the trees $T_0=\{a,a^2,\dots\}$ and
	\[
		T_1=T_0\cup\bigcup_{n\ge1}D_1a^{ns_1}\;.
	\]

For any segment $I\subset T_0$ and any tree $T$ containing $T_0$, the tree $T\cap F_3\{b^{\pm1},c^{\pm1}\}I$ is called the \emph{$\langle b,c\rangle$-saturation} \index{$\langle b,c\rangle$-saturation} of $I$ in $T$.

Beginning with the above trees $T_0\subset T_1$ and proceeding by induction, we continue constructing an increasing sequence, $T_0\subset T_1\subset\dots$, in $\TT(F_3)$. For some $n\ge1$, assume that we have constructed $T_0\subset T_1\subset\dots\subset T_n$ such that:
	\begin{itemize}
	
		\item they depend only on the terms $r_1,\dots,r_n$ of the above sequence;
		
		\item for all $i=0,1,\dots,n-1$, the segment $I_i:=\{a,\dots,a^{s_i}\}$ has the same $\langle b,c\rangle$-saturation in $T_i$ and in $T_{i+1}$, which is denoted by $S_i$;
		
		\item  for all $i=0,1,\dots,n$ and $k\in\Z^+$, the $\langle b,c\rangle$-saturation of the segment $I_ia^{(k-1)s_i}$ in $T_i$ is $S_ia^{ks_i}$; and,
		
		\item if $r_1,\dots,r_n$ are large enough, we can get
			\begin{equation}\label{|overline B_T_i(1,r)| le sum_j=0^i 2^-j r}
				|\ol B_{T_i}(a,r)|\le\sum_{j=0}^i2^{-j}r
			\end{equation}
		for all $i=0,1,\dots,n$ and $r>0$. 
	\end{itemize}
These properties are obviously satisfied for $n=1$. Then
	\[
		T_{n+1}=T_n\cup\bigcup_{k=1}^\infty S_{n-1}D_{n+1}a^{ks_{n+1}}
	\]
satisfies the desired properties; in particular, taking $r_{n+1}$ large enough, we can assume that~\eqref{|overline B_T_i(1,r)| le sum_j=0^i 2^-j r} also holds for $i=n+1$ and all $r>0$.

Now, it is easy to check that $T_\infty=\bigcup_{n=0}^\infty T_n$ is an $F_3$-repetitive $F_3$-limit-aperiodic tree in $\TT(F_3)$, which satisfies
	\begin{equation}\label{|overline B_T_infty(a,r)| le 2r}
		|\ol B_{T_\infty}(a,r)|\le2r
	\end{equation}
for all $r>0$ by~\eqref{|overline B_T_i(1,r)| le sum_j=0^i 2^-j r}. But, for any $n\in\Z^+$, the ball $\ol B_{T_\infty}(a^{s_n},n)$ contains $D_na^{s_n}$, which is isometric to $D_n$ using the right translation by $a^{s_n}$. Thus
	\begin{equation}\label{|overline B_T_infty(a^s_n,r)| ge 3^r}
		|\ol B_{T_\infty}(a^{s_n},r)|\ge|\ol B_{\langle b,c\rangle}(1,r)|=1+2(3^r-1)\ge3^r
	\end{equation}
for $0\le r\le R_n$. Since $R_n\uparrow\infty$, comparing~\eqref{|overline B_T_infty(a,r)| le 2r} and~\eqref{|overline B_T_infty(a^s_n,r)| ge 3^r} it follows that $T_\infty$ cannot be growth symmetric.

\subsection{An example with equi-amenable dense leaves and other non-amenable leaves}\label{ss: ex with equi-amenable dense leaves}

In Sections~\ref{ss: Ghys-Kenyon matchbox mfd} and~\ref{ss: ex with uncountably many growth types}, it is easy to see that $T_\infty$ is amenable, and indeed all trees in $Z$ are jointly amenably symmetric. Hence all leaves of $X$ are jointly amenably symmetric.

Now, with the notation of Section~\ref{ss: ex with uncountably many growth types}, construct an increasing sequence, $T'_0\subset T'_1\subset\dots$, in $\TT(F_3)$ by induction, setting $T'_0=T_0$ and
	\[
		T'_{n+1}=T'_n\cup\bigcup_{k=1}^\infty D_{n+1}a^{ks_{n+1}}\,.
	\]
 Let $T'_\infty=\bigcup_nT'_n\in\TT(F_3)$, let $Z'$ denote the closure of its class in $\TT(F_3)$, and $X'=\MM Z'\subset\MM\TT(F_3)$. The tree $T'_\infty$ is $F_3$-non-periodic, but it is neither $F_3$-limit-aperiodic nor $F_3$-repetitive; for instance, $Z'$ contains the non-dense class that consists only of the tree $\langle b,c\rangle$. This class corresponds to a compact leaf with holonomy of $X'$.

By Theorem~\ref{t: there exists alpha}, there exists some limit-aperiodic coloring $\alpha$ of $T'_\infty$ by $6$ colors, which is also $F_3$-limit-aperiodic. But $(T_\infty,\alpha)$ is not $F_3$-repetitive because $T_\infty$ is not $F_3$-repetitive. Let $\widehat Z'$ be the closure of the class of $(T_\infty,\alpha)$ in $\widehat\TT(F_3,6)$, and let $\widehat X'=\MM\widehat Z'\subset\MM\widehat\TT(F_3)$. Now, the leaves of $\widehat X'$ have no holonomy, and all dense leaves of $\widehat X'$ are equi-amenable by Theorem~\ref{t: amenable}-\eqref{i: if some leaf in X_0 is amenable, then all leaves in X_0,d are equi-amenable}. But there are colored trees in $\widehat Z'$ of the form $(\langle b,c\rangle,\alpha')$ for some $F_3$-limit-aperiodic colorings $\alpha'$, whose classes are not dense in $\widehat Z'$. Since $F_2$ is not F\o lner, the corresponding leaves are not amenable. This does not contradict Theorem~\ref{t: amenable}-\eqref{i: if some leaf in X_0 is amenable, then all leaves in X_0,d are equi-amenable} because these leaves are not dense.

\subsection{An example with with leaves of infinite asymptotic dimension}\label{ss: ex with leaves of infinite as dim}

With the notation of Section~\ref{s: graph matchbox mfds}, in $\MM\TT(\Gamma)$ and $\MM\widehat\TT(\Gamma,c)$, all unbounded leaves have asymptotic dimension one, and all  bounded leaves have asymptotic dimension zero. This holds because the asymptotic dimension of trees is one in the unbounded case, and zero in the bounded case \cite{BellDranishnikov2011}.

To produce an example of graph matchbox manifold with leaves of infinite asymptotic dimension, take a finitely generated group $\Gamma$ with infinite asymptotic dimension, like the reduced wreath product of $\Z$ by $\Z$ \cite[Section~5]{BellDranishnikov2011}. Then, in the matchbox manifold $\MM\{0,1\}^\Gamma$, the leaves corresponding to non-periodic colorings are coarsely quasi-isometric to $\Gamma$, of infinite asymptotic dimension, and the leaves corresponding to constant colors are of zero asymptotic dimension because they are bounded. 

Take a $\Gamma$-repetitive $\Gamma$-limit-aperiodic coloring $\alpha\in\{0,1\}^\Gamma$ (Section~\ref{s: TT}). Then the subshift matchbox manifold $X=\MM\ol{\Gamma\cdot\alpha}\subset\MM\{0,1\}^\Gamma$ is minimal and without holonomy. Moreover the leaves of $X$ are coarsely quasi-isometric to $\Gamma$, of infinite asymptotic dimension.

\section{Foliations of codimension one}\label{s: codim 1}

Now, we consider only transversely oriented $C^2$ foliations of codimension one on closed manifolds, unless otherwise stated. Then, unlike Theorem~\ref{t: realization}, there are examples of connected Riemannian manifolds of bounded geometry whose quasi-isometry type cannot be realized as leaves \cite{AttieHurder1996}, \cite{Zeghib1994}, \cite{Schweitzer1995}, \cite{Schweitzer2011}. If the metric is not considered, any surface can be realized as a leaf of a codimension one foliation on a closed manifold \cite{CantwellConlon1987}, but this fails in higher dimension \cite{Ghys1985}, \cite{InabaNishimoriTakamuraTsuchiya1985}, \cite{AttieHurder1996}, \cite{Souza2011}, \cite{SchweitzerSouza2013}.

\subsection{Hector's example}\label{ss: Hector's ex}

Concerning the growth of the leaves, let us recall the construction of a striking example of a transitive $C^\infty$ foliation given by Hector \cite{Hector1977b}, whose properties were indicated in Chapter~\ref{c: intro}.  This example was originally described in his previous paper \cite{Hector1976}, which is a great source of examples that can be similarly analyzed in terms of our main theorems.

By using the suspension construction (Section~\ref{s: suspensions}), it is enough to describe a finitely generated group of orientation preserving $C^\infty$ diffeomorphisms of $S^1$; indeed, this is achieved by describing a finitely generated group $G$ of orientation preserving $C^\infty$ diffeomorphisms of $[-1,1]$ that are flat\footnote{They have the same derivatives of any order as the identity map at those points.} at $\pm1$. Let $E([-1,1])$ be the set of orientation preserving $C^\infty$ diffeomorphisms $f$ of $[-1,1]$ such that $f$ is flat at $\pm1$, and the support\footnote{The closure of the set of points where $f\ne\id$.} of $f$ is an interval of the form $[\bar a,a]$, with $\bar a<0<a$, so that $f'>1$ on $(\bar a,0)$, and $f'<1$ on $(0,a)$. There exists a sequence $k_n$ ($n\in\N$) in $E([-1,1])$, with corresponding supports $[\bar a_n,a_n]$, such that:
	\begin{itemize}
	
		\item $[\bar a_0,a_0]=[-1,1]$, $a_{n+1}=k_n(\bar a_{n+1})$ and $\bigcap_n[\bar a_n,a_n]=\{0\}$; and,
		
		\item if $h=k_0$, $k_{(n)}=h^{-n}k_nh^n$ and $l_n=k_{(n)}\cdots k_{(1)}$ ($n\ge1$), then $l=\lim_nl_n$ is a orientation preserving $C^\infty$ diffeomorphism of $[-1,1]$ which is flat at $\pm1$.
		
	\end{itemize}
Let $G$ is the group generated by $\Sigma=\{h^{\pm1},l^{\pm1}\}$. Every trajectory $T_n=G\cdot a_n$ is proper, and, for all $x\in\Omega=[-1,1]\sm\bigcup_n\ol{T_n}$, the trajectory $T_x=G\cdot x$ is dense.

Consider the length $|\cdot|$ and right invariant metric on $G$ defined by $\Sigma$, and consider the induced metric $\delta$ on every trajectory $T_x\equiv G/G_x$. A \emph{short cut} \index{short cut} from $x$ to $y$ in $T_x$ is an element $g\in G$ such that $y=g(x)$ and $|g|=\delta(x,y)$. The notation $\Gamma_x$ is used for the set of short cuts from $x$. The short cuts satisfy the following properties:
	\begin{itemize}

		\item For all $x\in[-1,1]$ and $y\in T_x$, there is a unique short cut $g_{x,y}$ from $x$ to $y$. 
		
		\item Any non-trivial $g\in G$ is in $\Gamma_n=\Gamma_{a_n}$ if and only if there exists some $p\in\Z^+$ and $(\alpha_1,\dots,\alpha_p)\in\Z^p$ such that:
			\begin{itemize}
	
				\item $g=h^{\alpha_1}$ if $p=1$; and
		
				\item $g=h^{\alpha_1}l^{\alpha_q}h\cdots hl^{\alpha_p}h^{-(p-1)}$ with $1<q\le p$, $\alpha_q\alpha_p\ne0$, and $\alpha_j=0$ for $1<j<q$, if $p>1$.
				
			\end{itemize}
		
		\item For all $u,v\in T_n$, we have $g_{u,v}=g_vg_u^{-1}$, where $g_u=g_{a_n,u}$.
		
		\item For all $x\in\Omega$ and $y\in T_x$, there exist $n\in\N$ and $u,v\in T_n$ such that $g_{x,y}=g_{u,v}$.
		
	\end{itemize}
Thus $\Gamma_x\equiv T_x\equiv G/G_x$ as graphs; in particular, $\gr(\Gamma_x)=\gr(T_x)=\gr(G/G_x)$. It also follows that, for each connected component $(\bar u,u)$ of $[-1,1]\sm\ol{T_n}$, there is a unique $g=h^{\alpha_1}l^{\alpha_q}h\cdots hl^{\alpha_p}h^{-(p-1)}\in\Gamma_n$ such that $(\bar u,u)=g(\bar a_n,a_n)$. This procedure assigns an $n$-tuple $(\alpha_1,\dots,\alpha_n)$ to such $(\bar u,u)$, taking $\alpha_j=0$ for $p<j\le n$ if $p<n$. If $(\bar u',u')\subset(\bar u,u)$ is a connected component of $[-1,1]\sm\ol{T_{n+1}}$, then the corresponding $(n+1)$-tuple $(\alpha'_1,\dots,\alpha'_{n+1})$ satisfies $\alpha'_j=\alpha_j$ for $1\le j\le n$. Since any point in $\Omega$ is the intersection of a decreasing sequence of intervals that are connected components of the complements $[-1,1]\sm\ol{T_n}$, for all $n$, this procedure defines a bijection $\Phi:\Omega\to\Z^{\Z^+}$, which turns out to be a homeomorphism. 

Considering $\Phi$ as an identity, for $x=(x_n)\in\Omega$ and $p\in\Z^+$, write $X_p=\sum_{n=1}^p|x_n|$. It is said that $x$ is \emph{weakly dominated} \index{weak domination} by another point $y=(y_n)\in\Omega$ if there is some $A\in\Z^+$ such that $X_p\le A(p+Y_p)$ for all $p$. If $x$ and $y$ weakly dominate each other, then they are said to be \emph{weakly equivalent}. Every weakly equivalence class has the cardinality of the continuum because it does not change by adding elements of $\{0,1\}^{\Z^+}$. The following properties hold:
	\begin{itemize}
	
		\item $T_x=T_y$ if and only if $(x_n)$ and $(y_n)$ are eventually equally.
		
		\item If $x$ is weakly dominated by $y$, then $\gr(\Gamma_x)\ge\gr(\Gamma_y)$.
	
	\end{itemize}
The properties indicated in Chapter~\ref{c: intro} follow by pursuing further this kind of concepts and arguments. More precisely, the following properties hold: 
	\begin{itemize}

		\item Every growth class of metric spaces $\Gamma_x$ ($x\in\Omega$) has the cardinality of the continuum. 
		
		\item Every $T_n$ has polynomial growth of exact degree $n$, and the growth of $T_x$ is non-polynomial for all $x\in\Omega$.
		
		\item $\Gamma_0$ has exponential growth.
		
		\item For $r>0$, let $x(r)=(x_n(r))\in\Omega$ with $x_n(r)=\lfloor n^r\rfloor$. Then $\Gamma_r=\Gamma_{x(r)}$ has non-polynomial and quasi-polynomial growth. Moreover $\gr(\Gamma_r)<\gr(\Gamma_s)$ if $r>s$, and therefore there is a continuum of growth classes of this kind. 
		
		\item For $r>0$, let $\tilde x(r)=(\tilde x_n(r))\in\Omega$ with
			\[
				\tilde x_n(r)=
					\begin{cases}
						X_p(r)-X_{p-1}(r) & \text{if $n=X_p(r)=\sum_{n=1}^p\lfloor n^r\rfloor$}\\
						0 & \text{if $n\notin\{\,X_p(r)\mid p\in\Z^+\,\}$}\;.
					\end{cases}
			\]
		Then $\widetilde\Gamma_r=\Gamma_{\tilde x(r)}$ has non-quasi-polynomial and non-exponential growth. Moreover $\gr(\widetilde\Gamma_r)<\gr(\widetilde\Gamma_s)$ if $r>s$, and therefore there is a continuum of growth classes of this kind. 
		
	\end{itemize}

We will show that the growth type of every leaf is comparable with the growth type of meagerly many leaves (the alternative~\eqref{i: there are uncountably many growth types of leaves in X_0,d} of Theorem~\ref{t:growth 2}); in particular, there are uncountably many growth types of leaves without holonomy. Actually, a stronger property will be proved in Section~\ref{ss: generic ergodicity in Hector's ex}, whose statement and proof requires concepts and results recalled in Sections~\ref{ss: generic ergodicity of rels} and~\ref{ss: generic ergodicity of metric rels}.

\subsection{Generic ergodicity of equivalence relations}\label{ss: generic ergodicity of rels}

The following concepts are used. The orbits of an action of a group $G$ on a space $X$ define an equivalence relation denoted by $E_G^X$. A metric with possible infinite values\footnote{It is defined like a usual metric, except that the infinite distance between points is possible.} on a space $X$, $d:X\times X\to[0,\infty]$, defines an equivalence relation $E_d^X=d^{-1}([0,\infty))$ on $X$, called a \emph{metric equivalence relation}. Any metric with possible infinite values induces a topology like usual metrics. The \emph{composite} \index{relation!composite} of relations is the obvious extension of the composite of maps. For $E\subset X^2$, $x\in X$ and $S\subset X$, let $E(x)=\{\,y\in X\mid(x,y)\in E\,\}$ and $E(S)=\bigcup_{z\in X}E(z)$. The \emph{identity relation} at $X$ is the diagonal $\Delta_X\subset X^2$.

Let $E$ and $F$ be equivalence relations on respective spaces $X$ and $Y$. A map $\theta:X\to Y$ is called \emph{$(E,F)$-invariant} \index{relation!invariant map} if $(\theta(x),\theta(x'))\in F$ for all $(x,x')\in E$; this means that $\theta$ induces a map $\bar\theta:X/E\to Y/F$. It is said that $E$ is \emph{Borel reducible} \index{relation!Borel reducible} to $F$, written $E\le_BF$, if there is an $(E,F)$-invariant Borel map $\theta:X\to Y$ such that, for all $x,x'\in X$, we have $(x,x')\in E$ if $(\theta(x),\theta(x'))\in F$; this means that $\bar\theta:X/E\to Y/F$ is injective.  If \(E\le_B F\le_BE\), then \(E\) is said to be \emph{Borel bi-reducible} with \(F\), and the notation \(E\sim_BF\) is used. On the contrary, it is said that $E$ is \emph{generically $F$-ergodic} \index{relation!generically ergodic} if, for any $(E,F)$-invariant Baire measurable map $\theta:X\to Y$, there is some residual saturated $C\subset X$ such that $\bar\theta:C/(E\cap C^2)\to Y/F$ is constant.

The partial pre-order relation $\le_B$ establishes a hierarchy on the complexity of equivalence relations on Polish spaces. A first rank of this hierarchy consists of the \emph{concretely classifiable} \index{relation!concretely classifiable} or \emph{smooth} \index{relation!smooth} equivalence relations, defined by the condition of being $\le_B\Delta_\R$; this means that their equivalence classes can be distinguished by some Borel map to $\R$.

Consider the Polish space $\prod_{n=1}^\infty\{0,1\}^{\N^n}$, with the canonical action of the Polish group $S_\infty$ of permutations of $\N$. Each element of $\prod_{n=1}^\infty\{0,1\}^{\N^n}$ can be considered as a structure on $\N$ defined by a sequence of relations $R_n$ with arity $n$ (subsets of $\N^n$); the term \emph{countable model} \index{countable model} is used for $\N$ with this structure. The $S_\infty$-action defines the isomorphism relation $\cong$ between countable models. A second classification rank consists of equivalence relations $\le_B{\cong}$, which are said to be \emph{classifiable by countable models} \index{relation!classifiable by countable models}. On the contrary, we have the generically $\cong$-ergodic relations; indeed, these relations are generically $E_{S_\infty}^Y$-ergodic for any Polish $S_\infty$-space $Y$ \cite{Hjorth2000} (see also  \cite{BeckerKechris1996,KechrisMiller2004}).

Consider the equivalence relation $E_G^X$ defined by a Polish action\footnote{The actionof a Polish group on a Polish space.}. For open neighborhoods, $U$ of a point $x$ in $X$ and $V$ of the identity element in $G$, let $\OO(x,U,V)$ denote the set of points $y\in U$ such that there is a finite sequence $x=x_0,x_1,\dots,x_n=y$ in $U$, for some $n\in\Z^+$, so that $x_i\in V\cdot x_{i-1}$ for all $i=1,\dots,n$. This set $\OO(x,U,V)$ is called a \emph{local orbit}. \index{local orbit} This action is said to be \emph{turbulent} \index{turbulent!action} if its orbits are dense and meager, and its local orbits are somewhere dense\footnote{The closure has nonempty interior.}. Hjorth has introduced this dynamical property, and used it to give a precise characterization of the classification of $E_G^X$ by countable models and its generic $\cong$-ergodicity \cite{Hjorth2000}, \cite{Hjorth2002}.

\subsection{Generic ergodicity of metric equivalence relations}\label{ss: generic ergodicity of metric rels}

We recall our partial extension of Hjorth's work to metric equivalence relations \cite{AlvarezCandel-turbulent}. 

On a Polish space $X$, consider the metric equivalence relation $E_d^X$ defined by a metric $d$ with possible infinite values. For every open neighborhood $U$ of any point $x$ in $X$, and all $\epsilon>0$, let $E_d^X(x,U,\epsilon)$ denote the set of points $y\in U$ such that there is a finite sequence $x=x_0,x_1,\dots,x_m=y$ in $U$, for some $m\in\Z^+$, so that $d(x_{i-1},x_i)<\epsilon$ for all $i=1,\dots,m$. This set $E_d^X(x,U,\epsilon)$ is called a \emph{local equivalence class}. \index{local equivalence class} It is said that $E_d^X$ is \emph{turbulent} \index{turbulent!relation} if its equivalence classes are dense and meager, and its local equivalence classes are somewhere dense. If $d$ is of certain class, called \emph{type~I} in \cite{AlvarezCandel-turbulent}, and $E_d^X$ is turbulent, then $E_d^X$ is generically $E_{S_\infty}^Y$-ergodic for any Polish $S_\infty$-space $Y$ \cite[Theorem~5.5]{AlvarezCandel-turbulent}. 

Now, let $X$ be a set, and let $\UU=\{\,U_{R,r}\subset X^2\mid R,r>0\,\}$ be a set of relations on $X$. The following list of hypotheses are used to define a metric equivalence relation satisfying the above conditionss. 

\begin{hyp}\label{h:U_R,r}
    	\begin{enumerate}[(i)]
  
      		\item\label{i: bigcap_R,r>0 U_R,r = Delta_X} $\bigcap_{R,r>0}U_{R,r}=\Delta_X$;
      
      		\item\label{i: U_R,r symmetric} each $U_{R,r}$ is symmetric;
      
     		\item\label{i: U_R,r supset U_S,r} if $R\le S$, then $U_{R,r}\supset U_{S,r}$ for all $r>0$;
    
      		\item\label{i: U_R,r = bigcup_s<r U_R,s} $U_{R,r}=\bigcup_{s<r}U_{R,s}$ for all $R,r>0$; and
    
      		\item\label{i: phi ...} there is some function $\phi:(\R^+)^2\to\R^+$ such that, for all $R,S,r,s>0$,
        			\begin{gather*} 
				R\le\phi(R,r),\\ 
        				(R\le S,\ r\le s)\Longrightarrow\phi(R,r)\le\phi(S,s),\\ 
				U_{\phi(R,r+s),r}\circ U_{\phi(R,r+s),s}\subset U_{R,r+s}.
        			\end{gather*}
  
    \end{enumerate}
\end{hyp}

Under Hypothesis~\ref{h:U_R,r}, the sets $U_{R,r}$ form a base of entourages for a Hausdorff metrizable uniformity $\mathcal{U}$ on $X$, and, setting $E_r=\bigcap_{R>0}U_{R,r}$ ($r>0$), a metric with possible infinite values, $d:X\times X\to[0,\infty]$, is defined by
	\[
  		d(x,y)=\inf\{\,r>0\mid (x,y)\in E_r\,\}\;.
	\]
	
\begin{hyp}\label{h:type I}
  	\begin{enumerate}[(i)]
      
  		\item\label{i: Polish} $X$ is a Polish space with the topology induced by $\mathcal{U}$;
      
  		\item\label{i: U_T,t(y) subset E_s circ U_R,r(x)} for all $R,r,s>0$ and $x\in X$, if $y\in E_s(x)$, then there are some $T,t>0$ such that $U_{T,t}(y)\subset E_s\circ U_{R,r}(x)$; and,
      
  		\item\label{i: E_r(W) cap E_r(E_s(x)) subset E_r(V cap E_s(x))} for all $r,s>0$ and $(x,y)\in E_s$, and any neighborhood $V$ of $y$ in $X$, there is some neighborhood $W$ of $y$ in $X$ such that
      			\[
      				E_r(W)\cap E_r(E_s(x))\subset E_r(V\cap E_s(x))\;.
      			\]
      
    	\end{enumerate}
 \end{hyp}
 
 Under Hypothesis~\ref{h:type I}, $d$ is of type~I.
 
 \begin{hyp}\label{h:turbulent}
  	\begin{enumerate}[(i)]
    
  		\item\label{i: more than one class} $E_d^X$ has more than one equivalence class;
      
  		\item\label{i: U_R,r(x) cap E_s(y) neq emptyset} for all $x,y\in X$ and $R,r>0$, there is $s>0$ such that $U_{R,r}(x)\cap E_s(y)\neq\emptyset$; and,
      
  		\item\label{i: the local equivalence classes are somewhere dense}  for each \(x\in X\) and each \(R,r>0\), there are \(S,s>0\), a dense subset \(\DD\subseteq U_{S,s}(x)\cap E_d^X(x)\), and a \(d\)-dense subset of \(\DD\) such that every pair of points in \(\DD\) can be joined by a \(d\)-continuous path in \(U_{R,r}(x)\).
    
  	\end{enumerate}
\end{hyp}

Under Hypothesis~\ref{h:turbulent}, $E_d^X$ is turbulent. Thus, assuming Hypotheses~\ref{h:U_R,r}--\ref{h:turbulent}, the relation $E_d^X$ is generically $E_{S_\infty}^Y$-ergodic for any Polish $S_\infty$-space $Y$ \cite[Proposition~6.10]{AlvarezCandel-turbulent}.

\begin{rem}
	Actually, a stronger version of Hypotheses~\ref{h:turbulent}-\eqref{i: the local equivalence classes are somewhere dense} was required in \cite[Hypothesis~3-(iii)]{AlvarezCandel-turbulent}. That condition was simpler to state and check, and it was satisfied in the applications of that paper, and it was used in \cite[Proposition~6.8]{AlvarezCandel-turbulent} to obtain that the local equivalence classes are somewhere dense. But now we need the weaker requirement Hypotheses~\ref{h:turbulent}-\eqref{i: the local equivalence classes are somewhere dense}. Clearly, it also implies that the local equivalence classes are somewhere dense, and therefore this change can be made. 
\end{rem}

\subsection{Generic ergodicity of the growth type relation in Hector's example}\label{ss: generic ergodicity in Hector's ex}

Consider the notation of Section~\ref{ss: Hector's ex}.

\begin{thm}\label{t: Hector's foln is generically ergodic}
	In the Hector's foliation, the relation of being in leaves with the same growth type is generically $E_{S_\infty}^Y$-ergodic for any Polish $S_\infty$-space $Y$.
\end{thm}

Let $\Omega_+=\N^{\Z^+}\subset\Omega$. The weak equivalence relation on $\Omega$ is Borel bi-reducible to its restriction to $\Omega_+$. This can be seen by using the Borel maps $\Omega_+\hookrightarrow\Omega$ and $\Omega\to\Omega_+$, $(x_n)\mapsto(|x_n|)$. In turn, the weak equivalence relation on $\Omega_+$ has an obvious extension to the larger space $\widetilde\Omega_+=[0,\infty)^{\Z^+}$. The weak equivalence relations on $\Omega_+$ and $\widetilde\Omega_+$ are Borel bi-reducible, as can be shown with the Borel maps $\Omega_+\hookrightarrow\widetilde\Omega_+$ and $\widetilde\Omega_+\to\Omega_+$, $(x_n)\mapsto(\lfloor x_n\rfloor)$. So, according to Section~\ref{ss: Hector's ex}, Theorem~\ref{t: Hector's foln is generically ergodic} is a consequence of the following proposition.

\begin{prop}\label{p: weak equivalence on Omega is generically ergodic}
	The weak equivalence relation on $\widetilde\Omega_+$ is generically $E_{S_\infty}^Y$-ergodic for any Polish $S_\infty$-space $Y$.
\end{prop}

Using $\widetilde\Omega_+$ has some advantages in the proof. First, the assignment $(x_n)\mapsto(X_p)$ defines a bijection between $\Omega_+$ (respectively, $\widetilde\Omega_+$) and the set of non-decreasing sequences in $\N$ (respectively, $[0,\infty)$); second, it is easier to construct elements in $\widetilde\Omega_+$ than in $\Omega_+$ or $\Omega$; and, third, $\widetilde\Omega_+$ has nontrivial continuous paths.

Proposition~\ref{p: weak equivalence on Omega is generically ergodic} follows by checking the hypotheses of Section~\ref{ss: generic ergodicity of metric rels} are satisfied by the relations
	\[
		U_{R,r}=\{\,(x,y)\in\widetilde\Omega_+^2\mid e^{-r}(p+Y_p)<p+X_p<e^r(p+Y_p)\ \forall p=1,\dots,\lfloor R\rfloor\,\}\;,
	\]
for $R,r>0$. These sets obviously satisfy Hypothesis~\ref{h:U_R,r}; in particular, its condition~\eqref{i: phi ...} holds with $\phi(R,r)=R$ because
	\begin{equation}\label{U_R,r circ U_S,s subset ...}
		U_{R,r}\circ U_{S,s}\subset U_{\min\{R,S\},r+s}\;.
	\end{equation}
Note that Hypothesis~\ref{h:U_R,r}-\eqref{i: bigcap_R,r>0 U_R,r = Delta_X} would not be true with the same definition of sets $U_{R,r}$ in $\Omega$. The uniformity defined by the sets $U_{R,r}$ induces the topology of $\widetilde\Omega_+$.

According to Section~\ref{ss: generic ergodicity of metric rels}, the sets $U_{R,r}$ induce the relations 
	\[
		E_r=\{\,(x,y)\in\widetilde\Omega_+^2\mid e^{-r}(p+Y_p)<p+X_p<e^r(p+Y_p)\ \forall p\in\Z^+\,\}\;,
	\]
for $r>0$, which in turn induce the metric with possible infinite values, $d:\widetilde\Omega_+^2\to[0,\infty]$, determined by
	\[
		e^{d(x,y)}=\inf\{\,A\in[1,\infty)\mid A^{-1}(p+Y_p)<p+X_p<A(p+Y_p)\ \forall p\in\Z^+\,\}\;.
	\]
Note that $E_d^{\widetilde\Omega_+}=\bigcup_{r>0}E_r$ is the weak equivalence relation on $\widetilde\Omega_+$.

\begin{lemma}\label{l: E_r circ E_s = E_r+s}
	$E_r\circ E_s=E_{r+s}$ for all $r,s>0$.
\end{lemma}

\begin{proof}
	The inclusion ``$\subset$'' follows from~\eqref{U_R,r circ U_S,s subset ...}. To prove ``$\supset$'', take any $(x,y)\in E_{r+s}$. An element $z\in E_r(x)\cap E_s(y)$ is determined by the condition 
		\begin{equation}\label{p + Z_p = (p + X_p)^s/(r+s) (p + Y_p)^r/(r+s)}
			p+Z_p=\left(p+X_p\right)^{\frac{s}{r+s}}\left(p+Y_p\right)^{\frac{r}{r+s}}\;,
		\end{equation}
	for all $p\in\Z^+$. Therefore $(x,y)\in E_r\circ E_s$.
\end{proof}

Hypothesis~\ref{h:type I}-\eqref{i: Polish} is true because $\widetilde\Omega_+$ is Polish. The following lemma shows that Hypothesis~\ref{h:type I}-\eqref{i: U_T,t(y) subset E_s circ U_R,r(x)} is also satisfied.

\begin{lemma}\label{l:U_R,r circ E_s=E_s circ U_R,r=U_R,r+s}
  	For all $R,r,s>0$, $U_{R,r}\circ E_s=E_s\circ U_{R,r}=U_{R,r+s}$.
\end{lemma}

\begin{proof} 
  If $S\ge R$, then, by~\eqref{U_R,r circ U_S,s subset ...},
    	\[
      		(U_{R,r}\circ E_s)\cup(E_s\circ U_{R,r})\subset(U_{R,r}\circ U_{S,s})\cup(U_{S,s}\circ U_{R,r})\subset U_{R,r+s}\;.
    	\]
 
 Let us prove that
 	\[
		U_{R,r+s}\subset(U_{R,r}\circ E_s)\cap(E_s\circ U_{R,r})\;.
	\]  
For any $(x,y)\in U_{R,r+s}$, there is some $z\in U_{R,r}(x)\cap E_s(y)$, which can be defined by using~\eqref{p + Z_p = (p + X_p)^s/(r+s) (p + Y_p)^r/(r+s)} for $p\le\lfloor R\rfloor$, and taking $z_n=y_n$ for $n>\lfloor R\rfloor$. Similarly, there is some $z'\in U_{R,r}(y)\cap E_s(x)$. Thus $(x,y)\in(U_{R,r}\circ E_s)\cap(E_s\circ U_{R,r})$.
\end{proof}

Note that the equality in~\eqref{U_R,r circ U_S,s subset ...} follows from Lemma~\ref{l:U_R,r circ E_s=E_s circ U_R,r=U_R,r+s}.

\begin{lemma}\label{l: U_T',t'+r(y) cap E_r+s(x) subset E_r(U_T,t(y) cap E_s(x))} 
	For all $T,r,s,t>0$ and $(x,y)\in E_s$, if $U_{T,t}(y)\subset U_{T,s}(x)$, then
        		\[
          		U_{T,t+r}(y)\cap E_{r+s}(x)=E_r(U_{T,t}(y)\cap E_s(x))\;.
        		\]
\end{lemma}

\begin{proof} 
	The inclusion ``$\supset$'' holds for all $t>0$ by Lemmas~\ref{l: E_r circ E_s = E_r+s} and~\ref{l:U_R,r circ E_s=E_s circ U_R,r=U_R,r+s}. 
	
	To prove ``$\supset$'', without lost of generality, we can assume that $T\in\Z^+$. Let $z\in U_{T,t+r}(y)\cap E_{r+s}(x)$. By Lemmas~\ref{l: E_r circ E_s = E_r+s} and~\ref{l:U_R,r circ E_s=E_s circ U_R,r=U_R,r+s}, there are elements $x'\in E_r(z)\cap E_s(x)$ and $y'\in E_r(z)\cap U_{T+1,t}(y)$. Take a sequence $0<\epsilon_p\downarrow0$ such that
		\begin{gather*}
			e^{\epsilon_p-s}(p+X_p),e^{\epsilon_p-r}(p+Z_p)<p+X'_p<e^{s-\epsilon_p}(p+X_p),e^{r-\epsilon_p}(p+Z_p)\;,\\
			e^{\epsilon_p+\epsilon_{p+1}}(p+X_p)<p+1+X_{p+1}\;,\\
			\epsilon_p+\epsilon_{p+1}<r,t\;,
		\end{gather*}
	for all $p$, and
		\[
			e^{\epsilon_p-t}(p+Y_p),e^{\epsilon_p-r}(p+Z_p)<p+Y'_p<e^{t-\epsilon_p}(p+Y_p),e^{r-\epsilon_p}(p+Z_p)\;,
		\]
	for $p\le T+1$. Then other elements, $x''\in E_r(z)\cap E_s(x)$ and $y''\in E_r(z)\cap U_{T,t}(y)$, can be defined by
		\begin{align*}
			p+X''_p&=\min\{e^{s-\epsilon_p}(p+X_p),e^{r-\epsilon_p}(p+Z_p)\}\;,\\
			p+Y''_p&=
				\begin{cases}
					\max\{e^{\epsilon_p-t}(p+Y_p),e^{\epsilon_p-r}(p+Z_p)\} & \text{if $p\le T$}\\
					p+Y''_{p-1}+z_p & \text{if $p>T$}\;.
				\end{cases}
		\end{align*}
	For $p\le T$, if
		\[
			p+1+X''_{p+1}=e^{s-\epsilon_{p+1}}(p+1+X_{p+1})\;,\quad p+Y''_p=e^{\epsilon_p-p}(p+Y_p)\;,
		\]
	then
		\[
			\frac{p+Y''_p}{p+1+X''_{p+1}}<\frac{e^{\epsilon_p-t}(p+X_p)}{e^{-\epsilon_{p+1}}(p+1+X_{p+1})}<e^{-t}<1\;.
		\]
	If
		\[
			p+1+X''_{p+1}=e^{r-\epsilon_{p+1}}(p+1+Z_{p+1})\;,\quad p+Y''_p=e^{\epsilon_p-r}(p+Z_p)\;,
		\]
	then
		\[
			\frac{p+Y''_p}{p+1+X''_{p+1}}<e^{\epsilon_p+\epsilon_{p+1}-2r}<1\;.
		\]
	If
		\[
			p+1+X''_{p+1}=e^{s-\epsilon_{p+1}}(p+1+X_{p+1})\;,\quad p+Y''_p=e^{\epsilon_p-r}(p+Z_p)\;,
		\]
	then
		\[
			\frac{p+Y''_p}{p+1+X''_{p+1}}<\frac{e^{\epsilon_p}(p+X_p)}{e^{-\epsilon_{p+1}}(p+1+X_{p+1})}<1\;.
		\]
	If
		\[
			p+1+X''_{p+1}=e^{r-\epsilon_{p+1}}(p+1+Z_{p+1})\;,\quad p+Y''_p=e^{\epsilon_p-t}(p+Y_p)\;,
		\]
	then
		\[
			\frac{p+Y''_p}{p+1+X''_{p+1}}<\frac{e^{\epsilon_p-t}(p+Z_p)}{e^{-\epsilon_{p+1}}(p+1+Z_{p+1})}
			<e^{\epsilon_p+\epsilon_{p+1}-t}<1\;.
		\]
	In any case, we have $p+Y''_p<p+1+X''_{p+1}$. Therefore an element $z'\in E_r(z)\cap U_{T,t}(y)\cap E_s(x)$ can be defined by
		\[
			p+Z'_p=
				\begin{cases}
					p+Y''_p & \text{if $p\le T$}\\
					p+X''_p & \text{if $p>T$}\;.
				\end{cases}
		\]
	 Hence $z\in E_r(U_{T,t}(y)\cap E_s(x))$.
\end{proof}

Hypothesis~\ref{h:type I}-(\ref{i: E_r(W) cap E_r(E_s(x)) subset E_r(V cap E_s(x))}) follows from Lemma~\ref{l: U_T',t'+r(y) cap E_r+s(x) subset E_r(U_T,t(y) cap E_s(x))} in the following way. Given $r,s>0$ and $(x,y)\in E_s$, and any neighborhood $V$ of $y$ in $\widetilde\Omega_+$. Since $V$ can be chosen as small as desired, we can assume that $V=U_{T,t}(y)$ for some $T,t>0$. From $y\in E_s(x)\subset U_{T,s}(x)$, we easily get some $t>0$ such that $U_{T,t}(y)\subset U_{T,s}(x)$. Then Lemmas~\ref{l: E_r circ E_s = E_r+s} and~\ref{l:U_R,r circ E_s=E_s circ U_R,r=U_R,r+s} and~\ref{l: U_T',t'+r(y) cap E_r+s(x) subset E_r(U_T,t(y) cap E_s(x))} give the inclusion of Hypothesis~\ref{h:type I}-\eqref{i: E_r(W) cap E_r(E_s(x)) subset E_r(V cap E_s(x))} with $W=V$.

Hypothesis~\ref{h:turbulent}-\eqref{i: more than one class} means that there are more than one weak equivalence class in $\widetilde\Omega_+$, which is indicated in Section~\ref{ss: Hector's ex}. Hypothesis~\ref{h:turbulent}-\eqref{i: U_R,r(x) cap E_s(y) neq emptyset},\eqref{i: the local equivalence classes are somewhere dense} are consequences of the following lemmas, completing the proof of Proposition~\ref{p: weak equivalence on Omega is generically ergodic}.

\begin{lemma}\label{l: U_R,r(x) cap E_s(y) neq emptyset}
  For all $R,r,s>0$ and $(x,y)\in U_{R,s}$, we have $U_{R,r}(x)\cap E_s(y)\neq\emptyset$.
\end{lemma}

\begin{proof} 
	 An element  $z\in U_{R,r}(x)\cap E_s(y)$ can be defined by
		\[
			z_n=
				\begin{cases}
					x_n & \text{if $n\le\lfloor R\rfloor$}\\
					y_n & \text{if $n>\lfloor R\rfloor$}\;.
				\end{cases}
		\]
\end{proof}

\begin{lemma}\label{l: d-path connected}
	For every \(R,r>0\) and every \(x\in\widetilde\Omega_+\), the set \(U_{R,r}(x)\cap E_d^{\widetilde\Omega_+}(x)\) is \(d\)-path connected.
\end{lemma}

\begin{proof} For every \(y\in U_{R,r}(x)\cap E_d^{\widetilde\Omega_+}(x)\),  a \(d\)-continuous path \(t\mapsto z(t)\) in
\(U_{R,r}(x)\cap E_d^{\widetilde\Omega_+}\) from \(y\) to \(x\) can be defined by
	\[
		p+Z_p(t)=(p+X_p)^t(p+Y_p)^{1-t}\;.
	\]
\end{proof}

\subsection{The theory of levels}\label{ss: levels}

For the reader's convenience, we briefly recall some concepts of this theory \cite{CantwellConlon1981}, which which are needed to understand the examples indicated in Section~\ref{ss: growth at finite level}.

Let $\FF$ be a foliation on a manifold $M$ satisfying the current conditions, and let $\LL$ be a $C^\infty$ foliation of dimension one transverse to $\FF$. For any saturated open connected subset $U\subset M$, the minimal sets of $\FF|_U$ are called \emph{local minimal sets}. \index{local minimal set} For every leaf $L\subset U$, $\ol L\cap U$ contains a nonzero finite number of minimal sets of $\FF|_U$. 

Any proper leaf is a local minimal set. A nonempty saturated open connected set $U$ is a local minimal set if $\FF|_U$ is minimal, which is called of \emph{locally dense type}. \index{local minimal set!locally dense type} All other local minimal sets are called of \emph{exceptional type}, \index{local minimal set!exceptional type} and meet the leaves of $\LL$ in a sets homeomorphic to open subsets of the Cantor set. Obviously, minimal sets are local minimal sets.

A minimal set (and each of its leaves) is said to be at \emph{level $0$}. \index{level} A local minimal set $X$ (and each of its leaves) is said to be at \emph{level $k\in\Z^+$} if the highest level of any local minimal set in $\ol X\sm X$ is $k-1$. The leaves that are not at finite level are said to be at \emph{infinite level}. For $k\in\N$, the union $M_k$ of leaves at levels $\le k$ is compact. For every leaf $L$, the set $\ol L\cap M_k$ is a nonempty finite union of local minimal sets. Every local minimal set $U$ is at finite level. 

If $L$ is at infinite level, then $\ol L\cap(M_k\sm M_{k-1})\ne\emptyset$ (taking $M_{-1}=\emptyset$). Let $X=\bigcup_k\ol L\cap M_k$ and $Z=\ol L\sm X$. Then $X$ is dense in $\ol L$, $Z$ is an uncountable union of leaves, each leaf in $Z$ is dense in $\ol L$, and no leaf of $Z$ has a proper side.

The \emph{substructure} \index{substructure} of a leaf $L$ is the union of leaves $F\subset\ol L$ with $\ol F\ne\ol L$. It is a union of local minimal sets, none of which are of locally dense type. If every leaf in $S(L)$ is proper, then $L$ is said to have a \emph{totally proper substructure}. \index{totally proper!substructure} If every leaf in $\ol L$ is proper, then $L$ is said to be \emph{totally proper}. \index{totally proper!leaf}

For instance, in Hector's example (Section~\ref{ss: Hector's ex}), the leaves corresponding to orbits in $\Omega$ are at infinite level, the leaf corresponding to every orbit $T_n$ ($n\in\Z^+$) is at level $n$, and the points $\pm1$ correspond to one compact leaf at level $0$.

\subsubsection{Growth of leaves at finite level}\label{ss: growth at finite level}

Let us recall some results and examples about the growth of leaves at finite level due to Cantwell and Conlon \cite{CantwellConlon1982} (see also \cite{Tsuchiya1979}, \cite{Tsuchiya1980}).

With the notation of Section~\ref{ss: levels}, let $L$ be a leaf of $\FF$ of non-exponential growth. Then $L$ has totally proper substructure \cite[Corollary~3.4]{CantwellConlon1982}. If moreover $L$ is semi-proper, then $L$ is totally proper and has exactly polynomial growth of degree equal to its level \cite[Corollaries~3.5 and~3.6]{CantwellConlon1982}. In this case, the transitive compact foliated space $X=\ol L$ obviously satisfies the alternative~(i) of Theorems~\ref{t: coarsely q.i. leaves 2} and~\ref{t:growth 2} ($L$ is its unique dense leaf).

Let $L$ be a non-exponential leaf of $\FF$ at finite level $k$. If $L$ does not have a proper side, then $L$ is in a local minimal set $U$ of locally dense type, $\FF|_U$ has trivial holonomy with the leaves mutually diffeomorphic and the same growth type, and $\ol U\sm U$ is a finite union of totally proper leaves whose maximum level is $k-1$ \cite[Theorem~3.7]{CantwellConlon1982}.

Cantwell and Conlon described a family $\bfG$ of growth types, which contains a continuum infinity of distinct quasi-polynomial but non-polynomial types, a continuum of distinct non-exponential but non-quasi-polynomial types, and the exponential growth type. As indicated in Chapter~\ref{c: intro}, they proved that, for all closed $3$-manifold and $\gamma\in\bfG$, there is a $C^\infty$ foliation $\FF$ in $M$ containing a local minimal set $U$ of locally dense type such that $\ol U\sm U$ is a finite union of totally proper leaves, and $\FF|_U$ has trivial holonomy with the leaves mutually diffeomorphic and growth type $\gamma$ \cite[Theorem~5.1]{CantwellConlon1982}. They also observed that the construction can be adapted to produce leaves of polynomial growth of any degree $\ge3$ in $U$.

In these cases, the diffeomorphisms between the leaves in $U$ is given by a $C^0$ foliated\footnote{The flow maps leaves to leaves.} flow, which is non-singular precisely on $U$, and with differentiable restrictions to the leaves. Therefore the leaves in $U$ are also mutually differentiable quasi-isometric. Then $\ol U$ is a transitive compact foliated space satisfying the alternatives Theorem~\ref{t: coarsely q.i. leaves 2}-\eqref{i: all leaves in X_0,d are equi-coarsely quasi-isometric  to each other} and Theorem~\ref{t:growth 2}-\eqref{i: all leaves in X_0,d have equi-equivalent growth}.

More results are proved by Cantwell and Conlon with the same interpretation in terms of Theorems~\ref{t: coarsely q.i. leaves 2} and~\ref{t:growth 2}  \cite[Theorems~5.5,~6.2,~6.13 and~7.1]{CantwellConlon1982}.

A similar theory of levels was considered by Lukina \cite{Lukina2012} for the graph matchbox manifold $\TT(F_n)$. But, concerning growth types of leaves, it is quite different from the case of  the results of codimension one; for instance, there is a totally proper leaf at level $1$ of exponential growth \cite[Theorem~1.10]{Lukina2012}.

\section{Open problems}\label{s: problems}

Consider the notation and general conditions of Chapter~\ref{c: intro}.

\begin{problem}\label{prob: L_x is differentiably quasi-isometric to L_y}
  Prove versions of Theorems~\ref{t: coarsely q.i. leaves 1}--\ref{t:
    coarsely q.i. leaves 3} for the relation ``$x\sim y$ if and only
  if $L_x$ is differentiably quasi-isometric to $L_y$'' on $X$,
  assuming that $\FF$ is differentiable.
\end{problem}

Problem~\ref{prob: L_x is differentiably quasi-isometric to L_y}
should not be difficult: the usual local Reeb stability and
Arzela-Ascoli theorems should be enough to adapt the proofs of
Theorems~\ref{t: coarsely q.i. leaves 1}--\ref{t: coarsely q.i. leaves
  3}.

\begin{problem}\label{prob: Cantor space of ends}
  Suppose that $\FF$ is minimal and residually many leaves have a
  Cantor space of ends. What can be said about the possible coarse
  quasi-isometry types of the leaves in $X_0$? Are they equi-coarsely
  quasi-isometric one another (the alternative~\eqref{i: all leaves in
    X_0,d are equi-coarsely quasi-isometric to each other} of Theorem~\ref{t: coarsely q.i. leaves 2})? Is it
  possible to characterize the possible coarsely
  quasi-isometric types that can be realized in this way? What can be
  said about the leaves with holonomy? The differentiable version of
  this problem can be also considered, specially for dimension two.
\end{problem}

As indicated in Chapter~\ref{c: intro}, the version of
Problem~\ref{prob: Cantor space of ends} for $2$-ended leaves was
solved by Blanc \cite{Blanc2003}. In the case of a Cantor space of
ends, one should try to prove, for the leaves in
$X_0$, a coarsely quasi-isometric version of Stallings' description of finitely generated
groups with a Cantor space of ends, using amalgamated free products or
HNN extensions \cite{Stallings1968} (see also
\cite[Theorem~8.32]{BridsonHaefliger1999}). A measure theoretic
version of such a description is given in
\cite[Theorem~D]{Ghys1995}. The case of $1$-ended leaves is much more
difficult.

\begin{problem}\label{prob: classification by countable models}
  Suppose that $\XF$ is transitive (or minimal). Consider the following equivalence relations on
  $X$:
  \begin{enumerate}[(a)]
		
  \item\label{i: coarsely quasi-isometric rel} ``$x\sim y$ if and only if $L_x$ is coarsely quasi-isometric
    to $L_y$.''
			
  \item\label{i: growth type rel} ``$x\sim y$ if and only if $L_x$ has the same growth type as
    $L_y$.''
			
  \item\label{i: diff quasi-isometric rel} ``$x\sim y$ if and only if $L_x$ is differentiably
    quasi-isometric to $L_y$'' (assuming that $\FF$ is
    differentiable).
			
  \end{enumerate}
  Assume also that $\XF$ satisfies the
  alternative~\eqref{i: there are uncountably many coarse
    quasi-isometry types of leaves} of Theorem~\ref{t: coarsely
    q.i. leaves 2} in the case of the relation~\eqref{i: coarsely quasi-isometric rel}, the
  alternative~\eqref{i: there are uncountably many growth types of leaves in X_0,d} of Theorem~\ref{t:growth 2} in the case of the relation~\eqref{i: growth type rel}, and a similar alternative in the case of the relation~\eqref{i: diff quasi-isometric rel}. With the terminology of Section~\ref{ss: generic ergodicity of rels}, is any of these relations generically ergodic with respect to the
  isomorphism relation on countable models? Is there any example where
  some of them is classifiable by countable models?
\end{problem}

As suggested by Hector, Problem~\ref{prob: classification by countable
  models} is especially interesting in the case of foliations of
codimension one, confronting it with their Poincar\'e-Bedixson theory
\cite{HectorHirsch1987-B}, \cite{CandelConlon2000-I}. The techniques from \cite{AlvarezCandel-turbulent} may be useful to address Problem~\ref{prob: classification by countable models} (see Section~\ref{ss: generic ergodicity in Hector's ex}).

\begin{problem}
  What can be said about the coarse cohomology and other coarse
  algebraic invariants \cite{Roe1993} of the generic leaf of a minimal
  foliated space?
\end{problem}

\begin{problem}
  It is not hard to show that the generic leaf of an exceptional
  minimal set of a codimension one \(C^2\)-foliation of a compact
  manifold has either one end or a Cantor set of ends. (This is false
  in class \(C^1\).)  A famous but unpublished result of Duminy (see
  Cantwell-Conlon~\cite{CC_Duminy} for the statement and proof) states
  that every semiproper leaf of an exceptional minimal set of a
  codimension one foliation of class \(C^2\) has a Cantor set of
  ends. The natural conjecture is that the generic leaf of an
  exceptional minimal set of a codimension one \(C^2\)-foliation of a
  compact manifold has a Cantor set of ends, cf.~\cite[Remark,
  page~131]{CandelConlon2003-II}. Cantwell and Conlon have proved this
  and much more for a large class of exceptional minimal sets~\cite{CantwellConlon1988}.
\end{problem}




\backmatter
\bibliographystyle{plain}

\printindex

\end{document}